\documentclass[a4paper,twoside,openright,11pt,titlepage]{book}

\usepackage[T1]{fontenc}
\usepackage{lmodern}
\usepackage[utf8]{inputenc}
\usepackage[english]{babel}
\usepackage{url}
\usepackage{hyperref}
\usepackage{xcolor,color}
\usepackage{geometry}
\usepackage[all]{xy}
\usepackage{enumitem}
\usepackage{booktabs}
\usepackage{slashed}
\usepackage{etoolbox}

\usepackage{makeidx}
\makeindex
\usepackage[nottoc,notlot,notlof]{tocbibind}
\usepackage[toc,page]{appendix}

\usepackage{soul}

\usepackage{amsmath}
\usepackage{amssymb,graphicx}
\usepackage{amsthm}
\usepackage{latexsym}
\usepackage{amsfonts}
\usepackage{bbm}
\usepackage{mathrsfs}

\usepackage{etex}
\usepackage{tikz}
\usetikzlibrary{arrows,cd,decorations.pathreplacing,decorations.markings,shapes.geometric}

\usepackage{hyphenat}
\theoremstyle{plain}
\newtheorem{lemma}{Lemma}[chapter]
\newtheorem{proposition}[lemma]{Proposition}
\newtheorem{proposition/definition}[lemma]{Proposition/Definition}
\newtheorem{theorem}[lemma]{Theorem}
\newtheorem{corollary}[lemma]{Corollary}

\newtheorem*{theorem*}{Theorem}
\newtheorem*{definition*}{Definition}

\theoremstyle{definition}
\newtheorem{definition}[lemma]{Definition}
\newtheorem{remark}[lemma]{Remark}
\newtheorem{example}[lemma]{Example}

\numberwithin{equation}{chapter}

\newcommand\abstractname{Abstract}  
\makeatletter
\if@titlepage
\newenvironment{abstract}{%
	\titlepage
	\null\vfil
	\@beginparpenalty\@lowpenalty
	\begin{center}%
		\bfseries \abstractname
		\@endparpenalty\@M
	\end{center}}%
	{\par\vfil\null\endtitlepage}
	\else
	\newenvironment{abstract}{%
		\if@twocolumn
		\section*{\abstractname}%
		\else
		\small
		\begin{center}%
			{\bfseries \abstractname\vspace{-.5em}\vspace{\z@}}%
		\end{center}%
		\quotation
		\fi}
	{\if@twocolumn\else\endquotation\fi}
	\fi
	\makeatother

\DeclareMathOperator{\id}{id}
\DeclareMathOperator{\im}{im}

\DeclareMathOperator{\der}{\textit D}
\DeclareMathOperator{\Der}{\mathscr{D}}
\DeclareMathOperator{\VDer}{V\frakX}

\DeclareMathOperator{\Diff}{\mathscr{D}}
\DeclareMathOperator{\End}{End}
\DeclareMathOperator{\Aut}{Aut}

\DeclareMathOperator{\MC}{MC}
\DeclareMathOperator{\BFV}{BFV}

\DeclareMathOperator{\BRST}{BRST}
\DeclareMathOperator{\Ham}{Ham}
\DeclareMathOperator{\Jac}{Jac}

\DeclareMathOperator{\weight}{\mathsf{weight}_{\nabla}}
\newcommand{\J}{J}
\newcommand{\vder}{\textnormal{V}\frakX}

\newcommand{\Mod}{\ \operatorname{mod}\ }

\newcommand{\wwedge}{ \wedge \dots \wedge}

\newcommand{\bfA}{\mathbf{A}}
\newcommand{\bfB}{\mathbf{B}}
\newcommand{\bfC}{\mathbf{C}}

\newcommand{\calA}{\mathcal{A}}

\newcommand{\calC}{\mathcal{C}}

\newcommand{\calE}{\mathcal{E}}
\newcommand{\calF}{\mathcal{F}}
\newcommand{\calG}{\mathcal{G}}

\newcommand{\calK}{\mathcal{K}}
\newcommand{\calL}{\mathcal{L}}

\newcommand{\calO}{\mathcal{O}}

\newcommand{\calS}{\mathcal{S}}
\newcommand{\calT}{\mathcal{T}}

\newcommand{\scrL}{\mathscr{L}}
\newcommand{\scrM}{\mathscr{M}}

\newcommand{\scrP}{\mathscr{P}}
\newcommand{\scrQ}{\mathscr{Q}}

\newcommand{\bbC}{\mathbb{C}}
\newcommand{\bbD}{\mathbb{D}}

\newcommand{\bbF}{\mathbb{F}}
\newcommand{\bbG}{\mathbb{G}}

\newcommand{\bbN}{\mathbb{N}}

\newcommand{\bbQ}{\mathbb{Q}}
\newcommand{\bbR}{\mathbb{R}}
\newcommand{\bbS}{\mathbb{S}}
\newcommand{\bbT}{\mathbb{T}}

\newcommand{\bbW}{\mathbb{W}}

\newcommand{\bbY}{\mathbb{Y}}
\newcommand{\bbZ}{\mathbb{Z}}

\newcommand{\frakX}{\mathfrak{X}}

\newcommand{\fraka}{\mathfrak{a}}

\newcommand{\frakg}{\mathfrak{g}}
\newcommand{\frakh}{\mathfrak{h}}

\newcommand{\frakl}{\mathfrak{l}}
\newcommand{\frakm}{\mathfrak{m}}
\newcommand{\frakn}{\mathfrak{n}}

\newcommand{\meanphi}{\phi}
\renewcommand{\phi}{\varphi}

\newcommand{\eps}{\varepsilon}
\renewcommand{\theta}{\vartheta}
\newcommand{\injects}{\hookrightarrow}

\renewcommand{\tilde}[1]{\widetilde{#1}}
\renewcommand{\hat}[1]{\widehat{#1}}
\renewcommand{\bar}[1]{\overline{#1}}
\newcommand{\ldsb}{[\![}
\newcommand{\rdsb}{]\!]}

\begin{document}

\pagestyle{empty}    
\begingroup

\frontmatter

\begin{titlepage}
%
%
%
%
\linespread{1.1}
\pagestyle{myheadings}

\thispagestyle{empty}
\newgeometry{top=3cm,bottom=2cm}
\begin{center}
	\vspace{-4cm}\includegraphics[width = 1\linewidth]{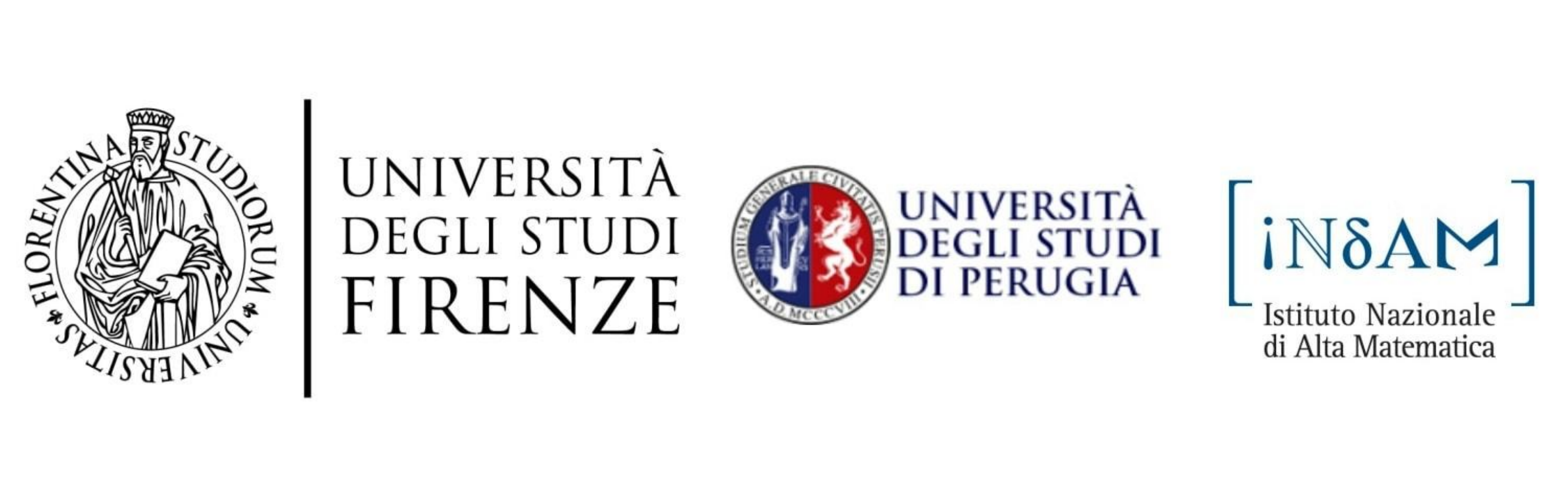}
	\par
	 { 
	  Universit\`a di Firenze, Universit\`a di Perugia, INdAM  consorziate nel {\small CIAFM}\\}
	 \vspace{1mm}
\vskip .8cm
	\par \vspace{2mm}
	\large
	\textbf{DOTTORATO DI RICERCA\\
	IN MATEMATICA, INFORMATICA, STATISTICA}\\
		 \vskip.2cm
		 CURRICULUM IN MATEMATICA\\
	CICLO XXIX
	\par \vspace{5mm}
	\large
	
		 
	 {\bf Sede amministrativa Universit\`a degli Studi di Firenze}\\
	Coordinatore Prof.~Graziano Gentili
	\par \vspace{8mm}
	\huge
	\textbf{Deformations of coisotropic\\submanifolds in Jacobi manifolds
	}
	\par \vspace{5mm}
	
	\large
	Settore Scientifico Disciplinare MAT/03
\end{center}
\par \vspace{10mm}

\normalsize
\hspace{1cm}\begin{minipage}{0.42\linewidth}
	\textbf{Dottorando}:
	\\
	{Alfonso Giuseppe Tortorella}
\end{minipage}
\hspace{2cm}
\begin{minipage}[t]{0.42\linewidth}
	\textbf{Tutore}
	\\
	{Dott.~Luca~Vitagliano}
	\vspace{14mm}
	\newline
	{Prof.~Paolo~de~Bartolomeis}
	
\end{minipage}
\par \vspace{10mm}

\begin{center}
	\begin{minipage}[t]{0.30\linewidth}
		\textbf{Coordinatore}
		\\
		{Prof.~Graziano Gentili}
	\end{minipage}
\end{center}
\par \vspace{9mm}
\begin{center}
	\hrule
	\par \vspace{5mm}
	Anni 2013/2016
	
\end{center}
\restoregeometry

\end{titlepage}

\cleardoublepage
\thispagestyle{empty}
\vspace*{\stretch{1}}
\begin{flushright}
	\itshape
	In memory of\\
	Paolo de Bartolomeis
\end{flushright}
\vspace{\stretch{3}}
\cleardoublepage


\begin{abstract}
%
%
	In this thesis, we investigate deformation theory and moduli theory of coisotropic submanifolds in Jacobi manifolds.
	Originally introduced by Kirillov as local Lie algebras with one dimensional fibers, Jacobi manifolds encompass, unifying and generalizing, locally conformal symplectic manifolds, locally conformal Poisson manifolds, and non-necessarily coorientable contact manifolds.
	
	We attach an $L_\infty$-algebra to any coisotropic submanifold in a Jacobi manifold.
	Our construction generalizes and unifies analogous constructions by Oh--Park (symplectic case), Cattaneo--Felder (Poisson case), and L\^e--Oh (locally conformal symplectic case).
	As a completely new case we also associate an $L_\infty$-algebra with any coisotropic submanifold in a contact manifold.
	The $L_\infty$-algebra of a coisotropic submanifold $S$ controls the \emph{formal} coisotropic deformation problem of $S$, even under Hamiltonian equivalence, and provides criteria both for the obstructedness and for the unobstructedness at the formal level.
	Additionally we prove that if a certain condition (``fiberwise entireness'') is satisfied then the algebra controls the \emph{non-formal} coisotropic deformation problem, even under Hamiltonian equivalence.
	
	We associate a BFV-complex with any coisotropic submanifold in a Jacobi manifold.
	Our construction extends an analogous construction by Sch\"atz in the Poisson setting, and in particular it also applies in the locally conformal symplectic/Poisson setting and the contact setting.
	Unlike the $L_\infty$-algebra, we prove that, with no need of any restrictive hypothesis, the BFV-complex of a coisotropic submanifold $S$ controls the \emph{non-formal} coisotropic deformation problem of $S$, even under both Hamiltonian equivalence and Jacobi equivalence.
	
	Notwithstanding the differences there is a close relation between the approaches to the coisotropic deformation problem via $L_\infty$-algebra and via BFV-complex.
	Indeed both the $L_\infty$-algebra and the BFV-complex of a coisotropic submanifold $S$ provide a cohomological reduction of $S$.
	Moreover they are $L_\infty$-quasi-isomorphic and so they encode equally well the moduli space of formal coisotropic deformations of $S$ under Hamiltonian equivalence.
	
	In addition we exhibit two examples of coisotropic submanifolds in the contact setting whose coisotropic deformation problem is obstructed at the formal level.
	Further we provide a conceptual explanation of this phenomenon both in terms of the $L_\infty$-algebra and in terms of the BFV-complex.
\end{abstract}


\chapter*{Acknowledgements}
\thispagestyle{empty}
I would not been able to complete the work on this thesis without the help and support received from all the people close to me during this journey.
I can only mention some of them.

First and foremost my sincere and deep gratitude goes to my adviser Luca Vitagliano.
Thank you, first for having introduced me, as a MSc student, to this beautiful field of differential geometry, and second for having proposed me this fascinating problem as PhD research topic.
For your trust in me, your generosity, and your patience I thank you again.
During these years you have constantly been an inspiring example and a firm point of reference not only from the scientific perspective but even from the human point of view.

I am grateful to Paolo de Bartolomeis for having accepted to be my ``co-tutor'', and this thesis is dedicated to him who suddenly passed away on the 29th of November 2016 leaving a deep scientific heritage.

I wish to thank also H\^ong V\^an L\^e and Yong-Geun Oh for having generously shared with me, from the very first stages of my PhD research, their knowledge of deformation and moduli theory of coisotropic submanifolds.

I am deeply indebted to Janusz Grabowski, who kindly accepted to be my mentor during the WCMCS PhD-internship at IMPAN: his helpful explanations and suggestions, and the pleasant atmosphere of his research group have decisively helped me to make progress with my PhD research.

I wish to thank Andrew J.~Bruce for having been always friendly both in talking about mathematics and in enjoying some free time together (in Warsaw, in Salerno and during several conferences).

I am also pleased to thank Florian Sch\"atz and Marco Zambon for useful comments and suggestions about the preprints~\cite{LOTV,LTV,tortorella2016rigidity} where we communicated large part of the results presented in this thesis.

A particular thanks goes to Graziano Gentili, the coordinator of the (Firenze--INdAM--Perugia) PhD program in Mathematics, Computer Science and Statistics: for the kindness and patience shown in his wisely guiding and advising all of us PhD students.

A special thanks is due to Giorgio Patrizio, for having accepted the r\^ole of ``referente interno'' for both me and my collegue Carlo: in this three years he has never missed to give us his encouragement, support and wise advise. 

I would also like to express my deep gratitude to the referees, Joana Margarida Nunes da Costa and Marco Zambon, for having kindly accepted to invest their time reading this thesis and providing their helpful suggestions and comments.

Thanks to my fellow PhD students: Agnese Baldisseri, Andrea Tamagnini, Carlo Collari, Davide Vanzo, Francesco Geraci, Gianluca Frasca-Caccia, Gioia Fioriti, Giovanni Zini, Mauro Maccioni, Kanishka Ariyapala, Majid Alamdari, Massimo Iuliani, and Sara Saldi.
Their youthful enthusiasm has contributed to create such a friendly atmosphere in the offices of viale Morgagni.

Thanks to Gianluca: he has been a great office-mate during my stay in Florence.
Sharing room T8 with him allowed me to balance the long working hours with interesting conversations with a kind friend.

Many thanks are due to all the members and the staff of the three mathematical institutions that hosted me while I was working on my PhD research: the Department of Mathematics and Computer Science ``Ulisse Dini'' of the University of Florence; the Department of Mathematics of the University of Salerno; the Institute of Mathematics of the Polish Academy of Science (IMPAN) in Warsaw.
I would like to thank, in particular, Sabina del Fonso from the secretary of the (Florence--INdAM--Perugia) PhD program in Mathematics, Computer Science and Statistics, and Marlena Nowi\'nska from the administration of the Mathematical Center for Mathematics and Computer Science (WCMCS) in Warsaw.

I would also like to thank my teachers at the University of Salerno, especially: Sergio de Filippo, Giovanni Sparano, and Alexandre Vinogradov.

Many thanks are due to my family for having always been close to me through some very difficult periods.

Finally, I gratefully acknowledge the financial support received from the  Gruppo Nazionale per le Strutture Algebriche, Geometriche e le loro Applicazioni (GNSAGA) of the Istituto Nazionale di Alta Matematica (INdAM) ``Francesco Severi'' during all three years of PhD studies, and from the Warsaw Center for Mathematics and Computer Science (WCMCS) during the semester I have spent at IMPAN.

\tableofcontents
\addtocontents{toc}{\protect\thispagestyle{empty}}
\addtocontents{toc}{\protect\thispagestyle{empty}}

\endgroup 

\cleardoublepage
\pagestyle{plain}      
\pagenumbering{arabic} 

\mainmatter

\chapter*{Introduction}
\label{chap:introduction}

\addcontentsline{toc}{chapter}{Introduction}

This thesis aims at studying deformation theory and moduli theory of coisotropic submanifolds in Jacobi manifolds.
Such aim is reached by means of two algebraic invariants we associate to each coisotropic submanifold $S$ of a Jacobi manifold, namely the $L_\infty[1]$-algebra and the BFV-complex of $S$.

In recent years, many authors have studied the coisotropic deformation problem in several geometric settings (symplectic, locally conformal symplectic and Poisson) which all share Jacobi geometry as a common extension.
However Jacobi geometry has been far less investigated than Poisson geometry.
Hence this thesis also contributes to a wider project which aims at filling the existing gap between Poisson and Jacobi geometry.

It is quite hard to overestimate the importance of Poisson, symplectic and contact manifolds both in geometry and in applications to physics.
Hence the main reason of interest in Jacobi structures is that they provide a combined generalization of (locally conformal) Poisson/symplectic structures and (non-necessarily coorientable) contact structures.
In this thesis a Jacobi manifold $(M,L,\{-,-\})$ is understood as a manifold $M$ equipped with a Jacobi bundle $(L,\{-,-\})$ over it. 
Here, following Marle~\cite{marle1991jacobi}, by Jacobi bundle over $M$ we mean a line bundle $L\to M$ equipped with a Jacobi structure $\{-,-\}$, i.e.~a Lie bracket on the module of its sections $\Gamma(L)$ which is a first order differential operator in each entry.
In this sense Jacobi manifolds are exactly the same thing as \emph{local Lie algebras with one-dimensional fibers} which were originally introduced and studied by Kirillov~\cite{kirillov1976local}.
Actually, two years later, and independently of Kirillov, Jacobi structures were rediscovered and deeply investigated by Lichnerowicz~\cite{Lich1978}.
However, on this regard, some comments about terminology are necessary.
On the one hand, Lichnerowicz's Jacobi structures are given by so-called \emph{Jacobi pairs}, and they are recovered as special case of above definition when the underlying line bundle is the trivial one.
On the other hand, Lichnerowicz's notion of locally conformal Jacobi structures coincides with our notion of Jacobi structures on non-necessarily trivial line bundles.
At a first sight it could seem that the line bundle approach à la Kirillov is an unnecessary generalization producing undue complications with respect to just work with Jacobi pairs.
On the contrary, as also shown in this thesis, the line bundle approach we adopt produces many simplifications and it is very satisfactory from conceptual point of view.
Indeed, within the Jacobi setting, geometric structures really live on the underlying line bundle and this gained geometrical insight gets obscured when we forget about such line bundle.

It is well-known that coisotropic submanifolds of symplectic and Poisson manifolds play a significant r\^ole both in geometry and in physics.
Let us mention here just some evidence of their relevance.
In many respects, coisotropic submanifolds are the most convenient way to generalize Lagrangian submanifolds from the symplectic case to the Poisson case~\cite{Weinstein1988}.
Graphs of Poisson maps are coisotropic submanifolds of product Poisson manifolds.
Coisotropic submanifolds naturally appear in symplectic and Poisson reduction.
In Hamiltonian mechanics, systems with gauge symmetries or Dirac (first class) constraints are represented by means of coisotropic submanifolds.
In topological field theory (cf.~\cite{cattaneo2004coisotropic}), coisotropic submanifolds of symplectic manifolds (resp.~Poisson manifolds) provide the D-branes for the A-model (resp.~Poisson sigma model).

Coisotropic submanifolds play a fundamental r\^ole in Jacobi and contact geometry as already in Poisson and symplectic geometry.
For instance graphs of Jacobi maps are coisotropic submanifolds in suitable product Jacobi manifolds.
Moreover coisotropic submanifolds naturally appear as zero level sets of  equivariant moment maps for Hamiltonian Lie group actions on Jacobi bundles.
Again, as in the Poisson setting, a coisotropic submanifold $S$ of a Jacobi manifold determines a Lie algebroid whose characteristic foliation allows to perform (singular) Jacobi reduction of $S$.

Note that Jacobi manifolds can be understood as homogeneous Poisson manifolds (of a special kind) via the ``Poissonization trick'' (see, e.g.~\cite{dazord1991structure, marle1991jacobi}, see also Remark~\ref{rem:poiss}).
However, not all coisotropic submanifolds in the Poissonization come from coisotropic submanifolds in the original Jacobi manifold.
On the other hand, if we regard a Poisson manifold as a Jacobi manifold, all its coisotropic submanifolds are coisotropic in the Jacobi sense as well.
In particular, the deformation problem of a coisotropic submanifold in a Jacobi manifold is genuinely more general than its analogue in the Poisson setting.

The first main result of the present thesis is that we attach an $L_\infty[1]$-algebra to any coisotropic submanifold $S$ in a Jacobi manifold, generalizing and unifying analogous constructions by Oh--Park~\cite{oh2005deformations} (symplectic case), Cattaneo--Felder\cite{cattaneo2007relative} (Poisson case), and L\^e--Oh~\cite{le2012deformations} (locally conformal symplectic case).
Our construction via higher derived brackets encompasses all the known cases as special cases and
reveals the prominent r\^ole of the Atiyah algebroid $\der L$ of a line bundle $L$.
In all previously known cases $L$ is a trivial line bundle while it is not necessarily so for general Jacobi manifolds.
As a new special case, our construction canonically applies to coisotropic submanifolds in any (non-necessarily coorientable) contact manifold.
The $L_\infty[1]$-algebra of $S$ can be seen, up to décalage, as a graded line bundle equipped with a higher homotopy Jacobi structure~\cite{brucekirillov2016}.
Further it provides a cohomological resolution of the (non-graded) Gerstenhaber-Jacobi algebra obtained from $S$ by singular Jacobi reduction.

We prove that the $L_\infty[1]$-algebra controls the formal coisotropic deformation problem of $S$, even under Hamiltonian equivalence.
This means that there exists a canonical one-to-one correspondence between formal coisotropic deformations of $S$ and formal coisotropic Maurer--Cartan elements of the $L_\infty[1]$-algebra.
Additionally such one-to-one correspondence intertwines Hamiltonian equivalence of formal coisotropic deformations with gauge equivalence of formal Maurer--Cartan elements.


For any coisotropic submanifold $S$ in a Jacobi manifold, its $L_\infty[1]$-algebra only depends on the infinite jet of the Jacobi structure along $S$.
If the Jacobi structure is fiber-wise entire along $S$, we prove that the $L_\infty[1]$-algebra also controls the non-formal coisotropic deformation problem of $S$, even under Hamiltonian equivalence.
However the fiber-wise entireness condition is not at all trivial, and so in general the $L_\infty[1]$-algebra fails to convey any information about non-formal coisotropic deformations of $S$.
This gap is filled up by means of another algebraic invariant of $S$, i.e.~its BFV-complex, whose construction is inspired by BRST and BFV formalisms.

The BRST formalism was originally introduced by Becchi, Rouet, Stora~\cite{Becchi1976} and Tyutin~\cite{Tyutin1975} as a method to deal, both on the classical and the quantum level, with physical systems possessing gauge symmetries or Dirac (first class) constraints.
The Hamiltonian counterpart of this formalism was developed by Batalin, Fradkin and Vilkovisky~\cite{Batalin1983,Batalin1977}.
It was soon realized that, for systems with finitely many degrees of freedom, the BFV-formalism is intimately related to symplectic and Poisson reduction~\cite{kostant1987symplectic}.
The construction of the underlying BFV-complex was recast in the context of homological perturbation theory by Henneaux and Stasheff~\cite{stasheff1988constrained,Stasheff1997}.
More recently, simplified versions (without ``ghosts of ghosts'') of the BFV-complex of a coisotropic submanifold have been constructed by Bordemann~\cite{Bordemann2000} (symplectic case), Herbig~\cite{herbig2007} and Sch\"atz~\cite{schatz2009bfv} (Poisson case).

The second main result of the present thesis is that we extend the construction of the BFV-complex of a coisotropic submanifold from the Poisson setting to the far more general Jacobi setting.
The resulting BFV-complex can be seen as a graded Jacobi bundle equipped with a cohomological Hamiltonian derivation.
Moreover it provides a cohomological resolution of the (non-graded) Gerstenhaber-Jacobi algebra obtained from $S$ by singular Jacobi reduction.
Our results are inspired by and encompass as special cases those of Herbig~\cite{herbig2007} and Sch\"atz~\cite{schatz2009bfv}.
However, we stress that we do not follow Sch\"atz in all our proofs.
In fact, we fully rely on the Homological Perturbation Lemma and a ``step-by-step obstruction'' method from homological perturbation theory rather than on homotopy transfer.
As a new case, our BFV-complex applies also to coisotropic submanifolds in contact and locally conformal symplectic/Poisson manifolds.

We prove that the BFV-complex controls the non-formal coisotropic deformation problem of $S$, even under both Hamiltonian equivalence and Jacobi equivalence.
We follow Sch\"atz and single out a special class of ``geometric'' Maurer-Cartan elements wrt the BFV-Jacobi bracket.
In this way we are able to establish a one-to-one correspondence between coisotropic deformations of $S$ and geometric Maurer--Cartan elements, modulo a certain equivalence.
Additionally, such one-to-one correspondence intertwines Hamiltonian/Jacobi equivalence of coisotropic deformations and Hamiltonian/Jacobi equivalence of geometric Maurer--Cartan elements.

The $L_\infty[1]$-algebra and the BFV-complex of a coisotropic submanifold $S$ in a Jacobi manifold are closely related.
Indeed, similarly as in the Poisson case~\cite{schatz2009bfv}, the $L_\infty[1]$-algebra of $S$ can be reconstructed starting from the BFV-complex by means of homotopy transfer along suitable contraction data.
As a consequence, the BFV-complex is $L_\infty$-quasi-isomorphic to the $L_\infty[1]$-algebra up to décalage.
So they encode equally well the moduli spaces of formal and infinitesimal coisotropic deformations under Hamiltonian equivalence.

We stress that we have taken seriously the issue of finding non-trivial examples where our theory applies.
In particular we have searched new examples in the contact setting where Oh--Park~\cite{oh2005deformations}, L\^e--Oh~\cite{le2012deformations} and Sch\"atz~\cite{schatz2009bfv} constructions do not apply.
Working in this direction, we have constructed in~\cite{tortorella2016rigidity} a first example of coisotropic submanifold in a contact manifold whose coisotropic deformation problem is formally obstructed.
In this thesis we provide a conceptual re-interpretation of this first obstructed example both in terms of the $L_\infty[1]$-algebra and in terms of the BFV-complex.
Actually this first obstructed example has been inspired by Zambon's example~\cite{zambon2008example} of an obstructed coisotropic submanifold in the symplectic setting.
Going on in the same direction, we have constructed a second example of an obstructed coisotropic submanifold in the contact setting (cf.~the revised version of~\cite{LOTV}).
Actually this second obstructed example has been inspired by an analogous one in the symplectic setting first considered by Oh--Park~\cite{oh2005deformations}, and later discussed in more details by Kieserman~\cite{kieserman2010liouville}.
In this thesis we provide a conceptual re-interpretation of this second obstructed example in terms of the associated $L_\infty[1]$-algebra.
We plan to analyze it in terms of the associated BFV-complex in the next future.

This thesis is organised as follows.
Chapter~\ref{chap:preliminaries} collects basic facts, including conventions and notations, about the algebraic and differential-geometric structures which are necessary to our presentation of Jacobi geometry.
This necessary machinery essentially consists of derivations of vector bundles, Jacobi algebroids and (graded) Gerstenhaber--Jacobi algebras.
Chapter~\ref{sec:abstract_jac_mfd} provides a self-contained introduction to Jacobi manifolds and their coisotropic submanifolds.
We present both classical well-known results and more specific ones concerning the purposes of the thesis.
In particular we attach important algebraic and geometric invariants to Jacobi manifolds and their coisotropic submanifolds which will be useful in studying the coisotropic deformation problem.
Our approach to Jacobi geometry, via Atiyah algebroids and first order multi-differential calculus on non-trivial line bundles, unifies and simplifies previous, analogous constructions for Poisson manifolds and locally conformal symplectic manifolds.
In Chapter~\ref{chap:L-infinity-algebra}, using results in Chapter~\ref{sec:abstract_jac_mfd}, we attach an $L_\infty[1]$-algebra to any closed coisotropic submanifold $S$ in a Jacobi manifold.
Then we prove that the $L_\infty[1]$-algebra controls the formal coisotropic deformation problem of $S$, even under Hamiltonian equivalence.
Further we point out that if the Jacobi structure is fiber-wise entire along $S$ then the associated $L_\infty[1]$-algebra controls also the non-formal coisotropic deformation problem of $S$, even under Hamiltonian equivalence.
Additionally, in the last section of Chapter~\ref{chap:L-infinity-algebra}, we study at the formal level the problem of simultaneously deforming both the Jacobi structure and the coisotropic submanifold.
In Chapter~\ref{chap:contact} we apply the theory developed in Chapter~\ref{chap:L-infinity-algebra} to the special class of Jacobi manifolds formed by contact manifolds.
In the contact setting we get pretty efficient formulas for the multi-brackets of the $L_\infty[1]$-algebra associated to a coisotropic submanifold.
These formulas are analogous to those of Oh--Park in the symplectic case~\cite{oh2005deformations} and L\^e--Oh in the locally conformal symplectic case~\cite{le2012deformations}.
Chapter~\ref{chap:graded_Jacobi_manifolds} is a direct continuation of the first two chapters, and sets the stage for the construction of the BFV-complex.
We introduce Jacobi structures on graded line bundles and also propose a suitable notion of lifting of Jacobi structures from a line bundle $L\to M$ to a certain graded line bundle $\hat L\to\hat M$.
Then we prove existence and uniqueness of such liftings extending to the Jacobi setting analogous results by Rothstein~\cite{Rothstein1991} (symplectic case), Herbig~\cite{herbig2007} and Sch\"atz~\cite{schatz2009bfv} (Poisson case).
In Chapter~\ref{chap:BFV_complex}, using results in Chapter~\ref{chap:graded_Jacobi_manifolds}, we attach a BFV-complex to any closed coisotropic submanifold $S$ in a Jacobi manifold.
Then we prove that the BFV-complex of $S$ controls the non-formal coisotropic deformation problem of $S$, even under Hamiltonian and Jacobi equivalence.
Additionally the BFV-complex also encodes the moduli spaces of infinitesimal and formal coisotropic deformations under Hamiltonian equivalence.
Finally, the thesis contains an appendix collecting auxiliary material for constructing and studying the BFV-complex.
In the first and second section we recall some tools from Homological Perturbation Theory, namely the Homological Perturbation Lemma and a ``step-by-step obstruction'' method.
The third and last section provides two very technical results necessary to prove some properties of the BFV-complex.

\chapter{Preliminaries}
\label{chap:preliminaries}


Here we collect basic facts, including conventions and notations, concerning the algebraic and differential-geometric machinery on which our presentation of Jacobi geometry is based.
 
In the first section we construct a functor from the category of vector bundles (with ``regular'' vector bundle morphisms) to the category of Lie algebroids, which associates each vector bundle $E$ with its Atiyah algebroid $\der E$.
The sections of $\der E$ are the so-called derivations of $E$, i.e.~a special kind of (linear) first order differential operators from $E$ to $E$.
Equivalently derivations of $E$ can also be seen as infinitesimal vector bundle automorphisms of $E$.

The Atiyah algebroid $DL$ of a line $L$ equipped with its tautological representation in $L$ represents a first example of Jacobi algebroid and it is denoted by $(DL,L)$.
Notice that our definition of Jacobi algebroid (Definition~\ref{def:jacb}) recovers as special cases what is called a Jacobi algebroid in~\cite{GM2001} and a Lie algebroid with a $1$-cocycle in~\cite{iglesias2000some}.
As we will see in the next chapter the r\^ole of Jacobi algebroids in Jacobi geometry is very similar to the one played by Lie algebroids in Poisson geometry.

There exists a canonical one-to-one correspondence between Jacobi algebroids and (graded) Gerstenhaber--Jacobi algebras (Proposition~\ref{prop:Jacobi_algbds_GJ-algbs}).
Our notion of Gerstenhaber--Jacobi algebra (Definition~\ref{definition:Gerstenhaber--Jacobi}) slightly generalizes the analogous notion as defined in~\cite{GM2001}.
In the setting of Jacobi geometry, and for the aims of this thesis, the most relevant example of Gerstenhaber--Jacobi algebra is formed by the multi-derivations of a line bundle $L$ equipped with the Schouten--Jacobi bracket.
The latter is exactly the Gerstenhaber--Jacobi algebra corresponding to the Jacobi algebroid $(DL,L)$.
Finally in Section~\ref{sec:GJ_algb_multi-derivations} we provide explicit formulas for the Schouten--Jacobi bracket of multi-derivations of a line bundle.

\section{Derivations and infinitesimal automorphisms of vector bundles}
\label{sec:app_0}

Let $M$ be a smooth manifold and let $E,F$ be vector bundles over $M$.
Recall (see, e.g.,~\cite{N2003}) that a (linear) $k$-th order differential operator from $E$ to $F$ is an $\bbR$-linear map $\Delta : \Gamma (E) \to \Gamma (F)$ such that
\[
[[ \ldots [[\Delta, a_0], a_1] \ldots ] , a_k] = 0
\]
for all $a_0, a_1, \ldots, a_k \in C^\infty (M)$, where we interpret the functions $a_i$ as operators (multiplication by $a_i$). 
There is a natural isomorphism between the $C^\infty (M)$-module $\mathrm{Diff}_k (E,F)$ of $k$-th order differential operators from $E$ to $F$ and the $C^\infty (M)$-module of sections of the vector bundle $\mathrm{diff}_k (E,F) := \mathrm{Hom} (J^k E , F)$, where $J^k E$ is the vector bundle of $k$-jets of sections of $E$.
The isomorphism $\Gamma (\mathrm{Hom} (J^k E , F)) \simeq \mathrm{Diff}_k (E,F) $ is given by $\phi \mapsto \phi \circ j^k$, where $\phi : \Gamma (J^k E) \to \Gamma (F)$ is a $C^\infty (M)$-linear map, and $j^k : \Gamma (E) \to \Gamma (J^k E)$ is the $k$-th jet prolongation.
Throughout this thesis, the trivial line bundle over $M$ is denoted by $\bbR_M:=M\times\bbR\to M$, then, in particular, $\mathrm{diff}_k (E, \bbR_M)$ is the dual bundle of $J^k E$ (see, e.g.,~\cite[Chapter 11]{N2003} for more details).
From now on we often denote $J_1 E := \mathrm{diff}_1 (E, \bbR_M) = (J^1 E)^\ast$.

Let $\Delta : \Gamma (E) \to \Gamma (F)$ be a $k$-th order differential operator.
The correspondence
\[
(a_1, \ldots , a_k) \longmapsto [[ \ldots [\Delta, a_1] \ldots ] , a_k],
\]
$a_1, \ldots, a_k \in C^\infty (M)$, is a well-defined symmetric, $k$-multi-derivation of the algebra $C^\infty (M)$ with values in $C^\infty (M)$-linear maps $\Gamma (E) \to \Gamma (F)$.
In other words, it is a section of the vector bundle $S^k TM \otimes \mathrm{Hom} (E , F)$, called the \emph{symbol of $\Delta$} and denoted by $\sigma(\Delta)$, or equivalently, in a more compact way, by $\sigma_\Delta$.
The symbol map $\sigma : \Delta \mapsto \sigma_\Delta$ sits in a short exact sequence
\begin{equation}\label{eq:seq_co-Spenc}
0 \longrightarrow \mathrm{Diff}_{k-1} (E,F) \longrightarrow \mathrm{Diff}_{k} (E,F) \overset{\sigma}{\longrightarrow} \Gamma (S^k TM \otimes \mathrm{Hom} (E , F)) \longrightarrow 0,
\end{equation}
of $C^\infty (M)$-modules.
Note that sequence~\eqref{eq:seq_co-Spenc} can be also obtained applying the contravariant functor $\mathrm{Hom ( - , \Gamma (F))}$ to the \emph{Spencer sequence}
\[
0 \longleftarrow \Gamma (J^{k-1} E) \longleftarrow \Gamma (J^k E) \overset{\gamma}{\longleftarrow} \Gamma (S^k T^\ast M \otimes E) \longleftarrow 0,
\]
where the inclusion $\gamma$, sometimes called the \emph{co-symbol}, is given by
\[
da_1 \cdot \cdots \cdot da_k \otimes e \longmapsto [[ \cdots [j^k, a_1] \cdots ] , a_k] e,
\]
$a_1, \ldots, a_k \in C^\infty (M)$, and $e \in \Gamma (E)$.

Now we focus on first order differential operators.
In general, there is no natural $C^\infty (M)$-linear splitting of the Spencer sequence
\begin{equation}
	\label{eq:Spencer}
	0 \longleftarrow \Gamma (E) \longleftarrow \Gamma (J^1 E) \overset{\gamma}{\longleftarrow} \Gamma (T^\ast M \otimes E) \longleftarrow 0.
\end{equation}
However, Sequence~\eqref{eq:Spencer} splits via the first order differential operator $j^1 : \Gamma (E) \to \Gamma (J^1 E)$.
In particular, $\Gamma (J^1 E) = \Gamma (E) \oplus \Gamma (T^\ast M \otimes E)$, and any section $\alpha$ of $J^1 E$ can be uniquely written as $\alpha = j^1 \lambda + \gamma (\eta)$, for some $\lambda \in \Gamma (E)$, and $\eta \in \Gamma (T^\ast M \otimes E)$.

Now, let $\Delta : \Gamma (E) \to \Gamma (E)$ be a first order differential operator.
The symbol of $\Delta$ is \emph{scalar-type} if it is of the kind $X \otimes \mathrm{id}_{\Gamma (E)}$ for some (necessarily unique) vector field $X$.
In other words $\Delta (fe) = X(f) e + f\Delta e$, for all $f \in C^\infty (M)$, and $e \in \Gamma (E)$.
In this case we identify $\sigma_\Delta$ with $X$, and call $\Delta$ a \emph{derivation} of the vector bundle $E$ (over the vector field $X$).
The space of derivations of $E$ will be denoted by $\Der E$.
It is the space of sections of a (transitive) Lie algebroid $\der  E \to M$ over $M$, sometimes called the \emph{Atiyah algebroid of} $E$, whose Lie bracket is the commutator of derivations, and whose anchor is the symbol $\sigma : \der  E \to TM$ (see, e.g., \cite[Theorem 1.4]{KM2002} for details).
The fiber $\der_x E$ of $\der E$ through $x \in M$ consists of $\bbR$-linear maps $\delta : \Gamma (E) \to E_x$ such that there exists a, necessarily unique, tangent vector $\xi \in T_x M$, called the \emph{symbol of $\delta$} and also denoted by $\sigma(\delta)$, satisfying the obvious Leibniz rule $\delta (f e) = \xi (f) e (x) + f (x) \delta (e),$ for all $f \in C^\infty (M)$ and $e \in \Gamma (E)$.

\begin{remark}
	\label{rem:derivations_of_trivial_line_bundle}
	If $E$ is a line bundle, then every first order differential operator $\Gamma (E) \to \Gamma (E)$ is a derivation of $E$.
	For the trivial line bundle $\bbR_M$, we have $\Gamma (\bbR_M) = C^\infty (M)$.
	Then first order differential operators $\Gamma (\bbR_M) \to \Gamma (\bbR_M)$ or, equivalently, derivations of $\bbR_M$, are the operators of the form $X + a : C^\infty (M) \to C^\infty (M)$, where $X$ is a vector field on $M$ and $a \in C^\infty (M)$ is interpreted as an operator (multiplication by $a$).
	Accordingly, in this case, there is a natural direct sum decomposition of $C^\infty(M)$-modules $\Der \bbR_M = \mathfrak{X} (M) \oplus C^\infty (M)$, the projection $\Der  \bbR_M \to C^\infty (M)$ being given by $\Delta \mapsto \Delta 1$.
\end{remark}

The construction of the Atiyah algebroid of a vector bundle is functorial, in the following sense.
Let $\phi : E \to F$ be a morphism of vector bundles $E \to M$, $F \to N$, over a smooth map $\underline{\phi} : M \to N$.
We assume that $\phi$ is \emph{regular}, in the sense that it is an isomorphism when restricted to fibers.
In particular a section $f$ of $F$ can be pulled-back to a section $\phi^\ast f$ of $E$, defined by $(\phi^\ast f)(x) := (\phi|_{E_x}^{-1} \circ f \circ \underline{\phi}) (x)$, for all $x \in M$.
Then $\phi$ induces a morphism of Lie algebroids $\der  \phi : \der E \to \der F$ defined by
\[
((\der  \phi) \delta) f := \delta (\phi^\ast f), \quad \delta \in \der E, \quad f \in \Gamma (F).
\]
We also denote $\phi_\ast := \der  \phi$.

Derivations of a vector bundle $E$ can be also understood as infinitesimal automorphisms of $E$ as follows.
First of all, a derivation $\Delta$ of $E$ determines a derivation $\Delta^\ast$ of the dual bundle $E^\ast$, with the same symbol as $\Delta$.
Derivation $\Delta^\ast$ is defined by
$
\Delta^\ast \varphi := \sigma_\Delta\circ \varphi - \varphi \circ \Delta
$,
where $\varphi : \Gamma (E) \to C^\infty (M)$ is a $C^\infty (M)$-linear map, i.e.~a section of $E^\ast$.
Now, recall that an automorphism of $E$ is a fiber-wise linear, bijective bundle map $\phi : E \to E$.
In particular, $\phi$ is a regular morphism (see above) and it covers a (unique) diffeomorphism $\underline{\phi} : M \to M$.
An \emph{infinitesimal automorphism} of $E$ is a vector field $Y$ on $E$ whose flow consists of (local) automorphisms.
In particular, $Y$ projects onto a (unique) vector field $\underline{Y} \in \mathfrak{X} (M)$.
Infinitesimal automorphisms of $E$ are sections of a (transitive) Lie algebroid over $M$, whose Lie-bracket is the commutator of vector fields on $E$, and whose anchor is $Y \mapsto \underline{Y}$.
It can be proven that a vector field $Y$ on $E$ is an infinitesimal automorphism iff it preserves fiber-wise linear functions on $E$, i.e.~sections of the dual bundle $E^\ast$.
Finally, note that the restriction of an infinitesimal automorphism to fiber-wise linear functions $Y|_{\Gamma (E^\ast)} : \Gamma (E^\ast) \to \Gamma (E^\ast)$ is a derivation of $E^\ast$, and the correspondence $Y \mapsto Y|^\ast_{\Gamma (E^\ast)}$ is a well-defined isomorphism between the Lie algebroid of infinitesimal automorphisms and the Atiyah algebroid of $E$.

If $\Delta$ is a derivation of $E$, $Y$ is the corresponding infinitesimal automorphism, and $\{ \phi_t \}$ is its flow, then we will also say that $\Delta$ \emph{generates the flow of automorphisms $\{ \phi_t\}$}.
We have
\[
\left. \frac{d}{dt}\right|_{t = 0} \phi_t^\ast e = \Delta e,
\]
for all $e \in \Gamma (E)$.
Similarly, if $\{ \Delta_t \}$ is a smooth one parameter family of derivations of $E$, $\{ Y_t \}$ is the corresponding one parameter family of infinitesimal automorphisms, and $\{ \psi_t \}$ is the associated one parameter family of automorphisms, then we will say that $\{ \Delta_t \}$ \emph{generates} $\{ \psi_t \}$.
We have
\[
\frac{d}{dt} \psi_t^\ast e = (\psi_t^\ast \circ \Delta_t) e .
\]

\begin{remark}
	\label{rem:derivation_along_morphism}
	Let us recall, for the reader's convenience, that, if $M_i$ is a (graded) module over a (graded) algebra $A_i$, $i=0,1$, then a graded module morphism $\phi:M_0\to M_1$, covering a graded algebra morphism $\underline{\smash{\phi}}:A_0\to A_1$, is a linear map $\phi:M_0\to M_1$ such that $\phi(am)=\underline{\smash{\phi}}(a)\phi(m)$, for all $a\in A_0$, $m\in M_0$.
	For future use we remark here also what follows.
	Let $\phi:M_0\to M_1$ is a degree $0$ graded module morphism covering a degree $0$ algebra morphism $\underline{\smash{\phi}}:A_0\to A_1$.
	A \emph{degree $k$ graded derivation covering $\underline{\smash{\phi}}$} is a degree $k$ graded linear map $X:A_0\to A_1$ such that $X(aa')=X(a)\underline{\smash{\phi}}(a')+(-)^{k|a|}\underline{\smash{\phi}}(a)X(a')$ for all homogeneous $a,a'\in A_0$.
	Additionally, \emph{a degree $k$ graded derivation covering $\phi$}, with symbol $X$, is a degree $k$ graded linear map $\square:M_0\to M_1$ such that $\square(am)=X(a)\phi(m)+(-)^{k|a|}\underline{\smash{\phi}}(a)\square(m)$, for all homogeneous $a\in A_0,m\in M_0$.
\end{remark}

\section{Jacobi algebroids and Gerstenhaber-Jacobi algebras}
\label{sec:app}


Let $A$ be a Lie algebroid with anchor $\rho$ and Lie bracket $[-,-]_A$.
Recall that a representation of $A$ in a vector bundle $E \to M$ is a \emph{flat $A$-connection in $E$}, i.e.~a Lie algebroid morphism $\nabla : A \to \der E$, written $\alpha \mapsto \nabla_\alpha$, with values in the Atiyah algebroid $\der  E$ of $E$.
In other words $\nabla$ is an $\bbR$-linear map $\Gamma (A) \to \mathrm{Diff}_1 (E,E)$, where $\mathrm{Diff}_1 (E,E)$ is the module of first order differential operators $\Gamma (E) \to \Gamma (E)$, such that
\[
\begin{aligned}
\nabla_{f\alpha} e & = f \nabla_\alpha e, \\
\nabla_\alpha (f e) &= \rho(\alpha)(f) e + f \nabla_\alpha e, \\
[\nabla_\alpha , \nabla_\beta ] e & = \nabla_{[\alpha, \beta]_A} e,
\end{aligned}
\]
for all $\alpha, \beta \in \Gamma (A)$, $f \in C^\infty (M)$, and $e \in \Gamma (E)$.
The next definition introduce the notion of Jacobi algebroid, which will play a relevant r\^ole in Jacobi geometry.
Notice that this notion, well-suited for our aims, is slightly more general than the definition of Jacobi algebroid as proposed in~\cite{GM2001}, or the equivalent notion of \emph{Lie algebroid with a $1$-cocycle}~\cite{iglesias2000some} (for more details, see Remark~\ref{rem:Jacobi_algbd}).
\begin{definition}
	\label{def:jacb}
	A \emph{Jacobi algebroid} is a pair $(A,L)$ where $A\to M$ is a Lie algebroid, with Lie bracket $[-,-]_A$ and anchor map $\rho$, and $L\to M$ is a line bundle equipped with a representation $\nabla$ of $A$.
\end{definition}

\begin{remark}
	Definition~\ref{def:jacb} of Jacobi algebroids is equivalent to Grabowski's \emph{Kirillov algebroids}~\cite[Section 8]{grabowski2013graded}.
\end{remark}

Let $(A,L)$ be a Jacobi algebroid.
Denote by $(\Gamma (\wedge^\bullet A^\ast),d_A)$ the de Rham complex of the Lie algebroid $A$.
The representation $\nabla$ of $A$ in $L$ defines a degree one differential on sections of $\wedge^\bullet A^\ast \otimes L$, denoted by $d_{A,L}$, and uniquely determined by
\begin{equation}\label{eq:d_A,L}
\begin{aligned}
d_{A,L} \lambda & = \nabla \lambda, \\
d_{A,L} (\omega \wedge \Omega) &= d_A \omega \wedge \Omega + (-)^{| \omega |}\omega \wedge d_{A,L} \Omega,
\end{aligned}
\end{equation}
for all $\lambda \in \Gamma (L)$, and all homogeneous $\omega \in \Gamma (\wedge^\bullet A^*)$ and $\Omega \in \Gamma (\wedge^\bullet A^\ast \otimes L)$, where we used the obvious $\Gamma (\wedge^\bullet A^\ast)$-module structure on $\Gamma (\wedge^\bullet A^\ast \otimes L)$.
Here, as throughout this thesis, we denoted by $|v|$ the degree of an homogeneous element $v$ of a graded vector space.
Throughout this thesis, the complex $(\Gamma (\wedge^\bullet A^\ast \otimes L), d_{A,L})$ will be called \emph{de Rham complex of the Jacobi algebroid $(A,L)$}.
Similarly its cohomology, denoted by $H(A,L)$, will be called \emph{de Rham cohomology of of the Jacobi algebroid $(A,L)$}.

Proposition~\ref{prop:Jacobi_algebroids_and_deRham_complex} below points out that the datum of a Jacobi algebroid structure on $(A,L)$ is fully encoded in its de Rham complex $(\Gamma(\wedge^\bullet A^\ast\otimes L), d_{A,L})$.
\begin{proposition}[\cite{vaintrob1997lie}]
	\label{prop:Jacobi_algebroids_and_deRham_complex}
	Let $A\to M$ be vector bundle and denote by $\scrM$ the graded manifold with algebra of functions $C^\infty(\scrM)=\Gamma(\wedge^\bullet A^\ast)$. Let $L\to M$ be line bundle, and denote by $\scrL\to\scrM$ the graded line bundle with $C^\infty(\scrM)$-module of sections $\Gamma(\scrL)=\Gamma(\wedge^\bullet A^\ast\otimes L)$.
	Then a canonical one-to-one correspondence between
	\begin{itemize}
		\item Jacobi algebroid structures on $(A,L)$, and
		\item \emph{cohomological derivations} $\Delta\in\Der\scrL$, i.e.~$|\Delta|=1$ and $[\Delta,\Delta]=2\Delta\circ\Delta=0$, (with symbol \emph{cohomological vector fields} $X\in\frakX(\scrM)$, i.e.~$|X|=1$ and $[X,X]=2X\circ X=0$),
	\end{itemize}
	is established by the following relation: $d_{A,L}=\Delta$ (and $d_A=X$).
\end{proposition}
Indeed, as a standard result in the theory of Lie algebroids~\cite{crainic2011lectures,mackenzie2005general}, Proposition~\ref{prop:Jacobi_algebroids_and_deRham_complex} continues to hold, mutatis mutandis, even after having replaced $L\to M$ with an arbitrary vector bundle $E\to M$.

\begin{remark}
	\label{rem:Jacobi_algbd}
	In the case $L = \bbR_M$, a representation of $A$ in $L$ is the same as a $1$-cocycle in the de Rham complex $(\Gamma (\wedge^\bullet A^\ast) , d_A)$ of $A$.
	Namely, in this case $\Gamma (\wedge^\bullet A^\ast \otimes L) = \Gamma (\wedge^\bullet A^\ast)$ and, in view of~\eqref{eq:d_A,L} $\nabla$ is completely determined by $\omega_\nabla := d_{A,L} 1 \in \Gamma (A^\ast)$.
	It is easy to see that $\omega_\nabla$ is a $d_A$-cocycle, i.e.~$d_A \omega_\nabla = 0$.
	Conversely, a $d_A$-cocycle $\omega \in \Gamma (A^\ast)$ determines a unique representation $\nabla$ in $L = \bbR_M$ such that $d_{A,L} 1 = \omega$.
	This shows that, in the case $L = \bbR_M$, Definition~\ref{def:jacb} recovers the definition of Lie algebroid with a $1$-cocycle proposed in~\cite{iglesias2000some}, which is in turn equivalent to the definition of Jacobi algebroid proposed in~\cite{GM2001}.
\end{remark}

\subsection{The der-complex of a line bundle}
\label{sec:der_complex}

Let $M$ be a manifold, and let $L\to M$ be a line bundle.
In Section~\ref{sec:app_0} we have introduced the Atiyah algebroid $\der L$ of $L\to M$.
Further the \emph{tautological representation} $\nabla$ of $\der L$ in $L$ is defined by $\nabla_\square\lambda=\square\lambda$, in other words it is given by the Lie algebroid morphism $\id:\der L\to\der L$.
In this way $(\der L,L)$ becomes a Jacobi algebroid.
The associated de Rham complex $(\Gamma(\wedge^\bullet(\der L)^\ast\otimes L),d_{\der L,L})$ of $\der L$ with values in $L$ is also known as the \emph{der-complex} of the line bundle $L\to M$ (cf.~\cite{rubtsov1980cohomology}).
Following~\cite{vitagliano2015dirac}, we denote by $d_D$ the de Rham differential $d_{\der L,L}$, and by $\Omega^\bullet_L$ the graded module $\Gamma(\wedge^\bullet(\der L)^\ast\otimes L)$ over the graded algebra $\Gamma(\wedge^\bullet(\der L)^\ast)$.
Moreover the elements of $\Omega_L^\bullet$ are called the \emph{$L$-valued Atiyah forms}.

\begin{remark}
	Notice that $\Omega_L^0=\Gamma(L)$, and $\Omega_L^1=\Gamma((\der L)^\ast\otimes L)$ identifies with $\Gamma(J^1L)$ by means of the $L$-valued duality pairing between $\der L$ and $J^1L$.
	In view of this, $d_D:\Omega_L^0\to\Omega_L^1,\ \lambda\mapsto d_D\lambda$, agrees with the first jet prolongation $j^1:\Gamma(L)\to\Gamma(J^1L),\ \lambda\mapsto j^1\lambda$, i.e.~$\langle d_D\lambda,\square\rangle=\langle\square, j^1\lambda\rangle$ for all $\square\in\Der L$.
\end{remark}

As it is well-known in the theory of Lie algebroids (cf., e.g.,~\cite{crainic2011lectures,mackenzie2005general}), the datum of a Lie algebroid structure on $A\to M$ and a representation of $A$ in $E\to M$ determines a Cartan calculus on $A$ with values in $E$.
We focus here only on a particular instance of this result: that one which is the most relevant for the purposes of this thesis.
Specifically, for any line bundle $L\to M$, the Jacobi algebroid $(\der L,L)$ determines an $L$-valued \emph{Cartan calculus} on $\der L$ whose structural operations are the following:
\begin{itemize}
	\item the de Rham differential $d_D:\Omega_L^\bullet\to\Omega_L^\bullet$,
\end{itemize}
and, for every $\square\in\Der L$,
\begin{itemize}
	\item the \emph{contraction} $\iota_\square:\Omega_L^\bullet\to\Omega_L^\bullet$, i.e.~the degree $-1$ derivation of the graded module $\Omega_L^\bullet$, with symbol a degree $-1$ derivation of the graded algebra $\Gamma(\wedge^\bullet(\der L)^\ast)$ again denoted by $\iota_\square$, which is uniquely determined by $\iota_\square\omega=\omega(\square)$, for all $\omega\in\Omega^1_L$,
	\item the \emph{Lie derivative} $\calL_\square:\Omega_L^\bullet\to\Omega_L^\bullet$, i.e.~the degree $0$ derivation of the graded module $\Omega_L^\bullet$, with symbol a degree $0$ derivation of the graded algebra $\Gamma(\wedge^\bullet(\der L)^\ast)$ again denoted by $\calL_\square$, which is uniquely determined by
	\begin{align*}
	\calL_\square\lambda=\square \lambda,\qquad
	(\calL_\square\omega)(\Delta)=\square(\omega(\Delta))-\omega([\square,\Delta]),
	\end{align*}
	for all $\lambda\in\Omega^0_L=\Gamma(L)$, $\omega\in\Omega^1_L$, and $\Delta\in\Der L$.
\end{itemize}
Additionally these structural operations satisfy the following relations:
\begin{equation*}
\calL_\square=\left[d_D,\iota_\square\right],\qquad
[\calL_\square,\iota_\Delta]=\iota_{[\square,\Delta]},\qquad
\left[\calL_\square,\calL_\Delta\right]=\calL_{\left[\square,\Delta\right]},\qquad \left[i_\square,i_\Delta\right]=0,
\end{equation*}
for all $\square,\Delta\in\Der L$, where $[-,-]$ denotes the graded commutator.

\begin{remark}
	Der-complex is always acyclic.
	Specifically there exists a canonical contracting homotopy for $(\Omega_L^\bullet,d_D)$ provided by the contraction $\iota_{\mathbbm{1}}$, where $\mathbbm{1}:\Gamma(L)\to\Gamma(L)$ denotes the derivation of $L$ given by the identity map, i.e.~$\mathbbm{1}\lambda=\lambda$.
\end{remark}

Although we do not need it in this thesis, notice that the der-complex, and the corresponding Cartan calculus, continue to be well-defined, mutatis mutandis, even though we replace the line bundle $L\to M$ with an arbitrary vector bundle $E\to M$.
For more details see, e.g.,~\cite{rubtsov1980cohomology} and~\cite{vitagliano2015dirac}.

\color{black}
\section{Gerstenhaber--Jacobi algebras}
\label{sec:GJ_algb}

Let $A \to M$ be a vector bundle, and let $L \to M$ be a line bundle.
Consider vector bundle $A_L := A \otimes L^\ast$.
The parallel between Lie algebroid structures on $A$ and Gerstenhaber algebra structures on the associated Grassmann algebra $\Gamma (\wedge^\bullet A)$ is well-known (see e.g.~\cite{KS1995}, \cite[Theorem 3]{GM2001}).
There is an analogous parallel between Jacobi algebroid structures on $(A,L)$ and Gerstenhaber-Jacobi algebra structures on $(\Gamma (\wedge^\bullet A_L), \Gamma (\wedge^\bullet A_L \otimes L)[1])$.


\begin{definition}
	\label{definition:Gerstenhaber--Jacobi}
	A \emph{Gerstenhaber--Jacobi algebra} is given by a graded commutative, (associative) unital algebra $\calA$, a graded $ \mathcal A$-module $\calL $, and, moreover, a graded Lie bracket $[-,-]$ on $\calL $ and an action by derivations, $\lambda\mapsto X_\lambda$, of $\calL$ on $\calA$ such that
	\begin{equation}\label{eq:genleib2}
	[\lambda,a\mu]=X_\lambda(a)\mu+(-)^{|\lambda||a|}a[\lambda,\mu],
	\end{equation}
	for all homogeneous $a\in\calA$, and $\lambda, \mu\in\calL$.
	In particular $[\lambda,-]$ is a degree $|\lambda|$ graded first order differential operator with scalar-type symbol $X_\lambda$.
\end{definition}

In the following if a structure of Gerstenhaber--Jacobi algebra is given on the pair $(\calA,\calL)$, then we will also say that $\calL$ is a Gerstenhaber--Jacobi algebra over $\calA$. 
\begin{remark}
	\label{remark:Gerstenhaber--Jacobi1}
	In the case $\calL = \calA[1]$, Definition~\ref{definition:Gerstenhaber--Jacobi} recovers the notion of Gerstenhaber--Jacobi algebra as defined in~\cite{GM2001}.
	If $\calL$ is a faithful $\calA$-module, then condition $X_{[\lambda,\mu]}=[X_\lambda,X_\mu]$, for any $\lambda,\mu\in\calL $, in Definition~\ref{definition:Gerstenhaber--Jacobi}, is redundant.
\end{remark}

\begin{example}
	\label{ex:Gerstenhaber--Jacobi}
	Gerstenhaber--Jacobi algebras encompass several well-known notions, here we just mention some of them.
	\begin{enumerate}[label=(\alph*)]
		\item
		\label{enumitem:ex:Gerstenhaber--Jacobi_1}
		Let $M$ be a manifold, and $L\to M$ be a line bundle.
		A Gerstenhaber--Jacobi algebra structure on $(\calA,\calL)$, with $\calA=C^\infty(M)$ and $\calL=\Gamma(L)$, is the same as a Jacobi structure $\{-,-\})$ on $L\to M$ (cf.~Definition~\ref{def:Jacobi_structure}).
		\item
		\label{enumitem:ex:Gerstenhaber--Jacobi_2}
		A (graded) Jacobi algebra is the same as a Gerstenhaber--Jacobi algebra with $\calL =\calA$,
		\item
		\label{enumitem:ex:Gerstenhaber--Jacobi_3}
		A Gerstenhaber algebra is the same as a Gerstenhaber--Jacobi algebra with $\calL =\calA[1]$ and $X_a=[ a,-]$, for all $a\in \calA$,
		\item
		\label{enumitem:ex:Gerstenhaber--Jacobi_4}
		A graded Lie--Rinehart algebra (see, e.g.,~\cite{huebschmann2004lie} and~\cite{vitagliano2015homotopy}) is the same thing as a Gerstenhaber--Jacobi algebra such that $\lambda\mapsto X_\lambda$ is $\calA$-linear.
	\end{enumerate}
\end{example}


Proposition~\ref{prop:Jacobi_algbds_GJ-algbs} below clarifies the parallel between Jacobi algebroid structures and Gerstenhaber--Jacobi algebra structures.

\begin{proposition}[see also~{\cite[Theorem 5]{GM2001}}]
	\label{prop:Jacobi_algbds_GJ-algbs}
	Let $A \to M$ be a vector bundle, and $L \to M$ be a line bundle.
	Set $A_L := A \otimes L^\ast\to M$, and consider the graded commutative unital algebra $\calA:=\Gamma ( \wedge^\bullet A_L)$ and the graded $\calA$-module $\calL:=\Gamma ( \wedge^\bullet A_L \otimes L)[1]$.
	Then there is a canonical one-to-one correspondence between:
	\begin{itemize}
		\item Jacobi algebroid structures $([-,-]_A,\rho,\nabla)$ on $(A,L)$, and
		\item Gerstenhaber--Jacobi algebra structures $([-,-],X_{(-)})$ on $(\calA,\calL)$,
	\end{itemize}
	Such one-to-one correspondence is canonically established by the following relations:
	\begin{equation}
	\label{eq:GJ_jaca}
	\begin{aligned}
	\rho (\alpha) (f) & = X_\alpha (f), \\
	[\alpha, \beta]_A & = [\alpha, \beta], \\
	\nabla_\alpha \lambda & = [\alpha, \lambda],
	\end{aligned}
	\end{equation}
	for all $\alpha, \beta \in \calL^0 = \Gamma (A)$, $f \in C^\infty (M) = \calA^0$, and $\lambda \in\calL^{-1}=\Gamma (L)$.
\end{proposition}

\begin{proof}
	Note preliminarily that a structure of Gerstenhaber--Jacobi algebra on $(\calA,\calL):=(\Gamma (\wedge^\bullet A_L), \Gamma (\wedge^\bullet A_L \otimes L)[1])$ is completely determined by
	\begin{enumerate}
		\item the action of degree zero elements of $\calL$ on degree zero elements of $\calA$,
		\item the Lie bracket between degree zero elements of $\calL$, and
		\item the Lie bracket between degree zero elements and degree $-1$ elements of $\calL$.
	\end{enumerate}

	Assume that $(A,L)$ possesses a structure of Jacobi algebroid with Lie bracket $[-,-]_A$, anchor map $\rho$, and representation $\nabla$.
	Then a structure of Gerstenhaber--Jacobi algebra on $(\calA,\calL)$ is obtained as prolongation to the higher degree terms of the operations defined by reading Equations~\eqref{eq:GJ_jaca} from the right to the left.
	
	Conversely, assume that $(\calA,\calL)$ possesses the structure of Gerstenhaber--Jacobi algebra, with graded Lie bracket $[-,-]$ and action of $\calL$ on $\calA$ written $\alpha \mapsto X_\alpha$.
	Read Equations~\eqref{eq:GJ_jaca} from the left to the right to define a Lie bracket, an anchor map, and a flat connection. 
	It is immediate to see that the above operations form a well-defined Jacobi algebroid structure on $(A,L)$.
\end{proof}

\subsection{The Gerstenhaber--Jacobi algebra of multiderivations of a line bundle}
\label{sec:GJ_algb_multi-derivations}

%
%

In Jacobi geometry, as we will see in the next chapters of this thesis, a central r\^ole is  played by the Gerstenhaber--Jacobi algebra of multi-derivations of a line bundle $L\to M$.
In view of Proposition~\ref{prop:Jacobi_algbds_GJ-algbs}, the latter can be exactly seen as the Gerstenhaber--Jacobi algebra corresponding to the Jacobi algebroid $(A,L)$ given by the Atiyah algebroid $\der L$ of $L$ equipped with the tautological representation $\id_{\der L}$ in $L$ (see Proposition~\ref{prop:Jacobi_Gerstenhaber_multi-differential}).
This section describes in detail the structure maps of this most relevant Gerstenhaber--Jacobi algebra.

Let $A = \der  L$ be the Atiyah algebroid of a line bundle $L\to M$.
In this case, $\Gamma (\wedge^\bullet A_L) = \Gamma (\wedge^\bullet J_1 L)$ and it consists of skew-symmetric, first order multi-differential operators from $\Gamma(L)$ to $C^\infty(M)$, i.e.~skew-symmetric $\bbR$-multi-linear maps which are first order differential operators in each entry.
In view of this we often denote $\operatorname{Diff}_1^\bullet (L,\bbR_M):=\Gamma(\wedge^\bullet J_1L)$, where $\operatorname{Diff}_1^0(L,\bbR_M):=C^\infty(M)$ and $\operatorname{Diff}_1^1(L,\bbR_M):=\operatorname{Diff}_1 (L,\bbR_M)$.
Let $\Delta\in\Gamma(\wedge^k J_1 L)$, and $\Delta^\prime \in \Gamma (\wedge^{k^\prime} J_1 L)$.
If we interpret $\Delta$ and $\Delta^\prime$ as multi-differential operators, then their exterior product is given by
\begin{equation}
\label{eq:multi-differential_operators}
(\Delta\wedge\Delta^\prime)(\lambda_1,\ldots,\lambda_{k+k^\prime})
= \sum_{\tau\in S_{k,k^\prime}}(-)^\tau\Delta(\lambda_{\tau(1)},\ldots,\lambda_{\tau(k)})\Delta^\prime(\lambda_{\tau(k+1)},\ldots,\lambda_{\tau(k+k^\prime)})
\end{equation}
where $\lambda_1,\ldots,\lambda_{k+k^\prime}\in\Gamma(L)$, and $S_{k, k^\prime}$ denotes the set of $(k,k^\prime)$-unshuffles.
Similarly, $\Gamma (\wedge^\bullet A_L \otimes L) = \Gamma (\wedge^\bullet J_1 L \otimes L)$ and it consists of skew-symmetric, multi-derivations of $L\to M$, i.e.~skew-symmetric $\bbR$-multi-linear maps from $\Gamma(L)$ to itself which are first order differential operators, hence derivations, in each entry.
For this reason we often denote $\Der^\bullet L := \Gamma (\wedge^\bullet J_1 L \otimes L)$, where $\Der^0 L := \Gamma (L)$ and $\Der^1 L := \Der (L)$.
Beware that an element of $\Der^k L$ is a multi-derivation with $k$-entries but its degree in $(\Der^\bullet L)[1]$ is $k-1$.
The $\Gamma (\wedge^\bullet J_1 L)$-module structure on $(\Der^\bullet L)[1]$ is given by the same formula~\eqref{eq:multi-differential_operators} as above.

\begin{remark}
	A Jacobi bracket $\{-,-\}$ on $L$ (see Chapter~\ref{sec:abstract_jac_mfd}) will be also interpreted as an element $J$ of $\Der^2 L$.
\end{remark}

\begin{proposition}
	\label{prop:Jacobi_Gerstenhaber_multi-differential}
	For every line bundle $L\to M$, there is a natural Gerstenhaber--Jacobi algebra structure $(\ldsb-,-\rdsb,X_{(-)})$ on $(\Gamma (\wedge^\bullet J_1 L), (\Der^\bullet L)[1])$, uniquely determined by
	\begin{equation}
	\ldsb \square,\square'\rdsb=[\square,\square'],\qquad
	\ldsb \square,\lambda\rdsb=\square(\lambda),\qquad
	\ldsb \lambda,\mu\rdsb=0,
	\end{equation}
	for all $\square,\square'\in\Der^1 L=\Der L$, and $\lambda,\mu\in\Der^0 L=\Gamma(L)$.
	The Lie bracket $\ldsb-,-\rdsb$ is called \emph{the Schouten--Jacobi bracket}.
\end{proposition}

\begin{proof}
	Since the Atiyah algebroid of a line bundle, equipped with its tautological representation, is a Jacobi algebroid, the proposition is a straightforward consequence of Proposition~\ref{prop:Jacobi_algbds_GJ-algbs}.
\end{proof}

Finally, we describe explicitly the structure of Gerstenhaber--Jacobi algebra existing on the pair $(\Gamma (\wedge^\bullet J_1 L), (\Der^\bullet L)[1])$.
The Lie bracket $\ldsb-,-\rdsb$ on $(\Der^\bullet L)[1]$ is called Schouten-Jacobi bracket because it can be seen as a ``Jacobi version'' of the Schouten--Nijenhuis bracket between multi-vector fields.
It is easy to see that
\begin{equation}
	\label{eq:SJ_bracket_and_Gerstenhaber_product}
	\ldsb\square,\square'\rdsb:=(-)^{kk'}\square\circ\square^\prime-\square'\circ\square,
\end{equation}
where $\square \in \Der^{k+1} L$, $\square^\prime \in \Der^{k^\prime +1} L$, and $\square \circ \square^\prime$ is given by the following ``\emph{Gerstenhaber product}'':
\begin{equation}
	\label{eq:Gerstenhaber_product}
	(\square \circ \square^\prime) (\lambda_1, \ldots, \lambda_{k+k^\prime +1})\!
	=\!\!\! \sum_{\tau \in S_{k^\prime+1, k}}\!\!(-)^{\tau} \square (\square^\prime (\lambda_{\tau(1)},\ldots, \lambda_{\tau(k^\prime+1)}) ,\lambda_{\tau (k^\prime+2)}, \ldots, \lambda_{\tau (k+k^\prime +1)}),
\end{equation}
where $\lambda_1,\ldots, \lambda_{k+k^\prime +1} \in \Gamma (L)$.

A direct computation shows that the action $\square \mapsto X_\square$ of $(\Der^\bullet L)[1]$ on $ \Gamma (\wedge^\bullet J_1 L)$ is defined as follows.
For $\square \in \Der^{k+1} L$, the \emph{symbol} of $\square $, denoted by $\sigma_\square \in \Gamma (TM \otimes \wedge^k J_1 L )$, is, by definition, the $\wedge^k J_1 L$-valued vector field on $M$ implicitly defined by:
\[
\sigma_\square (f)(\lambda_1 , \ldots, \lambda_{k}) \lambda := \square (f \lambda, \lambda_1,\ldots,\lambda_{k}) - f \square ( \lambda, \lambda_1,\ldots,\lambda_{k}),
\]
where $f\in C^\infty (M)$.
Finally, for any $\Delta \in \Gamma (\wedge^l J_1 L)$, and $\square \in \Der^{k+1} L$, section $X_\square(\Delta)\in\Gamma (\wedge^{k+l} J_1 L)$ is given by
\begin{equation}
\begin{aligned}
X_\square (\Delta) (\lambda_1, \ldots, \lambda_{k+l})
:= {}& (-)^{k(l-1)}\sum_{\tau \in S_{l,k}}(-)^\tau\sigma_\square (\Delta (\lambda_{\tau(1)}, \ldots, \lambda_{\tau (l)}))(\lambda_{\tau(l+1)}, \ldots, \lambda_{\tau (k+l)})\\
& - \sum_{\tau\in S_{k+1,l-1}}(-)^\tau \Delta (\square (\lambda_{\tau(1)}, \ldots, \lambda_{\tau(k+1)}), \lambda_{\tau(k+2)}, \ldots, \lambda_{\tau (k+l)}).
\end{aligned}
\end{equation}


\begin{remark}
	\label{rem:SJ_and_SN}
	Let $M$ be a manifold.
	Recall that we denoted by $\bbR_M$ the trivial line bundle $M\times\bbR\to M$.
	The tangent Lie algebroid $TM\to M$ becomes a Jacobi algebroid when equipped with the representation in $\bbR_M\to M$ provided by its anchor map.
	Then the Gerstenhaber--Jacobi algebra corresponding to $TM$ according with Proposition~\ref{prop:Jacobi_algbds_GJ-algbs} reduces to the Gerstenhaber algebra $\mathfrak{X}^\bullet (M) = \Gamma (\wedge^\bullet TM)$ of (skew-symmetric) multi-vector fields on $M$.
	Here the associative product in $\frakX^\bullet(M)$ is the exterior product given by formulas similar to~\eqref{eq:multi-differential_operators}, and the Lie bracket on $\frakX^\bullet(M)[1]$ is the Schouten-Nijenhuis bracket of multi-vector fields given, in terms of a certain ``\emph{Gerstenhaber product}'', by formulas similar to~\eqref{eq:SJ_bracket_and_Gerstenhaber_product} and~\eqref{eq:Gerstenhaber_product}.
	
	In view of Remark~\ref{rem:derivations_of_trivial_line_bundle}, the canonical splitting of vector bundles $J^1\bbR_M\simeq T^\ast M\oplus\bbR_M$ induces a $C^\infty(M)$-module isomorphism between $\Der^\bullet(\bbR_M)=\Gamma(\wedge^\bullet(J^1\bbR_M)^\ast)$ and $\Gamma(\wedge^\bullet(TM\oplus\bbR_M))=\frakX^\bullet(M)\oplus\frakX^\bullet(M)[1]$.
	Such isomorphism identifies any $\square\in\Der^{k+1}\bbR_M$ with the pair $(P,Q)\in\frakX^{k+1}(M)\oplus\frakX^k(M)$ which is uniquely determined by the condition
	\begin{equation}
	\square=P-Q\wedge\id,
	\end{equation}
	where $\id$ denotes the identity map of $\Gamma(\bbR_M)=C^\infty(M)$, or more explicitly by
	\begin{equation*}
	\square(f_1,\ldots,f_{k+1})=P(f_1,\ldots,f_{k+1})+\sum_{i=1}^{k+1}(-)^{k-i}Q(f_1,\ldots,\hat{f_i},\ldots,f_{k+1})f_i,
	\end{equation*}
	where $f_1,\ldots,f_{k+1}\in\Gamma(\bbR_M)=C^\infty(M)$, and a hat ``$\widehat{-}$'' denotes omission.
	Understanding the $C^\infty(M)$-module isomorphism $\Der^\bullet\bbR_M\simeq\frakX^\bullet(M)\oplus\frakX^\bullet(M)[1]$, through a straightforward computation, one can re-express the Schouten--Jacobi bracket of multi-derivations of $\bbR_M\to M$ in terms of the exterior product and the Schouten--Nijenhuis bracket of multi-vector fields on $M$.
	Namely, for all $(P,Q)\in\frakX^{k+1}(M)\oplus\frakX^k(M)$ and $(P',Q')\in\frakX^{k'+1}(M)\oplus\frakX^{k'}(M)$, we have that
	\begin{align}
	\label{eq:rem:SJ_and_SN}
	\ldsb P-Q\wedge\id,P'-Q'\wedge\id\rdsb&=\left(\ldsb P,P'\rdsb-(-)^{k'}kP\wedge Q'+k'Q\wedge P'\right)\\
	&\phantom{=}-\left(\ldsb P,Q'\rdsb+(-)^{k'}\ldsb Q,P'\rdsb-(k-k')Q\wedge Q'\right)\wedge\id.\nonumber
	\end{align}
	In this way, the Gerstenhaber algebra $\frakX^\bullet(M)$ of multi-vector fields on $M$ identifies with the subalgebra of $\Der^\bullet(\bbR_M)$ formed by those multi-derivations $\square$ of $\bbR_M\to M$ such that $\square(1,-,\ldots,-)=0$.
\end{remark}
\chapter{Jacobi manifolds and coisotropic submanifolds}
\label{sec:abstract_jac_mfd}

	In order to set the general background of the thesis, we provide in this chapter an elementary introduction to Jacobi geometry.
	We present both classical well-known results and more specific results which are necessary to the development, in the next chapters, of the deformation and moduli theory of coisotropic submanifolds.

In the first part of this chapter we recall the definitions of Jacobi structures on a line bundle $L\to M$ and morphisms of Jacobi manifolds, and present important examples of them (Definitions~\ref{def:Jacobi_structure} and~\ref{def:Jacobi_mfd_imorphism}, and Examples~\ref{ex:1}).
Our primary sources are~\cite{kirillov1976local}, \cite{Lich1978}, \cite{marle1991jacobi}, \cite{GM2001}, and the recent paper by Crainic and Salazar~\cite{crainic2013jacobi} whose line bundle approach \emph{à la} Kirillov we adopt.
Accordingly, we recover Lichnerowicz's notion of Jacobi pairs as special case of the Jacobi structures when $L$ is the trivial line bundle $\bbR_M$.
Generically non-trivial line bundles and first order multi-differential calculus on them (cf.~Chapter~\ref{chap:preliminaries}) play a prominent r\^ole in Jacobi geometry.
We propose some equivalent characterizations of the Jacobi structures on a line bundle $L\to M$: as local Lie brackets on $\Gamma(L)$ (see Proposition~\ref{prop:local_Lie_algebras}) and as Maurer--Cartan (MC) elements of $((\Der^\bullet L)[1],\ldsb-,-\rdsb)$ (see Proposition~\ref{prop:J_as_MC_element}).
In particular the latter allows to construct the Lichnerowicz--Jacobi and Chevalley--Eilenberg cohomologies of a Jacobi manifolds $(M,L,\{-,-\})$.

We also associate important algebraic and geometric structures with Jacobi manifolds: in doing this we stress the strong relationship existing between Jacobi manifolds and Jacobi algebroids.
On the one hand, a Jacobi algebroid structure on $(J^1 L,L)$ is canonically associated with every Jacobi structure on a line bundle $L\to M$ (see Proposition~\ref{prop:associated_Jacobi_algbd}), as first discovered by Kerbrat and Souici-Benhammadi~\cite{kerbrat1993} in the special case of trivial line bundle $L=\bbR_M$ (see~\cite{crainic2013jacobi} for the general case).
On the other hand, we discuss the existence, for any Jacobi algebroid structure on $(A,L)$, of a fiber-wise linear Jacobi structure on $\pi^\ast L\to A_L{}^*$ with $\pi:A_L{}^\ast:=A^\ast\otimes L\to M$ (Proposition~\ref{prop:lijac}), as first discovered by Iglesias and Marrero~\cite{iglesias2000some} in the special case $L=\bbR_M$.
As a consequence, these constructions provide a natural lifting of a Jacobi structure $\{-,-\}$ on a line bundle $L\to M$ to a Jacobi structure on $\pi^\ast L\to\der L$ where $\pi:\der  L\to M$ is the bundle map of the Atiyah algebroid of $L$ (Example~\ref{ex:lijac1}).

Further, after having recalled the definition of the characteristic distribution of a Jacobi manifold, we state and provide the proofs of two classical results in Jacobi geometry, first discovered by Kirillov~\cite{kirillov1976local}.
The first one is the complete characterization of transitive Jacobi manifolds $(M,L,\{-,-\})$ according to which if $\dim M$ is odd then they are nothing but contact manifolds (see Proposition~\ref{prop:odd-dim_transitive}) and if $\dim M$ is even then they are nothing more than locally conformal symplectic (lcs) manifolds (see Proposition~\ref{prop:even-dim_transitive}).
The second one is the foliation theorem according to which every Jacobi manifold is essentially a union of contact manifolds and lcs manifolds glued together in a smooth manner (see Theorem~\ref{theor:characteristi_foliation}).

In the second part of this chapter we recall the definition of coisotropic submanifold $S\subset M$ of a Jacobi manifold $(M,L,J)$ which encompasses, unifying and generalizing, the analogous notions that are specific to the several instances of Jacobi manifolds (like the contact, symplectic and Poisson ones).
We propose some equivalent characterizations of coisotropic submanifolds (Lemma~\ref{lem:cois} and Corollary~\ref{cor:cois}~\ref{enumitem:cor:cois_3}).
We also associate important geometric and algebraic invariants with coisotropic submanifolds.
In particular we establish a one-to-one correspondence between coisotropic submanifolds of a Jacobi manifold $(M,L,J)$ and a certain kind of Jacobi subalgebroids of the associated Jacobi algebroid $(J^1L,L)$ (Proposition~\ref{prop:conormal}).
The latter also yields a way how to introduce the characteristic distribution of a coisotropic submanifold.
Finally, in view of their natural r\^ole in the reduction of Jacobi manifolds, we associate a reduced Gerstenhaber--Jacobi algebra with every coisotropic submanifold (see Section~\ref{sec:Jacobi_reduction}).
These geometric and algebraic invariants of a coisotropic submanifold $S$ will be central to the construction of both its two further invariants (the $L_\infty[1]$-algebra in Chapter~\ref{chap:L-infinity-algebra} and the BFV-complex in Chapter~\ref{chap:BFV_complex}) which will turn out to control, at different levels, the coisotropic deformation problem of $S$. 

\section{Jacobi manifolds}
\label{sec:a_jac_mfd}

Let $M$ be a smooth manifold.

\begin{definition}
	\label{def:Jacobi_structure}
	\index{Jacobi structure}
	A \emph{Jacobi structure}, or a \emph{Jacobi bracket}, on a line bundle $L\to M$ is a Lie bracket $\{-,-\} :\Gamma(L)\times\Gamma (L)\rightarrow \Gamma(L)$, which is a first order differential operator, hence a derivation, in both entries.
	A \emph{Jacobi bundle} (over $M$) is a line bundle (over $M$) equipped with a Jacobi bracket.
	A \emph{Jacobi manifold} is a manifold equipped with a Jacobi bundle over it.
\end{definition}


\begin{remark}
	\label{rem:GJ_algebras_of_Jacobi_manifolds}
	A Jacobi structure on a line bundle $L\to M$ is exactly the same thing as a structure of Gerstenhaber--Jacobi algebra (concentrated in degree $0$) on $(C^\infty(M),\Gamma(L))$ (see Definition~\ref{definition:Gerstenhaber--Jacobi}).
	In the case of the trivial line bundle $L=\bbR_M$, we have $\Gamma(L)=C^\infty(M)$.
	In this case we speak about a \emph{Jacobi algebra} structure on $C^\infty(M)$ (see~\cite{GM2002}).
\end{remark}


The following proposition immediately leads to an equivalent description of the Jacobi structures on a line bundle.

\begin{proposition}
	\label{prop:local_Lie_algebras}
	Let $L\to M$ be a line bundle, and $\{-,-\}$ be a Lie algebra structure on $\Gamma(L)$.
	The following two conditions are equivalent.
	\begin{enumerate}[label=(\arabic*)]
		\item
		\label{enumitem:prop:local_Lie_algebras_1}
		The Lie bracket $\{-,-\}$ is a fist-order differential operator (hence a derivation) in both entries, i.e.~it is a Jacobi bracket.
		\item
		\label{enumitem:prop:local_Lie_algebras_2}
		The Lie bracket $\{-,-\}$ is \emph{local} in the sense that, for all $\lambda,\mu\in\Gamma(L)$,
		\begin{equation*}
		\operatorname{supp}\left(\{\lambda,\mu\}\right)\subset\operatorname{supp}(\lambda)\cap\operatorname{supp}(\mu).
		\end{equation*}
	\end{enumerate}
\end{proposition}

\begin{proof}
	The implication $(1)\Longrightarrow(2)$ is pretty obvious.
	As pointed out by Kirillov~\cite[Lemmas 1 and 2]{kirillov1976local}, the converse implication $(2)\Longrightarrow(1)$ follows from the combined use of Petree's Theorem (see, e.g.,~\cite[Section 19]{Kolar1993natural}) and the Jacobi identity of $\{-,-\}$.
\end{proof}

\begin{remark}
	\label{rem:local_Lie_algebras}
	As originally introduced by Kirillov~\cite{kirillov1976local}, a structure of \emph{local Lie algebra} on a line bundle $L\to M$ consists of a Lie bracket $\{-,-\}$ on $\Gamma(L)$ which is local in the sense of condition~\ref{enumitem:prop:local_Lie_algebras_2} above.
	Proposition~\ref{prop:local_Lie_algebras} tells us that a Jacobi structure on $L\to M$ is exactly the same as a local Lie algebra structure on $L\to M$.
	Hence our definition of Jacobi manifolds in terms of Jacobi structures (Definition~\ref{def:Jacobi_structure}) can also be equivalently rephrased in terms of local Lie algebras (see, e.g., \cite[Definition 3.1]{crainic2013jacobi}).
\end{remark}

\begin{example}
	\label{ex:1}
	Jacobi structures encompass, unifying and generalizing, several well-known geometric structures.
	Here are some examples.
	\begin{enumerate}
		\item Poisson structures on a manifold $M$ are exactly Jacobi structures $\{-,-\}$ on the trivial line bundle $\bbR_M:=M\times\bbR\to M$ such that $\{1,-\}=0$.
		\item Jacobi pairs on a manifold $M$, as introduced by Lichnerowicz~\cite{Lich1978} (see also Definition~\ref{def:Jacobi_pair}), are nothing but Jacobi structures on $\bbR_M\to M$ (cf.~Proposition~\ref{prop:Jacobi_pairs}).
		\item For any (non-necessarily coorientable) contact manifold $(M,C)$, the (non-nec\-es\-sar\-i\-ly trivial) line bundle $TM/C$ is naturally equipped with a Jacobi structure (for details see Section~\ref{subsec:odd-dim_transitive}).
		\item For any locally conformal symplectic (lcs) manifold $(M,L,\nabla,\omega)$, the (non-nec\-es\-sar\-i\-ly trivial) line bundle $L\to M$ is naturally equipped with a Jacobi structure (for details see Section~\ref{subsec:even-dim_transitive}).
		Notice that, following~\cite[Appendix A]{vitagliano2014vector}, we adopt a line bundle approach to lcs structures which slightly generalizes the analogous notion as defined in~\cite{vaisman1985locally}, in the same spirit as our line bundle approach to Jacobi structures encompasses Lichnerowicz's notion of Jacobi pairs (see also Proposition~\ref{prop:Jacobi_pairs} and Remark~\ref{rem:Vaisman_lcs_structures}).
	\end{enumerate}
\end{example}

Basic facts, including our notations and conventions, about (multi-)differential operators have been recollected in Chapter~\ref{chap:preliminaries}.
In the following, we will freely refer to them for details.
Here is just a very brief summary.
Let $L \to M$ be a line bundle, and denote by $J^1 L$ the bundle of $1$-jets of sections of $L$.
Let $j^1 : \Gamma (L) \to \Gamma (J^1 L)$ be the first jet prolongation, and denote by $J_1 L$ the dual bundle of $J^1 L$.
Sections of $J_1 L$ are first order differential operators $\Gamma (L) \to C^\infty (M)$.
Moreover, denote by $\Der^\bullet L = \Gamma (\wedge^\bullet J_1 L \otimes L)$ the space of multi-differential operators $\Gamma (L) \times \cdots \times \Gamma (L) \to \Gamma (L)$.
The Atiyah algebroid $\der L$ of $L$ is equipped with the tautological representation $\id_{\der  L}$ in $L$.
Accordingly, $(\der  L, L)$ is a Jacobi algebroid.
It follows from Proposition~\ref{prop:Jacobi_algbds_GJ-algbs} that there is a Gerstenhaber--Jacobi algebra structure on $( \Gamma (\wedge^\bullet J_1 L),  (\Der^\bullet L)[1])$.
The Lie bracket on $ (\Der^\bullet L)[1]$ will be also called the \emph{Schouten--Jacobi bracket} and denoted by $\ldsb-,-\rdsb$ (see Section~\ref{sec:GJ_algb_multi-derivations} for explicit formulas).

\begin{remark}
	\label{rem:J_denotes_Jacobi_structure}
	Let $\{-,-\}$ be a Jacobi bracket on a line bundle $L\to M$.
	From now on, when we want to stress that $\{-,-\}$ is a bi-differential operator, i.e.~a degree $1$ element of the graded Lie algebra $((\Der^\bullet L)[1],\ldsb-,-\rdsb)$, we also denote it by $J$.
\end{remark}

The Gerstenhaber--Jacobi algebra of multi-derivations of a line bundle plays a central r\^ole both in developing the theory of Jacobi manifolds and in analyzing the coisotropic deformation problem in the Jacobi setting.
As a starting point, Proposition~\ref{prop:J_as_MC_element} below provides a very useful characterization of Jacobi structures $J$ on a line bundle $L\to M$ as Maurer--Cartan (MC) elements of the graded Lie algebra $((\Der^\bullet L)[1],\ldsb-,-\rdsb)$.


\begin{proposition}[{\cite[Theorem 1.b, (28), (29)]{GM2001}}]
	\label{prop:J_as_MC_element}
	Let $L\to M$ be a line bundle, and $J=\{-,-\}:\Gamma(L)\times\Gamma(L)\to\Gamma(L)$ be a skew-symmetric, first order bi-differential operator, i.e.~$J \in \Der^2 L$.
	Then the following two conditions are equivalent.
	\begin{enumerate}
		\item $\{-,-\}$ is a Jacobi structure on $L\to M$, i.e.~it satisfies the Jacobi identity
		\begin{equation}
		\label{eq:prop:J_as_MC_element_1}
		\{\{\lambda,\mu\},\nu\}+\{\{\mu,\nu\},\lambda\}+\{\{\nu,\lambda\},\mu\}=0.
		\end{equation}
		\item $J$ is a MC element of $((\Der^\bullet(L)[1],\ldsb-,-\rdsb)$, i.e.~it satisfies the MC equation
		\begin{equation}
		\label{eq:prop:J_as_MC_element_2}
		\ldsb J,J\rdsb=0.
		\end{equation}
	\end{enumerate}
\end{proposition}

\begin{proof}
	Notice that the lhs of~\eqref{eq:prop:J_as_MC_element_1} defines a skew-symmetric first-order tri-differential operator, called the Jacobiator of $\{-,-\}$, and denoted by $\operatorname{Jac}(J)\in\Der^3L$.
	We aim at proving that $\ldsb J,J\rdsb=2\operatorname{Jac}(J)$.
	
	As a consequence of the explicit expression of the Schouten--Jacobi bracket in terms of the Gerstenhaber product (cf.~Section~\ref{sec:GJ_algb_multi-derivations}), we have
	\begin{equation}
	\label{eq:Jlm}
	\{\lambda, \mu\}=-\ldsb\ldsb J,\lambda\rdsb,\mu\rdsb,
	\end{equation}
	for all $\lambda,\mu\in\Gamma(L)$.
	Similarly, through a straightforward computation in the graded Lie algebra $((\Der^\bullet L)[1],\ldsb-,-\rdsb)$, we also get that
	\begin{align*}
	\ldsb J,J\rdsb(\lambda,\mu,\nu)&=\ldsb\ldsb\ldsb\ldsb J,J\rdsb,\lambda\rdsb,\mu\rdsb,\nu\rdsb\\
	&=2\left(\ldsb\ldsb J,\ldsb\ldsb J,\lambda\rdsb,\mu\rdsb\rdsb,\nu\rdsb+\ldsb\ldsb\ldsb J,\mu\rdsb,\ldsb J,\lambda\rdsb\rdsb,\nu\rdsb\right)\\
	&=2\left(\{\{\lambda,\mu\},\nu\}+\{\{\mu,\nu\},\lambda\}+\{\{\nu,\lambda\},\mu\}\right),
	\end{align*}
	for all $\lambda,\mu,\nu\in\Gamma(L)$.
	This proves that $\ldsb J,J\rdsb=2\operatorname{Jac}(J)$ and concludes the proof.
\end{proof}

Now, in view of Remark~\ref{rem:SJ_and_SN} and Proposition~\ref{prop:J_as_MC_element}, the notion of Jacobi pair on a manifold $M$, as introduced by Lichnerowicz~\cite{Lich1978}, can be recovered as special case of Definition~\ref{def:Jacobi_structure}.
Namely the former is obtained from the latter when $L=\bbR_M$.

\begin{definition}
	\label{def:Jacobi_pair}
	A \emph{Jacobi pair} $(\Lambda,\Gamma)$ on a manifold $M$ consists of a bi-vector field $\Lambda\in\frakX^2(M)$ and a vector field $\Gamma\in\frakX(M)$ satisfying the following compatibility condition
	\begin{equation}
	\label{eq:jac}
	\ldsb\Gamma,\Lambda\rdsb\equiv\calL_\Gamma\Lambda=0\quad\text{ and }\quad \ldsb\Lambda,\Lambda\rdsb+2\Gamma\wedge\Lambda=0.
	\end{equation}
	with $\ldsb-,-\rdsb$ denoting here the Schouten--Nijenhuis bracket on multi-vector fields.
\end{definition}

\begin{proposition}
	\label{prop:Jacobi_pairs}
	Let $M$ be a manifold.
	There exists a canonical one-to-one correspondence between
	\begin{itemize}
		\item Jacobi structures $J=\{-,-\}$ on the trivial line bundle $\bbR_M$,
		\item Jacobi pairs $(\Lambda,\Gamma)$ on $M$.
	\end{itemize}
	Such one-to-one correspondence is established by the relation $J=\Lambda-\Gamma\wedge\id$, where $\id$ denotes the identity map on $\Gamma(\bbR_M)=C^\infty(M)$, or more explicitly by
	\begin{equation*}
	\{f,g\}=\Lambda(f,g)+f\Gamma(g)-g\Gamma(f),
	\end{equation*}
	for all $f,g\in\Gamma(\bbR_M)=C^\infty(M)$.
\end{proposition}

\begin{proof}
	Let us recall from Remark~\ref{rem:SJ_and_SN} that there is a canonical direct sum decomposition of $C^\infty(M)$-module $\Der^\bullet(\bbR_M)\simeq\frakX^\bullet(M)\oplus\frakX^\bullet(M)[1]$.
	As a consequence, any $J\in\Der^2(\bbR_M)$ identifies with the pair $(\Lambda,\Gamma)\in\frakX^2(M)\oplus\frakX(M)$ such that $J=\Lambda-\Gamma\wedge\id$.
	Moreover $\ldsb J,J\rdsb\in\Der^3(\bbR_M)$ identifies with the pair $(P,Q)\in\frakX^3(M)\oplus\frakX^2(M)$ such that $\ldsb J,J\rdsb=P-Q\wedge\id$.
	According with~\eqref{eq:rem:SJ_and_SN}, $P$ and $Q$ are determined by $\Lambda$ and $\Gamma$, in terms of the exterior product and the Schouten--Nijenhuis bracket on multi-vector fields, as follows
	\begin{equation*}
	P=\ldsb\Lambda,\Lambda\rdsb+2\Lambda\wedge\Gamma\quad\text{ and }Q=-2\ldsb\Gamma,\Lambda\rdsb.
	\end{equation*}
	This concludes the proof.
\end{proof}

\begin{remark}
	\label{rem:Jacobi_pairs}
	Even if, as just shown, a Jacobi pair can be seen as a special case of a Jacobi structure, on the other hand, since every line bundle is locally trivial, every Jacobi structure on a line bundle looks locally like a Jacobi pair.
\end{remark}

\begin{remark}
	\label{rem:cohalg}
	Let $(M, \Lambda)$ be a Poisson manifold, with Poisson bi-vector $\Lambda$, and Poisson bracket $\{-,-\}_\Lambda$.
	Differential $d_\Lambda:=\ldsb\Lambda,-\rdsb:\frakX^\bullet (M)\to\frakX^\bullet(M)$, with $\ldsb-,-\rdsb$ denoting here the Schouten--Nijenhuis bracket on multi-vectors fields, has been introduced by Lichnerowicz to define what is known as the Lichnerowicz-Poisson cohomology of $(M, \Lambda)$.
	Note that complex $(\frakX^\bullet (M), d_\Lambda)$ can be seen as a subcomplex of the Chevalley--Eilenberg complex associated with the Lie algebra $(C^\infty(M), \{-,-\}_\Lambda)$ and its adjoint representation.
	For more general Jacobi manifolds $(M, L, J = \{-,-\})$
	it is natural to replace multi-vector fields, with a suitable subcomplex of the Chevalley-Eilenberg complex associated with the Lie algebra $(\Gamma(L),\{-,-\})$ and its adjoint representation, specifically, the subcomplex of first order, multi-differential operators, i.e.~elements of $ \Der^\bullet L$.
	In particular, the Lichnerowicz--Poisson differential is replaced with the differential $d_J :=\ldsb J, -\rdsb$.
	The resultant cohomology is called the \emph{Chevalley--Eilenberg cohomology} of $(M, L, \{-,-\})$, and we denote it by $H_{CE} (M, L, J)$ (see~\cite{guedira1984geometrie, Lich1978}).
	Furthermore, for a general Jacobi manifold $(M, L, J = \{-,-\})$, the action of $ (\Der^\bullet L)[1]$ on $ \Gamma (\wedge^\bullet J_1 L)$ gives rise to another cohomology, namely cohomology of the complex $( \Gamma (\wedge^\bullet J_1 L), X_J)$ (see Section~\ref{sec:GJ_algb_multi-derivations} for a definition of $X_J$), also called the \emph{Lichnerowicz--Jacobi (LJ-)cohomology of $(M, L, \{-,-\})$} (see, e.g., \cite{LLMP1999}).
	Notice that in Section~\ref{subsubsec:Jacobi_algbd_of_a_Jacobi_manifold} a Jacobi algebroid structure $([-,-]_J,\rho_J,\nabla_J)$ on $(J^1L,L)$ will be associated with any Jacobi manifold $(M,L,J=\{-,-\})$ so that the complex $(\Der^\bullet L, d_J)$ is the de Rham complex of $(J^1 L , \rho_J, [-,-]_J)$ with values in $L$.
	Similarly complex $(\Gamma(\wedge^\bullet J_1 L), X_J)$ is nothing but the de Rham complex of the Lie algebroid $(J^1 L , \rho_J, [-,-]_J)$.
\end{remark}

\begin{remark}
	\label{rem:poiss}
	Many properties of Poisson manifolds have analogues for Jacobi manifolds.
	Sometimes these analogues can be found using the ``\emph{Poissonization trick}'' which consists in the remark that Jacobi brackets on a line bundle $L \to M$ are in one-to-one correspondence with homogeneous Poisson brackets on the principal $\bbR^\times$-bundle $L^\ast \smallsetminus \mathbf{0}$, where $\mathbf{0}$ is the (image of the) zero section of $L^\ast$ (cf.~\cite{dazord1991structure, marle1991jacobi}).
	For instance, using the Poissonization trick, Iglesias and Marrero established a one-to-one correspondence between Jacobi structures and certain Jacobi bi-algebroids (see~\cite[Theorem 3.9]{iglesias2001generalized}).
	Here we prefer to adopt an intrinsic approach to Jacobi structures in the spirit of~\cite{crainic2013jacobi} (see also Remarks~\ref{rem:cohalg}, \ref{rem:dcohm}, and Proposition~\ref{prop:conormal}).
\end{remark}

\section{Jacobi structures and their canonical bi-linear forms}
\label{sec:jac_bilinear_forms}

Let $(M, L, J = \{-,-\})$ be a Jacobi manifold.
For every section $\lambda \in \Gamma (L)$, the associated \emph{Hamiltonian derivation} $\Delta_\lambda\in DL$ is the first order differential operator, hence a derivation of $L$, defined by $\Delta_\lambda:=\{\lambda,-\}$.
Further, the associated \emph{Hamiltonian vector field} $X_\lambda\in\frakX(M)$ is given by the symbol of $\Delta_\lambda$ (see Section~\ref{sec:app_0}), and so it is determined by the following generalized Leibniz rule
\begin{equation}
\label{eq:genleib}
\{ \lambda , f \mu \} = f \{ \lambda , \mu \} + X_\lambda (f) \mu,
\end{equation}
for all $\lambda, \mu \in \Gamma (L)$, $f \in C^\infty (M)$.

The bi-differential operator $J$ can be also interpreted as an $L$-valued, skew-symmetric, bi-linear form $\widehat \Lambda{}_J : \wedge^2 J^1 L \to L$.
Namely, $\widehat \Lambda{}_J $ is uniquely determined by
\[
\widehat \Lambda{}_J (j^1 \lambda, j^1 \mu) = \{\lambda, \mu \},
\]
for all $\lambda, \mu \in \Gamma (L)$.
Denote by $\der  L = \mathrm{Hom} (J^1 L, L)$ the Atiyah algebroid of the line bundle $L$ (see Section~\ref{sec:app_0} for details).
Then, the bi-linear form $\widehat \Lambda{}_J $ determines an obvious morphism of vector bundles $\widehat \Lambda{}_J^\sharp  : J^1 L \to \der  L$, defined by $ \widehat \Lambda{}_J^\sharp  (\alpha) \lambda := \widehat \Lambda{}_J (\alpha, j^1\lambda )$, where $\alpha \in \Gamma (J^1 L)$ and $\lambda \in \Gamma (L)$.
The \emph{bi-symbol} $\Lambda_J$ of $J$ will be also useful.
It is defined as follows:
$
\Lambda_J : \wedge^2 (T^\ast M \otimes L) \to L
$
is the bi-linear form obtained by restriction of $\widehat \Lambda{}_J$ to the module $\Omega^1 (M, L)$ of $L$-valued one forms on $M$, regarded as a submodule in $\Gamma (J^1 L)$ via the \emph{co-symbol} $\gamma : \Omega^1 (M, L) \to \Gamma (J^1 L)$ (see Section~\ref{sec:app_0}).
Namely,
\[
\Lambda_J (\eta, \theta) := \widehat \Lambda{}_J (\gamma (\eta), \gamma (\theta)),
\]
for all $\eta,\theta \in \Omega^1 (M, L)$.
It immediately follows from the definition that
\begin{equation}\label{eq:LambdaJ}
\begin{aligned}
& \Lambda_J (df\otimes \lambda, dg\otimes \mu) \\
& = \{f \lambda, g\mu \} -fg \{\lambda, \mu\} - f X_\lambda (g) \mu + g X_\mu (f)\lambda = \left( X_{f\lambda} (g) -f X_\lambda (g)\right) \mu,
\end{aligned}
\end{equation}
where $f,g \in C^\infty (M)$, and $\lambda, \mu \in \Gamma (L)$.

\begin{remark}
	When $L = \bbR_M$, then $J$ is the same as a Jacobi pair $(\Lambda, \Gamma)$ as in Proposition~\ref{prop:Jacobi_pairs}, and $\Lambda_J$ is just a bi-vector field.
	Actually, we have $\Lambda_J=\Lambda$ and $X_1=\Gamma$.
\end{remark}

The skew-symmetric form $\Lambda_J$ determines an obvious morphism of vector bundles
$ \Lambda_J^\sharp  : T^\ast M \otimes L \to TM$, implicitly defined by
$
\langle \Lambda_J^\sharp  (\eta \otimes \lambda) , \theta \rangle \mu := \Lambda_J (\eta \otimes \lambda , \theta \otimes \mu)
$,
where $\eta, \theta \in \Omega^1 (M)$, $\lambda, \mu \in \Gamma (L)$, and $\langle -,- \rangle$ is the duality pairing.
In other words,
\begin{equation}\label{eq:Lambdash}
\Lambda_J^\sharp  (df \otimes \lambda) = X_{f\lambda} -f X_\lambda,
\end{equation}
$f \in C^\infty (M)$, $\lambda \in \Gamma (L)$.
Morphism $\Lambda^\sharp _J$ can be alternatively defined as follows.
Recall that $\der  L$ projects onto $TM$ via the symbol map $\sigma$.
It is easy to see that diagram
\begin{equation*}
	\begin{tikzcd}
		T^\ast M\otimes L \arrow[r, "\Lambda_J^\sharp "]\arrow[d, swap, "\gamma"]&TM\\
		J^L\arrow[r, "\widehat \Lambda{}^\sharp _J"]&\der L\arrow[u, swap, "\sigma"]
	\end{tikzcd}
\end{equation*}
commutes, i.e.~$\Lambda^\sharp _J = \sigma \circ \widehat \Lambda{}^\sharp _J \circ \gamma$, which can be used as an alternative definition of $\Lambda^\sharp _J$.
Finally, note that
\[
(\widehat \Lambda{}_J^\sharp  \circ \gamma)(df \otimes \lambda) = \Delta_{f\lambda} -f \Delta_\lambda.
\]

\section{Morphisms of Jacobi manifolds}
\label{subsec:Jacobi_morphisms}

\begin{definition}
	\label{def:Jacobi_mfd_morphism}
	A \emph{morphism of Jacobi manifolds}, or a \emph{Jacobi map},
	\[
	(M,L,\{-,-\}) \rightarrow (M',L',\{-,-\}')
	\]
	is a vector bundle morphism $\phi:L\to L'$, covering a smooth map $\underline{\smash{\phi}}: M \to M'$, such that $\phi$ is an isomorphism on fibers, and $ \phi^\ast \{ \lambda, \mu \}' = \{ \phi^\ast \lambda , \phi^\ast \mu \}$ for all $\lambda,\mu \in\Gamma(L')$.
\end{definition}
In the following $\operatorname{Aut}(M,L,J)$ will denote the group of Jacobi automorphisms of a given Jacobi manifold $(M,L,J=\{-,-\})$.

\begin{definition}
	\label{def:Jacobi_mfd_imorphism}
	An \emph{infinitesimal automorphism of a Jacobi manifold} $(M,L,\{-,-\})$, or a \emph{Jacobi derivation}, is a derivation $\Delta$ of the line bundle $L$, equivalently, a section of the Atiyah algebroid $\der L$ of $L$, such that $\Delta$ generates a flow of automorphisms of $(M, L, \{-,-\})$  (see Section~\ref{sec:app_0}).
	A \emph{Jacobi vector field} is the (scalar-type) symbol of a Jacobi derivation.
\end{definition}

\begin{remark}\label{rem:inf_aut}
	Let $\Delta$ be a derivation of $L$, $\{ \varphi_t \}$ be its flow, and let $\square$ be a first order multi-differential operator on $L$ with $k $ entries, i.e.~$\square \in \Der^k L$.
	Since $L$ is a line bundle, a derivation $\Delta$ of $L$ is the same as a first order differential operator on $\Gamma (L)$, i.e.~an element of $\Der  L = \Der^1 L$.
	It is easy to see that (similarly as for vector fields)
	\begin{equation}\label{Lie1}
	\left. \frac{d}{dt} \right |_{t=0} (\varphi_t)_\ast \square =\ldsb\square, \Delta\rdsb
	\end{equation}
	where we denoted by $\varphi_\ast \square$ the \emph{push forward} of $\square $ along a line bundle isomorphism $\varphi : L \to L^\prime$, defined by
	$(\varphi_\ast \square) (\lambda_1^\prime, \ldots, \lambda_{k}^\prime) := (\varphi^{-1})^\ast (\square (\varphi^\ast \lambda_1, \ldots, \varphi^\ast \lambda_{k}))$, for all $\lambda_1^\prime, \ldots, \lambda_k^\prime \in \Gamma (L^\prime)$ (see Section~\ref{sec:app_0} about pushing forward derivations along vector bundle morphisms).
	In particular, $\Delta$ is an infinitesimal automorphism of $(M,L,\{-,-\})$ iff $\ldsb J, \Delta\rdsb= 0$.
	Since
	\begin{equation}\label{eq:DeltaJ}
	\ldsb J, \Delta\rdsb (\lambda, \mu ) = \{ \Delta \lambda , \mu\} + \{ \lambda , \Delta \mu\} - \Delta \{\lambda , \mu \},
	\end{equation}
	we conclude that $\Delta$ is an infinitesimal automorphism of $(M,L,\{-,-\})$ iff
	\begin{equation}
	\Delta \{ \lambda, \mu \} = \{\Delta \lambda , \mu \} + \{ \lambda , \Delta \mu\}
	\end{equation}
	for all $\lambda, \mu \in \Gamma (L)$.
	In other words $\Delta$ is a \emph{derivation} of the Jacobi bracket.
	As a consequence Jacobi derivations and Jacobi vector fields of $(M,L,J=\{-,-\})$ form Lie subalgebras that we will denote by $\mathfrak{aut}(M,L,J)\subset\Der L$ and $\underline{\smash{\mathfrak{aut}}}(M,L,J)\subset\frakX(M)$ respectively.
\end{remark}

\begin{remark}
	More generally, let $\{ \Delta_t \}$ be a one parameter family of derivations of $L$, generating the one parameter family of automorphisms $\{ \varphi_t \}$, and let $\square \in \Der^\bullet L$.
	Then
	\begin{equation}\label{Lie2}
	\frac{d}{dt} (\varphi_t)_\ast \square = \ldsb(\varphi_t)_\ast \square, \Delta_t\rdsb.
	\end{equation}
\end{remark}

\begin{remark}
	\label{rem:conf}
	The notions of \emph{conformal morphisms} and \emph{infinitesimal conformal automorphisms} of manifolds equipped with Jacobi pairs (see, e.g.,~\cite{dazord1991structure}) are recovered as special cases, when $L=\bbR_M$, of Definitions~\ref{def:Jacobi_mfd_morphism} and~\ref{def:Jacobi_mfd_imorphism}, respectively.
	In particular two Jacobi pairs are \emph{conformally equivalent} iff they are isomorphic as Jacobi structures.
\end{remark}

Let $(M,L, J = \{-,-\})$ be a Jacobi manifold and $\lambda \in \Gamma (L)$.
Recall from the preceding section that
\begin{equation}
\Delta_\lambda = \{ \lambda , -\} = -\ldsb J , \lambda\rdsb. \label{eq:trianglelambda}
\end{equation}
As an immediate consequence of the Jacobi identity for the Jacobi bracket we get that not only $\Delta_\lambda$ is a derivation of $L$, but even more, it is an infinitesimal automorphism of $(M,L,J=\{-,-\})$, called the \emph{Hamiltonian derivation associated with the section $\lambda$}.
Similarly, the scalar symbol $X_\lambda$ of $\Delta_\lambda$ will be called the \emph{Hamiltonian vector field associated with $\lambda$}.
Clearly, for all $\lambda\in \Gamma (L)$, $\square\in\mathfrak{aut}(M,L,J)$, and $Y\in\underline{\smash{\mathfrak{aut}}}(M,L,J)$ with $Y=\sigma_\square$, we have
\begin{equation}
\label{eq:iso}
[\square , \Delta_\lambda] = \Delta_{\square\lambda}, \quad \text{and} \quad [Y,X_\lambda] = X_{\square\lambda}.
\end{equation}
Hence Hamiltonian derivations form a Lie ideal of $\mathfrak{aut}(M,L,J)$ that we will denote by $\mathfrak{ham}(M,L,J)$.
Similarly Hamiltonian vector fields form a Lie ideal of $\underline{\smash{\mathfrak{aut}}}(M,L,J)$ that we will denote by $\underline{\smash{\mathfrak{ham}}}(M,L,J)$.
Jacobi automorphisms $L \to L$ generated by Hamiltonian derivations will be called \emph{Hamiltonian automorphisms}.
Similarly, diffeomorphisms $M \to M$ generated by Hamiltonian vector fields will be called \emph{Hamiltonian diffeomorphisms}.
In the following we will denote by $\operatorname{Ham}(M,L,J)\subset\operatorname{Aut}(M,L,J)$ the subgroup of Hamiltonian automorphisms of $(M,L,J=\{-,-\})$.

Hamiltonian derivations are interpreted as \emph{inner infinitesimal automorphisms}.
The following proposition provides a geometric interpretation of the first and the second Chevalley--Eilenberg cohomologies of $(M,L, \{-,-\})$.

\begin{proposition}\label{lem:coh} \
	\begin{enumerate}
		\item A derivation $\Delta$ of $L\to M$ is an infinitesimal automorphism of $(M,L,\{-,-\})$ iff $d_J \Delta = 0$, hence the set of outer infinitesimal automorphisms of $(M, L, \{-,-\})$ is $H_{CE}^{1} (M, L, J)$.
		\item An infinitesimal deformation $\bar J$ of $J$ is a Jacobi deformation if and only if $d_J \bar J = 0$, hence the set of infinitesimal Jacobi deformations of $J$ modulo infinitesimal automorphisms of the bundle $L$ is $H_{CE}^{2} (M, L,J)$.
	\end{enumerate}
\end{proposition}
\begin{proof} \
	\begin{enumerate}
		\item The first part of the assertion follows from Remark~\ref{rem:inf_aut}.
		Using this and taking into account~\eqref{eq:trianglelambda}, which interprets inner infinitesimal automorphisms as degree one co-boundaries, we immediately obtain the second part.
		
		\item The first part of the assertion follows from Proposition~\ref{prop:J_as_MC_element}.
		To prove the second part it suffices to show that the trivial infinitesimal deformation of $J$ induced by an infinitesimal automorphism $Y \in \Der  L $ is of the form $\ldsb J, Y\rdsb$.
		Clearly~\eqref{Lie1} proves what we need and this completes the proof.
	\end{enumerate}
\end{proof}

\begin{remark}\label{rem:dcohm} Proposition~\ref{lem:coh} generalizes a known interpretation of Lichnerowicz--Poisson cohomology, see e.g.~\cite[\S 2.1.2]{DZ2005}, and fits into deformation theory of Lie algebras, since any infinitesimal Jacobi deformation $\bar J$ of a Jacobi bracket $J$ is also an infinitesimal deformation of the Lie algebra $(\Gamma (L), \{-,-\})$ and, therefore, $\bar J$ is closed in the Chevalley--Eilenberg complex (see also~\cite{NR1967}).
\end{remark}

\begin{example}
	\label{ex:jmap}
	Let $(M,L,\{-,-\})$ be a Jacobi manifold.
	Hamiltonian vector fields generate a distribution $\calK \subset TM$ whose rank, generically, is non-constant.
	Distribution $\calK$ is called \emph{the characteristic distribution} of $(M, L, \{-,-\})$.
	The Jacobi manifold $(M, L, \{-,-\})$ is said to be \emph{transitive} if its characteristic distribution $\calK$ is the whole tangent bundle $TM$.
\end{example}

\section{Jacobi manifolds and Jacobi algebroids}
\label{sec:Jacobi_mflds-Jacobi_algbds}

It is well-known the close relationship which exists between Poisson manifolds and Lie algebroids, and works in both directions.
On the one hand, for any manifold $M$, a Lie algebroid structure on the cotangent bundle $T^\ast M\to M$ is canonically associated with every Poisson structure on $M$.
On the other hand, for any vector bundle $A\to M$, a Lie algebroid structure on $A$ is nothing but a fiber-wise linear Poisson structure on $A^\ast$.
In this section such relationship will be generalized to a similar one closely interconnecting Jacobi manifolds and Jacobi algebroids.


\subsection{The Jacobi algebroid associated with a Jacobi manifold}
\label{subsubsec:Jacobi_algbd_of_a_Jacobi_manifold}

Let $L\to M$ be a line bundle.
First we show that a structure of Jacobi algebroid $([-,-]_J,\rho_J,\nabla^J)$ on $(J^1L,L)$ is canonically associated with every Jacobi structure $J$ on $L\to M$ (see Proposition~\ref{prop:associated_Jacobi_algbd}).
Later Proposition~\ref{prop:characterization_associated_Jacobi_algbd} provides a complete characterization of Jacobi algebroid structures on $(J^1L,L)$ obtained in this way.

\begin{lemma}
	\label{lem:associated_Jacobi_algbd}
	Fix a Jacobi structure $J=\{-,-\}$ on $L\to M$.
	Then, for every Jacobi algebroid structure $([-,-],\rho,\nabla)$ on $(J^1L,L)$,  the following conditions are equivalent:
	\begin{enumerate}[label=(\alph*)]
		\item\label{enumitem:lem:associated_Jacobi_algbd_a}
		the Jacobi algebroid structure on $(J^1L,L)$ and the Jacobi structure on $L\to M$ are related through the relations
		\begin{equation}
			\label{eq:lem:associated_Jacobi_algbd}
			[j^1\lambda,j^1\mu]=j^1\{\lambda,\mu\},\quad \rho(j^1\lambda)=X_\lambda,\quad \nabla_{j^1\lambda}=\Delta_\lambda,
		\end{equation}
		for all $\lambda,\mu\in\Gamma(L)$,
		\item\label{enumitem:lem:associated_Jacobi_algbd_b}
		the de Rham complex $(\Gamma(\wedge^\bullet(J^1L)^\ast\otimes L),d_{J^1L,L})$ of $(J^1L,L)$ coincides with the Chevalley--Eilenberg complex $(D^\bullet L,d_J)$ of $(M,L,J)$, i.e.~$d_{J^1L,L}=d_J$.
	\end{enumerate} 
\end{lemma}

\begin{proof}
	Notice that both $d_J$ and $d_{J^1L,L}$ are completely determined by their action on degree $0$ and degree $1$ elements of $D^\bullet L=\Gamma(\wedge^\bullet(J^1L)^\ast\otimes L)$.
	For all $\lambda\in\Gamma(L)$, the elements $d_J\lambda$ and $d_{J^1L,L}\lambda$ of $DL=\Gamma((J^1L)^\ast\otimes L)$ act on an arbitrary $\mu\in\Gamma(L)$ as follows
	\begin{equation}
		\label{eq:proof:lem:associated_Jacobi_algbd_1}
		\begin{aligned}
			(d_J\lambda)\mu&=\ldsb\ldsb J,\lambda\rdsb,\mu\rdsb=\Delta_\mu\lambda,\\
			(d_{J^1L,L}\lambda)\mu&=\langle d_{J^1L,L}\lambda,j^1\mu\rangle=\nabla_{j^1\mu}\lambda,
		\end{aligned}
	\end{equation}
	where $\langle-,-\rangle$ denotes the $L$-valued duality pairing between $DL$ and $J^1L$.
	For all $\square\in DL=\Gamma((J^1L)^\ast\otimes L)$, the elements $d_J\square$ and $d_{J^1L,L}\square$ of $D^2L=\Gamma(\wedge^2(J^1L)^\ast\otimes L)$ act on arbitrary $\mu,\nu\in\Gamma(L)$ as follows
	\begin{equation}
		\label{eq:proof:lem:associated_Jacobi_algbd_2}
		\begin{aligned}
			(d_J\square)(\mu,\nu)&=\ldsb J,\square\rdsb(\mu,\nu)=\Delta_\mu(\square\nu)-\Delta_\nu(\square\mu)-\square\{\mu,\nu\},\\
			(d_{J^1L,L}\square)(\mu,\nu)&
			=\nabla_{j^1\mu}(\square\nu)-\nabla_{j^1\nu}(\square\mu)-\langle\square,[j^1\mu,j^1\nu]\rangle.
			\end{aligned}
	\end{equation} 
	Now, the equivalence of~\ref{enumitem:lem:associated_Jacobi_algbd_a} and~\ref{enumitem:lem:associated_Jacobi_algbd_b} is straightforward from~\eqref{eq:proof:lem:associated_Jacobi_algbd_1} and~\eqref{eq:proof:lem:associated_Jacobi_algbd_2}.	
\end{proof}

\begin{proposition}
	\label{prop:associated_Jacobi_algbd}
	Let $J=\{-,-\}$ be a Jacobi structure on $L\to M$.
	There is a unique Jacobi algebroid structure $([-,-]_J,\rho_J,\nabla^J)$ on $(J^1L,L)$ satisfying the equivalent conditions~\ref{enumitem:lem:associated_Jacobi_algbd_a} and~\ref{enumitem:lem:associated_Jacobi_algbd_b} in Lemma~\ref{lem:associated_Jacobi_algbd}.
	Explicitly, for all $\alpha,\beta\in\Gamma(J^1L)$,
	\begin{align}
	[\alpha,\beta]_J&=\calL_{\hat\Lambda_J^\sharp\alpha}\beta-\calL_{\hat\Lambda_J^\sharp\beta}\alpha-j^1(\Lambda_J(\alpha,\beta))\nonumber\\
	&=\iota_{\hat\Lambda_J^\sharp\alpha}d_D\beta-\iota_{\hat\Lambda_J^\sharp\beta}d_D\alpha+j^1(\Lambda_J(\alpha,\beta))
	\label{eq:prop:associated_Jacobi_algbd_bracket}\\
	\rho_J(\alpha)&=(\sigma\circ\hat\Lambda_J^\sharp)\alpha,\label{eq:prop:associated_Jacobi_algbd_anchor}\\
	\nabla^J_\alpha&=\Lambda_J^\sharp\alpha,\label{eq:prop:associated_Jacobi_algbd_representation}
	\end{align}
	with~\eqref{eq:prop:associated_Jacobi_algbd_bracket} written in terms of the $L$-valued Cartan calculus on $\der L$ (cf.~Section~\ref{sec:der_complex}).
\end{proposition}

\begin{proof}
	\textsc{Existence and Uniqueness.}
	Denote by $\scrM$ the graded manifold with $C^\infty(\scrM)=\Gamma(\wedge^\bullet(J^1L)^\ast)$, and by $\scrL\to\scrM$ the graded line bundle with $\Gamma(\scrL)=\Gamma(\wedge^\bullet(J^1L)^\ast\otimes L)$.
	As pointed out in Remark~\ref{rem:cohalg} (cf.~also Proposition~\ref{prop:Jacobi_algebroids_and_deRham_complex}), $d_J$ is a cohomological derivation of $\scrL\to\scrM$, with symbol the cohomological vector field $X_J$ on $\scrM$.
	Hence, according with Proposition~\ref{prop:Jacobi_algebroids_and_deRham_complex}, there exists a unique Jacobi algebroid structure on $(J^1L,L)$ satisfying condition~\ref{enumitem:lem:associated_Jacobi_algbd_b}.
	
	\textsc{Explicit Expression.}
	Since the operations defined by Equations~\eqref{eq:prop:associated_Jacobi_algbd_bracket}--\eqref{eq:prop:associated_Jacobi_algbd_representation} satisfy condition~\ref{enumitem:lem:associated_Jacobi_algbd_a}, it remains to check that they determine a Jacobi algebroid structure on $(J^1L,L)$.
	The Leibniz rule relating $[-,-]_J$ and $\rho_J$ is satisfied.
	Indeed we have 
	\begin{equation*}
		[\alpha,f\beta]_J=\calL_{\hat\Lambda_J^\sharp\alpha}(f\beta)-\iota_{\hat\Lambda_J^\sharp(f\beta)}d_D\alpha=
		(\rho_J(\alpha)f)\beta+f[\alpha,\beta]_J,
	\end{equation*}
	for all $\alpha,\beta\in\Gamma(J^1L)$, and $f\in C^\infty(M)$.
	Hence, in particular, $[-,-]_J$ is a bi-derivation of the vector bundle $J^1L\to M$.
	The flatness condition on $\nabla^J$ holds, i.e.
	\begin{equation}
		\label{eq:proof:prop:associated_Jacobi_algbd}
		\nabla^J_{[\alpha,\beta]_J}-[\nabla^J_\alpha,\nabla^J_\beta]=0,
	\end{equation}
	for all $\alpha,\beta\in\Gamma(J^1L)$.
	Actually the lhs of~\eqref{eq:proof:prop:associated_Jacobi_algbd} is $C^\infty(M)$-bilinear in $\alpha$ and $\beta$, and vanishes on the $1$-jet prolongations which generate the $C^\infty(M)$-module $\Gamma(J^1L)$.
	Hence a fortiori $[\rho_J(\alpha),\rho_J(\beta)]=\rho_J[\alpha,\beta]^J$, for all $\alpha,\beta\in\Gamma(J^1L)$.
	In view of what above, the Jacobi identity for $[-,-]_J$ reduces to a straightforward consequence of the Jacobi identity for $\{-,-\}$.
\end{proof}

\begin{proposition}
	\label{prop:characterization_associated_Jacobi_algbd}
	The equivalent conditions~\ref{enumitem:lem:associated_Jacobi_algbd_a} and~\ref{enumitem:lem:associated_Jacobi_algbd_b} in Lemma~\ref{lem:associated_Jacobi_algbd} establish a one-to-one correspondence between
	\begin{itemize}
		\item Jacobi structures on $L\to M$, and
		\item Jacobi algebroid structures $([-,-],\rho,\nabla)$ on $(J^1L,L)$ satisfying
		\begin{equation}
		\label{eq:prop:characterization_associated_Jacobi_algbd}
		[j^1\lambda,j^1\mu]=j^1(\nabla_{j^1\lambda}\mu),\qquad\textnormal{for all }\lambda,\mu\in\Gamma(L).
		\end{equation}
	\end{itemize}
\end{proposition}

\begin{proof}
	Evidently, for any Jacobi structure $J=\{-,-\}$ on $L\to M$, the associated Jacobi algebroid structure $([-,-]_J,\rho_J,\nabla^J)$ on $(J^1L,L)$, uniquely determined by equivalent conditions~\ref{enumitem:lem:associated_Jacobi_algbd_a} and~\ref{enumitem:lem:associated_Jacobi_algbd_b}, satisfies~\eqref{eq:prop:characterization_associated_Jacobi_algbd}.
	
	Conversely, let $([-,-],\rho,\nabla)$ be a Jacobi algebroid structure on $(J^1L,L)$ satisfying~\eqref{eq:prop:characterization_associated_Jacobi_algbd}.
	Then there is a unique bracket $\{-,-\}:\Gamma(L)\times\Gamma(L)\to\Gamma(L)$ such that
	\begin{equation}
	\{\lambda,\mu\}=\nabla_{j^1\lambda}\mu=-\nabla_{j^1\mu}\lambda,
	\end{equation}
	and a fortiori $j^1\{\lambda,\mu\}=[j^1\lambda,j^1\mu]$, for all $\lambda,\mu\in\Gamma(L)$.
	It is immediate to check now that $\{-,-\}$ is actually a Jacobi structure on $L\to M$ with, in particular, $\Delta_\lambda:=\{\lambda,-\}$ and $X_\lambda:=\sigma(\Delta_\lambda)$ coinciding with $\nabla_{j^1\lambda}$ and $\rho(j^1\lambda)$ respectively.
\end{proof}

\begin{remark}
	\label{rem:coinciding_characteristic_distributions}
	Notice that, for any Jacobi manifold $(M,L,J)$, its characteristic distribution $\calK$ (cf.~Example~\ref{ex:jmap}) coincides with the image of the anchor map $\rho_{\J}$, i.e.~the characteristic distribution of the associated Jacobi algebroid $(J^1L,L)$.
\end{remark}

\subsection{The fiber-wise linear Jacobi structure on the adjoint bundle of a Jacobi algebroid}
\label{subsubsec:fiberwise_linear_Jacobi_structures}


Let us start specifying what we mean by fiber-wise linear Jacobi structure.
Thus, let $\pi: E \to M$ be a vector bundle, $L \to M$ a line bundle and $\pi^\ast L \to E$ the pull-back line bundle.
Set $E_L:=E\otimes L^\ast\to M$.
The vector bundle $E_L{}^\ast=E^\ast\otimes L\to M$ will be called the \emph{$L$-adjoint bundle of $E$}.
Then there is a natural $C^\infty(E)$-module isomorphism $\Gamma(\pi^\ast L)\simeq C^\infty (E) \otimes_{C^\infty(M)}\Gamma (L)$.
Accordingly, we have that
\begin{itemize}
	\item the Euler derivation $\Delta_{E,L}$ of the line bundle $\pi^\ast L\to E$ is defined by
	\begin{equation*}
	\Delta_{E,L}(f \lambda)=(Z_Ef) \lambda,
	\end{equation*}
	where $\lambda\in\Gamma(L)$, $f\in C^\infty(E)$, and $Z_E$ denotes the Euler vector field on $E$,
	\item sections $\lambda$ of $L \to M$ identify, by pull-back, with the \emph{fiber-wise constant} sections of $\pi^\ast L\to E$, i.e.~those $\ell\in\Gamma(\pi^\ast L)$ such that $\Delta_{E,L}\ell=0$,
	\item sections of the $L$-adjoint bundle $E_L{}^\ast = E^\ast \otimes L$ identify with the \emph{fiber-wise linear} sections of $\pi^\ast L\to E$, i.e.~those $\ell\in\Gamma(\pi^\ast L)$ such that $\Delta_{E,L}\ell=\ell$.
\end{itemize}
Now, a Jacobi structure $J=\{-,-\}$ on the pull-back line bundle $\pi^\ast L\to E$ is \emph{fiber-wise linear} if $\ldsb\Delta_{E,L},J\rdsb=-J$.
Moreover, in view of Equations~\eqref{eq:SJ_bracket_and_Gerstenhaber_product} and~\eqref{eq:Gerstenhaber_product}, it turns out that $J=\{-,-\}$ is fiberwise linear iff:
\begin{enumerate}[label=(\arabic*)]
	\item\label{enumitem:fiber-wise_linear_Jacobi_1}
	the Jacobi bracket $\{-,-\}$ between two fiber-wise linear sections  is fiber-wise linear as well,
	\item\label{enumitem:fiber-wise_linear_Jacobi_2}
	the Jacobi bracket $\{-,-\}$ between a fiber-wise constant section and a fiber-wise linear one is fiber-wise constant,
	\item\label{enumitem:fiber-wise_linear_Jacobi_3}
	the Jacobi bracket $\{-,-\}$ between two fiber-wise constant sections is zero.
\end{enumerate}

\begin{remark}
	\label{rem:lijac_1}
	Notice that condition~\ref{enumitem:fiber-wise_linear_Jacobi_3} above is actually redundant: it follows from conditions~\ref{enumitem:fiber-wise_linear_Jacobi_1} and~\ref{enumitem:fiber-wise_linear_Jacobi_2}.
	Moreover, if we concentrate our attention only to Poisson structures on $E$, i.e.~when $L=\bbR_M$ and $\{1,-\}=0$, then even condition~\ref{enumitem:fiber-wise_linear_Jacobi_2} above becomes redundant.
	However, as first noticed in~\cite[Remark 1]{iglesias2000some}, in the general case of a Jacobi structure on $\pi^\ast L\to E$, condition~\ref{enumitem:fiber-wise_linear_Jacobi_1} does not necessarily imply condition~\ref{enumitem:fiber-wise_linear_Jacobi_2}.
\end{remark}

\begin{remark}
	\label{rem:lijac_2}
	If $J=\{-,-\}$ is a Jacobi structure on $\pi^\ast L\to E$, then conditions~\ref{enumitem:fiber-wise_linear_Jacobi_1}--\ref{enumitem:fiber-wise_linear_Jacobi_3} above imply immediately that, for all $f\in C^\infty(M)$,
	\begin{equation*}
	X_\alpha f\in C^\infty(M),\qquad \textnormal{and}\qquad X_\lambda(f)=0,
	\end{equation*}
	where $\alpha\in\Gamma(E_L{}^\ast)$ is a fiber-wise linear section of $\pi^\ast L$, and $\lambda\in\Gamma(L)$ is a fiber-wise constant section of $\pi^\ast L$.
\end{remark}

\begin{proposition}[{\cite[Theorems 1--3]{iglesias2000some}}]
	\label{prop:lijac}
	Let $A\to M$ be a vector bundle, and $L\to M$ be a line bundle.
	Denote by $\pi$ the bundle map $A_L{}^\ast\to M$.
	Then there exists a canonical one-to-one correspondence between:
	\begin{itemize}
		\item Jacobi algebroid structures $([-,-],\rho,\nabla)$ on $(A,L)$, and
		\item fiberwise linear Jacobi structures $J=\{-,-\}$ on $\pi^\ast L\to A_L{}^\ast$.
	\end{itemize}
	Such one-to-one correspondence is established by the following relations
	\begin{equation}
	\label{eq:prop:lijac}
	\begin{aligned}
	\{\alpha,\beta \}&=[\alpha,\beta]_A,\\
	X_\alpha(f)&=\rho(\alpha)f,\\
	\{\alpha,\lambda\} &=\nabla_\alpha\lambda,
	\end{aligned}
	\end{equation}
	for all $\alpha,\beta\in\Gamma(A)$, $\lambda\in\Gamma(L)$, and $f\in C^\infty(M)$.
\end{proposition}

\begin{proof}	
	Let us note preliminarily that a Jacobi structure $J=\{-,-\}$ on $\pi^\ast L\to A_L{}^\ast$ is completely determined by
	\begin{enumerate}
		\item the Jacobi bracket between fiber-wise linear sections,
		\item the Jacobi bracket between a fiber-wise linear section and a fiber-wise constant one,  and
		\item the Jacobi bracket between fiber-wise constant sections.
	\end{enumerate}
	Hence if we assume that $(A,L)$ possess a structure of Jacobi algebroid with Lie bracket $[-,-]_A$, anchor map $\rho$, and representation $\nabla$, then, in view of Remark~\ref{rem:lijac_1}, a fiber-wise linear Jacobi structure on $\pi^\ast L\to A_L{}^\ast$ is well-defined by reading Equations~\eqref{eq:prop:lijac} from the left to the right.
	
	Conversely, assume that $\pi^\ast L\to A_L{}^\ast$ possesses a fiber-wise linear Jacobi structure $J=\{-,-\}$.
	Read Equations~\eqref{eq:prop:lijac} from the right to the left to define a Lie bracket, an anchor map, and a flat connection. 
	In view of Remark~\ref{rem:lijac_2} it is immediate to get that the above operations form a well-defined Jacobi algebroid structure on $(A,L)$.	
\end{proof}

\begin{example}
	\label{ex:J1L}
	Let $L \to M$ be a line bundle and let $\der  L$ be its Atiyah algebroid.
	Since $(\der  L, L)$ is a Jacobi algebroid, Proposition~\ref{prop:lijac} provides a fiber-wise linear Jacobi structure on the $L$-adjoint bundle of $\der  L$, which is $J^1 L$.
	This is nothing but the Jacobi structure determined by the canonical contact structure on $J^1 L$ (see Example~\ref{ex:1jet}).
\end{example}

\begin{example}
	\label{ex:lijac1} (cf.~\cite[Theorem 1, \S 3 Example 5]{iglesias2000some})
	Let $(M, L, \{-,-\})$ be a Jacobi manifold.
	Since $(J^1 L, L)$ is a Jacobi algebroid, there is a fiber-wise linear Jacobi structure on the $L$-adjoint bundle of $J^1 L$ which is $\der  L$.
\end{example}

\section{Characterization of transitive Jacobi manifolds}
\label{sec:transitive}

In this section, following Kirillov~\cite{kirillov1976local}, we provide a complete characterization of transitive Jacobi manifolds $(M,L,J)$, according with the dimension of $M$.

\subsection{Odd-dimensional transitive Jacobi manifolds}
\label{subsec:odd-dim_transitive}

Let us start the characterization of transitive Jacobi manifolds $(M,L,J)$ analyzing the case when $\dim M$ is odd.

Fix a smooth manifold $M$ and a line bundle $L\to M$.
Let $\theta$ be a no-where zero $L$-valued $1$-form on $M$, and $C$ be a hyperplane distribution on $M$.
Assume that $\ker\theta=C$, so that $\theta$ induces an isomorphism $TM/C\simeq L$.
These data determine the \emph{curvature form} $\omega\in\Gamma(\wedge^2 C^\ast\otimes L)$, which is defined by $\omega(X,Y):=\theta[X,Y]$, for all $X,Y\in\Gamma(C)$.
The $1$-form $\theta$, and the corresponding distribution $C$ are said to be \emph{contact} if $\omega$ is non-degenerate, i.e.~$\omega^\flat:C\to C^\ast\otimes L$, $X\mapsto\omega(X,-)$, is a vector bundle isomorphism; if this is the case, then $\dim M$ is odd, and the inverse of $\omega^\flat$ will be denoted by $\omega^\sharp:C^\ast\otimes L\to C$.
A \emph{contact manifold} is a manifold $M$ equipped with a contact structure which is equivalently given by a contact $1$-form $\theta$ or a contact distribution.

It was first pointed out by Kirillov~\cite{kirillov1976local} that a transitive Jacobi structure is canonically associated with every contact structure as described in the next proposition.

\begin{proposition}
	\label{prop:contact_to_Jacobi}
	Let $L\to M$ be a line bundle.
	A transitive Jacobi structure on $L\to M$ is canonically associated with every $L$-valued contact form on $M$.
\end{proposition}

\begin{proof}
	Fix an $L$-valued contact $1$-form $\theta$ on $M$, with corresponding contact distribution $C$ and curvature form $\omega$.
	Denote by $\frakX_C\subset\frakX(M)$ the Lie subalgebra of contact vector fields of $(M,C)$, i.e.~those $X\in\frakX(M)$ such that $[X,\Gamma(C)]\subset\Gamma(C)$.
	The first order differential operator $\meanphi_{(-)}:\frakX(M)\to\Gamma(C^\ast\otimes L)$, $X\mapsto\meanphi_X$, is defined by $\meanphi_X(Y):=\theta[X,Y]$, for all $X\in\frakX(M)$ and $Y\in\Gamma(C)$, and fits in the following short exact sequence of $\bbR$-linear maps
	\begin{equation*}
	\begin{tikzcd}
	0\arrow[r]&\frakX_C\arrow[r, "\textnormal{incl}"]&\frakX(M)\arrow[r, "\meanphi_{(-)}"]&\Gamma(C^\ast\otimes L)\arrow[r]&0.
	\end{tikzcd}
	\end{equation*}
	Since the latter splits through $\omega^\sharp:\Gamma(C^\ast\otimes L)\lhook\joinrel\longrightarrow\Gamma(C)$, we get the direct sum decomposition of $\bbR$-vector spaces
	\begin{equation}
	\label{eq:proof:contact_to_Jacobi0}
	\frakX(M)=\frakX_C\oplus\Gamma(C).
	\end{equation}
	Accordingly there is a unique $\bbR$-linear monomorphism $X_{(-)}:\Gamma(L)\to\frakX(M)$, $\lambda\mapsto X_\lambda$ such that $X_\lambda\in\frakX_C$, and $\theta(X_\lambda)=\lambda$,
	for all $\lambda\in\Gamma(L)$.
	Additionally $\phi_{(-)}$ is a first order differential operator satisfying
	\begin{equation}
	\label{eq:proof:contact_to_Jacobi1}
	X_{f\lambda}=fX_\lambda+\omega^\sharp\left((df)|_C\otimes\lambda\right),
	\end{equation}
	for all $\lambda\in\Gamma(L)$ and $f\in C^\infty(M)$.
	Pulling back the Lie algebra structure on $\frakX_C\subset\frakX(M)$ through $\phi_{(-)}$, we obtain the following Lie bracket on $\Gamma(L)$
	\begin{equation*}
	\{-,-\}:\Gamma(L)\times\Gamma(L)\longrightarrow\Gamma(L),\quad (\lambda,\mu)\longmapsto\{\lambda,\mu\}:=\theta[X_\lambda,X_\mu],
	\end{equation*}
	which is actually a Jacobi structure on $L\to M$.
	Indeed, from~\eqref{eq:proof:contact_to_Jacobi1}, it follows that, for all $\lambda,\mu\in\Gamma(L)$ and $f\in C^\infty(M)$,
	\begin{equation*}
	\{\lambda,f\mu\}=\theta[X_\lambda,fX_\mu+\omega^\sharp((df)|_C\otimes\mu)]=\theta[X_\lambda,fX_\mu]=X_\lambda(f)\mu+f\{\lambda,\mu\},
	\end{equation*}
	i.e.~$\{\lambda,-\}$ is a derivation of $L\to M$, with symbol given by $X_\lambda$.
	Moreover such Jacobi structure on $L\to M$ is transitive because, according with~\eqref{eq:proof:contact_to_Jacobi1}, the $X_\lambda$'s generate the whole tangent space $TM$.
\end{proof}

The following complete characterization of odd-dimensional transitive Jacobi manifolds first appeared in~\cite[Section 4]{kirillov1976local}.

\begin{proposition}
	\label{prop:odd-dim_transitive}
	Let $(M,L,J)$ be a Jacobi manifold.
	The following conditions are equivalent:
	\begin{enumerate}[label=(\arabic*)]
		\item
		\label{enumitem:contact_to_Jacobi}
		there is an $L$-valued contact $1$-form on $M$ whose canonically associated Jacobi structure is $J$,
		\item
		\label{enumitem:odd-dim_transitive}
		the Jacobi manifold $(M,L,J)$ is transitive, and $M$ is odd-dimensional,
		\item
		\label{enumitem:non-deg_Jacobi_bi-derivation}
		the Jacobi structure $J$ on $L$, seen as the skew-symmetric bilinear form $\hat{\Lambda}_J\in\Gamma(\wedge^2(J^1L)^\ast\otimes L)$, is non-degenerate.
	\end{enumerate}	
\end{proposition}

\begin{proof}
	Implication $(1)\Longrightarrow(2)$ is contained within the discussion leading to Proposition~\ref{prop:contact_to_Jacobi}.
	Implication $(2)\Longrightarrow(3)$, up to counting dimensions, is a direct consequence of the skew-symmetry of $\hat{\Lambda}_J$.
	Hence it only remains to check that $(3)\Longrightarrow (1)$.
	
	Assume that $\hat{\Lambda}_J$ is non-degenerate.
	From the skew-symmetry of $\hat{\Lambda}_J$ and $\Lambda_J$, and the relation $\Lambda_J^\sharp=\sigma\circ\hat{\Lambda}_J^\sharp\circ\gamma$, with $\sigma$ surjective and $\gamma$ injective, we get that:
	\begin{enumerate}[label=(\alph*)]
		\item
		\label{enumitem:proof:odd-dim_transitive_a}
		the image of $\gamma:T^\ast M\otimes L\hookrightarrow J^1L$ contains the kernel of $\sigma\circ\hat{\Lambda}_J^\sharp:J^1L\twoheadrightarrow TM$,
		\item
		\label{enumitem:proof:odd-dim_transitive_b}
		the image of $\Lambda_J^\sharp:T^\ast M\otimes L\to TM$ is an hyperplane distribution $C$ on $M$.
	\end{enumerate}
	In view of~\ref{enumitem:proof:odd-dim_transitive_b}, there is a non-degenerate $\omega\in\Gamma(\wedge^2 C^\ast\otimes L)$ which is defined by setting
	$\omega(\Lambda_J^\sharp(\alpha),\Lambda_J^\sharp(\beta))=\Lambda_J(\alpha,\beta)$, for all $\alpha,\beta\in\Omega^1(M,L)$.
	From~\ref{enumitem:proof:odd-dim_transitive_a} it follows that the Lie algebra morphism $X_{(-)}:=\sigma\circ\hat{\Lambda}_J^\sharp\circ j^1:\Gamma(L)\to\underline{\smash{\mathfrak{ham}}}(M,L,J)$, $\lambda\mapsto X_\lambda$, is actually an isomorphism, with
	\begin{equation}
	\label{eq:proof:odd-dim_transitive1}
	X_{f\lambda}=fX_\lambda\Mod C,
	\end{equation}
	for all $f\in C^\infty(M)$, and $\lambda\in\Gamma(L)$.
	Moreover the Spencer sequence $0\to\Omega^1(M,L)\to\Gamma(J^1L)\to\Gamma(L)\to 0$, which is split by $j^1:\Gamma(L)\to\Gamma(J^1L)$, gives rise to a new short exact sequence of $\bbR$-linear maps $0\to\Gamma(C)\to\frakX(M)\to\underline{\smash{\mathfrak{ham}}}(M,L,J)\to 0$, which splits through the inclusion map $\underline{\smash{\mathfrak{ham}}}(M,L,J)\to\frakX(M)$, and fits in the following commutative diagram
	\begin{equation*}
	\begin{tikzcd}
	0\arrow[r]&\Omega^1(M,L)\arrow[d, two heads, "\Lambda_J^\sharp"]\arrow[r, "\gamma"]&\Gamma(J^1L)\arrow[d, two heads, "\sigma\circ \hat{\Lambda}_J^\sharp"]\arrow[r]&\Gamma(L)\arrow[d, hook, two heads, "X_{(-)}"]\arrow[r]&0\\
	0\arrow[r]&\Gamma(C)\arrow[r]&\frakX(M)\arrow[r, dotted]&\underline{\smash{\mathfrak{ham}}}(M,L,J)\arrow[r]&0
	\end{tikzcd}
	\end{equation*}
	Hence, in particular, we have the following direct sum decomposition of $\bbR$-vector spaces
	\begin{equation}
	\label{eq:proof:odd-dim_transitive2}
	\frakX(M)=\Gamma(C)\oplus\underline{\smash{\mathfrak{ham}}}(M,L,J).
	\end{equation}
	In view of~\eqref{eq:proof:odd-dim_transitive1} and~\eqref{eq:proof:odd-dim_transitive2} there is a unique $\theta\in\Omega^1(M,L)$ such that $\ker\theta=C$, and $\theta(X_\lambda)=\lambda$ for all $\lambda\in\Gamma(L)$.
	Further a straightforward computation, involving the Jacobi algebroid structure on $(J^1L,L)$ canonically associated with $J$, allows to get that
	\begin{enumerate}[label=(\roman*)]
		\item\label{enumitem:proof:odd-dim_transitive_i} $\omega$ is the curvature form associated with $\theta$,
		\item\label{enumitem:proof:odd-dim_transitive_ii} $X\in\underline{\smash{\mathfrak{ham}}}(M,L,J)$ if and only if $\theta[X,\Gamma(C)]=0$.
	\end{enumerate}
	Finally~\ref{enumitem:proof:odd-dim_transitive_i} and~\ref{enumitem:proof:odd-dim_transitive_ii} imply that $\theta$ is an $L$-valued contact form on $M$, and $J$ coincides with the Jacobi structure on $L\to M$ canonically associated with $\theta$ on the basis of Proposition~\ref{prop:contact_to_Jacobi}.
\end{proof}

\subsection{Even-dimensional transitive Jacobi manifolds}
\label{subsec:even-dim_transitive}

In this section we complete the characterization of the transitive Jacobi manifolds $(M,L,J)$ considering the case when $\dim M$ is even.

Fix a smooth manifold $M$ and a line bundle $L\to M$.
Let $\nabla$ be a flat $TM$-connection in $L\to M$, and denote by $(\Omega^\bullet(M,L),d_\nabla)$ the de Rham complex of the tangent Lie algebroid $TM\to M$ with values in its representation $\nabla$ in $L\to M$.
Let $\omega$ an $L$-valued $2$-form on $M$, and denote by $\omega^\flat:TM\to T^\ast M\otimes L$ the vector bundle morphism defined by $\omega^\flat(X):=\omega(X,-)$, for all $X\in\frakX(M)$.
The pair $(\nabla,\omega)$ is said to be a \emph{locally conformal symplectic} (or \emph{lcs} for short) \emph{structure} on $L\to M$ if $d_\nabla\omega=0$, and $\omega$ is non-degenerate, i.e.~$\omega^\flat:TM\to T^\ast M\otimes L$ is a vector bundle isomorphism.
If this is the case, then $\dim M$ is even, and the inverse of $\omega^\flat$ will be denoted by $\omega^\sharp:T^\ast M\otimes L\to TM$.
A \emph{lcs manifold} $(M,L,\nabla,\omega)$ is a manifold $M$ equipped with a line bundle $L\to M$, and an lcs structure $(\nabla,\omega)$ over it.

\begin{remark}
	\label{rem:Vaisman_lcs_structures}
	Let $M$ be a manifold.
	Recall that a flat connection in the trivial line bundle $\bbR_M$ is the same thing as a closed $1$-form on $M$.
	Hence an lcs structure on $\bbR_M\to M$ is nothing but a pair $(\omega,\theta)$, with $\omega\in\Omega^2(M)$ and $\theta\in\Omega^1(M)$, such that $d\theta=0$ and $d\omega+\omega\wedge\theta=0$.
	In this way we recover, as a special case of the definition above, Vaisman's notion of lcs structure on a manifold $M$ (cf.~\cite{vaisman1985locally}).
\end{remark}

The following proposition shows that, as first pointed out by Kirillov~\cite{kirillov1976local}, a transitive Jacobi structure is canonically associated with every lcs structure.

\begin{proposition}
	\label{prop:lcs_to_Jacobi}
	Let $L\to M$ be a line bundle.
	A transitive Jacobi structure on $L\to M$ is canonically associated with every lcs structure on $L\to M$.
\end{proposition}

\begin{proof}
	Fix an lcs structure $(\nabla,\omega)$ on $L\to M$.
	Since $\omega$ is non-degenerate, for any section $\lambda\in\Gamma(L)$ , the associated \emph{Hamiltonian vector field} $X_\lambda\in\frakX(M)$ is well-defined by $X_\lambda:=\omega^\sharp(d_\nabla\lambda)$.
	Notice that the map $X_{(-)}:\Gamma(L)\to\frakX(M)$, $\lambda\mapsto X_\lambda$, is a first order differential operator; indeed it is $\bbR$-linear, and satisfies, in particular,
	\begin{equation}
	\label{eq:proof:lcs_to_Jacobi1}
	X_{f\lambda}=fX_\lambda+\omega^\sharp(df\otimes\lambda),
	\end{equation}
	for all $\lambda\in\Gamma(L)$ and $f\in C^\infty(M)$.
	Accordingly a skew-symmetric $\bbR$-bilinear map $\{-,-\}:\Gamma(L)\times\Gamma(L)\longrightarrow\Gamma(L)$ is defined by
	\begin{equation}
	\label{eq:proof:lcs_to_Jacobi2}
	\{\lambda,\mu\}:=\nabla_{X_\lambda}\mu=\omega(X_\mu,X_\lambda),
	\end{equation}
	for all $\lambda,\mu\in\Gamma(L)$.
	Indeed $\{-,-\}$ is a bi-derivation because, for all $\lambda\in\Gamma(L)$, Equation~\eqref{eq:proof:lcs_to_Jacobi2} tells us, in particular, that $\{\lambda,-\}$ is a derivation of $L\to M$ with symbol given by $X_\lambda$.
	Moreover $\{-,-\}$ is a Jacobi structure on $L\to M$.
	Indeed, just using the flatness of $\nabla$ and Equation~\eqref{eq:proof:lcs_to_Jacobi2}, we obtain, for all $\lambda,\mu,\nu\in\Gamma(L)$, the following identity
	\begin{equation}
	\label{eq:closedness=Jacobi_identity}
	\begin{aligned}
	d_\nabla\omega(X_\lambda,X_\mu,X_\nu)&=\nabla_{X_\lambda}(\omega(X_\mu,X_\nu))-\nabla_{X_\mu}(\omega(X_\lambda,X_\nu))+\nabla_{X_\nu}(\omega(X_\lambda,X_\mu))\\
	&\phantom{=}-\omega([X_\lambda,X_\mu],X_\nu)+\omega([X_\lambda,X_\nu],X_\mu)-\omega([X_\mu,X_\nu],X_\lambda)\\
	&=\nabla_{X_\lambda}\nabla_{X_\nu}\mu+\nabla_{X_\mu}\nabla_{X_\lambda}\nu+\nabla_{X_\nu}\nabla_{X_\mu}\lambda\\
	&\phantom{=}+{[\nabla_{X_\lambda},\nabla_{X_\mu}]}\nu-{[\nabla_{X_\lambda},\nabla_{X_\nu}]}\mu+{[\nabla_{X_\mu},\nabla_{X_\nu}]}\lambda\\
	&=\{\lambda,\{\mu,\nu\}\}+\{\mu,\{\nu,\lambda\}\}+\{\nu,\{\lambda,\mu\}\}.
	\end{aligned}
	\end{equation}
	Hence, from~\eqref{eq:closedness=Jacobi_identity}, it turns out that the Jacobi identity for $\{-,-\}$ is just a rephrasing of $d_\nabla\omega=0$.
	Finally such Jacobi structure $\{-,-\}$ on $L\to M$ is transitive because of~\eqref{eq:proof:lcs_to_Jacobi1} and of $\omega$ being non-degenerate.
\end{proof}

The following complete characterization of even-dimensional transitive Jacobi manifolds first appeared in~\cite[Section 3]{kirillov1976local}.

\begin{proposition}
	\label{prop:even-dim_transitive}
	Let $(M,L,J=\{-,-\})$ be a Jacobi manifold.
	The following conditions are equivalent:
	\begin{enumerate}
		\item
		\label{enumitem:lcs_to_Jacobi}
		there is an lcs structure on $L\to M$ whose canonically associated Jacobi structure is $J$,
		\item
		\label{enumitem:even-dim_transitive}
		the Jacobi manifold $(M,L,J)$ is transitive, and $M$ is even-dimensional,
		\item
		\label{enumitem:non-deg_Jacobi_bi-symbol}
		the bi-symbol of the Jacobi bi-derivation $J$, seen as the skew-symmetric bilinear form $\Lambda_J\in\Gamma(\wedge^2(T^\ast M\otimes L)^\ast\otimes L)$, is non-degenerate.
	\end{enumerate}
\end{proposition}

\begin{proof}
	The implication $(1)\Longrightarrow(2)$ is contained within the discussion preceding Proposition~\ref{prop:lcs_to_Jacobi}.
	The implication $(2)\Longrightarrow(3)$, up to counting dimensions, follows directly from the skew-symmetry of $\Lambda_J$ and the injectivity of the co-symbol map $\gamma:T^\ast M\otimes L\to J^1L$.
	Hence it only remains to prove that $(3)\Longrightarrow(1)$.
	
	Assume that $\Lambda_J$ is non-degenerate.
	Hence there is a non-degenerate $\omega\in\Omega^2(M,L)$ such that $\omega^\flat:TM\longrightarrow T^\ast M\otimes L$ is the inverse of $\Lambda_J^\sharp:T^\ast M\otimes L\longrightarrow TM$.
	Moreover now, through the skew-symmetry of $\hat{\Lambda}_J$ and $\Lambda_J$, and the relation $\Lambda_J^\sharp=\sigma\circ\hat{\Lambda}_J^\sharp\circ\gamma$, we also get that $D^1L$ is the direct sum of the $C^\infty(M)$-modules $\ker\sigma=\langle\mathbbm{1}\rangle$ and $\im\hat{\Lambda}^\sharp_J$
	\begin{equation*}
	D^1L=\ker\sigma\oplus\im\hat{\Lambda}^\sharp_J.
	\end{equation*}
	Accordingly there exists a $TM$-connection $\nabla$ in $L\to M$ well-defined by $\nabla_{\sigma\Delta}=\Delta$, for all $\Delta\in\im\hat{\Lambda}_J^\sharp$.
	Moreover such connection $\nabla$ is flat because $\im\hat{\Lambda}_J^\sharp\subset D^1L$ is a Lie subalgebroid.
	Indeed, by its very definition, $\hat{\Lambda}_J^\sharp:J^1L\to D^1L$ is a Jacobi algebroid morphism from $(J^1L,L)$ to $(D^1L,L)$, where $(J^1L,L)$ carries the Jacobi algebroid structure canonically associated with $J$.
	Denote by $(\Omega^\bullet(M,L),d_\nabla)$ the de Rham complex of the tangent Lie algebroid $TM\to M$ with values in its representation $\nabla$ in $L\to M$.
	
	We want to show that
	\begin{enumerate}[label=(\alph*)]
		\item
		\label{enumitem:proof:even-dim_transitive1}
		$d_\nabla\omega=0$, so that $(\omega,\nabla)$ is an lcs structure on $L\to M$,
		\item
		\label{enumitem:proof:even-dim_transitive2}
		$\nabla_{X_\lambda}\mu=\{\lambda,\mu\}$ and $d_\nabla\lambda=\omega^\flat(X_\lambda)$, for all $\lambda,\mu\in\Gamma(L)$, so that $J=\{-,-\}$ is exactly the Jacobi structure on $L\to M$ canonically associated with $(\omega,\nabla)$.
	\end{enumerate}
	The first part of~\ref{enumitem:proof:even-dim_transitive2} is straightforward from the definition of $\nabla$, whilst the second part follows from
	\begin{equation*}
	\omega(X_\lambda,\Lambda_J^\sharp(df\otimes\mu))=- X_\lambda(f)\mu=f\{\lambda,\mu\}-\{\lambda,f\mu\}=\langle d_\nabla\lambda,\Lambda_J^\sharp(df\otimes\mu)\rangle,
	\end{equation*}
	for all $\lambda,\mu\in\Gamma(L)$ and $f\in C^\infty(M)$.
	Finally, from~\ref{enumitem:proof:even-dim_transitive2} and Formula~\eqref{eq:closedness=Jacobi_identity}, it turns out that~\ref{enumitem:proof:even-dim_transitive1} is equivalent to the Jacobi identity satisfied by $\{-,-\}$.
\end{proof}

\section{Global structure of Jacobi manifolds}
\label{subsec:global_structure}

This section aims to investigate the global structure of Jacobi manifolds,  decomposing them in their smallest building blocks.
With this aim in view, given a Jacobi manifold $(M,L,J)$ and a submanifold $S\subset M$, we have to understand under what conditions $J$ induces a Jacobi structure on the restricted line bundle $L|_S\to S$ according with the following definition.

\begin{definition}
	\label{def:Jacobi_submanifolds}
	Let $(M,L,J=\{-,-\})$ be a Jacobi manifold, and $S\subset M$ be an immersed submanifold.
	Set $\ell:=L|_S\to S$.
	The submanifold $S$ is said to be a \emph{Jacobi submanifold} of $(M,L,J)$ if $J$ induces a Jacobi structure on $\ell\to S$, i.e.~there is a (unique) Jacobi structure $J_S=\{-,-\}_S$ on $\ell\to S$ such that $i:\ell\to L$, the regular line bundle morphism given by the inclusion, is a Jacobi map from $(S,\ell,J_S)$ to $(M,L,J)$.
\end{definition}

Helpful characterizations of Jacobi submanifolds are contained in the following proposition.

\begin{proposition}
	\label{prop:Jacobi_submanifolds}
	Let $(M,L,J=\{-,-\})$ be a Jacobi manifold, with characteristic distribution $\calK$, and $S\subset M$ be an embedded submanifold.
	Set $\ell:=L|_S\to S$, and denote by $\Gamma_S\subset\Gamma(L)$ the submodule of sections vanishing on $S$.
	Then the following conditions are equivalent:
	\begin{enumerate}[label=(\arabic*)]
		\item
		\label{enumitem:prop:Jacobi_submanifolds1}
		$S$ is a Jacobi submanifold of $(M,L,J)$,
		\item
		\label{enumitem:prop:Jacobi_submanifolds2}
		the bi-derivation $J\in D^2L$ is tangent to $S$, i.e.~$J|_S\in D^2\ell\subset(D^2L)|_S$,
		\item
		\label{enumitem:prop:Jacobi_submanifolds3}
		$\Gamma_S$ is a Lie ideal of $\Gamma(L)$, i.e.~$\{\Gamma_S,\Gamma(L)\}\subset\Gamma_S$,
		\item
		\label{enumitem:prop:Jacobi_submanifolds4}
		$X_\lambda|_S\in\frakX(S)$, for all $\lambda\in\Gamma(L)$, i.e.~$\calK|_S\subset TS$.
	\end{enumerate}
\end{proposition}

\begin{proof}
	Both~\ref{enumitem:prop:Jacobi_submanifolds2} and~\ref{enumitem:prop:Jacobi_submanifolds3} can be immediately recognized as direct rephrasings of~\ref{enumitem:prop:Jacobi_submanifolds1} via the Definition~\ref{def:Jacobi_mfd_morphism} of Jacobi maps: actually, for the inclusion map $i:\ell\to L$, we have that $\Gamma_S$ is the kernel of $i^\ast:\Gamma(L)\to\Gamma(\ell)$, and $Di:D^\bullet\ell\to D^\bullet L$ is the inclusion map.
	Hence it remains to check that $(3)\Longleftrightarrow (4)$.
	
	Working locally, we may assume, without loss of generality, that $\Gamma_S=I_S\cdot\Gamma(L)$, where $I_S\subset C^\infty(M)$ is the ideal of functions vanishing on $S$.
	For all $\lambda,\mu\in\Gamma(L)$, and $f\in I_S$, we have:
	\begin{equation*}
	\left.\{\lambda,f\mu\}\right|_S=\left.\left(X_\lambda(f)\mu+f\{\lambda,\mu\}\right)\right|_S=X_\lambda(f)|_S\mu|_S.
	\end{equation*}
	The latter shows that, under the current assumptions,~\ref{enumitem:prop:Jacobi_submanifolds3} is equivalent to~\ref{enumitem:prop:Jacobi_submanifolds4}.
\end{proof}

\begin{remark}
	\label{rem:Jacobi_submanifolds}
	The equivalences $(1)\Longleftrightarrow (2)\Longleftrightarrow (4)$ from Proposition~\ref{prop:Jacobi_submanifolds} continue to hold for an immersed submanifold $S\subset M$ as well.
\end{remark}

\begin{corollary}
	\label{cor:Jacobi submanifolds}
	Let $(M,L,J)$ be a Jacobi manifold, with characteristic distribution $\calK$, and $S\subset M$ be an immersed Jacobi submanifold.
	Set $\ell:=L|_S\to S$, and denote by $\Gamma_S\subset\Gamma(L)$ the submodule of sections vanishing on $S$.
	Then
	\begin{enumerate}[label=(\alph*)]
		\item\label{enumitem:cor:Jacobi_submanifolds_a}
		$J_S:=J|_S\in D^2\ell$ is the (unique) Jacobi structure induced by $J$ on $\ell\to S$,
		\item\label{enumitem:cor:Jacobi_submanifolds_b}
		$\calK_S:=\calK|_S$ is the characteristic distribution of $(S,\ell,J_S)$,
		\item\label{enumitem:cor:Jacobi_submanifolds_c} $X_\lambda|_S=0$ for all $\lambda\in\Gamma_S$.
	\end{enumerate}
\end{corollary}

The next theorem, describing the global structure of Jacobi manifolds, was first obtained by Kirillov~\cite[Theorem 1]{kirillov1976local}, and it also encompass, as its special case, the symplectic foliation theorem for Poisson manifolds.
In both the statement and the proof of the next theorem we will make use of terminology and results coming from the Stefan--Sussmann theory for the integrability of singular tangent distributions (for details see, e.g.,~\cite[Section 3]{Kolar1993natural} and~\cite[Chapter 2]{vaisman1994lectures}).

\begin{theorem}[Characteristic Foliation Theorem]\cite[Theorem 1]{kirillov1976local}
	\label{theor:characteristi_foliation}
	The characteristic distribution $\calK\subset TM$ of a Jacobi manifold $(M,L,J)$ is completely integrable à la Stefan--Sussmann.
	Hence it determines a singular foliation $\calF$ of $M$ whose leaves are called the \emph{characteristic leaves} of $(M,L,J)$.
	Moreover each characteristic leaf $\calC$ is a Jacobi submanifold of $(M,L,J)$, with $J$ inducing a transitive Jacobi structure on $L|_\calC\to\calC$.
\end{theorem}

\begin{proof}
	Recall that the characteristic distribution $\calK$ of $(M,L,J)$ is the smooth singular distribution on $M$ generated by the Hamiltonian vector fields $X_\lambda$, with $\lambda\in\Gamma(L)$.
	As remarked in Section~\ref{subsec:Jacobi_morphisms} (see, e.g., Equation~\eqref{eq:iso}), Hamiltonian vector fields form a Lie subalgebra $\underline{\smash{\mathfrak{ham}}}(M,L,J)\subset\frakX(M)$.
	Hence $\calK$ is involutive, and, in order to prove its complete integrability, it remains to check that the rank of $\calK$ is constant along the flow lines of Hamiltonian vector fields.
	
	Fix an arbitrary $\lambda\in\Gamma(L)$.
	Denote by $\phi_t$ (resp.~$\underline{\smash{\phi}}_t$) the smooth $1$-parameter group of Jacobi automorphisms of $(M,L,J)$ (resp.~diffeomorphisms of $M$) which is generated by $\Delta_\lambda$ (resp.~$X_\lambda$).
	Clearly the local line bundle automorphism $\phi_t:L\to L$ covers the local diffeomorphism $\underline{\smash{\phi}}_t:M\to M$.
	Then, as remarked in Section~\ref{subsec:Jacobi_morphisms}, we have:
	\begin{equation*}
	(D\phi_t)\circ\Delta_\mu=\Delta_\nu\circ\underline{\smash{\phi}}_t,\qquad (d\phi_t)\circ X_\mu=X_\nu\circ\underline{\smash{\phi}}_t,
	\end{equation*}
	for all $\mu,\nu\in\Gamma(L)$ with $\mu=\phi_t^\ast\nu$.
	As a consequence, we obtain that
	\begin{equation*}
	(d_x\underline{\smash{\phi}}_t)\calK_x=\calK_{\underline{\smash{\phi}}_t(x)}.
	\end{equation*}
	for all $t\in\bbR$ and $x\in\operatorname{dom}(\underline{\smash{\phi}}_t)$.
	This proves that the rank of $\calK$ remains constant along flow lines of Hamiltonian vector fields.
	All the assumptions of the Frobenius Theorem for singular distributions are satisfied.
	So $\calK$ is completely integrable à la Stefan--Sussmann.	
	
	Now let $\calC$ be an integral leaf of $\calK$.
	Since, by its very definition, $\calC$ is an immersed submanifold of $M$ with $T\calC=\calK|_\calC$, from Proposition~\ref{prop:Jacobi_submanifolds} and Corollary~\ref{cor:Jacobi submanifolds}, it follows that $\calC$ is a Jacobi submanifold of $(M,L,J)$, and moreover, if $J_\calC$ denotes the Jacobi structure on $L|_\calC\to\calC$ induced by $J$, then the characteristic distribution of $(\calC,L|_\calC,J_\calC)$ is given by $\calK|_\calC=T\calC$.
\end{proof}

\begin{remark}
	The characteristic distribution $\calK$ of a Jacobi manifold $(M,L,J)$ coincides with $\im(\rho_J)$, i.e.~the characteristic distribution of the associated Jacobi algebroid $(J^1L,L)$ (cf.~Remark~\ref{rem:coinciding_characteristic_distributions}).
	This leads to two different points of view on the first part of Theorem~\ref{theor:characteristi_foliation}.
	On the one hand, the integrability of $\calK$ derives from the more general result according to which the characteristic distribution of a Lie algebroid is always integrable à la Stefan--Sussmann~(cf., e.g.,~\cite[Section 2.2]{crainic2011lectures}).
	On the other hand, the latter general result can be proven by just slightly adapting the above argument which was originally adopted by Kirillov~\cite{kirillov1976local} in the special case of Jacobi (and Poisson) manifolds.
\end{remark}

Combining together the characteristic foliation theorem for Jacobi manifolds (Theorem~\ref{theor:characteristi_foliation}) and the characterization of transitive Jacobi manifolds (Propositions~\ref{prop:odd-dim_transitive} and~\ref{prop:even-dim_transitive}), we obtain an alternative equivalent description of Jacobi structures.

\begin{proposition}
	\label{prop:Dirac-Jacobi_bundles}
	A Jacobi structure on a line bundle $L\to M$ is equivalent to the datum of a singular foliation $\calF$ of $M$, and
	\begin{itemize}
		\item for every odd-dimensional leaf $\calC$ of $\calF$, an $L|_\calC$-valued contact $1$-form on $\calC$, with associated Jacobi structure $\{-,-\}_\calC$ on $L|_\calC\to\calC$,
		\item for every even-dimensional leaf $\calC$ of $\calF$, an lcs structure on  $L_\calC\to\calC$, with associated Jacobi structure $\{-,-\}_\calC$ on $L|_\calC\to\calC$,
	\end{itemize}
	such that Jacobi structures $\{-,-\}_\calC$ on $L|_\calC\to\calC$, with $\calC\in\calF$, fit together in a smooth way, i.e., for all $\lambda,\mu\in\Gamma(L)$, a smooth section $\{\lambda,\mu\}\in\Gamma(L)$ is well-defined by $\{\lambda,\mu\}|_\calC=\{\lambda|_\calC,\mu|_\calC\}_\calC$, for all $\calC\in\calF$.
\end{proposition}

\begin{remark}
\label{rem:Dirac-Jacobi_bundles}
The equivalent description of Jacobi bundles given by Proposition~\ref{prop:Dirac-Jacobi_bundles} can rephrased in more geometrical terms within the extended setting provided by Dirac--Jacobi bundles~\cite{vitagliano2015dirac}.
\end{remark}

\section{Coisotropic submanifolds}
\label{subsec:coisotropic_submanifolds}

In the context of Jacobi manifolds there exists a notion of coisotropic submanifolds which unifies the analogous notions in the Poisson, l.c.s., and contact settings.
In this section we propose some characterizations of coisotropic submanifolds in a Jacobi manifold (Lemma~\ref{lem:cois} and Corollary~\ref{cor:cois}~(3)).

Let $(M, L, J= \{-,-\})$ be a Jacobi manifold, and let $x \in M$.
A subspace $T \subset T_xM$ is said to be \emph{coisotropic} (wrt~the Jacobi structure $J=\{-,-\}$ on $L$, if $\Lambda_J ^\sharp  (T^0 \otimes L_{{x}}) \subset T$, where $T^0 \subset T^*_xM$ denotes the annihilator of $T$ (cf.~\cite[Definition 4.1]{ILMM1997}).
Equivalently, $T^0 \otimes L_{{x}}$ is isotropic wrt~the $L$-valued bi-linear form $\Lambda_J$.

A submanifold $S\subset M$ is called \emph{coisotropic} (wrt~the Jacobi structure $(L,J=\{-,-\})$, if its tangent space $T_xS$ is coisotropic for all $ x \in S$.

\begin{lemma}
	\label{lem:cois}
	Let $S \subset M$ be an embedded submanifold, and let $\Gamma_S$ denote the set of sections $\lambda$ of the Jacobi bundle such that $\lambda |_S = 0$.
	The following three conditions are equivalent:
	\begin{enumerate}[label=(\arabic*)]
		\item
		\label{enumitem:lem:cois_1}
		$S$ is a coisotropic submanifold,
		\item
		\label{enumitem:lem:cois_2}
		$\Gamma_S$ is a Lie subalgebra in $\Gamma (L)$,
		\item
		\label{enumitem:lem:cois_3}
		$X_\lambda$ is tangent to $S$, for all $\lambda \in \Gamma_S$.
	\end{enumerate}
\end{lemma}

\begin{proof}
	Let $S \subset M$ be a submanifold.
	We may assume, without loss of generality, that $L$ is trivial.
	Then $\Gamma_S=I_S\cdot\Gamma(L)$, where $I_S$ denotes the ideal in $C^\infty(M)$ consisting of functions that vanish on $S$.
	In particular, if $\lambda$ is a generator of $\Gamma (L)$, then every section in $\Gamma_S$ is of the form $f \lambda$ for some $f \in I_S$.
	Now, let $f,g \in I_S$.
	Putting $\mu = \lambda$ in~\eqref{eq:LambdaJ} and restricting to $S$, we find
	\[
	\{ f\lambda, g\lambda \}|_S = \langle \Lambda_J^\sharp  (df \otimes \lambda) , dg \rangle \lambda |_S .
	\]
	This shows that $(1) \Longleftrightarrow (2)$.
	The equivalence $(2) \Longleftrightarrow (3)$ follows from the identity $X_\lambda(f)\mu|_S=\{\lambda,f\mu\}|_S$, for all $\lambda\in\Gamma_S$, $\mu\in\Gamma(L)$, and $f\in I_S$.
\end{proof}

\begin{example}
	\label{ex:simul}
	\leavevmode
	\begin{enumerate}
		\item Any coisotropic submanifold (in particular a Legendrian/Lagrangian submanifold) in a contact/lcs manifold is a coisotropic submanifold wrt~the associated Jacobi structure (see Section~\ref{sec:cois_cont} for details).
		\item As it is well-known~\cite[\S 2]{cattaneo2007relative}, in the Poisson setting, coisotropic submanifolds admit a characterization which is just a special case of Lemma~\ref{lem:cois}.
		\item Let $S$ be a coisotropic submanifold of the Jacobi manifold $(M,L,\{-,-\})$, and let $X \in \frakX(M)$ be a Jacobi vector field such that $X_x \notin T_x S$, for all $x\in S$.
		Then the flowout of $S$ along $X$, is a coisotropic submanifold as well.
		Indeed, let $\{\phi_t\}$ be the flow of $X$.
		Clearly, whenever defined, $\phi_t (S)$ is a coisotropic submanifold, and the claim immediately follows from Lemma~\ref{lem:cois}.
	\end{enumerate}
\end{example}

Now, let $S \subset M$ be a coisotropic submanifold and let $T^0 S \subset T^\ast M |_S$ be the annihilator of $TS$.
The (generically singular) distribution $\mathcal{K}_S := \Lambda^\sharp _J (T^0 S \otimes L) \subset TS$ on $S$ is called the \emph{characteristic distribution of $S$}.
In view of~\eqref{eq:Lambdash}, $\calK_S$ is generated by the (restrictions to $S$ of) the Hamiltonian vector fields of the kind $X_\lambda$, with $\lambda \in \Gamma_S$.

From Lemma~\ref{lem:cois} we derive the following

\begin{corollary}\label{cor:cois}\
	\begin{enumerate}[label=(\arabic*)]
		\item
		\label{enumitem:cor:cois_1}
		(cf.~\cite[\S 2]{cattaneo2007relative})
		The characteristic distribution $\mathcal K_S$ of any coisotropic submanifold $S$ is completely integrable à la Stefan--Sussmann (hence, it determines a singular foliation on $S$, called the \emph{characteristic foliation} of $S$).
		
		\item
		\label{enumitem:cor:cois_2}
		If $S$ is a Jacobi submanifold, then $S$ is also a coisotropic submanifold, with characteristic distribution $\calK_S=0$.
		
		\item
		\label{enumitem:cor:cois_3}
		A submanifold $S \subset M$ is coisotropic, iff $TS \cap T\calC$ is coisotropic in the tangent bundle $T\calC$, for all characteristic leaves $\calC$ intersecting $S$, where $\calC$ is equipped with
		the induced Jacobi structure.
	\end{enumerate}
\end{corollary}

\begin{proof} \
	(1)
	Vector fields $X_\lambda|_S$, with $\lambda\in\Gamma_S$, generate the singular distribution $\calK_S$, and form a Lie subalgebra of $\underline{\smash{\mathfrak{ham}}}(M,L,J)$.
	Indeed, for $\lambda, \mu \in \Gamma_S$,
	\begin{equation}
	\label{eq:com1}
	[X_\lambda, X_\mu] = X_{\{\lambda, \mu\}}
	\end{equation}
	is again in $\calK_S$, because of Lemma~\ref{lem:cois}~\ref{enumitem:lem:cois_3} and in view of what remarked in Section~\ref{subsec:Jacobi_morphisms} (see, e.g., Equation~\eqref{eq:iso}).
	Hence $\calK_S$ is involutive, and it only remains to check that the rank of $\calK_S$ is constant along the flow lines of these generating vector fields.
	
	Fix an arbitrary $\lambda\in\Gamma_S$.
	Denote by $\phi_t$ (resp.~$\underline{\smash{\phi}}_t$) the smooth $1$-parameter group of Jacobi automorphisms of $(M,L,J)$ (resp.~diffeomorphisms of $M$) which is generated by $\Delta_\lambda$ (resp.~$X_\lambda$).
	Clearly the local line bundle automorphism $\phi_t:L\to L$ covers the local diffeomorphism $\underline{\smash{\phi}}_t:M\to M$, and the latter fix $S$ point-wise.
	Then, as remarked in Section~\ref{subsec:Jacobi_morphisms}, we have:
	\begin{equation*}
	(D\phi_t)\circ\Delta_\mu=\Delta_\nu\circ\underline{\smash{\phi}}_t,\qquad (d\phi_t)\circ X_\mu=X_\nu\circ\underline{\smash{\phi}}_t,
	\end{equation*}
	for all $\mu,\nu\in\Gamma(L)$, with $\mu=\phi_t^\ast\nu$, and additionally $\phi_t^\ast\Gamma_S=\Gamma_S$.
	As a consequence, the tangent map to $\underline{\smash{\phi}}_t$ preserves the fibers of $\calK_S$.
	
	All the assumptions of the Frobenius Theorem for singular distributions are satisfied.
	So $\calK$ is completely integrable à la Stefan--Sussmann.
	
	(2) If $S$ is a Jacobi submanifold, i.e.~$X_\lambda$ is tangent to $S$ for all sections $\lambda\in\Gamma(L)$ (cf.~Proposition~\ref{prop:Jacobi_submanifolds}~\ref{enumitem:prop:Jacobi_submanifolds4}), then, a fortiori, $S$ is a coisotropic submanifold, with additionally $\calK_S=0$.
	Indeed $X_\lambda$ vanishes on $S$ for all section $\lambda\in\Gamma(L)$ vanishing on $S$ (cf.~Corollary~\ref{cor:Jacobi submanifolds}~\ref{enumitem:cor:Jacobi_submanifolds_c} and Lemma~\ref{lem:cois}~\ref{enumitem:lem:cois_3}).
	
	(3) For $V \subset T\calC$ let $V^{0 _\calC}$ denote the annihilator of $V$ in $T^*\calC$.
	Noting that $\calC$ is coisotropic, the transitivity of $(\calC, L|_\calC, \{-,-\}_\calC)$ and~\eqref{eq:com1} imply that the restriction to $\calC$ of an Hamiltonian vector field on $M$ is an Hamiltonian vector field.
	Denote by $i: \calC \injects M$ the inclusion.
	Then for any $\xi \in T^*M$, $\lambda \in L$ and for any submanifold $S$ in $M$ we have
	\begin{equation}
	\Lambda_{J_\calC}^\sharp  (i^*\xi \otimes \lambda)= \Lambda_J^\sharp  (\xi \otimes \lambda) \quad \text{and} \quad (TS\cap T\calC)^{0_\calC} \otimes L|_\calC = i ^* (T^0 S ) \otimes L|_\calC.\label{eq:int1}
	\end{equation}
	Hence, if $S$ is coisotropic, we have
	\[
	\Lambda _{J_\calC}^\sharp  ((TS \cap T\calC)^{0_\calC} \otimes L|_\calC) \subset TS\cap T\calC,
	\]
	i.e.~$TS \cap T\calC$ is coisotropic in $T\calC$.
	Conversely, assume that $TS \cap T\calC$ is coisotropic in $T\calC$, i.e.~$\Lambda_{J_\calC}^\sharp  ((TS \cap T\calC)^{0_\calC} \otimes L|_\calC) \subset TS\cap T\calC$.
	Using~\eqref{eq:int1} we obtain
	immediately
	\[
	\Lambda ^\sharp  (T^0 S \otimes L\, |_\calC) = \Lambda_{J_\calC} ^\sharp  (i^* (T^0 S ) \otimes L|_\calC) = \Lambda^\sharp _{J_\calC} ((TS \cap TC ) ^ {0_\calC} \otimes L|_\calC)\subset TS \cap T\calC.
	\]
\end{proof}

\section{Jacobi subalgebroid associated with a closed co\-iso\-tro\-pic submanifold}
\label{sec:Jacobi_subalgbd_of_a_coisotropic_submanifold}

Since in this thesis we will be interested in deformations of a \emph{closed} coisotropic submanifold, from now on we assume that $S$ is a closed submanifold in a smooth manifold $M$.
Let $A \to M$ be a Lie algebroid.
Recall that a \emph{subalgebroid of $A$ over $S$} is a vector subbundle $B \to S$, with embeddings $j : B \injects A$ and $\underline{j} : S \injects M$, such that the anchor $\rho : A \to TM$ descends to a (necessarily unique) vector bundle morphism $\rho_B : B \to TS$, making diagram
\begin{equation*}
	\begin{tikzcd}
		B \arrow[r, "j"]\arrow[d, swap, "\rho_B"] & A\arrow[d, "\rho"]\\
		TS \arrow[r, swap, "d\underline{\smash{j}}"] & TM
	\end{tikzcd}
\end{equation*}
commutative and, moreover, for all $\beta, \beta^\prime \in \Gamma (B)$ there exists a (necessarily unique) section $[\beta,\beta^\prime]_B \in \Gamma (B)$ such that whenever $\alpha, \alpha^\prime \in \Gamma (A)$ are $j$-related to $\beta, \beta^\prime$ (i.e.~$j \circ \beta = \alpha \circ \underline{j}$, in other words $\alpha |_S = \beta$, and similarly for $\beta^\prime, \alpha^\prime$) then $[\alpha, \alpha^\prime]_A$ is $j$-related to $[\beta, \beta^\prime]_B$.
In this case $B$, equipped with $\rho_B$ and $[-,-]_B$, is a Lie algebroid itself.
One can also give a notion of \emph{Jacobi subalgebroid} as follows.

Let $(A,L)$ be a Jacobi algebroid with flat $A$-connection $\nabla$ in $L$.

\begin{definition}
	A \emph{Jacobi subalgebroid of $(A,L)$ over S} is a pair $(B, \ell)$, where $B \to S$ is a Lie subalgebroid of $A$ over $S \subset M$, and $\ell := L|_S \to S$ is the pull-back line subbundle of $L$, such that $\nabla$ descends to a (necessarily unique) vector bundle morphism $\nabla|_\ell$ making diagram
	\begin{equation*}
	\begin{tikzcd}
		B \arrow[r, "j"]\arrow[d, swap, "\nabla|_\ell"] & A\arrow[d, "\nabla"]\\
		\der\ell\arrow[r, swap, "\der j_\ell"]&\der L
	\end{tikzcd}
	\end{equation*}	
	commutative.
	Here $j_\ell : \ell \injects L$ is the inclusion (cf.~Section~\ref{sec:app_0} for a definition of morphism $\der  j_\ell$).
\end{definition}

If $(B,\ell)$ is a Jacobi subalgebroid, then the restriction $\nabla|_\ell$ is a representation so that $(B, \ell)$, equipped with $\nabla|_\ell$, is a Jacobi algebroid itself.

Now, let $(M,L,J=\{-,-\})$ be a Jacobi manifold, and let $S$ be a submanifold.
In what follows, we denote by
\begin{itemize}
	\item $\ell :=L|_S$ the restricted line bundle,
	\item $NS:=TM|_S/TS$ the normal bundle of $S$ in $M$,
	\item $N^\ast S:= (NS)^\ast\simeq T^0 S \subset T^\ast M$ the conormal bundle of $S$ in $M$,
	\item $N_\ell S := NS \otimes \ell^\ast$, and by
	\item $N_\ell{}^\ast S := (N_\ell S)^\ast = N^\ast S \otimes \ell$ the $\ell$-adjoint bundle of $NS$.
\end{itemize}
Vector bundle $N_\ell {}^\ast S$ will be also regarded as a vector subbundle of $(J^1 L)|_S$ via the vector bundle embedding
\begin{equation*}
	N_\ell {}^\ast S \lhook\joinrel\longrightarrow (T^\ast M \otimes L)|_S \overset{\gamma}{\longrightarrow} J^1 L|_S,
\end{equation*}
where $\gamma$ is the co-symbol.
If $\lambda\in\Gamma(L)$, we have that $(j^1\lambda)|_S\in \Gamma (N_\ell {}^\ast S) $ if and only if $\lambda|_S=0$, i.e.~$\lambda\in\Gamma_S$.

The following Proposition establishes a one-to-one correspondence between co\-iso\-tro\-pic submanifolds and certain Lie subalgebroids of $J^1 L$.

\begin{proposition}
	\label{prop:conormal}
	(cf.~\cite[Proposition 5.2]{iglesias2003jacobi}) The submanifold $S \subset M$ is coisotropic iff $(N_\ell {}^\ast S, \ell)$ is a Jacobi subalgebroid of $(J^1 L, L)$.
\end{proposition}

\begin{proof} Let $S \subset M$ be a coisotropic submanifold.
	We want to show that $N_\ell {}^\ast S$ is a Jacobi subalgebroid of $J^1 L$.
	We propose a proof which is shorter than the one in~\cite{iglesias2003jacobi}.
	Since $S$ is coisotropic, we have
	\begin{equation}
		\rho_J ( N_\ell {}^\ast S ) \subset TS , \label{eq:anchor1}
	\end{equation}
	and similarly
	\begin{equation}
		\nabla^J ( N_\ell {}^\ast S ) \subset \der  \ell .
	\end{equation}
	Next we shall show that for any $\alpha, \beta \in \Gamma (J^1 L)$ such that $\alpha|_S , \beta |_S \in \Gamma (N_\ell {}^\ast S)$ we have
	\begin{equation}
		[\alpha, \beta]_{J} |_S \in \Gamma (N_\ell {}^\ast S).
	\end{equation}
	First we note that if $\alpha |_S \in \Gamma (N_\ell {}^\ast S)$ then $\alpha = \sum f j^1 \lambda$ for some $\lambda \in \Gamma_S$.
	Using the Leibniz properties of the Jacobi bracket we can restrict to the case $\alpha, \beta \in j^1 \Gamma_S$.
	The latter case can be handled taking into account~\eqref{eq:lem:associated_Jacobi_algbd} and Lemma~\ref{lem:cois}.
	Moreover, using~\eqref{eq:prop:associated_Jacobi_algbd_bracket}, we easily check that
	\[
	[\alpha, \beta]_{J} |_S = 0 \text{ if } \alpha|_S = 0 \text{ and } \beta |_S \in \Gamma (N_\ell {}^\ast S).
	\]
	This completes the ``only if part'' of the proof.
	
	To prove the ``if part'' it suffices to note that condition~\eqref{eq:anchor1}, regarded as a condition on the image of the anchor map of the Lie subalgebroid $N_\ell {}^\ast S$, implies, in view of~\eqref{eq:prop:associated_Jacobi_algbd_anchor}, that
	$S$ is a coisotropic submanifold.
\end{proof}

\begin{remark}
	\label{rem:defcois}
	Proposition~\ref{prop:conormal} extends to the Jacobi setting a similar well-known result for coisotropic submanifolds of Poisson manifolds.
	See, e.g., \cite[Proposition 3.1.3]{Weinstein1988}, \cite[Proposition 5.1]{cattaneo2004integration}, and~\cite[Theorem 10.4.2]{mackenzie2005general}.
\end{remark}

\begin{remark}
	Let $(M,L,\{-,-\})$ be a Jacobi manifold, and let $S\subset M$ be a closed coisotropic submanifold.
	In view of Proposition~\ref{prop:conormal}, the characteristic distribution $\calK_S$ of $S$ coincides with $\rho_J(N_\ell{}^\ast S)$, i.e.~the characteristic distribution of the Jacobi algebroid $(N_\ell{}^\ast S,\ell)$  associated with $S$.
	This lead to a different point of view on Corollary~\ref{cor:cois}~\ref{enumitem:cor:cois_1}.
	Indeed the integrability of $\calK_S$ can also be seen as a further instance of the general result according to which the characteristic distribution of a Lie algebroid is always integrable à la Stefan--Sussmann~(cf., e.g.,~\cite[Section 2.2]{crainic2011lectures}).
\end{remark}

\section{Coisotropic submanifolds and Jacobi reduction}
\label{sec:Jacobi_reduction}

The reduction of Jacobi manifolds can be loosely understood as the replacement of a Jacobi manifold with a new ``smaller'' one obtained ``quotienting out some degrees of freedom''.
Actually a reason of interest in coisotropic submanifolds is that they naturally appear in the reduction of Jacobi manifolds.

Our initial setting consists of a Jacobi manifold $(M,L,J=\{-,-\})$, and a closed submanifold $S\subset M$, with restricted line bundle $\ell:=L|_S\to S$.
Generically $S$ is not a Jacobi submanifold of $(M,L,J)$, i.e.~$J$ does not induce a Jacobi structure on $\ell\to S$.
However under suitable hypotheses is still possible, by Jacobi reduction, to determine a Jacobi structure on a quotient of $\ell\to S$ as formalized in the following definition.

\begin{definition}
	\label{def:Jacobi_reduction}
	Let $L'\to M'$ be a line bundle, and $p:\ell\to L'$ be a regular line bundle morphism, covering a surjective submersion $\underline{\smash{p}}:S\to M'$.
	The submanifold $S\subset M$ is said to be \emph{Jacobi reducible} via $p$ if there is a (unique) Jacobi structure $J'=\{-,-\}'$ on $L'\to M'$ such that
	\begin{equation}
	\label{eq:def:Jacobi_reduction}
	\{\lambda,\mu\}|_S=p^\ast\{\lambda',\mu'\}'
	\end{equation}
	for all $\lambda,\mu\in\Gamma(L)$ and $\lambda',\mu'\in\Gamma(L')$ with $\lambda|_S=p^\ast\lambda'$ and $\mu|_S=p^\ast\mu'$.
	In such case we will also say that the \emph{Jacobi reduction} of $S$ is performed via $p$, and $(M',L',J')$ is the \emph{reduced Jacobi manifold}.
\end{definition}

\begin{remark}
	Jacobi structures live on non-necessarily trivial line bundles.
	Hence generically the surjective submersion $\underline{\smash{p}}:S\to M'$ is not enough to perform Jacobi reduction: it is also necessary a regular line bundle morphism $p:\ell\to L'$ covering $\underline{\smash{p}}$.
	So if $S$ is Jacobi reducible, not only $S$ is a foliated manifold with leaf space $M'$ and quotient map $\underline{\smash{p}}:S\to M'$, but additionally the reduced line bundle $\ell$ is isomorphic to the pull-back line bundle $\underline{\smash{p}}^\ast L'$.
\end{remark}

The next proposition gives an equivalent characterization of Jacobi reducibility.
\begin{proposition}
	\label{prop:Jacobi_reduction}
	Let $L'\to M'$ be a line bundle, and $p:\ell\to L'$ be a regular line bundle morphism, covering a surjective submersion $\underline{\smash{p}}:S\to M'$.
	Then $S$ is Jacobi reducible via $p$ if and only if the following conditions are satisfied:
	\begin{enumerate}[label=(\arabic*)]
		\item\label{enumitem:prop:Jacobi_reduction_1}
		$\{\lambda,\mu\}|_S\in p^\ast\Gamma(L')$ for all $\lambda,\mu\in\Gamma(L)$ such that $\lambda|_S,\mu|_S\in p^\ast\Gamma(L')$,
		\item\label{enumitem:prop:Jacobi_reduction_2}
		$\Delta_\lambda$, or equivalently $X_\lambda$, is tangent to $S$, for all $\lambda\in\Gamma(L)$ with $\lambda|_S\in p^\ast\Gamma(L')$.
	\end{enumerate}
	Moreover, in the affirmative case, for all $\lambda\in\Gamma(L)$, $\lambda'\in\Gamma(L')$ with $\lambda|_S=p^\ast\lambda'$, we have \begin{equation}
	\label{eq:prop:Jacobi_reduction}
	(Dp)\circ(\Delta_\lambda|_S)=\Delta_{\lambda'}\circ\underline{\smash{p}},\qquad (T\underline{\smash{p}})\circ (X_\lambda|_S)=X_{\lambda'}\circ\underline{\smash{p}}.
	\end{equation}
\end{proposition}

\begin{proof}
	On the one hand,~\ref{enumitem:prop:Jacobi_reduction_1} and~\ref{enumitem:prop:Jacobi_reduction_2} represent the necessary and sufficient condition for the existence of a (unique) binary operation $\{-,-\}'$ on $\Gamma(L')$ satisfying~\eqref{eq:def:Jacobi_reduction}, or equivalently fitting in the following commutative diagram
	\begin{equation*}
	\begin{tikzcd}
	\Gamma(L)\times\Gamma(L)\arrow[d, swap, "{\{-,-\}}"]&\Gamma_p\times\Gamma_p\arrow[l, hook']\arrow[r, two heads]\arrow[d, dotted]&\Gamma(L')\times\Gamma(L')\arrow[d, dotted, "{\{-,-\}'}"]\\
	\Gamma(L)&\Gamma_p\arrow[l, hook']\arrow[r, two heads]&\Gamma(L')
	\end{tikzcd}
	\end{equation*}
	where $\Gamma_p$ denotes the space of those sections $\lambda\in\Gamma(L)$ such that $\lambda|_S\in p^\ast\Gamma(L')$.
	On the other hand, if there is a binary operation $J'=\{-,-\}'$ on $\Gamma(L')$ satisfying~\eqref{eq:def:Jacobi_reduction}, then it is automatically a Jacobi structure on $L'\to M'$ as well.
	In particular, for arbitrary  $\lambda'\in\Gamma(L')$ and $\lambda\in\Gamma(L)$ with $p^\ast\lambda'=\lambda|_S$, from~\eqref{eq:def:Jacobi_reduction} it follows that $\Delta_{\lambda'}:=\{\lambda',-\}'$ is a derivation of $L'\to M'$ with symbol $X_{\lambda'}\in\frakX(M')$, and they are completely determined by
	\begin{equation*}
	p^\ast(\Delta_{\lambda'}\mu')=(\Delta_\lambda \mu)|_S,\qquad\underline{\smash{p}}^\ast(X_{\lambda'}f')=(X_\lambda f)|_S,
	\end{equation*}
	for all $\mu\in\Gamma(L)$, $\mu'\in\Gamma(L')$, and $f\in C^\infty(M)$, $f'\in C^\infty(M')$ such that $p^\ast\mu'=\mu|_S$ and $\underline{\smash{p}}^\ast f'=f|_S$, i.e.~\eqref{eq:prop:Jacobi_reduction} is satisfied.
\end{proof}

\begin{corollary}
	\label{cor:Jacobi_reduction}
	\leavevmode
	\begin{enumerate}[label=(\alph*)]
		\item
		\label{enumitem:cor:Jacobi_reduction_a} 
		The submanifold $S$ is Jacobi reducible via the identity map $id_\ell:\ell\to\ell$ if and only if $S$ is a Jacobi submanifold of $(M,L,J)$.
		\item
		\label{enumitem:cor:Jacobi_reduction_b}
		If $S$ is Jacobi reducible, then it is a coisotropic submanifold of $(M,L,J)$.
	\end{enumerate}
\end{corollary}

\begin{proof}
	(a) In Proposition~\ref{prop:Jacobi_reduction}, with $M'=S$, $L'=\ell$, and $p=\id_\ell$, condition~\ref{enumitem:prop:Jacobi_reduction_1} is empty, while condition~\ref{enumitem:prop:Jacobi_reduction_2} becomes $X_\lambda|_S\in\frakX(S)$, for all $\lambda\in\Gamma(L)$, i.e.~$S$ is a Jacobi submanifold of $(M,L,J)$.
	
	(b) Condition~\ref{enumitem:prop:Jacobi_reduction_2} from Proposition~\ref{prop:Jacobi_reduction} implies that $X_\lambda|_S\in\frakX(S)$, in particular, for all $\lambda\in\Gamma(L)$ such that $\lambda|_S=0$, i.e.~$S$ is a coisotropic submanifold of $(M,L,J)$.
\end{proof}

Adapting our initial setting, we will assume now, till the end of the section, that the submanifold $S$ is coisotropic in $(M,L,J)$.
Denote by $I_S\subset C^\infty(S)$ the ideal of functions vanishing on $S$ and by $\Gamma_S\subset\Gamma(L)$ the $C^\infty(M)$-submodule and Lie subalgebra of sections vanishing on $S$.
Define the associative subalgebra $N(I_S)\subset C^\infty(M)$, and the Lie subalgebra $N(\Gamma_S)\subset\Gamma(L)$ by setting
\begin{align*}
N(\Gamma_S)&:=\{\nu\in\Gamma(L)\colon\{\Gamma_S,\nu\}\subset\Gamma_S\},\\
N(I_S)&:=\{f\in C^\infty(M)\colon X_\lambda(f)\in I_S,\ \textnormal{for all}\ \lambda\in\Gamma_S\}.
\end{align*}
Clearly, $N(\Gamma_S)$ is the normalizer of $\Gamma_S$ in $\Gamma (L)$, and consists of those sections $\lambda\in\Gamma(L)$ such that $X_\lambda$ is tangent to $S$.
Moreover $N(I_S)$ consists of those functions $f\in C^\infty(M)$ which are constant along the leaves of $\calF_S$.

The pair $(C^\infty(M_{\textnormal{red}}),\Gamma(L_{\textnormal{red}})):=(N(I_S)/I_S,N(\Gamma_S)/\Gamma_S)$ admits an obvious structure of Gerstenhaber-Jacobi algebra (concentrated in degree $0$), that we call the \emph{reduced Gerstenhaber-Jacobi algebra of $S$}.
The latter is morally the Gerstenhaber--Jacobi algebra of the ``singular'' Jacobi manifold $(M_{\textnormal{red}},L_{\textnormal{red}},\{-,-\}_{\textnormal{red}})$ obtained by performing a singular reduction of $S$ wrt its characteristic foliation $\calF_S$.

\begin{remark}
	\label{rem:linfty_cohom_resolution}
	For future reference let us point out here that, if $S$ is closed, then there exists a canonical module isomorphism $\phi:\Gamma(L_{\textnormal{red}})\to H^0(N_\ell{}^\ast S,\ell)$, covering an algebra isomorphism $\underline{\smash\phi}:C^\infty(M_{\textnormal{red}})\to H^0(N_\ell{}^\ast S)$, defined by $\underline{\smash\phi}(f+I_S)=[f|_S]$, and $\phi(\lambda+\Gamma_S)=[\lambda|_S]$, for all $f\in N(I_S)$, and $\lambda\in N(\Gamma_S)$.
\end{remark}

In general the converse of the implication in Corollary~\ref{cor:Jacobi_reduction}(b) does not hold, that is a coisotropic submanifold generically is not Jacobi reducible.
However a generically non-trivial sufficient condition for the Jacobi reducibility of a coisotropic submanifold is pointed out by the next proposition.
\begin{proposition}
	\label{prop:reduction_of_coisotropic_submanifolds}
	Assume that $S$ is coisotropic in $(M,L,J)$.
	Let $L'\to M'$ be a line bundle, and $p:\ell\to L'$ be a regular line bundle morphism, covering a surjective submersion $\underline{\smash{p}}:S\to M'$.
	If the fibers of $\underline{\smash{p}}:S\to M'$ are connected, and
	\begin{equation}
	\label{eq:prop:reduction_of_coisotropic_submanifolds_1}
	\ker(Dp)=(\hat{\Lambda}_J^\sharp\circ\gamma)(T^0S\otimes\ell),
	\end{equation}
	then $S$ is Jacobi reducible via $p$.
	Additionally, denoting by $(M',L',J'=\{-,-\}')$ the reduced Jacobi manifold, there is a Gerstenhaber--Jacobi algebra isomorphism
	\begin{equation}
	\label{eq:prop:reduction_of_coisotropic_submanifolds_2}
	(C^\infty(M_{\textnormal{red}}),\Gamma(L_{\textnormal{red}}),\{-,-\}_{\textnormal{red}})\simeq(C^\infty(M'),\Gamma(L'),\{-,-\}'),
	\end{equation}
	such that $f\Mod I_S\simeq f'$ and $\lambda\Mod\Gamma_S\simeq\lambda'$ iff $f|_S=p^\ast f'$ and $\lambda|_S=p^\ast\lambda'$.
\end{proposition}

\begin{proof}
	Applying the symbol map $\sigma$ to both sides of~\eqref{eq:prop:reduction_of_coisotropic_submanifolds_1}, and because of the skew-symmetry of $\Lambda_J^\sharp$, we get $\Lambda_J^\sharp((\ker T\underline{\smash{p}})^0\otimes\ell)\subset TS$, which is just a rewording of condition~\ref{enumitem:prop:Jacobi_reduction_2} in Proposition~\ref{prop:Jacobi_reduction}.
	Going further, equation~\eqref{eq:prop:reduction_of_coisotropic_submanifolds_1} means that the sub-bundle $\ker Dp\subset(DL)|_S$ is generated by $\Delta_\lambda|_S$, for all $\lambda\in\Gamma_S$.
	As a consequence, since the fibers of $\underline{\smash{p}}$ are connected, it turns out, for all $\lambda\in\Gamma(L)$, that $\lambda|_S\in p^\ast\Gamma(L')$ if and only if $\lambda\in N(\Gamma_S)$.
	Therefore even condition~\ref{enumitem:prop:Jacobi_reduction_1} in Proposition~\ref{prop:Jacobi_reduction} is verified, so that $S$ is Jacobi reducible via $p$.
	The remaining part of the proof follows immediately.
\end{proof}

\begin{remark}
	\label{rem:Jacobi_reduction}
	Jacobi reduction of coisotropic submanifolds is a special instance of a wider reduction scheme: the Marsden--Ratiu reduction of Jacobi manifolds.
	The latter allows, under suitable hypotheses, to construct a Jacobi structure even on quotients of submanifolds which are not necessarily coisotropic submanifolds of the ambient Jacobi manifold.
	Moreover many interesting applications of Jacobi reduction emerge in presence of a Lie group acting on a Jacobi manifold by Jacobi automorphisms.
	More details about reduction of Jacobi manifolds can be found in~\cite{dacosta1989reduction} and~\cite{dacosta1990reduction}.
\end{remark}

\nocite{vitagliano2016holomorphic,LTV,LOTV}

\chapter{The \texorpdfstring{$L_\infty[1]$}{L∞[1]}-algebra of a coisotropic submanifold}
\label{chap:L-infinity-algebra}

In this chapter, aiming to control the coisotropic deformation problem, a natural $L_\infty$-isomorphism class of $L_\infty[1]$-algebras is associated with each coisotropic submanifold $S$ of a Jacobi manifold $(M,L,J)$.

For any choice of a fat tubular neighborhood (Definition~\ref{def:fat_tubular_neighborhod}), an $L_\infty[1]$-algebra is associated with $S$ (Proposition~\ref{prop:linfty}) by means of Th.~Voronov's technique of higher derived brackets~\cite{Voronov2005higher1}.
Such $L_\infty[1]$-algebra is an enrichment of the de Rham complex of the Jacobi algebroid of $S$, and provides a cohomological resolution of the reduced Gerstenhaber--Jacobi algebra of $S$ (Proposition~\ref{prop:linfty_cohomological_resolution}).
Such $L_\infty[1]$-algebra is proven to be independent, up to $L_\infty$-isomorphisms, from the chosen fat tubular neighborhood (Proposition~\ref{prop:gauge_invariance}) extending to the Jacobi setting the argument given by Cattaneo and Sch\"atz~\cite{cattaneo2008equivalences} in the Poisson setting.

For a coisotropic submanifold $S$ of a Jacobi manifold, analysing the notions of coisotropic deformations (Proposition~\ref{prop:coiss}) and Hamiltonian equivalence of coisotropic deformations (Proposition~\ref{prop:Hameq}), we are lead to introduce the definitions of formal coisotropic deformations (Definition~\ref{def:coisoformal}) and Hamiltonian equivalence of formal co\-iso\-tro\-pic deformations (Definition~\ref{def:formHeq}).
We prove that the $L_\infty[1]$-algebra of $S$ controls its formal coisotropic deformation problem, even under Hamiltonian equivalence.
Namely there is a one-to-one correspondence between formal coisotropic deformations of $S$ and (degree $0$) formal Maurer--Cartan elements
of the associated $L_\infty[1]$-algebra (Proposition~\ref{prop:mcformal}), which moreover intertwines Hamiltonian equivalence of formal coisotropic deformations with gauge equivalence of the corresponding formal Maurer--Cartan elements (Proposition~\ref{prop:MC_gauge}).
As a consequence of this fact there are criteria for the unobstructedness (Corollary~\ref{prop:rigid}) and the obstructedness (Proposition~\ref{prop:Kuranishi}) of an infinitesimal coisotropic deformation of $S$. 
We also give a necessary and sufficient condition for the convergence of the formal Maurer--Cartan series $MC(s)$ for any smooth section $s$ (Proposition~\ref{prop:fana}), extending a previous sufficient condition given by Sch\"atz and Zambon in~\cite{SZ2012}.
As a consequence of the latter, we prove that, when the Jacobi structure $J$ is fiber-wise entire, the $L_\infty[1]$-algebra of $S$ controls also its (non-formal) coisotropic deformation problem even under Hamiltonian equivalence (Corollaries~\ref{cor:conver} and~\ref{cor:hequi}).

Finally, for any choice of a fat tubular neighborhood, the associated $L_\infty[1]$-algebra of $S$ is extended to a larger $L_\infty[1]$-algebra (Proposition~\ref{prop:extended_linfty}) again by higher derived brackets.
We prove that the extended $L_\infty[1]$-algebra of $S$ controls, at the formal level, the problem of deforming simultaneously $J$ into a new Jacobi structure $J'$ on $L\to M$ and $S$ into a new submanifold $S'$ which is now coisotropic wrt $J'$.
Namely there is a one-to-one correspondence between formal simultaneous coisotropic deformations of $(J,S)$ and (degree $0$) formal Maurer--Cartan elements
of the extended $L_\infty[1]$-algebra (Proposition~\ref{prop:mc_formal_simultaneous}).
The latter also encodes the infinitesimal simultaneous coisotropic deformations of $(J,S)$.
As a consequence there are criteria for the unobstructedness (Corollary~\ref{prop:simultaneous_rigid}) and the obstructedness (Proposition~\ref{prop:simultaneous_Kuranishi}) of an infinitesimal simultaneous coisotropic deformation of $(J,S)$. 


%

\section{Fat tubular neighborhoods}
\label{sec:fat_tubular_nghb}

Let us start with some heuristic considerations.
As it is well-known (see, e.g., \cite{hamilton1982inverse}), for any manifold $M$, there is a natural structure of Fréchet manifold on the space $\calS(M)$ of all compact submanifolds of $M$.
Let $S\subset M$ be a submanifold.
Locally, around $S$, the space $\calS(M)$ is modelled on the linear space $\Gamma(NS)$ of sections of the normal bundle $NS\to S$.
Recall that a \emph{tubular neighborhood} of $S$ in $M$ consists of an open embedding $\underline{\smash{\tau}}:NS\to M$ identifying the zero section of the normal bundle with the inclusion map $S\to M$.
Then a tubular neighborhood $\underline{\smash{\tau}}$ of $S$ in $M$ determines the local coordinate chart of $\calS(M)$, centered at $S$, which identifies any section $s\in\Gamma(NS)$ with the image of $\underline{\smash{{\tau}}}\circ s$.
Given a Jacobi manifold $(M,L,J)$, the global description of the space of all coisotropic submanifolds is out of reach, but we can still hope to describe its local models around a given coisotropic submanifold $S$.
However it is not enough to work within a tubular neighborhood of $S$ in $M$ because we have also to take care of the line bundle $L\to M$.
This leads us to introduce the notion of \emph{fat tubular neighborhood}.

Let $L\to M$ be a line bundle and $S\subset M$ be a submanifold, with restricted line bundle $\ell:=L|_S\to S$ and normal bundle $\pi:NS\to S$.
Denote by $i:S\to M$ and $i_L:\ell\to L$ the inclusion maps, and consider the pull-back line bundle $L_{NS}:=\pi^\ast\ell=NS\times_{S}\ell\to NS$.

\begin{definition}
	\label{def:fat_tubular_neighborhod}
	A \emph{fat tubular neighborhood} $(\tau,\underline{\smash{\tau}})$ of $\ell\to S$ in $L\to M$ consists of two layers:
	\begin{itemize}
		\item a tubular neighborhood $\underline{\smash{\tau}}$ of $S$ in $M$, and 
		\item a regular line bundle morphism $\tau:L_{NS}\to L$, covering $\underline{\smash{\tau}}$,
	\end{itemize}
	such that the following diagram commutes
	\begin{equation}
	\label{eq:fat_tubular_neighborhod}
		\begin{tikzcd}
			L_{NS} = \pi^\ast \ell \ar[rr, "\tau"]\arrow[dd]& &L\arrow[dd]\\
			&\ell\arrow[ul]\ar[ur, swap, "i_L"]&\\
			NS\ar[dr, "\pi", shift left=0.7 ex]\arrow[rr, "\underline{\smash{\tau}}"]& & M \\
			& S\arrow[from=uu, crossing over] \ar[ul, "\mathbf 0", shift left=0.7 ex]\arrow[ur, swap, "i"]&
		\end{tikzcd}
	\end{equation}
\end{definition}
A fat tubular neighborhood can be understood as a ``tubular neighborhood in the category of line bundles''.
As already for tubular neighborhoods, there are existence and uniqueness results for fat tubular neighborhoods as well.
We need a preliminary lemma.

\begin{lemma}
	\label{lem:fat_tubular_neighborhood}
	Let $q:A\to N$ be a vector bundle.
	Regard $N$ as a submanifold of $A$ identifying it with the image of the zero section $\mathbf 0 :N\to A$.
	For any vector bundle $B\to A$, with restricted vector bundle $B_N:=B|_N\to N$, there exists a (non-canonical) vector bundle isomorphism $\phi$ from $B\to A$ to $q^\ast B_N=A\times_N B_N\to A$, covering the identity map of $A$, which agrees with the identity map on $ B_N$.
\end{lemma}

\begin{proof}
	It is a straightforward consequence of the fact that, if we pull-back the vector bundle $B\to A$ along the homotopic maps $\mathbf 0\circ q:A\to A$ and $\id_A$, then we obtain isomorphic vector bundles.
\end{proof}

In the following we regard $S$ as a submanifold of $NS$ identifying it with the image of the zero section $\mathbf 0 : S \to NS$.

\begin{proposition}[Existence]
	There exist fat tubular neighborhoods of $\ell$ in $L$.
\end{proposition}

\begin{proof}
	Let $\underline \tau : NS \injects M$ be a tubular neighborhood of $S$.
	According to Lemma~\ref{lem:fat_tubular_neighborhood}, the pull-back bundle $\underline \tau {}^\ast L \to NS$ is (non-canonically) isomorphic to $L_{NS}$.
	Pick any isomorphism $\phi : L_{NS} \to \underline \tau {}^\ast L $.
	Then the composition
	\[
	L_{NS} \overset{\phi}{\longrightarrow} \underline \tau {}^\ast L \longrightarrow L,
	\]
	where the second arrow is the canonical map, is a fat tubular neighborhood of $\ell$ over $\underline \tau$.
\end{proof}

\begin{proposition}[Uniqueness]
	\label{prop:isotopy}
	Any two fat tubular neighborhoods $ \tau{}_0$ and $\tau{}_1$ of $S$ are \emph{isotopic}, i.e.~there is a smooth one parameter family of fat tubular neighborhoods $\mathcal T{}_t $ of $\ell$ in $L$, and an automorphism $\psi : L_{NS} \to L_{NS}$ of $L_{NS}$ covering an automorphism $ \underline \psi : NS \to NS$ of $NS$ over the identity, such that $\mathcal T_0 = \tau_0$, and $\mathcal T_1 = \tau_1 \circ \psi$.
\end{proposition}

\begin{proof}
	In view of the tubular neighborhood Theorem~\cite[Theorem 5.3]{hirsch}, there is a smooth one parameter family of tubular neighborhoods $\underline{\mathcal T}{}_t : NS \injects M$ of $S$ in $M$, and an automorphism $\underline \psi : NS \to NS$ over the identity such that $\underline{\mathcal T}{}_0 = \underline{\tau}{}_0$, and $\underline{\mathcal T}{}_1 = \underline \tau{}_1 \circ \underline \psi $.
	Denote by $\underline{\mathcal T}:NS\times[0,1] \to M$ the map defined by $\underline{\mathcal T} (\nu,t)=\underline{\mathcal T}{}_t(\nu)$ and consider the line bundle
	\[
	p : L_{NS}^\ast \otimes_{NS} \underline{\mathcal T}^\ast L \longrightarrow NS \times [0,1].
	\]
	Note that $L_{NS}^\ast \otimes_{NS} \underline{\mathcal T}^\ast L $ reduces to $\operatorname{End} \ell \times [0,1] = \bbR_{S \times [0,1]}$ over $S \times [0,1]$.
	Hence, according to Lemma~\ref{lem:fat_tubular_neighborhood}, $L_{NS}^\ast \otimes_{NS} \underline{\mathcal T}^\ast L$ is isomorphic to the pull-back over $NS \times [0,1]$ of the trivial line bundle $\bbR_{S \times [0,1]}$ over $S \times [0,1]$.
	In particular, $p$ is a trivial bundle.
	Moreover, $p$ admits a nowhere zero section $\upsilon$ defined on $(S\times[0,1])\cup(NS\times\{0,1\})$ and given by $\operatorname{id}_{\ell} $ on $S\times[0,1]$, by ${\mathcal T}_0$ on $NS\times\{0\}$ and by ${\mathcal T}_1$ on $NS \times \{ 1 \}$.
	By triviality, $\upsilon$ can be extended to a nowhere zero section $\Upsilon$ on the whole $NS \times [0,1]$.
	Section $\Upsilon$ is the same as a one parameter family of vector bundle isomorphisms $\Upsilon_t : L_{NS} \to \underline{\mathcal T}{}_t^\ast L$ over the identity of $NS$.
	Denote by ${\mathcal T}_t : L_{NS} \to L$ the composition
	\[
	L_{NS} \overset{\Upsilon_t}{\longrightarrow} \underline{\mathcal T}{}_t^\ast L \lhook\joinrel\longrightarrow L,
	\]
	where the second arrow is the natural inclusion.
	By construction, the ${\mathcal T}_t$'s are line bundle embeddings covering the $\underline{\mathcal T}{}_t$'s.
	Finally, there exists a unique automorphism $\psi : L_{NS} \to L_{NS}$ over $\underline \psi$ such that ${\mathcal T}_1 = \tau_1 \circ \psi$.
	We conclude that the ${\mathcal T}_t$'s and $\psi$ possess all the required properties.
\end{proof}

\section{\texorpdfstring{$L_\infty[1]$}{L∞[1]}-algebra associated with a coisotropic submanifold}
\label{sec:linfty_algebra}

This section aims to show that, for a closed coisotropic submanifold of a Jacobi manifold, any choice of a fat tubular neighborhood determines a set of \emph{V-data} (as defined by Frégier--Zambon~\cite{fregier2014simultaneous}).
Hence an $L_\infty[1]$-algebra is constructed out of this set of V-data through Th.~Voronov's technique~\cite{Voronov2005higher1} of higher derived brackets.
Notice that our conventions about $L_\infty$ and $L_\infty[1]$-algebras are the same as those in~\cite{fregier2014simultaneous}, so that, in particular, the multi-brackets in an $L_\infty[1]$-algebra are graded symmetric of degree $1$.

Let $(M,L,J=\{-,-\})$ be a Jacobi manifold, and $S \subset M$ be a closed submanifold.
There exists a unique degree $0$ graded module morphism $P:\Der^\bullet L\longrightarrow\Gamma(\wedge^\bullet N_\ell S\otimes\ell)$, covering a degree $0$ graded algebra morphism $\underline{\smash{P}}:\Gamma(\wedge^\bullet(J^1L)^\ast)\longrightarrow\Gamma(N_\ell S)$, which is completely determined by
\begin{equation}
	\label{eq:projection_P}
	P(\lambda)=\lambda|_S,\qquad P(\square)=\sigma_\square|_S\Mod TS,
\end{equation}
for all $\lambda\in\Gamma(L)=\Der^0L$ and $\square\in\Der L=\Der^1L$.
Now, it is not hard to see that, for all $k\geq 0$, and $\square\in\Der^kL$, the projection $P\square\in\Gamma(\wedge^k N_\ell S\otimes\ell)$ is explicitly given by
\begin{equation}
	\label{eq:projection_P_bis}
	(P\square)(\alpha_1,\ldots,\alpha_k)=\langle\square,\gamma(\alpha_1)\wedge\ldots\wedge\gamma(\alpha_k)\rangle|_S,
\end{equation}
for all $\alpha_1,\ldots,\alpha_k\in\Gamma(N_\ell{}^\ast S)$.
As in the Poisson case (see, e.g., \cite{cattaneo2008equivalences}), projection $P:(\Der^\bullet L)\to\Gamma(\wedge^\bullet N_\ell S\otimes\ell)$ allows to formulate a further characterization of coisotropic submanifolds.
\begin{proposition}
	\label{prop:PJ=0}
	\leavevmode
	\begin{enumerate}[label=(\arabic*)]
		\item
		\label{enumitem:prop:PJ=0_1}
		$\ker P$ is a graded Lie subalgebra of $(\Der^\bullet(L)[1],\ldsb-,-\rdsb)$.
		\item
		\label{enumitem:prop:PJ=0_2}
		$S$ is coisotropic in $(M,L,J)$ iff $P(J)=0$.
	\end{enumerate}
\end{proposition}

\begin{proof}
	Clearly the kernel $\ker P$ is a graded $\bbR$-vector subspace of $\Der^\bullet L$.
	For all $k\geq 0$, and $\square\in\Der^k L$, Equation~\eqref{eq:projection_P_bis} implies that $\square\in\ker P$ iff $\square(\lambda_1,\ldots,\lambda_k)\in\Gamma_S$, for all $\lambda_1,\ldots,\lambda_k\in\Gamma_S$, where $\Gamma_S\subset\Gamma(L)$ denotes the submodule of sections vanishing on $S$.
	So we immediately get both condition~\ref{enumitem:prop:PJ=0_2} (by means of Lemma~\ref{lem:cois}) and condition~\ref{enumitem:prop:PJ=0_1} (by means of the expression~\eqref{eq:SJ_bracket_and_Gerstenhaber_product} of the Schouten--Jacobi bracket in terms of the Gerstenhaber product~\eqref{eq:Gerstenhaber_product}).
\end{proof}

\begin{corollary}
	\label{cor:PJ=0}
	If $S$ is coisotropic in $(M,L,J)$ then $\ker P$ is preserved by $d_J:=\ldsb J,-\rdsb$.
\end{corollary}

\begin{remark}
	\label{rem:PJ}
	Let $S \subset M$ be any submanifold, then $P (J)$ does only depend on the bi-symbol $\Lambda_J$ of $J$.
	To see this, note, first of all, that the symbol $\sigma : \Der  L \to \mathfrak X (M)$ induces an obvious projection $ (\Der^\bullet L) \to \Gamma (\wedge^\bullet (TM \otimes L^\ast) \otimes L)$.
	Moreover, in view of its very definition, $P :  (\Der^\bullet L)[1] \to \Gamma (\wedge^\bullet N_\ell S \otimes \ell)[1]$ descends to an obvious projection
	\[
	\Gamma (\wedge^\bullet (TM \otimes L^\ast) \otimes L)[1] \longrightarrow \Gamma (\wedge^\bullet N_\ell S \otimes \ell)[1],
	\]
	which, abusing the notation, we denote again by $P$.
	Now, recall that $\Lambda_J \in \Gamma (\wedge^2 (TM \otimes L^\ast) \otimes L)$.
	It immediately follows from the definition of $P$ that, actually,
	\[
	P(J) = P(\Lambda_J).
	\]
	In particular $S$ is coisotropic iff $P (\Lambda_J)=0$.
\end{remark}

Choose once and for all a fat tubular neighborhood $(\tau,\underline{\smash{\tau}})$  of $\ell\to S$ in $L\to M$.
We identify $NS$ with the open neighborhood $\underline{\smash{\tau}}(NS)$ of $S$ in $M$.
Similarly, we identify $L_{NS}$ with $L|_{\underline{\smash{\tau}}(NS)}$.
In particular the line bundle $L_{NS}\to NS$ inherits a Jacobi structure from $L\to M$ by means of pull-back via $\tau$.
Abusing the notation we denote by $J$ again the Jacobi bracket on $\Gamma (L_{NS})$.
Moreover, in view of Proposition~\ref{prop:PJ=0}, there is a projection $P : (\Der^\bullet L_{NS})[1] \to \Gamma (\wedge^\bullet N_\ell S \otimes \ell)[1]$, with $\ldsb\ker P,\ker P\rdsb\subset\ker P$, such that $S$ is coisotropic iff $P(J)=0$.

Now, regard the vertical bundle $V (NS) := \ker d\pi$ as a Lie algebroid and note preliminarily that
\begin{enumerate}
	\item
	\label{enumitem:natural_splitting}
	There is a natural splitting $T(NS)|_S = TS \oplus NS$: projection $T(NS)|_S \to TS$ is $d \pi$, while projection $T (NS)|_S \to NS$ is the natural one.
	In particular, sections of $NS$ can be understood as vector fields on $NS$ along the submanifold $S$ and vertical wrt~$\pi$.
	\item Since $\pi : NS \to S$ is a vector bundle, the vertical bundle $V(NS)$ identifies canonically with the induced bundle $\pi^\ast NS \to NS$.
	In particular, there is an embedding $\pi^\ast : \Gamma (NS) \injects \frakX (NS)$ that takes a section $\nu$ of $NS$ to the unique vertical vector field $\pi^\ast \nu$ on $NS$, which is constant along the fibers of $\pi$, and agrees with $\nu$ on $S$.
	\item Since $L_{NS} = \pi^\ast \ell = NS \times_S \ell$, there is a natural flat connection $\mathbb D$ in {$L_{NS}$}, along the Lie algebroid $V(NS)$, uniquely determined by $\mathbb D_X \pi^\ast \lambda = 0$, for all vertical vector fields $X$ on $NS$, and all fiber-wise constant sections $\pi^\ast \lambda$ of $L_{NS}$, $\lambda \in \Gamma (\ell)$.
\end{enumerate}

After these preliminary remarks it is easy to see that there exists a unique degree $0$ graded module morphism $I:\Gamma(\wedge^\bullet N_\ell S\otimes\ell)\to\Der^\bullet(L_{NS})$, covering a degree $0$ graded algebra morphism $\underline{\smash{I}}:\Gamma(\wedge^\bullet N_\ell S)\to\Gamma(\wedge^\bullet(J^1L_{NS})^\ast)$, which is completely determined by
\begin{equation}
\label{eq:embedding_I}
I(\lambda)=\pi^\ast\nu,\qquad I(\nu)=\bbD_{\pi^\ast\nu},
\end{equation}
for all $\lambda\in\Gamma(\ell)=\Gamma(\wedge^0 N_\ell S\otimes\ell)$ and $\nu\in\Gamma(NS)=\Gamma(\wedge^1 N_\ell S\otimes\ell)$.

\begin{proposition}
	\label{prop:embeddingI}
	\leavevmode
	\begin{enumerate}[label=(\arabic*)]
		\item
		\label{enumitem:prop:embeddingI_1}
		$\im I$ is an abelian graded Lie subalgebra of $(\Der^\bullet(L_{NS})[1],\ldsb-,-\rdsb)$.
		\item
		\label{enumitem:prop:embeddingI_2}
		$I$ is a right inverse of $P$, i.e.~$P\circ I=\id$ on $\Gamma(\wedge^\bullet N_\ell S\otimes\ell)$.
	\end{enumerate}
\end{proposition}

\begin{proof}
	(1) Since the Schouten--Jacobi bracket $\ldsb-,-\rdsb$ satisfies the (generalized) Leibniz rule~\eqref{eq:genleib2}, and $I$ is a graded module morphism covering the graded algebra morphism $\underline{\smash{I}}$, it is enough to check that $\ldsb I(\alpha),I(\beta)\rdsb$ vanishes for all $\alpha,\beta\in\Gamma(\ell)\cap\Gamma(NS)$.
	Indeed the latter is a consequence of the very definitions of $I$ and $\ldsb-,-\rdsb$, and the fact that any two fiber-wise constant vertical vector fields on $NS$ commute.
	
	(2) Since $P\circ I:\Gamma(\wedge^\bullet N_\ell S\otimes\ell)\to\Gamma(\wedge^\bullet N_\ell S\otimes\ell)$ is a graded module morphism, covering the graded algebra morphism $\underline{\smash{P}}\circ \underline{\smash{I}}:\Gamma(\wedge^\bullet N_\ell S)\to\Gamma(\wedge^\bullet N_\ell S)$, it is enough to check that $P\circ I$ agrees with the identity map on $\Gamma(\ell)\cap\Gamma(NS)$.
	Indeed this is exactly the case because of~\eqref{eq:projection_P} and~\eqref{eq:embedding_I}.
\end{proof}
%

According to Frégier--Zambon~\cite[Definition~1.7]{fregier2014simultaneous} a \emph{set of V-data} is a quadruple $(\calL,\fraka,P,\Delta)$ where:
\begin{itemize}
	\item $\calL$ is a graded Lie algebra, with Lie bracket denoted by $[-,-]$, 
	\item $\fraka\subset\calL$ is an abelian Lie subalgebra,
	\item $P:\calL\to\fraka$ is a projection such that $\ker P\subset\calL$ is a Lie subalgebra,
	\item $\Delta\in\ker P$ is a Maurer--Cartan element of $\calL$, i.e.~$|\Delta|=1$ and $[\Delta,\Delta]=0$.
\end{itemize} 
From now on we assume that $S$ is coisotropic.
In this case, as summarized by the next Lemma~\ref{lem:linfty}, a set of V-data is singled out by Propositions~\ref{prop:PJ=0} and~\ref{prop:embeddingI}.

\begin{lemma}
	\label{lem:linfty}
	The quadruple $((\Der^\bullet L_{NS})[1],\im I,P,J)$ is a set of V-data
\end{lemma}

An $L_\infty[1]$-algebra structure on $\Gamma(\wedge^\bullet N_\ell S\otimes\ell)[1]$ can be constructed out of this set of V-data via higher derived brackets construction~\cite{Voronov2005higher1}.

\begin{proposition}
	\label{prop:linfty}
	There is an $L_\infty[1]$-algebra structure on $\Gamma (\wedge^\bullet N_\ell S \otimes \ell)[1]$ with multi-brackets $\frakm _k:\Gamma (\wedge^\bullet N_\ell S \otimes \ell)[1]^{\otimes k} \to \Gamma (\wedge^\bullet N_\ell S \otimes \ell)[1]$ given by the following higher derived brackets
	\begin{equation}
	\label{eq:higher_derived_brackets}
	\frakm_k (\xi_1,\ldots,\xi_k): =P\ldsb\ldsb\ldots\ldsb J, I(\xi_1)\rdsb,\ldots\rdsb,I(\xi_k)\rdsb,
	\end{equation}
	for all $k>0$, and $\xi_1,\ldots,\xi_k\in\Gamma(\wedge^\bullet N_\ell S\otimes\ell)[1]$.
\end{proposition}

A version of Proposition~\ref{prop:linfty} for Poisson manifolds is well-known (see~\cite{cattaneo2007relative}, and also~\cite[Lemma 2.2]{fregier2014simultaneous}).

%
%

\begin{remark}
	\label{rem:multi-brackets}
	By their very definition~\eqref{eq:higher_derived_brackets}, the multi-brackets $\mathfrak m_k$ of the $L_\infty[1]$-algebra satisfy the following properties:
	\begin{enumerate}[label=(\alph*)]
		\item
		\label{enumitem:rem:multi-brackets_a}
		$\mathfrak m_k$ is a degree $1$ graded-symmetric $\bbR$-linear map,
		\item
		\label{enumitem:rem:multi-brackets_b}
		$\mathfrak m_k$ is a graded derivation of the graded line bundle $\scrL\to\scrM$, in each entry, with $C^\infty(\scrM)=\Gamma(\wedge^\bullet N_\ell S)$ and  $\Gamma(\scrL)=\Gamma(\wedge^\bullet N_\ell S\otimes\ell)$.
	\end{enumerate}
	As a consequence, the constructed $L_\infty[1]$-algebra structure can be seen as a \emph{Jacobi structure up to homotopy} on $\scrL\to\scrM$ (cf.~\cite{brucekirillov2016} where these structures have been introduced and studied for the first time).
\end{remark}

The $L_\infty[1]$-algebra provides a cohomological resolution of $S$, and its unary bracket $\frakm_1$ has an intrinsic meaning in the sense of the next proposition.
 
\begin{proposition}
	\label{prop:linfty_cohomological_resolution}
	\leavevmode
	\begin{enumerate}[label=(\arabic*)]
		\item
		\label{enumitem:prop:linfty_cohomological_resolution_1}
		The unary bracket $\frakm_1$ coincides with de Rham differential $d_{N_\ell {}^\ast S, \ell}$ of the Jacobi algebroid $(N_\ell {}^\ast S, \ell)$.
		\item
		\label{enumitem:prop:linfty_cohomological_resolution_2}
		The binary bracket $\frakm_2$ induces on $(H^\bullet(N_\ell S),H^\bullet(N_\ell S,\ell))$, up to décalage isomorphism, a structure of graded Gerstenhaber--Jacobi algebra.
		\item
		\label{enumitem:prop:linfty_cohomological_resolution_3}
		The degree $0$ component of $(H^\bullet(N_\ell S),H^\bullet(N_\ell S,\ell),\frakm_2)$ is canonically isomorphic to the reduced Gerstenhaber--Jacobi of $S$, i.e.
		\begin{equation*}
		(H^0(N_\ell S),H^0(N_\ell S,\ell),\frakm_2)\simeq (C^\infty(M_{\textnormal{red}}),\Gamma(L_{\textnormal{red}}),\{-,-\}_{\textnormal{red}}).
		\end{equation*}
	\end{enumerate}	
\end{proposition}

\begin{proof}
	(1) Propositions~\ref{prop:conormal} and~\ref{prop:embeddingI}~\ref{enumitem:prop:embeddingI_2} imply that, for all $\alpha\in\Gamma(\wedge^\bullet N_\ell S\otimes\ell)$,
	\begin{equation}
	\label{eq:d*2}
	\frakm_1(\alpha):=P\ldsb J,I(\alpha)\rdsb=(P\circ d_J\circ I)(\alpha)=d_{N_\ell{}^\ast S,\ell}(\alpha).
	\end{equation}
	(2) It is a straightforward consequence of Remark~\ref{rem:multi-brackets}.
	\newline\noindent
	(3) Recall from Remark~\ref{rem:linfty_cohom_resolution} that there exists a module isomorphism $\phi:\Gamma(L_{\textnormal{red}})\to H^0(N_\ell{}^\ast S,\ell)$, covering an algebra isomorphism $\underline{\smash{\phi}}:C^\infty(M_{\textnormal{red}})\to H^0(N_\ell{}^\ast S)$, which is canonically defined by $\phi(\lambda\Mod\Gamma_S)=[\lambda|_S]$, for all $\lambda\in N(\Gamma_S)$.
	Now it is easy to check that
	\begin{equation*}
	\phi(\{\lambda\Mod\Gamma_S,\mu\Mod\Gamma_S\}_{\textnormal{red}})=\left[\{\pi^\ast(\lambda|_S),\pi^\ast(\mu|_S)\}|_S\right]=\frakm_2(\phi(\lambda),\phi(\mu)),
	\end{equation*}
	for all $\lambda,\mu\in N(\Gamma_S)$.
	This concludes the proof.
\end{proof}

\section{Coordinate formulas for the multi-brackets}
\label{sec:multi-brackets_coordinates}

In this section we propose some more efficient formulas for the multi-brackets in the $L_\infty[1]$-algebra of a coisotropic submanifold.
Let $(M, L, J = \{-,-\})$ be a Jacobi manifold and let $S \subset M$ be a coisotropic submanifold.
Moreover, as in the previous section, we equip $S$ with a fat tubular neighborhood $\tau : L_{NS} \injects L$.

	Because of Remark~\ref{rem:multi-brackets}~\ref{enumitem:rem:multi-brackets_b} the $\mathfrak m_k$'s are completely determined by their action on all $\lambda\in\Gamma(\ell) = \Gamma (\wedge^0 N_\ell S \otimes \ell)$, and on all $s \in \Gamma (NS) = \Gamma (\wedge^1 N_\ell S \otimes \ell)$.
	Moreover Remark~\ref{rem:multi-brackets}~\ref{enumitem:rem:multi-brackets_a} implies that, if $\xi_1,\ldots,\xi_k \in \Gamma (\wedge^\bullet N_\ell S \otimes \ell)[1]$ have non-positive degrees, then $\mathfrak m_k (\xi_1, \ldots, \xi_k) = 0$ whenever more than two arguments have degree $-1$.

From now on, in this section, we identify
\begin{itemize}
	\item a section $\lambda \in \Gamma (\ell)$, with its pull-back $\pi^\ast \lambda \in \Gamma (L_{NS})$ (as already in Section~\ref{sec:linfty_algebra}),
	\item a section $s \in \Gamma (NS)$, with the corresponding vertical vector field $\pi^\ast s \in \Gamma (\pi^\ast NS) \simeq \Gamma (V (NS))$ (as already in Section~\ref{sec:linfty_algebra}),
	\item a section $\varphi \in \Gamma (N_\ell {}^\ast S)$ of the $\ell$-adjoint bundle $N_\ell {}^\ast S = N^\ast S \otimes \ell$ with the corresponding fiber-wise linear section of $L_{NS}$.
\end{itemize}
Moreover, we denote by $\langle-,-\rangle : NS \otimes N_\ell {}^\ast S \to \ell$ the obvious ($\ell$-twisted) duality pairing.

\begin{proposition}
	\label{prop:CF}
	The multi-bracket $\mathfrak m_{k+1}$ is completely determined by
	\begin{equation}\label{eq:CF1}
	\frakm_{k+1}(s_1, \ldots, s_{k-1}, \lambda,\nu )  =(-)^{k} \bbD_{s_1} \cdots \bbD_{s_{k-1}} \{ \lambda , \nu \} |_S,
	\end{equation}
	\begin{equation}\label{eq:CF2}
	\begin{aligned}
	& \left \langle \frakm_{k+1} (s_1,\ldots,s_{k},\lambda ) , \varphi \right \rangle \\
	& = -(-)^{k} \left. \left( \bbD_{s_1} \cdots \bbD_{s_{k}} \{ \lambda , \varphi \} - \sum_{i} \bbD_{ s_1} \cdots \widehat{\bbD_{s_{i}}} \cdots \bbD_{s_{k}} \{ \lambda , \langle s_{i}, \varphi \rangle \} \right) \right |_S,
	\end{aligned}
	\end{equation}
	\begin{equation}\label{eq:CF3}
	\begin{aligned}
	&\left \langle \frakm_{k+1} (s_1,\ldots,s_{k+1}) , \varphi \otimes \psi \right \rangle 
	\\ & = -(-)^k \left( \bbD_{s_1} \cdots \bbD_{s_{k+1}} \{ \varphi, \psi \} \text{\textcolor{white}{$\sum_{i = 1}^{k+1}$}}\right.  \\
	& \quad \, + \sum_{i<j } \bbD_{s_1} \cdots \widehat{\bbD_{s_i}}\cdots \widehat{\bbD_{s_j}} \cdots \bbD_{s_{k+1}} \left(\{ \langle s_i, \varphi \rangle, \langle s_j , \psi \rangle \} + \{\langle s_j, \varphi \rangle, \langle s_i , \psi \rangle \} \right) \\
	& \quad \, \left. \left. - \sum_{i} \bbD_{s_1} \cdots \widehat{\bbD_{s_i}} \cdots \bbD_{s_{k+1}} \left(\{ \langle s_i , \varphi \rangle, \psi \} + \{ \varphi , \langle s_i , \psi \rangle \} \right) \right) \right |_S, 
	\end{aligned}
	\end{equation}
	where $\lambda, \nu \in \Gamma (\ell)$, $s_1, \ldots, s_{k+1} \in \Gamma (NS)$, $\varphi, \psi \in \Gamma (N_\ell {}^\ast S)$, and a hat ``$\widehat{-}$'' denotes omission.
\end{proposition}

\begin{proof}
	Equation~\eqref{eq:CF1} immediately follows from~\eqref{eq:higher_derived_brackets}, \eqref{eq:Jlm}, and the easy remark that $\ldsb\Delta, \lambda\rdsb = \Delta (\lambda)$ for all $\Delta \in \Der  L_{NS} = \Der^1 L_{NS}$, and $\lambda \in \Gamma (L_{NS}) = \Der^0 L_{NS}$.
	Equation~\eqref{eq:CF2} follows from~\eqref{eq:higher_derived_brackets}, \eqref{eq:DeltaJ}, and the obvious remark that
	$ \langle s, \varphi \rangle = \bbD_s \varphi $,
	hence $\bbD_{s_1} \bbD_{s_2} \varphi = 0$, for all $s, s_1, s_2 \in \Gamma (NS)$, and $\varphi \in \Gamma (N_\ell {}^\ast S)$.
	Equation~\eqref{eq:CF3} can be proven in a similar way.
\end{proof}

Let $z^\alpha$ be local coordinates on $M$, and let $\mu$ be a local generator of $\Gamma (L)$.
Define local sections $\mu^\ast$ and $\nabla_\alpha$ of $J_1 L$ by putting
\begin{equation*}
\mu^\ast(f\mu)= f,\quad \nabla_\alpha(f\mu)=\partial_\alpha f,
\end{equation*}
where $f\in C^\infty(M)$, and $\partial_\alpha = \partial / \partial z^\alpha$.
Then $ \Gamma (\wedge^\bullet J_1 L)$ is locally generated, as a $C^\infty (M)$-module, by
\begin{equation*}
\nabla_{\alpha_1}\wedge\ldots\wedge\nabla_{\alpha_k},\qquad\nabla_{\alpha_1}\wedge\ldots\wedge\nabla_{\alpha_{k-1}}\wedge\mu^\ast, \quad k>0,
\end{equation*}
with $\alpha_1<\ldots<\alpha_k$.
In particular, any $\Delta\in  \Gamma (\wedge^\bullet J_1 L)$ is locally expressed as
\begin{equation*}
\Delta=X^{\alpha_1\ldots \alpha_k}\nabla_{\alpha_1}\wedge\ldots\wedge\nabla_{\alpha_k}+g^{\alpha_1\ldots \alpha_{k-1}}\nabla_{\alpha_1}\wedge\ldots\wedge\nabla_{\alpha_{k-1}}\wedge\mu^\ast,
\end{equation*}
where $X^{\alpha_1\ldots \alpha_k},g^{\alpha_1\ldots \alpha_{k-1}}\in C^\infty(M)$.
Here and in what follows, we adopt the Einstein summation convention over pair of upper-lower repeated indexes.
Hence, $ (\Der^\bullet L)[1]$ is locally generated, as a $C^\infty(M)$-module, by
\begin{equation*}
\nabla_{\alpha_1}\wedge\ldots\wedge\nabla_{\alpha_k}\otimes\mu,\qquad\nabla_{\alpha_1}\wedge\ldots\wedge\nabla_{\alpha_{k-1}}\wedge\operatorname{id}, \quad k > 0,
\end{equation*}
with $\alpha_1<\ldots<\alpha_k$, and any $\square\in \mathcal  (\Der^\bullet L)[1]$ is locally expressed as
\begin{equation*}
\square=X^{\alpha_1\ldots \alpha_k}\nabla_{\alpha_1}\wedge\ldots\wedge\nabla_{\alpha_k}\otimes\mu+g^{\alpha_1\ldots \alpha_{k-1}}\nabla_{\alpha_1}\wedge\ldots\wedge\nabla_{\alpha_{k-1}}\wedge\operatorname{id}.
\end{equation*}

\begin{remark}
	\label{oss:components_of_J}
	Let $J\in \Der^2 L$.
	Locally,
	\begin{equation}
	\label{eq:components_of_J}
	J=J^{\alpha\beta}\nabla_{\alpha}\wedge\nabla_{\beta}\otimes\mu+J^\alpha\nabla_\alpha\wedge\operatorname{id},
	\end{equation}
	for some local functions $J^{\alpha \beta}, J^\alpha$.
\end{remark}

Now, identify $L_{NS}$ with its image in $L$ under $\tau$ and assume that:
\begin{itemize}
	\item coordinates $z^\alpha$ are fibered, i.e.~$z^\alpha = (x^i,y^a)$, with $x^i$ coordinates on $S$, and $y^a$ linear coordinates along the fibers of $\pi : NS \to S$,
	\item local generator $\mu$ is fiber-wise constant so that, locally, $\Gamma(\ell) \subset \Gamma (L_{NS})$ consists exactly of sections $\lambda$ which are vertical, i.e.~$\nabla_a\lambda=0$.
\end{itemize}
In particular, local expression~\eqref{eq:components_of_J} for $J$ expands as
\begin{equation}
\label{eq:components_of_J_bis}
J= \left( J^{a b}\nabla_a\wedge\nabla_b +2J^{ai}\nabla_a\wedge\nabla_i +J^{ij}\nabla_i\wedge\nabla_j \right) \otimes\mu+ \left( J^a\nabla_a +J^i\nabla_i \right)\wedge\operatorname{id}.
\end{equation}
We have the following

\begin{corollary}
	\label{prop:multi-brackets_coordinates}
	Locally, the multi-bracket $\mathfrak m_{k+1}$ is uniquely determined by
	\begin{align*}
	& \frakm_{k+1} \left(\partial_{a_1},\ldots,\partial_{a_{k-1}},f \mu,g \mu \right) =(-)^k \partial_{a_1} \cdots \partial_{a_{k-1}}\left.\left[2J^{ij} \partial_i f \partial_i g - J^i(f\partial_i g-g \partial_i f)\right]\right|_
	S \mu,\\
	& \frakm_{k+1} \left( \partial_{a_1},\ldots,\partial_{a_{k}}, f \mu \right) =(-)^k \partial_{a_1} \cdots \partial_{a_{k}}\left.\left(2J^{ai}\partial_i f+J^a f\right)\right|_S \partial_a,\\
	& \frakm_{k+1}\left(\partial_{a_1},\ldots,\partial_{a_{k+1}} \right) =-(-)^k\left. \partial_{a_1} \cdots \partial_{a_{k+1}}J^{ab} \right|_S \delta_a\wedge\delta_b\otimes\mu,
	\end{align*}
	where $f, g \in C^\infty (S)$, and $\delta_a := \partial_a \otimes \mu^\ast$.
\end{corollary}

\section{Independence of the tubular embedding}
\label{sec:gauge_invariance}
Now we show that, as already in the symplectic~\cite[Appendix]{oh2005deformations}, the Poisson~\cite{cattaneo2008equivalences}, and the l.c.s.~\cite[Theorem 9.5]{le2012deformations} cases, the $L_\infty[1]$-algebra in Proposition~\ref{prop:linfty} does not really depend on the choice of a fat tubular neighborhood, in the sense clarified by Proposition~\ref{prop:gauge_invariance} below.
As a consequence, its $L_\infty$-isomorphism class is an \emph{invariant} of the coisotropic submanifold.

\begin{proposition}
	\label{prop:gauge_invariance}
	Let $S$ be a coisotropic submanifold of the Jacobi manifold $(M,L,J=\{-,-\})$.
	Then the $L_\infty[1]$-algebra structures on $\Gamma (\wedge^\bullet N_\ell S \otimes \ell)[1]$ associated with $S$, through different choices of the fat tubular neighborhood $L_{NS} \injects L$ of $\ell$ in $L$, are $L_\infty$-isomorphic.
\end{proposition}

The proof is an adaptation of the one given by Cattaneo and Sch\"atz in the Poisson setting (see sections 4.1 and 4.2 of~\cite{cattaneo2008equivalences}, see also Remark~\ref{rem:CS} below) and it is based on Theorem 3.2 of~\cite{cattaneo2008equivalences} and the fact that \emph{any two fat tubular neighborhoods are isotopic} (in the sense of Proposition~\ref{prop:isotopy} above).
Before proving Proposition~\ref{prop:gauge_invariance}, let us recall Cattaneo--Sch\"atz Theorem.
We will present a ``minimal version'' of it, adapted to our purposes.
The main ingredients are the following.

We work in the category of real topological vector spaces.
Let $(\mathfrak h,\mathfrak a, P, \Delta_0)$ and $(\mathfrak h, \mathfrak a, P, \Delta_1)$ be V-data~\cite{fregier2014simultaneous}.
We identify $\mathfrak a$ with the target space of $P$.
Note that $(\mathfrak h,\mathfrak a, P, \Delta_0)$ and $(\mathfrak h, \mathfrak a, P, \Delta_1)$ differ for the last entry only.
Higher derived brackets construction associates $L_\infty[1]$-algebras to $(\mathfrak h,\mathfrak a, P, \Delta_0)$ and $(\mathfrak h, \mathfrak a, P, \Delta_1)$.
Denote them $\mathfrak a_0$ and $\mathfrak a_1$ respectively.
Cattaneo and Sch\"atz main idea is proving that when
\begin{itemize}
	\item $\Delta_0$ and $\Delta_1$ are gauge equivalent elements of the graded Lie algebra $\mathfrak h$, and
	\item they are intertwined by a gauge transformation preserving $\ker P$,
\end{itemize}
then $\mathfrak a_0$ and $\mathfrak a_1$ are $L_\infty$-isomorphic.
Specifically, $\Delta_0$ and $\Delta_1$ are \emph{gauge equivalent} if they are interpolated by a smooth family $\{ \Delta_t \}_{t \in [0,1]}$ of elements $\Delta_t \in \mathfrak h$, and there exists a smooth family $\{ \xi_t \}_{t \in [0,1]}$ of degree zero elements $\xi_t \in \mathfrak h$ such that the following evolutionary differential equation is satisfied:
\begin{equation}
\label{eq:gaugeMC00}
\frac{d}{dt} \Delta_t = [\xi_t , \Delta_t].
\end{equation}
One usually assumes that the family $\{ \xi_t \}_{t \in [0,1]}$ integrates to a family $\{ \phi_t \}_{t \in [0,1]}$ of automorphisms $\phi_t : \mathfrak h \to \mathfrak h$ of the Lie algebra $\mathfrak h$, i.e.~$\{ \phi_t \}_{t \in [0,1]}$ is a solution of the Cauchy problem
\begin{equation}
\label{eq:Cauchy}
\left\{
\begin{aligned}
& \frac{d}{dt}\phi_t (-)=[\phi_ t (-), \xi_t] \\
& \phi_0= \id
\end{aligned}
\right. .
\end{equation}

Finally we say that $\Delta_0$ and $\Delta_1$ are intertwined by a gauge transformation preserving $\ker P$ if family $\{ \xi_t \}_{t \in [0,1]}$ above satisfies the following conditions:
\begin{enumerate}
	\item the only solution $\{ a_t \}_{t \in [0,1]}$, where $a_t \in \mathfrak a$, of the Cauchy problem
	\begin{equation}
	\label{eq:equivalence_2a}
	\left\{
	\begin{aligned}
	& \frac{d}{dt}a_t =P [a_t , \xi_t] \\
	& a_0 =0
	\end{aligned}
	\right.
	\end{equation}
	is the trivial one: $a_t = 0$ for all $t \in [0,1]$,
	\item $[ \xi_t, \ker P] \subset \ker P$ for all $t \in [0,1]$.
\end{enumerate}

\begin{theorem}[{\cite[Theorem 3.2]{cattaneo2008equivalences}}]\label{theor:CS}
	Let $(\mathfrak h,\mathfrak a, P, \Delta_0)$ and $(\mathfrak h, \mathfrak a, P, \Delta_1)$ be V-data, and let $\mathfrak a_0$ and $\mathfrak a_1$ be the associated $L_\infty[1]$-algebras.
	If $\Delta_0$ and $\Delta_1$ are gauge equivalent and they are intertwined by a gauge transformation preserving $\ker P$, then $\mathfrak a_0$ and $\mathfrak a_1$ are $L_\infty$-isomorphic.
\end{theorem}

\begin{proof}[Proof of Proposition {\ref{prop:gauge_invariance}}]
	Let $\tau_0, \tau_1 : L_{NS} \injects L$ be fat tubular neighborhoods over tubular neighborhoods $\underline \tau{}_0, \underline \tau{}_1 : NS \injects M$.
	Denote by $J_0$ and $J_1$ the Jacobi brackets induced on $\Gamma (L_{NS})$ by $\tau_0$ and $\tau_1$ respectively, i.e.~$J_0 = (\tau_0^{-1})_\ast J$, and $J_1 = (\tau_1^{-1})_\ast J$ (see Remark~\ref{rem:inf_aut} about pushing forward a multi-differential operator along a line bundle isomorphism).
	In view of Proposition~\ref{prop:isotopy} it is enough to consider the following two cases:
	
	{\textbf Case I:} \emph{$\tau_1 = \tau_0 \circ \psi$ for some automorphism $\psi : L_{NS} \to L_{NS}$ covering an automorphism $ \underline \psi : NS \to NS$ of $NS$ over the identity}.
	Obviously, $\psi$ identifies the V-data $((\Der^\bullet L_{NS})[1], \im I, P, J_0)$ and $((\Der^\bullet L_{NS})[1], \im I, P, J_1)$.
	As an immediate consequence, the $L_\infty[1]$-algebra structures on $\Gamma (\wedge^\bullet N_\ell S \otimes \ell)[1]$ determined by $\tau_0$ and $\tau_1$ are (strictly) $L_\infty$-isomorphic.
	
	{\textbf Case II:} \emph{$\tau_0$ and $\tau_1$ are interpolated by a smooth one parameter family of fat tubular neighborhoods $\tau_t$.}
	Consider $\phi_t := \tau_t^{-1} \circ \tau_0 $.
	It is a local automorphism of $L_{NS}$ covering a local diffeomorphism $\underline \varphi{}_t = \underline \tau{}_t^{-1} \circ \underline \tau{}_0$, well defined in a suitable neighborhood of $S$ in $NS$, fixing $S$ point-wise and such that $\varphi_0 = \id$.
	Let $\xi_t$ be infinitesimal generators of the family $\{ \varphi_t \}$.
	They are derivations of $L_{NS}$ well defined around $S$.
	Our strategy is using $\xi_t$ and $\varphi_t$ to prove that $J_0$ and $J_1$ are gauge equivalent Maurer--Cartan elements of $(\Der^\bullet L_{NS})[1]$ intertwined by a gauge transformation preserving $\ker P$, and then applying Theorem~\ref{theor:CS}.
	However, the $\varphi_t$'s are well-defined only around $S$ in $NS$.
	In order to remedy this minor drawback, we slightly change the graded space $\Der^\bullet L_{NS}$ underlying our V-data, passing to the graded space $\Der_{\text for}^\bullet L_{NS}$ of \emph{skew-symmetric, first order, multi-differential operators on $L_{NS}$ in a formal neighborhood of $S$ in $NS$}.
	The space $\Der_{\text for}^\bullet L_{NS}$ is defined as the inverse limit
	\begin{equation*}
	\lim_{\longleftarrow} \Der^\bullet L_{NS}/ I(S)^n \Der^\bullet L_{NS},
	\end{equation*}
	where $I(S) \subset C^\infty (NS)$ is the ideal of functions vanishing on $S$, and consists of ``Taylor series normal to $S$'' of multi-differential operators.
	V-data $((\Der^\bullet L_{NS})[1], \im I, P, J)$ induce obvious V-data $((\Der_{\text for}^\bullet L_{NS})[1], \im I_{\text for}, P_{\text for}, J_{\text for})$.
	In particular, $I_{\text for} : \Gamma (\wedge^\bullet N_\ell S \otimes \ell)[1] \injects (\Der_{\text for}^\bullet L_{NS})[1]$ is the natural embedding, and $J_{\text for}$ is the class of $J$ in $(\Der_{\text for}^\bullet L_{NS})[1]$.
	Moreover, in view of Corollary~\ref{prop:multi-brackets_coordinates}, the $L_\infty[1]$-algebra determined by $((\Der^\bullet L_{NS})[1], \im I, P, J)$ does only depend on $J_{\text for}$.
	Therefore, V-data $((\Der_{\text for}^\bullet L_{NS})[1], \im I_{\text for}, P_{\text for}, J_{\text for})$ determine the same $L_\infty[1]$-algebra as $((\Der^\bullet L_{NS})[1], \im I, P, J)$.
	
	Now, being well defined around $S$, the $\varphi_t$'s determine well-defined automorphisms $\phi_t := (\varphi_t)_\ast : (\Der_{\text for}^\bullet L_{NS})[1]\longrightarrow(\Der_{\text for}^\bullet L_{NS})[1]$ such that $\phi_0 = \id$.
	Similarly the $\xi_t$'s descend to zero degree elements of $(\Der_{\text for}^\bullet L_{NS})[1]$ which we denote by $\xi_t$ again.
	Clearly, family $\{ \phi_t (J_0)_{\text for} \}$ interpolates between $(J_0)_{\text for}$ and $(J_1)_{\text for}$ and, in view of Equation~\eqref{Lie2}, the $\phi_t$'s satisfy Cauchy problem~\eqref{eq:Cauchy}.
	Finally,
	\begin{enumerate}
		\item from uniqueness of the one parameter family of automorphisms $\varphi_t$ generated by the one parameter family of derivation $\xi_t$, it follows that Cauchy problem~\eqref{eq:equivalence_2a} possesses a unique solution,
		\item $\varphi_t |_{\ell} = \id$ so that the $\xi_t$'s vanish on $S$, hence $[\xi_t, \ker P ] \subset \ker P$ for all $t$.
	\end{enumerate}
	The above considerations show that $(J_0)_{\text for}$ and $(J_1)_{\text for}$ are gauge equivalent and they are intertwined by a gauge transformation preserving $\ker P$.
	Hence, from Theorem~\ref{theor:CS}, the $L_\infty[1]$-algebra structures on $\Gamma (\wedge^\bullet N_\ell S \otimes \ell)[1]$ associated to the two choices $\tau_0$ and $\tau_1$ of the fat tubular neighborhood $L_{NS} \injects L$ are actually $L_\infty$-isomorphic.
\end{proof}

\begin{remark}\label{rem:CS}
	As already mentioned, in the proof of Proposition {\ref{prop:gauge_invariance}} we basically follow Cattaneo and Sch\"atz~\cite[Section 4.2]{cattaneo2008equivalences}.
	Actually, the structure of our proof looks slightly simpler than theirs, which contains some redundancy.
	On another hand, our case is a bit more complicated technically, in view of the presence of a generically non-trivial line-bundle.
	Namely, in our case, having a tubular neighborhood of $S$ is not enough.
	We also need a \emph{fat tubular neighborhood} of $\ell = L|_S$.
	The latter guarantees that $L$ is presented as a pull-back bundle from $\ell$ around $S$ which is crucial in our proof.
\end{remark}

\begin{remark}
	In the contact case, as already in the l.c.s.~one, there exists a tubular neighborhood theorem for coisotropic submanifolds.
	As a consequence, the proof of Proposition~\ref{prop:gauge_invariance} simplifies.
	In particular, it does not require using any \emph{formal neighborhood technique}.
\end{remark}

%

\section{Smooth coisotropic deformations}
\label{sec:cois_def}

Let $(M, L, J=\{-,-\})$ be a Jacobi manifold and let $S \subset M$ be a closed coisotropic submanifold.
We fix a fat tubular neighborhood $(\tau,\underline{\smash{\tau}})$ of $\ell\to S$ in $L\to M$ and use it to identify $L_{NS}\to NS$ with its image.
Accordingly, and similarly as above, from now on in this section, we abuse the notation and denote by $(L,J=\{-,-\})$ (instead of $(L_{NS},\tau_\ast^{-1}J)$) the Jacobi bundle on $NS$ (unless otherwise specified).
It is well known that locally, around $S$, a deformation of $S$ in $NS$ can be identified with a section of the normal bundle $\pi:NS\to S$.
We say that a section $s\in\Gamma(NS)$ is \emph{coisotropic} if its image $s (S)$ is a coisotropic submanifold in $(NS,L,J)$.

\begin{definition}
	\label{def:smcoiso}
	A smooth one parameter family of smooth sections of $NS \to S$ starting from the zero section is a \emph{smooth coisotropic deformation of} $S$ if each section in the family is coisotropic.
	A section $s$ of $NS \to S$ is an \emph{infinitesimal coisotropic deformation of} $S$ if $\eps s$ is a coisotropic section up to infinitesimals $\calO(\eps^2)$, where $\eps$ is a formal parameter.
\end{definition}

\begin{remark}
	Let $\{ s_t \}$ be a smooth coisotropic deformation of $S$.
	Then
	\[
	\left. \frac{d}{dt}\right|_{t=0} s_t
	\]
	is an infinitesimal coisotropic deformation.
\end{remark}

\begin{remark}
	\label{rem:coiss}
	Recall that a section $s\in\Gamma(NS)$ is mapped, via $I:\Gamma(\wedge^\bullet N_\ell S\otimes\ell)[1] \to  (\Der^\bullet L)[1]$, to a derivation $I(s):=\mathbb D_{\pi^\ast s}$ of $L\to NS$, where $\pi : NS \to S$ is the projection.
	Then  $I(s)$ generates a smooth $1$-parameter group $\{\Phi_t\}$ of automorphisms of $L\to NS$ which covers the smooth $1$-parameter group $\{\underline{\smash{\Phi}}_t\}$ of diffeomorphisms generated by $\sigma(Is)=\pi^\ast s$.
	Set $\exp I (s) := \Phi_{1}$ and $\underline{\smash{\exp I (s)}} := \underline{\smash{\Phi}}_{1}$.
	Clearly $\exp I(s)(\nu,\lambda)=(\underline{\smash{\exp I(s)}}(\nu), \lambda)=(\nu + s(x) , \lambda)$, for all $(\nu,\lambda)\in L = NS \times_S \ell$, $x = \pi (\nu)$.
	Further, let $\mathrm{pr} : J^1 L \to NS$ be the projection, denote by $j^1 \exp I(s) : J^1 L \to J^1 L$ the first jet prolongation of $\exp I(s)$, and consider the following commutative diagram

	\begin{equation*}
		\begin{tikzcd}[column sep=huge, row sep=large]
			N_\ell{}^\ast S\arrow[d]\arrow[r, "\gamma"]&J^1L\arrow[d]\arrow[r, "j^1 \exp I(s)"]&J^1L\arrow[d]\\
			S\arrow[r, shift left=0.7 ex, "\mathbf 0"]&NS\arrow[l, shift left=0.7 ex, "\pi"]\arrow[r, shift left=0.7 ex, "\underline{\smash{\exp I (s)}}"]&NS\arrow[l, shift left=0.7 ex, "\underline{{\exp I(-s)}}"]
		\end{tikzcd}
	\end{equation*}
	where $\mathbf 0$ is the zero section.
	Note that $s = \underline{\smash{\exp I(s)}} \circ \mathbf 0$.
\end{remark}

\begin{proposition}
	\label{prop:coiss}
	Let $s : S \to NS$ be a section of $\pi$.
	The following three conditions are equivalent
	\begin{enumerate}
		\item $s$ is coisotropic,
		\item $P (\exp I (-s)_* J) = 0$ (cf.~\cite{SZ2012}),
		\item vector bundle $\mathrm{pr} \circ j^1 \exp I(s) \circ \gamma : N_\ell {}^\ast S \to s(NS)$ is a Jacobi subalgebroid of $J^1 L$.
	\end{enumerate}
\end{proposition}

\begin{proof} \
	
	(1) $\Longleftrightarrow$ (2).
	Let $P^s : \Der L \to \Gamma(NS)$ be composition
	\begin{equation*}
	\Der  L \overset{\sigma}{\longrightarrow} \mathfrak{X}(M) \longrightarrow \Gamma (TM |_{s(S)})\longrightarrow \Gamma (NS),
	\end{equation*}
	where second arrow is the restriction, and last arrow is the canonical projection (cf.~Equation~\eqref{eq:projection_P}).
	Surjection $P^s$ extends to a surjection of graded modules $(\Der^\bullet L)[1] \to \Gamma (\wedge^\bullet N_\ell S \otimes \ell)[1]$ which we denote again by $P^s$ (and is defined analogously as $P : (\Der^\bullet L)[1] \to \Gamma (\wedge^\bullet N_\ell S \otimes \ell)[1]$).
	By Proposition~\ref{prop:PJ=0}, $s$ is coisotropic iff $ P^s (J) = 0$.
	Since
	\[
	\der  \ell = \exp I (-s)_* \der  L|_{s(S)} \quad \text{and} \quad \exp I(-s)_* NS = NS,
	\]
	we obtain
	\begin{equation}\label{eq:Ps}
	P^s = P \circ \exp I(-s)_*.
	\end{equation}
	In particular, $P^s (J) = P (\exp I(-s)_*J) = 0$ iff $s$ is coisotropic.
	
	(1) $\Longleftrightarrow$ (3).
	Note that $\mathrm{pr} \circ j^1 \exp I(s) \circ \gamma : N_\ell {}^\ast S \to s(N)$ is the $\ell$-adjoint bundle of the normal bundle of $s(S)$ in $NS$.
	Now the claim follows immediately from Proposition~\ref{prop:conormal}.
\end{proof}

\begin{corollary}
	\label{cor:inf1}
	Let $s : S \to NS$ be a section of $\pi$.
	The following two conditions are equivalent
	\begin{enumerate}
		\item $s$ is an infinitesimal coisotropic deformation of $S$,
		\item $s$ is a $1$-cocycle in the de Rham complex of the Jacobi algebroid $(N_\ell{}^\ast S,\ell)$, i.e.
		\begin{equation}
			\label{eq:infdef}
			d_{N_\ell{}^\ast S,\ell}s\equiv\frakm_1 s=0.
		\end{equation}
	\end{enumerate}
\end{corollary}

\begin{remark}
	\label{rem:cois_Lambda}
	Let $s$ be a section of $NS$.
	In view of Remark~\ref{rem:PJ}, $P^s (J) = P^s (\Lambda_J)$, where, in the rhs, $P^s$ denotes the extension $\Gamma (\wedge^\bullet (T(NS) \otimes L^\ast) \otimes L) \to \Gamma (\wedge^\bullet N_\ell S \otimes \ell)$ of composition $T(NS) \to T(NS)|_{s(S)} \to NS$ defined analogously as $P : (\Der^\bullet L)[1] \to \Gamma (\wedge^\bullet N_\ell S \otimes \ell)[1]$.
	Moreover, it is clear that
	\[
	\Lambda_{\exp I(-s)_\ast J} = \exp I(-s)_\ast \Lambda_J,
	\]
	where, in the lhs, $\Lambda_{\exp I(-s)_\ast J} $ denotes the bi-symbol of $\exp I(-s)_\ast J$ and, in the rhs, $\exp I(-s)_\ast : \Gamma (\wedge^\bullet (T(NS) \otimes L^\ast)\otimes L) \to \Gamma (\wedge^\bullet (T(NS) \otimes L^\ast)\otimes L)$ denotes the isomorphism induced by the line bundle automorphism $\exp I(-s)$.
	It immediately follows that $s$ is coisotropic iff $P(\exp I (-s)_\ast \Lambda_J)=0$.
\end{remark}

\section{Formal coisotropic deformations}
Let $\eps$ be a formal parameter.

\begin{definition}
	A formal series $s (\eps) = \sum _{i =0} ^ \infty \eps ^i s_i \in \Gamma (NS) [[ \eps]]$, $s_i \in \Gamma (NS)$, such that $s_0 = 0$, is called a \emph{formal deformation of $S$}.
\end{definition}

\begin{remark}
	\label{rem:Lieder}
	Formal series $I (s (\eps)) : = \sum_{i =0}^\infty \eps ^i I (s_i) \in (\Der L) [[ \eps ]]$ is a formal derivation of $L$.
	It is easy to see that the space $(\Der  L) [[\eps]]$ of formal derivations of $L$ is a Lie algebra, which has a linear representation in the space $ (\Der^\bullet L)[[\eps]]$ of formal first order multi-differential operators on $L$ via the following \emph{Lie derivative}:
	\begin{equation}
		\label{eq:Lieder}
		\calL_{\xi(\eps)} \Delta (\eps) \equiv \ldsb\xi (\eps), \Delta (\eps)\rdsb := \sum _{ k =0} ^\infty \eps ^{k} \sum_{i+j = k}\ldsb \xi_i , \Delta_j\rdsb,
	\end{equation}
	for $\xi (\eps)= \sum_{ i= 0} ^ \infty\eps ^i \xi_i$, $\xi_i \in \Der  L$, and $\Delta (\eps) = \sum _{ i =0}^\infty \eps ^ i \Delta_i$, $\Delta_i\in  \Der^\bullet L$.
	
	We define the exponential of the Lie derivative $\calL_{\xi (\eps)}$ as the following formal power series
	\begin{equation}
		\label{eq:exp}
		\exp \calL_{\xi(\eps)} = \sum_{n=0}^\infty \frac{1}{n!} \calL_{\xi (\eps)}^n.
	\end{equation}
\end{remark}

Proposition~\ref{prop:coiss} motivates the next definition of \emph{formal coisotropic deformation} of $S$.
For further motivation of this notion see also the following Remark~\ref{rem:motivation_formal_coisotropic_deformations}.

\begin{definition}
	\label{def:coisoformal}
	A formal deformation $s (\eps)$ of $S$ is \emph{coisotropic}, if
	$P (\exp \calL_{I(s (\eps))} J) = 0$.
\end{definition}

\begin{remark}
	\label{rem:motivation_formal_coisotropic_deformations}
	Let $s_t$ be a smooth one parameter family in $\Gamma(NS)$ starting from the zero section $\mathbf 0$ of $NS\to S$.
	Denote by $s(\eps)$ the Taylor series of $s_t$, at $t=0$, in the formal parameter $\eps$, i.e.
	\begin{equation*}
	s(\eps):=\sum_{i=1}^\infty\underbrace{\frac{1}{i!}\left(\left.\frac{d^i}{dt^i}\right|_{t=0}s_t\right)}_{:=s_i}\eps^i
	\end{equation*}
	Through an iterated application of Leibniz rule, we get, for all $n\geq 0$, the following
	\begin{equation*}
	\frac{1}{n!}\left.\frac{d^n}{dt^n}\right|_{t=0}P(\exp(I(-s_t))_\ast J)=\sum_{k=1}^n\frac{1}{k!}\sum_{\genfrac{}{}{0pt}{}{i_1,\ldots,i_k}{i_1+\ldots+i_k=n}}P\ldsb I(s_{i_1}),\ldsb\ldots\ldsb I(s_{i_k}),J\rdsb\ldots\rdsb\rdsb.
	\end{equation*}
	The latter means exactly that $P(\exp\calL_{I(s(\eps))}J)$ is the Taylor series of $P(\exp(I(-s_t))_\ast J)$, at $t=0$, in the formal parameter $\eps$, i.e.
	\begin{equation*}
	P(\exp\calL_{I(s(\eps))}J)=\sum_{n=0}^\infty\frac{1}{n!}\left(\left.\frac{d^n}{dt^n}\right|_{t=0}P(\exp(I(-s_t))_\ast J)\right)\eps^n.
	\end{equation*}
	Hence if $\{s_t\}$ is a smooth coisotropic deformation of $S$, then $s(\eps)$, its Taylor series at $t=0$, is a formal coisotropic deformation of $S$.
\end{remark}

\begin{remark}
	Let $\xi (\eps) \in (\Der  L) [[ \eps]]$.
	Define a Lie derivative
	\[
	\calL_{\xi (\eps)} : \Gamma (\wedge^\bullet (T(NS) \otimes L^\ast) \otimes L) [[\eps]] \to \Gamma (\wedge^\bullet (T(NS) \otimes L^\ast) \otimes L) [[\eps]],
	\]
	in the obvious way.
	It is easy to see that
	\begin{equation}\label{eq:PJvsPLambda}
	P (\exp \calL_{I(s (\eps))} J) = P (\exp \calL_{I(s (\eps))} \Lambda_J),
	\end{equation}
	for all formal deformations $s (\eps)$ of $S$ (cf.~Remarks~\ref{rem:PJ} and~\ref{rem:cois_Lambda}).
	In particular, $s (\eps)$ is coisotropic iff $ P (\exp \calL_{I(s (\eps))} \Lambda_J) = 0$.
\end{remark}

\begin{remark}
	\label{rem:formal_deformation_problem}
	[Formal deformation problem]
	The \emph{formal deformation problem for a coisotropic submanifold $S$} consists in finding formal coisotropic deformations of $S$.
	According with Deligne's principle~\cite{deligneletter1986},
	in characteristic zero, every deformation problem should be controlled by a differential graded Lie algebra, or by its higher homotopy version: an $L_\infty[1]$-algebra.
	Actually such viewpoint holds also for the formal deformation problem of a coisotropic submanifold $S$ in a Jacobi manifold $(M,L,J)$.
	Indeed the latter is \emph{controlled by the $L_\infty[1]$-algebra} $(\Gamma (\wedge^\bullet N_\ell S \otimes \ell)[1], \{ \frakm_k\})$ in the sense clarified by the next Proposition~\ref{prop:mcformal} (see also the following Proposition~\ref{prop:MC_gauge}).
	In this way what was first established in the symplectic setting by Oh--Park~\cite{oh2005deformations} is generalized to the Jacobi setting.
\end{remark}

\begin{proposition}
	\label{prop:mcformal}
	A formal deformation $s (\eps)$ of $S$ is coisotropic iff $-s (\eps)$
	is a solution of the (formal) Maurer--Cartan equation
	\begin{equation}
	\label{eq:MC}
	MC (-s (\eps)):=\sum _{ k =1} ^\infty \frac{1}{k!}\frakm_k (-s (\eps), \cdots , -s(\eps))= 0.
	\end{equation}
\end{proposition}

\begin{proof}
	The expression $MC (-s (\eps))$ should be interpreted as an element of $\Gamma (\wedge^2 N_\ell S \otimes \ell)[[ \eps]]$.
	The proposition is then a consequence of~\eqref{eq:exp}, $P(J) = 0$, and the following identities
	\begin{equation}
	\label{eq:mcl}
	P (\calL_{I(\xi)}^k J) = \frakm_k (-\xi, \cdots, -\xi) , \quad k \ge 1,
	\end{equation}
	for $\xi \in \Gamma (NS)$, which immediately follow from the definition of $\frakm_k$.
\end{proof}

%


\begin{remark}
	Let $s(\eps)$ be an arbitrary formal deformation of $S$.
	Hence $s(\eps)=\sum_{i=1}^\infty\eps^is_i$, for arbitrary $s_i\in\Gamma(NS)$, and we have $MC(s(-\eps))=\sum_{i=1}^\infty\eps^i\Gamma_i$, where $\Gamma_i\in\Gamma(\wedge^2N_\ell S\otimes\ell)$ is given by
		\begin{equation*}
		\Gamma_i=\sum_{k=1}^i\frac{(-1)^k}{k!}\sum_{\genfrac{}{}{0pt}{}{i_1,\ldots,i_k>0}{i_1+\ldots+i_k=i}}\frakm_k(s_{i_1},\ldots,s_{i_k}),
		\end{equation*}
	for all $i>0$.
	As a consequence the formal Maurer--Cartan equation~\eqref{eq:MC} for $-s(\eps)$ splits in the following infinite sequence of equations for the $s_i$'s
		\begin{equation}
		\label{eq:MC_k}
		\tag{$MC_k$}
		\frakm_1(s_k)=\sum_{h=2}^k\frac{(-1)^h}{h!}\sum_{\genfrac{}{}{0pt}{}{0<i_1,\ldots,i_h<k}{i_1+\ldots+i_h=k}}\frakm_h(s_{i_1},\ldots,s_{i_h}),
		\end{equation}
	for all $k>0$.
	In particular~\eqref{eq:MC_k}, for $k=1$, reads as $\frakm_1(s_1)=0$.
	In view of Proposition~\ref{prop:mcformal} and Corollary~\ref{cor:inf1}, the latter means that if $s(\eps)$ is a formal coisotropic deformation of $S$ then $s_1$ is an infinitesimal coisotropic deformation of $S$.
	In such case one also says that $s(\eps)$ is a prolongation of the infinitesimal coisotropic deformation $s_1$ to a formal one.
	However, in general, not all infinitesimal coisotropic deformations can be ``prolonged'' to a formal coisotropic deformation.
	If this is the case, one says that \emph{the formal deformation problem is unobstructed}.
	Otherwise, \emph{the formal deformation problem is obstructed}.
\end{remark}

Let $s(\eps)$ be an arbitrary formal deformation of $S$.
From the generalized Jacobi identities for the $L_\infty[1]$-algebra $(\Gamma(\wedge^\bullet N_\ell S\otimes\ell),\{\frakm_k\})$, through a straightforward computation, it follows that $s(\eps)$ must satisfy the following identity
\begin{equation}
\label{eq:obstructions}
\sum_{k=1}^\infty\frac{1}{(k-1)!}\frakm_k(MC(-s(\eps)),-s(\eps),\ldots,-s(\eps))=0.
\end{equation}  
Fix $n>0$.
Assume that $-s(\eps)$ satisfies the formal Maurer--Cartan equation~\eqref{eq:MC} up to infinitesimals $\calO(\eps^n)$, i.e.~Equation~\eqref{eq:MC_k} holds for all $k=1,\ldots,n-1$.
Notice that this assumption only concerns the first $(n-1)$-coefficients $s_1,\ldots,s_{n-1}$ of $s(\eps)$.
Now, from~\eqref{eq:obstructions}, it follows that $\frakm_1(MC(-s(\eps)))$ vanishes up to infinitesimals $\calO(\eps^{n+1})$, i.e.~the rhs of~\eqref{eq:MC_k}, for $k=n$, is a $2$-cocycle in the de Rham complex of the Jacobi algebroid $(N_\ell{}^\ast S,\ell)$.
Hence the cohomology class of the rhs of~\eqref{eq:MC_k}, for $k=n$, represents the obstruction to the prolongability of $s_1$ to a formal coisotropic deformation up to infinitesimals $\calO(\eps^{n+1})$.
Indeed, under the current assumption,~\eqref{eq:MC_k}, for $k=n$, admits a solution $s_n$ if and only if the cohomology class of its rhs is zero.
This leads to the following unobstructedness criterion for the formal deformation problem.

\begin{corollary}
	\label{prop:rigid}
	Assume that the second cohomology group $H^2 (N_\ell {}^\ast S, \ell)$ of the Jacobi algebroid $(N_\ell {}^\ast S,\ell)$ vanishes.
	Then every infinitesimal coisotropic deformation of $S$ can be prolonged to a formal coisotropic deformation, i.e.~for any given $\overline{s}_1 \in\Gamma(NS)$, with $\frakm_1 (\overline{s}_1)=0$, there is a formal coisotropic deformation $s(\eps)=\sum_{i=i}^\infty\eps^i s_i$ such that $s_1=\overline{s}_1$.
	In other words, the formal deformation problem is unobstructed.
\end{corollary}

There is also a simple criterion for the non-prolongability of an infinitesimal coisotropic deformation to a formal coisotropic deformation expressed in terms of the \emph{Kuranishi map}:
\begin{equation}
\label{eq:L_infty_kuranishi_map}
\operatorname{Kr} : H^1 (N_\ell {}^\ast S, \ell) \longrightarrow H^2 (N_\ell {}^\ast S, \ell), \quad [s] \longmapsto [\mathfrak m_2 (s, s)].
\end{equation}
Since $\frakm_1$ is a derivation of the binary bracket $\frakm_2$, the Kuranishi map is well-defined.
Moreover, Equation~\eqref{eq:MC_k}, for $k=2$, reads as follows
\begin{equation*}
\frakm_1 s_2=\frac{1}{2}\frakm_2(s_1,s_1).
\end{equation*}
Hence $\operatorname{Kr}([s_1])$ is the first obstruction we meet when we try to prolong an infinitesimal coisotropic deformation $s_1$ to a formal coisotropic deformation $s(\eps)$.
This leads to the following obstructedness criterion.
\begin{proposition}
	\label{prop:Kuranishi}
	Let $\alpha = [s] \in H^1 (N_\ell {}^\ast S, \ell)$, where $s \in \Gamma (NS)$ is an infinitesimal coisotropic deformation, i.e.~$d_{N_\ell {}^\ast S, \ell} s = \frakm_1 s = 0$.
	If $\operatorname{Kr} (\alpha) \neq 0$, then $s$ cannot be prolonged to a formal coisotropic deformation.
	In particular, the formal deformation problem is obstructed.
\end{proposition}

\begin{remark}
	Notice that, even when the $L_\infty[1]$-algebra reduces to a dgLa up to décalage, i.e.~$\frakm_k=0$ for all $k>2$, the above Equations~\eqref{eq:MC_k}, with $k>0$, still represent an infinite sequence of obstructions to the prolongability of an infinitesimal coisotropic deformation $s_1$ to a formal coisotropic deformation $s(\eps)$.
\end{remark}

%

\section{Moduli of coisotropic sections}
\label{subsec:Moduli_coisotropic_sections}

Jacobi diffeomorphisms, in particular Hamiltonian diffeomorphisms, preserve coisotropic submanifolds.
Two coisotropic submanifolds are \emph{Hamiltonian equivalent} if there is an Hamiltonian isotopy (i.e.~a one parameter family of Hamiltonian diffeomorphisms) interpolating them.
With this definition at hand one can define a moduli space of coisotropic submanifolds under \emph{Hamiltonian equivalence}.
Now, let $S$ be a coisotropic submanifold.
In this section we adapt the definition of Hamiltonian equivalence to the case of coisotropic sections of $NS \to S$~\cite[Definition 6.3]{le2012deformations}.
In this way we define a local version of the moduli space under Hamiltonian equivalence.

\begin{definition}
	\label{def:hequi}
	(cf.~\cite[Definition 10.2]{le2012deformations}).\
	\begin{enumerate}[label=(\arabic*)]
		\item\label{enumitem:def:hequi_1}
		Two coisotropic sections $s_0, s_1 \in \Gamma (NS)$ are called \emph{Hamiltonian equivalent} if they are interpolated by a smooth family of sections $s_t \in \Gamma (NS)$ and there exists a family of Hamiltonian diffeomorphisms $\psi_t : NS \to NS$ of $(NS, L, J = \{-,-\})$ (i.e.~the family $\{\psi_t\}$ is generated by a family $\{ X_{\lambda_t}\}$ of Hamiltonian vector fields, where the $\lambda_t$'s depend smoothly on $t$) and a family of diffeomorphisms $g_t : S \to S$, $t \in [0,1]$, such that $g_0 = \operatorname{id}_{S}$, $\psi_0 = \operatorname{id}_{NS}$ and
		$s_t = \psi_t \circ s_0 \circ g_t^{-1}$.
		A coisotropic deformation of $S$ is \emph{trivial} if it is Hamiltonian equivalent to the zero section.
		\item\label{enumitem:def:hequi_2}
		Two infinitesimal coisotropic deformations $s_0, s_1 \in \Gamma (NS)$ are called \emph{infinitesimally Hamiltonian equivalent} if $s_1 - s_0$ is the vertical component along $S$ of an Hamiltonian vector field.
		An infinitesimal coisotropic deformation is trivial if it is infinitesimally Hamiltonian equivalent to the zero section.
	\end{enumerate}
\end{definition}

Note that both Hamiltonian equivalence and infinitesimal Hamiltonian equivalence are equivalence relations.
The notion of infinitesimal Hamiltonian equivalence is motivated by the following remark.

\begin{remark}
	Let $s_0,s_1$ be Hamiltonian equivalent coisotropic sections, and let $s_t$ be the family of sections interpolating them as in Definition~\ref{def:hequi}.(1).
	Then $s_t$ is obviously a coisotropic section for all $t$.
	Moreover, $s_0$ and
	\[
	s_0 + \left. \frac{d}{dt}\right|_{t= 0} s_t
	\]
	are infinitesimally Hamiltonian equivalent coisotropic sections.
\end{remark}

\begin{proposition}\label{prop:Hameq}
	Let $s_0,s_1 \in \Gamma (NS)$ be Hamiltonian equivalent coisotropic sections.
	Then $s_0,s_1$ are interpolated by a smooth family of sections $s_t \in \Gamma (NS)$ and there exists a smooth family of sections $\lambda_t$ of the Jacobi bundle $L$ such that $s_t$ is a solution of the following evolutionary equation:
	\begin{equation}\label{eq:Hameq}
	\frac{d}{dt} s_t = P (\exp I(-s_t)_\ast \Delta_{\lambda_t}).
	\end{equation}
	If $S$ is compact the converse is also true.
\end{proposition}
\begin{proof}
	Denote by $\pi : NS \to S$ the projection.
	First of all, let $s_0,s_1$ be Hamiltonian equivalent coisotropic sections, and let $s_t$, $\psi_t$, $g_t$ be as in Definition~\ref{def:hequi}.(1).
	The $g_t$'s are completely determined by the $\psi_t$'s via $g_t = \pi \circ \psi_t \circ s_0$.
	In their turn, the $\psi_t$'s are generated by a smooth family $\{ X_{\lambda_t} \}$ of Hamiltonian vector fields, $\lambda_t \in \Gamma (L)$.
	Differentiating the identity $s_t = \psi_t \circ s_0 \circ g_t^{-1}$ with respect to $t$, one finds
	\begin{equation}\label{eq:HameqPs}
	\frac{d}{dt} s_t= P^{s_t} (\Delta_{\lambda_t}),
	\end{equation}
	where, for a generic section $s \in \Gamma (NS)$, the projection $P^s :  (\Der^\bullet L)[1] \to \Gamma (\wedge^\bullet N_\ell S \otimes \ell)[1]$ is defined as in the proof of Proposition~\ref{prop:coiss}.
	To see this, interpret the $s_t$'s as smooth maps, and consider their pull-backs $s_t^\ast : C^\infty (NS) \to C^\infty (S)$.
	Then $s_t^\ast = (g_t^{-1})^\ast \circ s_0^\ast \circ \psi^\ast_t$ and a straightforward computation shows that
	\[
	\frac{d }{dt} s_t^\ast= s_t^\ast \circ X_{\lambda_t} \circ (\id - \pi^\ast \circ s_t^\ast).
	\]
	which is equivalent to \eqref{eq:HameqPs}.
	Equation~\eqref{eq:Hameq} now follows from~\eqref{eq:Ps}.
	
	Conversely, let $S$ be compact, $s_t$ be a solution of Equation~\eqref{eq:Hameq} interpolating $s_0$ and $s_1$, and let $\{ \psi_t \}$ be the one parameter family of Hamiltonian diffeomorphisms $NS \to NS$ generated by $\{ X_{\lambda_t} \}$.
	The compactness assumption guarantees that $\psi_t$ is well-defined for all $t \in [0,1]$ (see, e.g.~\cite[Lemma 3.15]{SZ2014}).
	In view of~\eqref{eq:Ps} again, $s_t$ is the (unique) solution of~\eqref{eq:HameqPs} starting at $s_0$.
	In particular, $\psi_t$ maps diffeomorphically the image of $s_0$ to the image of $s_t$.
	Hence, the map $g_t = \pi \circ \psi_t \circ s_0$ is a diffeomorphism and $s_t = \psi_t \circ s_0 \circ g_t^{-1}$.
\end{proof}

Note that if $\{ s_t \}$ is a solution of~\eqref{eq:Hameq} interpolating coisotropic sections $s_0,s_1$, then $s_t$ is a coisotropic section for all $t$.
Proposition~\ref{prop:Hameq} motivates the following

\begin{definition}\label{def:formHeq}
	Two formal coisotropic deformations $s_0 (\eps), s_1(\eps) \in \Gamma (NS)[[\eps]]$ are called \emph{Hamiltonian equivalent} if they are interpolated by a smooth family of formal coisotropic deformations $s_t (\eps) \in \Gamma (NS)[[\eps]]$ (i.e.~$s_t (\eps) = \sum_i s_{t,i} \eps^i$ and the $s_{t,i}$'s depend smoothly on $t$) and there exists a smooth family of formal sections $\lambda_t (\eps) \in \Gamma (L)[[ \eps]]$ of the Jacobi bundle such that
	\[
	\frac{d}{dt} s_t (\eps)= P (\exp \calL_{I (s_t( \eps))} \Delta_{\lambda_t (\eps)}).
	\]
\end{definition}

We now show that formal coisotropic deformations $s_0 (\eps), s_1(\eps)$ are Hamiltonian equivalent iff $-s_0 (\eps), -s_1 (\eps)$ are gauge equivalent solutions of the Maurer--Cartan equation $MC(\xi(\eps)) = 0$.
Two solutions $\xi_0 (\eps) , \xi_1 (\eps)$ of the Maurer--Cartan equation are \emph{gauge equivalent} if, by definition, they are interpolated by a smooth family of formal sections $\xi_t (\eps) \in \Gamma (NS) [[\eps]]= \Gamma (\wedge^1 N_\ell S \otimes \ell) [[\eps]]$ and there exists a smooth family of formal sections $\lambda_t (\eps) \in \Gamma (\ell)[[ \eps]] =\Gamma (\wedge^0 N_\ell S \otimes \ell) [[\eps]]$ such that

\begin{equation}\label{eq:HameqMC}
\frac{d }{dt} \xi_t (\eps) = \sum_{k = 0}^\infty \frac{1}{k!} \frakm_{k+1} (\xi_t (\eps) , \ldots, \xi_t (\eps) , \lambda_t (\eps)).
\end{equation}
Gauge equivalence is an equivalence relation~\cite{doubek2007deformation,kontsevich2003deformation}.
Moreover, it follows from Equation~\eqref{eq:HameqMC} that the $\xi_t(\eps)$ is a solution of the Maurer--Cartan equation for all $t$.

\begin{proposition}
	\label{prop:MC_gauge}
	Two formal coisotropic deformations $s_0 (\eps), s_1 (\eps) \in \Gamma (NS)[[\eps]]$ are Hamiltonian equivalent iff they are gauge equivalent solutions of the Maurer--Cartan equation.
\end{proposition}

\begin{proof} Recall that $\ker P \subset  (\Der^\bullet L)[1]$ is a Lie subalgebra.
	As Th.~Voronov notes~\cite{Voronov2005higher1}, this can be rephrased as:
	\begin{equation}\label{eq:kerP}
	P\ldsb\square_1, \square_2\rdsb = P\ldsb IP\square_1 , \square_2\rdsb + P\ldsb\square_1, IP \square_2\rdsb,
	\end{equation}
	$\square_1, \square_2 \in  (\Der^\bullet L)[1]$.
	Now, let $\{ s_t (\eps) \}$ be a family of formal coisotropic deformations, and let $\{ \lambda_t (\eps) \}$ be a family of formal sections of $L$.
	Put
	\[
	J_k (\eps) := \ldsb\ldots \ldsb J, \underset{k \text{ times}}{\underbrace{I(-s(\eps))\rdsb \ldots, I(-s(\eps))\rdsb}},
	\]
	In particular, $P J_k (\eps) = \mathfrak m_k (-s (\eps), \ldots, -s(\eps))$.
	Compute
	\begin{align*}
	P(\exp \calL_{I (s_t (\eps))} \Delta_{\lambda_t (\eps)} )
	& = -\sum_{k=0}^\infty \frac{1}{k!} P \ldsb J_k (\eps), \lambda_t (\eps)\rdsb \\
	& = -\sum_{k=0}^\infty \frac{1}{k!} P \ldsb IPJ_k (\eps), \lambda_t (\eps)\rdsb
	-\sum_{k=0}^\infty \frac{1}{k!} P \ldsb J_k (\eps), IP\lambda_t (\eps)\rdsb \\
	& = - P \ldsb I (MC (-s_t(\eps))), \lambda_t (\eps)\rdsb
	-\sum_{k=0}^\infty \frac{1}{k!} P \ldsb J_k (\eps), I (\lambda_t (\eps)|_S)\rdsb \\
	& = -\sum_{k=0}^\infty \frac{1}{k!} \frakm_{k+1} (-s(\eps), \cdots, -s(\eps), \lambda_t (\eps) |_S ),
	\end{align*}
	where we used~\eqref{eq:kerP}, and the fact that $MC(-s_t (\eps)) = 0$ for all $t$.
	This concludes the proof.
\end{proof}

\begin{corollary}
	\label{cor:infequi}
	Two solutions of~\eqref{eq:infdef} are infinitesimally Hamiltonian equivalent iff they are cohomologous elements of the complex $(\Gamma (\wedge^\bullet N_\ell S \otimes \ell)[1], \frakm_1)$.
	Hence, the infinitesimal moduli space (i.e.~the set of infinitesimal Hamiltonian equivalence classes) of infinitesimal coisotropic deformations of $S$ is $H^0(\Gamma (\wedge^\bullet N_\ell S \otimes \ell)[1], \frakm_1) = H^1 (N_\ell {}^\ast S , \ell)$.
\end{corollary}

\section{Fiber-wise entireness: from formal deformations to non-formal deformations}
\label{sec:radially_entire}

In this section we establish a connection between (Hamiltonian equivalence of) formal coisotropic sections and (Hamiltonian equivalence of) non-formal coisotropic sections.
We do this introducing the notion of fiber-wise entireness for Jacobi structures and their bi-symbols, which slightly generalizes the analogous notion of fiber-wise entireness for Poisson structures introduced by Sch\"atz and Zambon in~\cite{SZ2012}.

Let $E \to S$ be a vector bundle. Recall that a smooth function on $E$ is called \emph{fiber-wise entire} if its restriction to each fiber of $E$ is \emph{entire}, i.e.~it is real analytic on the whole fiber.
Now, let $\ell \to S$ be a line bundle, and $L := E \times_S \ell$.
A section of $L$ is called \emph{fiber-wise entire} if it is a linear combination of fiber-wise constant sections, with coefficients being fiber-wise entire functions.
In particular, fiber-wise linear and the fiber-wise constant sections of $L$ (as in Section~\ref{subsubsec:fiberwise_linear_Jacobi_structures}) are fiber-wise entire. 
Let $\Theta \in \Gamma (\wedge^k (TE \otimes L^\ast)\otimes L)$.
We regard $\Theta$ as a multi-linear map
\[
\Theta : \wedge^{k} (T^\ast E \otimes L) \longrightarrow L.
\]
The multi-linear map $\Theta$ is called \emph{fiber-wise entire} if
\[
\Theta (df_1 \otimes \lambda_1 , \ldots, df_{k} \otimes \lambda_{k})
\]
is fiber-wise entire, whenever $f_1, \ldots, f_{k} \in C^\infty (E)$ and $\lambda_1, \ldots, \lambda_{k}$ are fiber-wise linear.
Equivalently $\Theta$ is fiber-wise entire if its components in some (and therefore any) system of vector bundle coordinates are fiber-wise entire functions (cf.~\cite[Lemmas 1.4, 1.7]{SZ2012}).
Let $\Delta\in\Der^k L$.
We regard $\Delta$ as a multi-linear map
\[
\Delta : \wedge^{k} (J^1L) \longrightarrow L.
\]
The multi-differential operator $\Delta$ is said to be \emph{fiber-wise entire} if
\[
\Delta (j^1\lambda_1 , \ldots, j^1\lambda_k)
\]
is a fiber-wise entire section of $L$, whenever $\lambda_1, \ldots, \lambda_k\in\Gamma(L)$ are fiber-wise linear.
Equivalently, $\Delta$ is fiber-wise entire if its component in vector bundle coordinates are fiber-wise entire.

Now, let $S$ and $(NS, L, J = \{-,-\})$ be as in Section~\ref{sec:cois_def}.
For any section $s \in \Gamma(NS)$, the \emph{Maurer--Cartan series of $s$} is defined by
\[
MC(-s) := \sum_{k=1}^\infty \frac{1}{k!} \frakm_k (-s, \ldots, -s).
\]
In general, $MC(-s)$ does not converge, not even for a coisotropic section $s$.
The following proposition generalizes the main result of~\cite{SZ2012} establishing a necessary and sufficient condition for the convergence of the Maurer--Cartan series $MC(-s)$ of a generic section $s \in \Gamma(NS)$.
In this way, we can describe (see Corollary~\ref{cor:conver}) coisotropic sections in terms of non-formal Maurer--Cartan elements.

\begin{proposition}
	\label{prop:fana}
	The bi-symbol $\Lambda_J$ of the Jacobi bi-differential operator $J$ is fiber-wise entire iff, for all sections $s \in \Gamma (NS)$, the Maurer--Cartan series $MC (-s)$ converges to $P ( \exp I(s)_\ast J) = P (\exp I(s)_\ast \Lambda_J )$ in the sense of point-wise convergence.
\end{proposition}

\begin{proof} Let $z^\alpha = (x^i, y^a)$ be vector bundle coordinates on $NS$, with $x^i$ coordinates on $S$, and $y^a$ linear coordinates along the fibers of $NS$.
	Moreover, let $\mu$ be a fiber-wise constant local generator of $\Gamma (L)$.
	The Jacobi bi-differential operator $J$ is locally given by Equation~\eqref{eq:components_of_J}, or, equivalently, Equation~\eqref{eq:components_of_J_bis}
	\begin{align*}
	J &= J^{\alpha \beta} \nabla_\alpha \wedge \nabla_\beta \otimes \mu + J^\alpha \nabla_\alpha \wedge \operatorname{id} \\
	& = \left( J^{a b}\nabla_a\wedge\nabla_b +2J^{ai}\nabla_a\wedge\nabla_i +J^{ij}\nabla_i\wedge\nabla_j \right) \otimes\mu+ \left( J^a\nabla_a +J^i\nabla_i \right) \wedge\operatorname{id}.
	\end{align*}
	Accordingly, the bi-symbol $\Lambda_J$ is locally given by
	\[
	\begin{aligned}
	\Lambda_J & = J^{\alpha \beta} \delta_\alpha \wedge \delta_\beta \otimes \mu \\
	& = \left( J^{a b}\delta_a\wedge\delta_b +2J^{ai}\delta_a\wedge\delta_i +J^{ij}\delta_i\wedge\delta_j \right) \otimes\mu
	\end{aligned}
	\]
	where $\delta_\alpha := \partial_\alpha \otimes \mu^\ast$.
	In particular, $\Lambda_J$ is fiber-wise entire iff its components $J^{ab}, J^{ai}, J^{ij}$ are fiber-wise entire functions.
	Now, let $s\in \Gamma (NS)$ and denote by $\{ \Phi_t \}$ the one parameter group of automorphisms of $L$ generated by $I(s)$.
	Then, from $P(J) = P (\Lambda_J) = 0$, Equations~\eqref{eq:mcl}, \eqref{eq:PJvsPLambda}, and the very definition of the Lie derivative, we get
	\begin{equation*}
	MC (-s) = P \sum_{k=0}^\infty \left. \frac{\partial^k (\Phi_{-t_1 - \cdots - t_k})_\ast \Lambda_J}{\partial t_1 \cdots \partial t_k}\right|_{t_1 = \cdots = t_k = 0}=P \sum_{k=0}^\infty \frac{1}{k!} \left. \frac{d^k}{d t^k} \right|_{t=0}(\Phi_{-t})_\ast \Lambda_J.
	\end{equation*}
	Let $(x,y,\lambda) \in L$, $x \in S$, $y \in N_x S$, $\lambda \in\ell_x $.
	Then $\Phi_{-t} (x,y, \lambda) = (x, y -t s(x) , \lambda)$ and
	\begin{align*}
	(\Phi_{-t})_\ast \Lambda_J&=(J^{a b} \circ \Phi_t)\delta_a\wedge\delta_b \otimes\mu +2(J^{ai} \circ \Phi_t) \delta_a\wedge(\delta_i - t s_{i}^b \delta_b ) \otimes\mu\\
	&\phantom{=} +(J^{ij}\circ \Phi_t )(\delta_i - t s_{i}^a \delta_a )\wedge (\delta_j - t s_{j}^b \delta_b) \otimes\mu,
	\end{align*}
	where $s^a_i$ denotes the partial derivative wrt~$x^i$ of the $a$-th local component of $s$ in the local basis $(\partial_a)$ of $\Gamma (NS)$.
	Hence
	\begin{equation}\label{eq:MC_Taylor}
	MC(-s) = \sum_{k=0}^\infty \frac{1}{k!}\left. \frac{d^k}{dt^k}\right|_{t=0}\left[ J^{ab} \circ ts - 2t s_i^b (J^{ai}\circ ts ) + t^2 s_i^a s_j^b (J^{ij} \circ ts) \right] \delta_a \wedge \delta_b \otimes \mu .
	\end{equation}
	Assume that $\Lambda_J$ is fiber-wise entire.
	Then the Taylor expansions in $t$, around $t = 0$, of $J^{ab} \circ ts$, $J^{ai} \circ ts$, and $J^{ij} \circ ts$ converge for all $t$'s, in particular for $t = 1$.
	It immediately follows that the series in the rhs of~\eqref{eq:MC_Taylor} converges as well.
	This proves the ``only if'' part of the proposition (cf.~the proof of the analogous proposition in~\cite{SZ2012}).
	
	For the ``if part'' of the proposition assume that the series in the rhs of~\eqref{eq:MC_Taylor} converges for all $s$.
	First of all, locally, we can choose $s$ to be ``constant'' wrt~coordinates $(x^i, y^a)$.
	Then $s^a_i = 0$ and~\eqref{eq:MC_Taylor} reduces to
	\begin{equation}\label{Taylor}
	MC(-s) = \sum_{k=0}^\infty \frac{1}{k!}\left. \frac{d^k}{dt^k}\right |_{t=0}\left( J^{ab} \circ ts \right) \delta_a \wedge \delta_b \otimes \mu .
	\end{equation}
	Since $s$ is arbitrary, \eqref{Taylor} shows that the $J^{ab}$'s are entire on any straight line through the origin in the fibers of $NS$.
	Since the Taylor series of the restriction to such a straight line is the same as the restriction of the Taylor series, we conclude that the $J^{ab}$'s are fiber-wise entire.
	Now, fix values $i_0, a_0$ for the indexes $i,a$ respectively, and choose $s$ so that $s^a_i = \delta_i^{i_0}\delta_{a_0}^a$ to see that the $J^{a_0 i_0}$'s are fiber-wise entire for all $a_0, i_0$.
	One can prove that the $J^{ij}$'s are fiber-wise entire in a similar way.
	This concludes the proof.
\end{proof}

\begin{corollary}
	\label{cor:conver}
	Let $(M, L, J= \{-,-\})$ be a Jacobi manifold, and let $S \subset M$ be a coisotropic submanifold equipped with a fat tubular neighborhood $\tau : L_{NS} \injects L$.
	If $\tau_\ast ^{-1} \Lambda_J $ is fiber-wise entire, then a section $s : S \to NS$ of $NS$ is coisotropic iff the Maurer--Cartan series $MC (-s)$ converges to zero.
\end{corollary}

Now, we establish a necessary and sufficient condition for the convergence of both the Maurer--Cartan series $MC(-s)$ and the series
\begin{equation}\label{eq:delta_MC}
\delta_\lambda MC(-s) := \sum_{k=0}^\infty \frac{1}{k!} \frakm_{k+1} (-s, \ldots, -s, \lambda )
\end{equation}
for generic sections $s \in \Gamma(NS)$ and $\lambda \in \Gamma (\ell)$.
In this way, we can describe (see Corollary~\ref{cor:hequi}) moduli of coisotropic sections under Hamiltonian equivalence in terms of gauge equivalence classes of non-formal Maurer--Cartan elements.
Again, let $S$ and $(NS, L, J = \{-,-\})$ be as in Section~\ref{sec:cois_def}. 

\begin{theorem}
	\label{prop:fana1}
	The Jacobi bi-differential operator $J$ is fiber-wise entire iff, for all sections $s \in \Gamma (NS)$, and $\lambda \in \Gamma (L)$, the Maurer--Cartan series $MC (-s)$ converges to $P (\exp I(s)_\ast J)$, and the series $\delta_{\lambda|_S} MC (-s)$ \eqref{eq:delta_MC} converges to $P(\exp I(s)_\ast \Delta_\lambda)$, in the sense of point-wise convergence.
\end{theorem}

\begin{proof}
	We already know that the bi-linear form $\Lambda_J$ is fiber-wise entire iff $MC(-s)$ converges for all $s$.
	Now, it is easy to see that $P (\exp \calL_{I(s)} \Delta_\lambda) = P (\exp \calL_{I(s)} X_\lambda)$ for all $s \in \Gamma (NS)$, and $\lambda \in \Gamma (L)$ (cf.~\eqref{eq:PJvsPLambda}).
	Moreover, from the proof of Proposition~\ref{prop:MC_gauge}, we get
	\[
	\delta_{\lambda|_S} MC (-s) = - P (\exp \calL_{I(s)} \Delta_\lambda) = - P (\exp \calL_{I(s)} X_\lambda).
	\]
	Therefore, similarly as in the proof of Proposition~\ref{prop:fana}, we find
	\[
	\delta_{\lambda|_S} MC (-s) = -P \sum_{k=0}^\infty \frac{1}{k!} \left. \frac{d^k}{d t^k} \right|_{t=0}(\Phi_{-t})_\ast X_\lambda.
	\]
	The bi-differential operator $J$ is locally given by~\eqref{eq:components_of_J_bis}, hence a straightforward computation shows that
	\[
	\begin{aligned}
	& \delta_{\lambda|_S} MC (-s) \\
	&= \sum_{k=0}^\infty \frac{1}{k!} \left. \frac{d^k}{dt^k}\right |_{t=0} \left[ 2\partial_i g (J^{ai} \circ ts)-2ts_j^a\partial_i g (J^{ij} \circ ts) + g (J^a \circ ts) -t s^a_i g (J^i \circ s) \right] \partial_a,
	\end{aligned}
	\]
	where we used the same notations as in the proof of Proposition~\ref{prop:MC_gauge}, and $g$ is the component of $\lambda|_S$ in the basis $\mu$.
	The assertion now follows in a very similar way as in the proof of Proposition~\ref{prop:MC_gauge}.
\end{proof}

\begin{corollary}
	\label{cor:hequi}
	Let $(M, L, J= \{-,-\})$ be a Jacobi manifold, and let $S \subset M$ be a compact coisotropic submanifold equipped with a fat tubular neighborhood $\tau : \ell \injects L$.
	If $\tau_\ast^{-1}J$ is fiber-wise entire, then two solutions $s_0, s_1 : S \to NS$ of the (well-defined) Maurer--Cartan equation $MC (-s) = 0$ are Hamiltonian equivalent iff they are interpolated by a smooth family of sections $s_t \in \Gamma (NS)$ and there exists a smooth family of sections $\lambda_t$ of $\ell$ such that $s_t$ is a solution of the following well-defined evolutionary equation:
	\[
	\frac{d}{dt} s_t = \delta_{\lambda_t} MC (-s_t).
	\]
\end{corollary}

\begin{remark}
	\label{rem:Schaetz-Zambon}
	Immediately after a preliminary version of these results appeared on arXiv (see~\cite{LOTV}), Sch\"atz and Zambon, independently, finalized a pre-print where they discuss the moduli space of coisotropic submanifolds of a symplectic manifold.
	In particular, they use our method to prove Corollary~\ref{cor:hequi} in the symplectic case (see~\cite[Theorem 3.21]{SZ2014}).
	Notice that $\tau_\ast^{-1}J$ is automatically fiber-wise entire in Sch\"atz--Zambon situation and, therefore, convergence issues do not appear in their work.
\end{remark}

\section{Simultaneous coisotropic deformations}
\label{sec:simultaneous_coisotropic_deformations}

Let $(M, L, J=\{-,-\})$ be a Jacobi manifold and let $S \subset M$ be a closed coisotropic submanifold.
Earlier in this chapter we only considered the problem of deforming $S$ into a new submanifold $S'\subset M$ which is still coisotropic wrt the \emph{fixed} Jacobi structure $J$ on $L\to M$.
In this section, we remove the constraint of $J$ being fixed, and we consider the extended problem of deforming simultaneously $J$ into a new Jacobi structure $J'$ on $L\to M$, and $S$ into a new submanifold $S'\subset M$ which is coisotropic wrt $J'$.
Starting again from the set of V-data singled out in Lemma~\ref{lem:linfty}, and implementing higher derived brackets construction for arbitrary derivations~\cite{Voronov2005higher2}, we extend the $L_\infty[1]$-algebra constructed in Proposition~\ref{prop:linfty} to a new $L_\infty[1]$-algebra controlling simultaneous deformations at the formal level.
Similar results were obtained first by Frégier--Zambon~\cite{fregier2014simultaneous} for fiber-wise polynomial Poisson structures and then by Sch\"atz--Zambon~\cite{SZ2012} for fiber-wise entire Poisson structures.
In this section no restrictive assumption is made on the Jacobi structures.
An explicit example, in the contact setting, where the simultaneous deformation problem is formally obstructed will be exhibited in  Section~\ref{sec:obstructed_example_contact_L-infinity}.

Fix a fat tubular neighborhood $(\tau,\underline{\smash{\tau}})$ of $\ell\to S$ in $L\to M$ and use it to identify multi-derivations on $L_{NS}\to NS$ with multi-derivations on its image.
Similarly as above, from now on in this section, we abuse the notation and denote by $(L,J=\{-,-\})$ (instead of $(L_{NS},\tau_\ast^{-1}J)$) the Jacobi bundle on $NS$ (unless otherwise specified).
Set $\frakg^\bullet(S):=\Gamma(\wedge^\bullet N_\ell S\otimes\ell)$ and $\frakh^\bullet(S):=\Der^\bullet(L_{NS})[1]\oplus\frakg^\bullet(S)$.
In Section~\ref{sec:linfty_algebra}, we used higher derived brackets technique~\cite{Voronov2005higher1} to construct an $L_\infty[1]$-algebra $(\frakg^\bullet(S),\{\frakm_k\})$ out of V-data $((\Der^\bullet L_{NS})[1],\im I,P,J)$.
Now, following Th.~Voronov~\cite[Theorem 2]{Voronov2005higher2}, we will adapt this technique to construct a larger $L_\infty[1]$-algebra out of the same set of V-data.
\begin{proposition}
	\label{prop:extended_linfty}
	There is an $L_\infty[1]$-algebra structure on $\frakh^\bullet(S)[1]$ whose degree $1$ graded symmetric multi-brackets $\frakn _k:\frakh^\bullet(S)[1]^{\otimes k}\to\frakh^\bullet(S)[1]$ are given by the following higher derived brackets
	\begin{align*}
	\frakn_1(\Delta)&:=(-\ldsb J,\Delta\rdsb,P\Delta),\\
	\frakn_2 (\Delta,\Delta')&:=(-)^{|\Delta|}\ldsb\Delta,\Delta'\rdsb,
	\end{align*}
	and, for $k\geq 1$,
	\begin{align*}
	\frakn_{k+1} (\Delta,\xi_1,\ldots,\xi_k)&:=P \ldsb\ldsb\ldots\ldsb\Delta, I(\xi_1)\rdsb,\ldots\rdsb,I(\xi_k)\rdsb,\\
	\frakn_k (\xi_1,\ldots,\xi_k)&:=P \ldsb\ldsb\ldots\ldsb J, I(\xi_1)\rdsb,\ldots\rdsb,I(\xi_k)\rdsb,
	\end{align*}
	for all $\Delta,\Delta'\in(\Der^\bullet L_{NS})[1]$, and $\xi_1,\ldots,\xi_k\in\frakg^\bullet(S)$.
	Up to graded symmetry and $\bbR$-multi-linearity, all the other multi-brackets are set to be zero. 
\end{proposition}

Notice that both the graded Lie algebra $((\Der^\bullet L_{NS})[1],\ldsb-,-\rdsb)$, up to décalage, and the $L_\infty[1]$-algebra $(\frakg^\bullet(S)[1],\{\frakm_k\})$ are naturally embedded into $(\frakh^\bullet(S)[1],\{\frakn_k\})$.

\subsubsection{Smooth simultaneous coisotropic deformations}
\label{subsubsec:smooth_simultaneous_coisotropic_deformations}

\begin{definition}
	\label{def:smooth_simultaneous_coisotropic_deformation}
	A smooth one parameter family $(J_t,s_t)$ in $\frakh^1(S)=\Der^2L_{NS}\oplus\Gamma(NS)$ starting from $(J,\mathbf 0)$ is a \emph{smooth simultaneous coisotropic deformation of} $(J,S)$ if $J_t$ is a Jacobi structures on $L_{NS}\to NS$ and $s_t$ is a coisotropic section wrt $J_t$, for all $t$.
	A pair $(\square,s)\in\frakh^1(S)$ is an \emph{infinitesimal simultaneous coisotropic deformation} if $(J,\mathbf 0)+\eps (\square,s)$ is a smooth simultaneous coisotropic deformation of $(J,S)$ up to infinitesimals $\calO(\eps^2)$, where $\eps$ is a formal parameter.
\end{definition}

\begin{remark}
	Let $(J_t,s_t)$ be a smooth simultaneous coisotropic deformation of $(J,S)$.
	Then
	\[
	\left. \frac{d}{dt}\right|_{t=0} (J_t,s_t)
	\]
	is an infinitesimal simultaneous coisotropic deformation of $(J,S)$.
\end{remark}


Recalling Remark~\ref{rem:coiss}, smooth simultaneous coisotropic deformations can be characterized as follows.
\begin{proposition}
	\label{prop:smooth_simultaneous_coisotropic_deformations}
	Let $(J_t,s_t)$ be a smooth one parameter family in $\frakh^1(S)$ starting from $(J,\mathbf 0)$.
	The following two conditions are equivalent
	\begin{enumerate}
		\item $(J_t,s_t)$ is a smooth simultaneous coisotropic deformation of $(J,S)$,
		\item $\ldsb J_t,J_t\rdsb=0$ and $P (\exp I (-s_t)_* J_t) = 0$.
	\end{enumerate}
\end{proposition}

\begin{proof}
	It is a straightforward consequence of Propositions~\ref{prop:J_as_MC_element} and~\ref{prop:coiss}.
\end{proof}

\begin{corollary}
	\label{cor:infinitesimal_simultaneous_coisotropic_deformations}
	For any pair $(\square,s)\in\frakh^1(S)$, the following two conditions are equivalent
	\begin{enumerate}
		\item $(\square,s)$ is an infinitesimal simultaneous coisotropic deformation of $(J,S)$,
		\item $(\square,-s)$ is a $1$-cocycle in $(\frakh^\bullet(S),\frakn_1)$, i.e. $\frakn_1(\square,-s)=0$, or explicitly $d_J\square=0$ and $d_{N_\ell{}^\ast S,\ell}s\equiv\frakm_1 s=P\square$.
	\end{enumerate}
\end{corollary}

\subsubsection{Formal simultaneous coisotropic deformations}
Let $\eps$ be a formal parameter.

\begin{definition}
	\label{def:formal_simultaneous_deformations}
	A formal series $(J(\eps),s(\eps)) = \sum_{i=0}^\infty\eps^i (J_i,s_i)\in\frakh^1(S)[[\eps]]$, with coefficients in $\frakh^1(S)$, such that $(J_0,s_0)=(J,\mathbf 0)$, is called a \emph{formal deformation of $(J,S)$}.
\end{definition}

%

Remark~\ref{rem:Lieder} and Proposition~\ref{prop:smooth_simultaneous_coisotropic_deformations} motivate the next definition of \emph{formal simultaneous coisotropic de\-for\-ma\-tion} of $(J,S)$.
For further motivations see also Remark~\ref{rem:motivation_formal_simultaneous_coisotropic_deformations}.

\begin{definition}
	\label{def:formal_simultaneous_coisotropic_deformations}
	A formal deformation $(J(\eps),s(\eps))$ of $(J,S)$ is said \emph{simultaneous co\-iso\-tro\-pic}, if $\ldsb J(\eps),J(\eps)\rdsb=0$ and $P (\exp \calL_{I(s (\eps))} J(\eps)) = 0$.
\end{definition}

\begin{remark}
	\label{rem:motivation_formal_simultaneous_coisotropic_deformations}
	Let $(J_t,s_t)$ be a smooth one parameter family in $\frakh^1(S)$ starting from $(J,\mathbf 0)$.
	Denote by $(J(\eps),s(\eps))$ the Taylor series of $(J_t,s_t)$, at $t=0$, in the formal parameter $\eps$, i.e.
	\begin{equation*}
	(J(\eps),s(\eps)):=\sum_{i=1}^\infty\underbrace{\frac{1}{i!}\left(\left.\frac{d^i}{dt^i}\right|_{t=0}(J_t,s_t)\right)}_{:=(J_i,s_i)}\eps^i.
	\end{equation*}
	Applying the Leibniz rule iteratively, we get, for all $n\geq 0$,
	\begin{align*}
	\frac{1}{n!}\left.\frac{d^n}{dt^n}\right|_{t=0}\ldsb J_t,J_t\rdsb&=\sum_{\genfrac{}{}{0pt}{}{h,k}{h+k=n}}\ldsb J_h,J_k\rdsb,\\
	\frac{1}{n!}\left.\frac{d^n}{dt^n}\right|_{t=0}P(\exp(I(-s_t))_\ast J_t)&=\sum_{k=1}^n\frac{1}{(k-1)!}\sum_{\genfrac{}{}{0pt}{}{i_1,\ldots,i_k}{i_1+\ldots+i_k=n}}P\ldsb I(s_{i_1}),\ldots\ldsb I(s_{i_{k-1}}),J_{i_k}\rdsb\ldots\rdsb.	
	\end{align*}
	The latter means that $\ldsb J(\eps),J(\eps)\rdsb$ and $P(\exp\calL_{I(s(\eps))}(J(\eps)))$ provide exactly the Taylor series, at $t=0$, of $\ldsb J_t,J_t\rdsb $ and $P(\exp(I(-s_t))_\ast J_t)$ respectively:
	\begin{align*}
	\ldsb J(\eps),J(\eps)\rdsb&=\sum_{n=0}^\infty\frac{1}{n!}\left(\left.\frac{d^n}{dt^n}\right|_{t=0}\ldsb J_t,J_t\rdsb\right)\eps^n,\\
	P(\exp\calL_{I(s(\eps))}(J(\eps)))&=\sum_{n=0}^\infty\frac{1}{n!}\left(\left.\frac{d^n}{dt^n}\right|_{t=0}P(\exp(I(-s_t))_\ast J_t)\right)\eps^n.
	\end{align*}
	Hence if $\{(J_t,s_t)\}$ is a smooth simultaneous coisotropic deformation of $(J,S)$, then its Taylor series $(J(\eps),s(\eps))$ at $t=0$ is a formal simultaneous coisotropic deformation of $(J,S)$. 
\end{remark}


\begin{remark}[Formal simultaneous deformation problem]
	\label{rem:formal_simultaneous_deformation_problem}
	The \emph{formal simultaneous deformation problem for a coisotropic submanifold $S$ of a Jacobi manifold $(M,L,J)$} consists in finding formal simultaneous coisotropic deformations of $(J,S)$.
	The guiding philosophy of deformation theory expressed by the Deligne's principle (see also Remark~\ref{rem:formal_deformation_problem}) works for formal simultaneous deformations of $(J,S)$ as well.
	Indeed the latter is \emph{controlled by the $L_\infty[1]$-algebra} $(\frakh^\bullet(S),\{\frakn_k\})$ in the sense clarified by next Proposition~\ref{prop:mc_formal_simultaneous}.
	In this way we generalize to the Jacobi setting what was first established by Frégier--Zambon~\cite{fregier2014simultaneous} (for fiber-wise polynomial Poisson structures) and Sch\"atz--Zambon~\cite{SZ2012} (for fiber-wise entire Poisson structures).
\end{remark}

\begin{proposition}
	\label{prop:mc_formal_simultaneous}
	Let $(\square(\eps),s(\eps))=\sum_{i=1}^\infty\eps^i (\square_i,s_i)\in\frakh^1(S)[[\eps]]$ be an arbitrary formal series, with coefficients in $\frakh^1(S)$, vanishing up to infinitesimals $\calO(\eps)$.
	Then the following two conditions are equivalent:
	\begin{itemize}
		\item $(J+\square(\eps),-s(\eps))$ is a formal simultaneous coisotropic deformation of $(J,S)$,
		\item $(\square(\eps),s(\eps))$ is a formal Maurer--Cartan element of $L_\infty[1]$-algebra $(\frakh^\bullet(S)[1],\{\frakn_k\})$, i.e.~it satisfies the formal Maurer--Cartan equation
			\begin{equation}
			\label{eq:extended_MC}
			MC(\square(\eps),s(\eps)):=\sum_{k=1}^\infty\frac{1}{k!}\frakn_k((\square(\eps),s(\eps)),\ldots,(\square(\eps),s(\eps)))=0.
			\end{equation}
	\end{itemize}
\end{proposition}

\begin{proof}
	The expression $MC(\square(\eps),s(\eps))$ should be interpreted as a formal series, with coefficients in $\frakh^2(S)=\Der^3(L_{NS})\oplus \Gamma(\wedge^2 N_\ell S \otimes \ell)$, vanishing up to infinitesimals $\calO(\eps)$.
	From Propositions~\ref{prop:linfty} and~\ref{prop:extended_linfty} it follows that
	\begin{align*}
	\frakn_1(\square(\eps),s(\eps))\!&=\!(-\ldsb J,\square(\eps)\rdsb,\frakm_1(s(\eps))+P\square(\eps)),\\
	\frakn_2((\square(\eps),s(\eps)),(\square(\eps),s(\eps)))\!&=\!(-\ldsb\square(\eps),\square(\eps)\rdsb,\frakm_2(s(\eps),s(\eps))+2P\ldsb\square(\eps),Is(\eps)\rdsb),
	\end{align*}
	and, for $k\geq 3$,
	\begin{multline*}
	\frakn_k((\square(\eps),s(\eps)),\ldots,(\square(\eps),s(\eps)))\\=(0,\frakm_k(s(\eps),\ldots,s(\eps))+kP\ldsb\ldots\ldsb\square(\eps),\underbrace{Is(\eps)\rdsb,\ldots,Is(\eps)}_{(k-1)-\textnormal{times}}\rdsb).
	\end{multline*}
	Moreover, from the definition~\eqref{eq:Lieder} of $\calL_{I(-s(\eps))}$, the basic assumption $PJ=0$, and the definition of the $\frakm_k$'s, we immediately get the following identities
	\begin{equation*}
	\begin{aligned}
	P (\calL_{I(-s(\eps))}^k J)&=
	\begin{cases}
	0,&\text{if } k=1,\\
	\frakm_k(s(\eps),\ldots,s(\eps)),&\text{if } k\geq 1,
	\end{cases}
	\\
	P (\calL_{I(-s(\eps))}^k\square(\eps))&=P\ldsb\ldots\ldsb\square(\eps),\underbrace{Is(\eps)\rdsb\ldots,Is(\eps)}_{k-\textnormal{times}}\rdsb,
	\end{aligned}
	\end{equation*}
	for all $k\geq 0$.
	Finally, from what above, the basic assumption $\ldsb J,J\rdsb=0$, and the definition~\eqref{eq:exp} of $\exp\calL_{I(-s(\eps))}$, it follows that
	\begin{equation*}
	MC(\square(\eps),s(\eps))=\left(-\frac{1}{2}\ldsb J+\square(\eps),J+\square(\eps)\rdsb,P(\exp\calL_{I(-s(\eps))}(J+\square(\eps)))\right),
	\end{equation*}
	which concludes the proof.
\end{proof}

\begin{remark}
	Let $(\square(\eps),s(\eps))$ be an arbitrary formal series, with coefficients in $\frakh^1(S)$, vanishing up to infinitesimals $\calO(\eps)$.
	Hence $(J(\eps),s(\eps)):=(J,\mathbf 0)+(\square(\eps),s(\eps))$ is an arbitrary formal deformation of $(J,S)$, and $(\square(\eps),s(\eps))=\sum_{i=1}^\infty\eps^i(\square_i,s_i)$, for arbitrary $(\square_i,s_i)\in\frakh^1(S)$.
	Now $MC(\square(\eps),-s(\eps))=\sum_{i=1}^\infty\eps^i\Gamma'_i$, where $\Gamma'_i\in\frakh^2(S)$ is given by
	\begin{equation*}
	\Gamma'_i=\sum_{k=1}^i\frac{1}{k!}\sum_{\genfrac{}{}{0pt}{}{i_1,\ldots,i_k>0}{i_1+\ldots+i_k=i}}\frakn_k((\square_{i_1},-s_{i_1}),\ldots,(\square_{i_1},-s_{i_k})),
	\end{equation*}
	for all $i>0$.
	As a consequence, formal Maurer--Cartan equation~\eqref{eq:extended_MC} for $(\square(\eps),-s(\eps))$ splits in the following infinite sequence of equations for the $\square_i$'s and the $s_i$'s
	\begin{equation}
	\label{eq:MC'_k}
	\tag{$MC'_k$}
	\frakn_1(\square_k,-s_k)=-\sum_{h=2}^k\frac{1}{h!}\sum_{\genfrac{}{}{0pt}{}{0<i_1,\ldots,i_h<k}{i_1+\ldots+i_h=k}}\frakn_h\left((\square_{i_1},-s_{i_1}),\ldots,(\square_{i_h},-s_{i_h})\right),
	\end{equation}
	for all $k>0$.
	In particular~\eqref{eq:MC'_k}, for $k=1$, reads as $\frakn_1(\square_1,-s_1)=0$.
	In view of Proposition~\ref{prop:mc_formal_simultaneous} and Corollary~\ref{cor:infinitesimal_simultaneous_coisotropic_deformations}, the latter means that, if $(J(\eps),s(\eps))$ is a formal simultaneous coisotropic deformation of $(J,S)$, then $(\square_1,s_1)$ is an infinitesimal simultaneous coisotropic deformation of $(J,S)$.
	In this case one says that $(J(\eps),s(\eps))$ is a prolongation of the infinitesimal simultaneous coisotropic deformation $(\square_1,s_1)$ to a formal one.
	The formal simultaneous deformation problem is said to be \emph{unobstructed} if every infinitesimal simultaneous coisotropic deformation of $(J,S)$ can be ``prolonged'' to a formal one, i.e.~for any given $(\overline{\square}_1,-\overline{s}_1) \in\frakh^1(S)$, with $\frakn_1(\overline{\square}_1,-\overline{s}_1)=0$, there is a formal simultaneous coisotropic deformation $(J(\eps),s(\eps))=(J,\mathbf 0)+(\square(\eps),s(\eps))$ of $(J,S)$ such that $(\square_1,s_1)=(\overline{\square}_1,\overline{s}_1)$.
	However, in general, not all infinitesimal simultaneous coisotropic deformations can be prolonged to a formal one.
	In this case the formal deformation problem is \emph{obstructed}.
\end{remark}

Let $(J(\eps),s(\eps):=(J,\mathbf 0)+(\square(\eps),s(\eps))$ be an arbitrary formal deformation of $(J,S)$,
From the generalized Jacobi identities for the $L_\infty[1]$-algebra $(\frakh^\bullet(S),\{\frakn_k\})$, through a straightforward computation, we get the following identity
\begin{equation}
\label{eq:simultaneous_obstructions}
\sum_{k=1}^\infty\frac{1}{(k-1)!}\frakn_k(MC(\square(\eps),-s(\eps)),(\square(\eps),-s(\eps)),\ldots,(\square(\eps),-s(\eps)))=0.
\end{equation}  
Fix $n>0$.
Assume that $(\square(\eps),-s(\eps))$ satisfies the formal Maurer--Cartan equation~\eqref{eq:extended_MC} up to infinitesimals $\calO(\eps^n)$, i.e.~Equation~\eqref{eq:MC'_k} holds for all $k=1,\ldots,n-1$.
Notice that this assumption only concerns the first $(n-1)$-coefficients $s_1,\ldots,s_{n-1}$ and $\square_1,\ldots,\square_{n-1}$ of $s(\eps)$ and $\square(\eps)$ respectively.
Now, from~\eqref{eq:simultaneous_obstructions}, it follows that $\frakn_1(MC(\square(\eps),-s(\eps)))$ vanishes up to infinitesimals $\calO(\eps^{n+1})$, i.e.~the rhs of~\eqref{eq:MC'_k}, for $k=n$, is a $2$-cocycle in $(\frakh^\bullet(S),\frakn_1)$.
Hence the cohomology class of the rhs of~\eqref{eq:MC'_k}, for $k=n$, represents the obstruction to the prolongability of the infinitesimal coisotropic deformation $(\square_1,s_1)$ to a formal one up to infinitesimals $\calO(\eps^{n+1})$.
Indeed, under the current assumption,~\eqref{eq:MC'_k}, for $k=n$, admits a solution $(\square_n,s_n)$ if and only if the cohomology class of its rhs is zero.
This leads to the following unobstructedness criterion for the formal simultaneous deformation problem.

\begin{corollary}
	\label{prop:simultaneous_rigid}
	Assume that the second cohomology group $H^2 (\frakh^\bullet(S),\frakn_1)$ of the $L_\infty[1]$-algebra $(\frakh^\bullet(S)[1],\{\frakn_k\})$ vanishes.
	Then every infinitesimal simultaneous coisotropic deformation of $(J,S)$ can be prolonged to a formal simultaneous coisotropic deformation, i.e.~the formal simultaneous deformation problem is unobstructed.
\end{corollary}

There is also a simple criterion for the non-prolongability of an infinitesimal simultaneous coisotropic deformation of $(J,S)$ to a formal simultaneous coisotropic deformation of $(J,S)$.
It is expressed in terms of the \emph{extended Kuranishi map}:
\begin{equation}
\label{eq:extended_L_infty_kuranishi_map}
\operatorname{Kr} : H^1 (\frakh^\bullet(S),\frakn_1) \longrightarrow H^2 (\frakh^\bullet(S),\frakn_1), \quad [(\square,s)] \longmapsto [\frakn_2 ((\square,s),(\square,s))].
\end{equation}
Since $\frakn_1$ is a derivation of the binary bracket $\frakn_2$, the extended Kuranishi map is well-defined.
Moreover, Equation~\eqref{eq:MC'_k}, for $k=2$, reads
\begin{equation*}
\frakn_1 (\square_2,-s_2)=\frac{1}{2}\frakn_2((\square_1,-s_1),(\square_1,-s_1)).
\end{equation*}
Hence $\operatorname{Kr}[(\square_1,-s_1)]$ is the first obstruction we meet when we try to prolong an infinitesimal simultaneous coisotropic deformation $(\square_1,s_1)$ of $(J,S)$ to a formal simultaneous coisotropic deformation $(J(\eps),s(\eps)))=(J,\mathbf 0)+(\square(\eps),s(\eps))$.
This leads to the following obstructedness criterion.

\begin{proposition}
	\label{prop:simultaneous_Kuranishi}
	Let $\beta=[(\overline{\square}_1,-\overline{s}_1)] \in H^1 (\frakh^\bullet(S),\frakn_1)$, where $(\overline{\square}_1,\overline{s}_1)\in\frakh^1(S)$ is an infinitesimal simultaneous coisotropic deformation of $(J,S)$, i.e.~$\frakn_1(\overline{\square}_1,-\overline{s}_1)= 0$.
	If $\operatorname{Kr}(\beta)\neq 0$, then $(\overline{\square}_1,\overline{s}_1)$ cannot be prolonged to a formal simultaneous coisotropic deformation of $(J,S)$.
	In particular, the formal simultaneous deformation problem is obstructed.
\end{proposition}


\chapter{The contact case}
\label{chap:contact}

As seen in Section~\ref{subsec:odd-dim_transitive} contact manifolds form a distinguished class of Jacobi manifolds.
In this chapter we get rather efficient formulas (from a computational point of view) for the multibrackets of the $L_\infty$-algebra associated to a coisotropic submanifold in a contact manifold.
These formulas are analogous to those of Oh--Park in the symplectic case~\cite[Equation (9.17)]{oh2005deformations} and L\^e--Oh in the locally conformal symplectic case~\cite[Equation (9.10)]{le2012deformations}.

In the first part of this chapter, after having briefly recalled the necessary preliminaries from (pre-)contact geometry, we introduce the special subclass of so called ``regular'' coisotropic submanifolds in a contact manifold $(M,C)$ (Definition~\ref{def:regular_coisotropic}).
In this case there is a normal form theorem (see Theorem~\ref{teor:coisotropic_embedding}) which goes back to Loose~\cite{loose1998neighborhoods}.
On the one hand, in a contact manifold, every regular coisotropic submanifold inherits a structure of pre-contact manifold.
On the other hand every pre-contact manifold admits a unique coisotropic embedding up to local contactomorphisms.

Further we investigate the implications of this normal form theorem on the $L_\infty[1]$-algebra, and hence on the deformation problem of a regular coisotropic submanifold $S$ in a contact manifold $(M,C)$.
Any distribution $G$ complementary to the characteristic distribution $T\calF$ of $S$ canonically determines a coisotropic tubular neighborhood of $S$ in $(M,C)$ (Proposition~\ref{prop:contact_thickening}), and so also a fat tubular neighborhood over it.
As a consequence we get formulas for the multibrackets of the associated $L_\infty[1]$-algebra which only depend on the intrinsic pre-contact geometry of $S$ and the transversal geometry to $T\calF$ (Theorem~\ref{theor:multi}).
As a direct application of these formulas we prove that if the pre-contact structure on $S$ is transversally integrable then the associated $L_\infty[1]$-algebra is represented by a differential graded Lie algebra up to décalage (Proposition~\ref{prop:transversally_integrable}).

Finally we use these formulas for the multibrackets for studying coisotropic deformations in explicit examples.
First, in Section~\ref{subsec:toy_examples}, we present two toy examples: a Legendrian submanifold, and its flowout wrt a transversal contact vector field.
Then, in Sections~\ref{sec:obstructed_example_contact_L-infinity} and~\ref{sec:second_obstructed_example_contact_L-infinity} we consider two examples whose coisotropic deformation problem is formally obstructed (cf.~\cite[Examples~3.5 and~3.8]{tortorella2016rigidity} and the revised version of~\cite{LOTV}).


\section{Coisotropic submanifolds in contact manifolds}
\label{sec:cois_cont}

Let $C$ be an hyperplane distribution on a smooth manifold $M$.
Denote by $L$ the quotient line bundle $TM/C$, and by $\theta:TM\to L$, $X \mapsto \theta (X) := X \Mod C$ the projection.
We will often interpret $\theta$ as an $L$-valued differential $1$-form, and call it the \emph{structure form of $C$}.
The \emph{curvature form} of $(M,C)$ is the vector bundle morphism $\omega:\wedge^2C\to L$ well-defined by
$\omega(X,Y)=\theta([X,Y])$, with $X,Y\in\Gamma(C)$.
Consider also the vector bundle morphism $\omega^\flat : C \to C^\ast \otimes L$, $X\mapsto\omega^\flat(X):=\omega(X,-)$.
The \emph{characteristic distribution} of $(M,C)$, is the (generically singular) distribution $\ker\omega^\flat = C^{\perp_\omega}$, where, as in the following, $V^{\perp_\omega}$ denotes the $\omega$-orthogonal complement of a subbundle $V \subset C$.

\begin{remark}
	\label{rem:curvature_form_X_distributions}
	Note that the definition of curvature form works verbatim for distribution of arbitrary co-rank.
\end{remark}

\begin{remark}
	\label{rem:characteristic_distribution_contact}
	The characteristic distribution of an hyperplane distribution $C$ is involutive, meaning that
	\begin{equation*}
	[\Gamma(\ker\omega^\flat),\Gamma(\ker\omega^\flat)]\subset \Gamma(\ker\omega^\flat).
	\end{equation*}
	However it is a singular distribution which is not smooth in general.
	Indeed, since its rank is upper semi-continuous, it is smooth if and only if its rank is locally constant, and in such case it is an ordinary distribution on each connected component, and it is integrable because of the ordinary Frobenius theorem.
\end{remark}

\begin{definition}
	\label{mydef:pre-contact}
	A \emph{pre-contact structure} on a smooth manifold $M$ is an hyperplane distribution $C$ on $M$ such that its characteristic distribution $\ker \omega^\flat$ has constant rank.
	A \emph{pre-contact manifold} $(M,C)$ is a smooth manifold $M$ equipped with a pre-contact structure $C$.
	The integral foliation of $\ker \omega^\flat$ is called the \emph{characteristic foliation} of $C$ and will be denoted by $\mathcal{F}$.
\end{definition}


\begin{remark}
	\label{rem:maximally-integrable}
	The curvature form $\omega$ of $(M,C)$ measures how far is $C$ from being integrable.
	Indeed, $C$ is integrable iff $\omega = 0$, or, equivalently, $\omega^\flat=0$.
	Accordingly, $C$ is said to be \emph{maximally non-integrable} when $\omega$ is non degenerate, or, equivalently, $\ker \omega^\flat =0$.
	If $C$ is maximally non-integrable, then the rank of $C$ is even, and $M$ is odd-dimensional.
	Additionally, in the maximally integrable case, $\omega^\flat$ is a vector bundle isomorphism, whose inverse will be denoted by $\omega^\sharp :C^\ast\otimes L\to C$, and $C\to M$ becomes a symplectic vector bundle with $L$-valued symplectic form $\omega$ on its fibers.
\end{remark}

\begin{definition}
	\label{mydef:contact}
	A \emph{contact structure} on a smooth manifold $M$ is a maximally non-integrable hyperplane distribution $C$ on $M$.
	A \emph{contact manifold} $(M,C)$ is a smooth manifold $M$ equipped with a contact structure $C$.
\end{definition}

\begin{example}\label{ex:1jet}
	Let $L \to M$ be a line bundle.
	There is a canonical contact structure $C$ on $J^1 L$, sometimes called the \emph{Cartan distribution} and defined as follows.
	Let $\pi : J^1 L \to M$, and $\mathrm{pr} : J^1 L \to L$ be canonical projections.
	Consider the pull-back line bundle $\pi^\ast L \to J^1 L$.
	There is a canonical $\pi^\ast L$-valued $1$-form $\theta$ on $J^1 L$ given by
	\begin{equation}
	\label{eq:ex:1jet}
	\theta (\xi_\alpha ) := (d \mathrm{pr} - d\lambda \circ d\pi) (\xi_\alpha), \quad \xi_\alpha \in T_\alpha J^1 L,
	\end{equation}
	where $\alpha = (j^1 \lambda )(x) \in J^1 L$, and $x = \pi(\alpha)$, $\lambda \in \Gamma (L)$.
	The rhs of~\eqref{eq:ex:1jet} is a tangent vector to $L$ at $\operatorname{pr}(\alpha)=\lambda(x)$ which is vertical with respect to the line bundle map $L\to M$.
	So it identifies with an element of $L_x\simeq(\pi^\ast L)_\alpha$.
	The Cartan distribution is then defined as the kernel of $\theta$.
	In particular, the line bundle $T(J^1 L)/C$ identifies canonically with $\pi^\ast L$.
	Finally notice that the Jacobi structure 
	 associated with the Cartan distribution on $J^1L$ coincides exactly with the fiberwise linear one coming from the Jacobi algebroid $(DL,L)$.
\end{example}

\begin{remark}
	\label{rem:coorientable}
	Let $(M,C)$ be a contact manifold.
	There exists a natural one-to-one correspondence between
	\begin{enumerate}
		\item local trivializations (or nowhere zero local sections) of the line bundle $L\to M$ and
		\item local contact forms of $(M,C)$, i.e.~$1$-forms $\alpha\in\Omega^1(U)$, with $U$ open in $M$, such that $C|_U=\ker\alpha$.
	\end{enumerate}
In particular, the contact structure is called \emph{coorientable} if the line bundle $L\to M$ is trivial, or equivalently if there is a \emph{global contact form}, i.e. a $1$-form $\alpha\in\Omega^1(M)$ such that $C=\ker\alpha$.
\end{remark}

Let $(M,C)$ and $(M^\prime, C^\prime)$ be contact manifolds.
A contactomorphism from $(M,C)$ to $(M',C')$ is a diffeomorphism $\phi:M\to M'$ such that
\begin{equation*}
(d\phi)C=C'.
\end{equation*}

An \emph{infinitesimal contactomorphism} (or \emph{contact vector field}) of a contact manifold $(M,C)$ is a vector field $X\in\mathfrak X(M)$ whose flow consists of local contactomorphisms.
Equivalently, $X \in \mathfrak X(M)$ is a contact vector field if
$
[X, \Gamma (C)] \subset \Gamma (C)
$.
Contact vector fields of $(M,C)$ form a Lie subalgebra of $\mathfrak X(M)$ which will be denoted by $\mathfrak X_C$, and we have the natural direct sum decomposition of $\bbR$-vector spaces $\mathfrak X(M)=\mathfrak X_C\oplus\Gamma(C)$ (see the proof of Proposition~\ref{prop:contact_to_Jacobi}).
As a consequence there exists an isomorphism of $\bbR$-vector spaces $X_{(-)}:\Gamma(L)\longrightarrow\frakX_C,\ \lambda\longmapsto X_\lambda$, uniquely determined by $\theta(X_\lambda)=\lambda$, for all $\lambda\in\Gamma(L)$.
Then, according to Proposition~\ref{prop:contact_to_Jacobi}, the contact structure on $M$ determines a Jacobi structure $J=\{-,-\}$ on $L\to M$ which is defined by $X_{\{\lambda,\mu\}}=\theta([X_\lambda,X_\mu])$, for all $\lambda,\mu\in\Gamma(L)$.
In particular, for all $\lambda\in\Gamma(L)$, the contact vector field $X_\lambda$ coincides exactly with the associated Hamiltonian vector field, i.e. the symbol of the associated Hamiltonian derivation $\Delta_\lambda:=\{\lambda,-\}\in\Der L$.

\begin{remark}
	Let $(M, L , J=\{-,-\})$ be a Jacobi manifold.
	For completeness and the reader's convenience, we recall here that, according to Proposition~\ref{prop:odd-dim_transitive}, $(L,J=\{-,-\})$ is the Jacobi structure associated with a (necessarily unique) contact structure iff the corresponding bi-linear form $\widehat \Lambda{}_J: \wedge^2 J^1 L \rightarrow L$ is non-degenerate.
	Moreover, in this case,  Hamiltonian derivations exhaust all infinitesimal Jacobi automorphisms, and Hamiltonian vector fields exhaust all Jacobi vector fields.
\end{remark}

For future reference, we recall here Darboux Lemma for (pre-)contact manifolds.
According to Darboux Lemma the dimension of a pre-contact manifold, and the rank of the characteristic distribution, are its only local invariants up to contactomorphisms.
For its proof we refer the reader to~\cite{geiges2008introduction} and \cite{libermann2012symplectic}).

\begin{proposition}[Darboux Lemma]
	\label{prop:Darboux_lemma}
	Let $(M,C)$ be a pre-contact manifold.
	Fix an arbitrary $p\in M$.
	Then there is an open neighborhood $U$ of $p$ in $M$, and a set of local coordinates $x^1,\ldots,x^k,z,u^1,\ldots,u^{2n}$ on $U$ such that 
	\begin{equation*}
	C|_U=\ker\left(dz-\sum_{a=1}^nu^{n+a}du^a\right).
	\end{equation*}
	So, in particular, the characteristic distribution of $(M,C)$ is locally generated by $\frac{\partial}{\partial x^i}$.
\end{proposition}

Now, let $(M,C)$ be a contact manifold, and let $S \subset M$ be a submanifold.
The intersection $C_S := C\cap TS$ is a generically singular distribution on $S$.
More precisely $S$ is the union of two disjoint subsets $S_0, S_1$ defined by
\begin{itemize}
	\item $p \in S_0$ iff $\dim(C_S)_p=\dim S$,
	\item $p \in S_1$ iff $\dim(C_S)_p=\dim S-1$.
\end{itemize}
If $S = S_0$ then $S$ is said to be an \emph{isotropic submanifold} of $(M,C)$.
In other words, an isotropic submanifold of $(M,C)$ is an integral manifold of the contact distribution $C$.
Locally maximal isotropic, or, equivalently, locally maximal integral submanifolds of $C$ are \emph{Legendrian submanifolds}.

\begin{proposition}
	\label{prop:coisotropics_contact_setting}
	Let $S = S_1$.
	The following conditions are equivalent:
	\begin{enumerate}[label=(\arabic*)]
		\item $C_S$ is a pre-contact structure on $S$, and its characteristic distribution coincides with $(C_S)^{\perp_\omega} \subset C|_S$,
		\item $(C_S)_p$ is coisotropic wrt $L_p$-valued symplectic form $\omega_p$ on $C_p$, i.e.~$(C_S)_p^{\perp_\omega} \subset (C_S)_p$, for all $p\in S$, (cf.~Remark~\ref{rem:maximally-integrable}),
		\item $S$ is a coisotropic submanifold of the associated Jacobi manifold $(M,L,J=\{-,-\})$.
	\end{enumerate}
\end{proposition}

\begin{proof}
	The equivalence $1)\Longleftrightarrow 2)$ amounts to a standard argument in symplectic linear algebra.
	The equivalence $2)\Longleftrightarrow 3)$ is based on the following facts.
	Let $(L, J = \{-,-\})$ be the Jacobi structure associated with $(M,C)$.
	For $\lambda\in\Gamma(L)$, and $f\in C^\infty(M)$ put $Y_{f,\lambda}:=\Lambda_J^\sharp (df \otimes \lambda) =X_{f\lambda}-fX_\lambda$.
	We have the following:
	\begin{itemize}
		\item $Y_{f, \lambda} \in \Gamma(C)$.
		
		\item Let $I_S \subset C^\infty (M)$ be the ideal of functions vanishing on $S$.
		Then
		$
		Y_{f,\lambda}$ is tangent to $S$ iff $X_{f\lambda}$ is tangent to $S$,
		for all $f\in I_S$, and $\lambda\in\Gamma(L)$.
		
		\item
		$
		\omega(Y_{f,\lambda},X)=X(f)\lambda,
		$
		for all $f\in\ C^\infty(M)$, $\lambda\in\Gamma(L)$, and $X\in\Gamma(C)$.
		
		\item Let $\Gamma_S \subset \Gamma(L)$ be the submodule consisting of sections vanishing on $S$.
		Then
		$
		\Gamma_S=I_S\cdot\Gamma(L)
		$.
		
	\end{itemize}
	Now it is easy to see that $(C_S)^{\perp_\omega}\subset C_S$ if and only if $S$ is coisotropic in $(M,L,\{-,-\})$.
\end{proof}

\begin{definition}
	\label{def:regular_coisotropic}
	Let $S = S_1$.
	If equivalent conditions $(1)$-$(3)$ in Proposition~\ref{prop:coisotropics_contact_setting} are satisfied, then $S$ is said to be a \emph{regular coisotropic submanifold} of $(M,C)$.
\end{definition}

\begin{remark}
	Unlike equivalence $(1) \Longleftrightarrow (2)$, in Proposition~\ref{prop:coisotropics_contact_setting}, equivalence $(2)\Longleftrightarrow (3)$ continues to hold also without assuming that $S = S_1$.
\end{remark}

\section{Coisotropic embeddings and \texorpdfstring{$L_\infty[1]$}{L∞[1]}-algebras from pre-con\-tact manifolds}
\label{sec:tubular_neighborhood} From now till the end of this section we consider only closed regular coisotropic submanifolds.
The intrinsic pre-contact geometry of a regular coisotropic submanifold $S$ in a contact manifold $M$ contains a full information about the coisotropic embedding of $S$ into $M$, at least locally around $S$.
This is an immediate consequence of the \emph{Tubular Neighborhood Theorem} in contact geometry (see~\cite{loose1998neighborhoods}, \cite[Section 6]{OW2013},
see also~\cite{gotay1982coisotropic} for the analogous result in symplectic geometry).

Let $(S,C_S)$ be a pre-contact manifold, with characteristic foliation $\calF$.

\begin{definition}
	\label{mydef:coisotropic_embedding}
	A \emph{coisotropic embedding} of $(S,C_S)$ into a contact manifold $(M,C)$ is an embedding $i: S \injects M$ such that
	$
	(di) C_S = C_{i(S)}$, and $(di)T\calF=(C_{i(S)})^{\perp_{\omega}}$,
	where $C_{i(S)}:=C\cap T(i(S))$, and $\omega $ is the curvature form of $(M,C)$.
\end{definition}

\begin{remark}
	Clearly, in view of Proposition~\ref{prop:coisotropics_contact_setting}, if $i : S \injects M$ is a coisotropic embedding of $(S,C_S)$ into $(M,C)$, then $i(S)$ is a regular coisotropic submanifold of $(M,C)$.
\end{remark}

Let $i_1$ and $i_2$ be coisotropic embeddings of $(S,C_S)$ into contact manifolds $(M_1,C_1)$ and $(M_2,C_2)$, respectively.
\begin{definition}
	\label{mydef:local_equivalence_coisotropic_embeddings}
	Coisotropic embeddings $i_1$ and $i_2$ are said to be \emph{locally equivalent} if there exist open neighborhoods $U_j$ of $\im i_j$ in $M_j$, $j=1,2$, and a contactomorphism $\phi:(U_1,C_1)\to(U_2,C_2)$ such that $\phi\circ i_1=i_2$.
\end{definition}

\begin{theorem}
	[Coisotropic embedding of pre-contact manifolds: existence and uniqueness]
	\label{teor:coisotropic_embedding}
	Every pre-contact manifold admits a coisotropic embedding.
	Additionally, any two coisotropic embeddings of a given pre-contact manifold are locally equivalent.
\end{theorem}

Theorem~\ref{teor:coisotropic_embedding} is a special case of Theorem 3 in~\cite{loose1998neighborhoods}.
We do not repeat the ``uniqueness part'' of the proof here.
The ``existence part'' can be proven constructively via \emph{contact thickening}.
This is done for later purposes in the next section.

\begin{corollary}[{$L_\infty[1]$}-algebra of a pre-contact manifold] \label{cor:L_infty_precont}
	Every pre-contact manifold determines a natural isomorphism class of $L_\infty[1]$-algebras.
\end{corollary}

\begin{proof}
	The ``existence part'' of Theorem~\ref{teor:coisotropic_embedding} and Proposition~\ref{prop:linfty} guarantee that a pre-contact manifold $(S,C_S)$ determines a unique $L_\infty[1]$-algebra up to the choice of a coisotropic embedding $(S,C_S)\subset (M,C)$, a fat tubular neighborhood $\tau : NS \times_S \ell \injects L$ of $\ell$ in $L$, where $\ell = TS / C_S$ and $L$ is the Jacobi bundle of $(M,C)$.
	Any two such $L_\infty[1]$-algebras are $L_\infty$-isomorphic because of Proposition~\ref{prop:gauge_invariance} and the ``uniqueness part'' of Theorem~\ref{teor:coisotropic_embedding}.
\end{proof}

\section{Contact thickening}
\label{sec:contact_thickening}
We now show that every pre-contact manifold $(S,C_S)$ admits a coisotropic embedding into a suitable contact manifold uniquely determined by $(S,C_S)$ up to the choice of a complementary distribution to the characteristic distribution.
Thus, let $(S,C_S)$ be a pre-contact manifold, $\calF $ its characteristic foliation, $\ell = TS/C_S$ the quotient line bundle, and let $\theta : TS \rightarrow \ell$ be the structure form.
Theorem~\ref{teor:coisotropic_embedding} is a ``contact version'' of a theorem by Gotay~\cite{gotay1982coisotropic} and can be proven by a similar technique as the \emph{symplectic thickening} of~\cite{oh2005deformations}.
Accordingly, we will speak about \emph{contact thickening}.
See also~\cite{OW2013} for a relevant discussion on contact thickening in a different context.

Pick a distribution $G$ on $S$ complementary to $T\calF $, and let $p_{T\calF;G} : TS \to T\calF $ be the projection determined by the splitting $TS = G \oplus T\calF $.
Put $T_\ell {}^\ast \calF := T^\ast\calF \otimes \ell$, and let $\tau : T_\ell {}^\ast\calF \to S$ be the natural projection.
We equip the manifold $T_\ell {}^\ast\calF$
with the line bundle $L := \tau^\ast \ell$.
The $\ell$-valued $1$-form $\theta$ can be pulled-back via $\tau$ to an $L$-valued $1$-form $\tau^\ast \theta$ on $T_\ell {}^\ast\calF $.
There is also another $L$-valued $1$-form $\theta_G$ on $T_\ell {}^\ast\calF$.
It is defined as follows: for $\alpha \in T_\ell {}^\ast\calF $, and $\xi \in T_\alpha(T_\ell {}^\ast\calF )$
\[
(\theta_G)_\alpha (\xi) := \left\langle\alpha, (p_{T\calF;G} \circ d \tau)\xi\right\rangle \in \ell_x = L_\alpha, \quad x := \tau (\alpha),
\]
where $\alpha$ is interpreted as a linear map $T_x\calF \to\ell_x$, and $\langle-,-\rangle$ denotes the $\ell$-twisted duality pairing between $T_\ell{}^\ast\calF$ and $T\calF$.
By definition, $\theta_G$ depends on the choice of splitting $G$.

\begin{proposition}
	\label{prop:contact_thickening}
	Distribution $C := \ker (\theta_G + \tau^\ast \theta )$ is a contact structure on a neighborhood $U$ of $\operatorname{im}\mathbf{0}$, the image of the zero section $\mathbf{0}$ of $\tau$.
	Additionally $\mathbf{0}$ is a coisotropic embedding of $(S,C_S)$ into the contact manifold $(U,C|_U)$.
\end{proposition}

\begin{proof}
	Use Darboux Lemma (see, e.g.,~Proposition~\ref{prop:Darboux_lemma}) and choose local coordinates $(x^i,u^a,z)$ on $S$ \emph{adapted to $C_S$}, i.e.
	\begin{gather*}
	\Gamma(T \calF)=\left\langle \partial/\partial x^i \right\rangle, \quad \Gamma(C_S)=\left\langle \partial/\partial x^i, \mathbb C_a\right\rangle, \quad \mathbb C_a = \frac{\partial}{\partial u^a} + C_a \frac{\partial}{\partial z},
	\end{gather*}
	where the $C_a$'s are \emph{linear functions} of the $u^b$'s only.
	Section $\mu:= \theta(\partial/ \partial z)$ is a local generator of $\Gamma (\ell)$.
	Moreover $\theta$ is locally given by
	$
	\theta = (dz - C_a du^a) \otimes \mu,
	$
	and the curvature form $\omega_S$ of $C_S$ is locally given by
	\[
	\omega_S = \frac{1}{2} \omega_{ab} du^a |_C \wedge du^b|_C \otimes \mu, \quad \omega_{ab} = \frac{\partial C_b}{\partial u^a} - \frac{\partial C_a}{\partial u^b}.
	\]
	In particular, the skew-symmetric matrix $(\omega_{ab})$ is non-degenerate.
	We will use the following local frame on $S$ adapted to both $C_S$ and $G$:
	\[
	\left(\frac{\partial}{\partial x^i}, \mathbb C'_a,Z \right),
	\]
	where
	$
	\mathbb C'_a := (\mathrm{id} - p_{T\calF;G}) (\mathbb C_a)
	$,
	and
	$
	Z := (\mathrm{id} - p_{T\calF;G})(\partial/\partial z)
	$.
	Now, let $p = (p_i)$ be linear coordinates along the fibers of $\tau:T_\ell {}^\ast\calF \to S$ corresponding to the local frame $(dx^i|_{T \mathcal F} \otimes \mu)$.
	Then $(\partial / \partial x^i,\mathbb C'_a, Z, \frac{\partial}{\partial p_i})$ is a local frame on $T_\ell {}^\ast\calF $.
	It is easy to check that locally
	\[
	\Gamma(C) = \left\langle X_i, \mathbb C'_a,\frac{\partial}{\partial p_i}\right\rangle ,
	\]
	where $X_i :=\partial / \partial x^i - p_i Z$.
	Finally, the representative matrix of the curvature of $C$ wrt the local frames $(X_i, \mathbb C'_a,\frac{\partial}{\partial p_i})$ of $C$ and $Z\, \mathrm{mod}\, C$ of $T (T_\ell {}^\ast \calF)/C = L$ is
	\begin{equation}\label{eq:cont_thick}
	\left( \begin{array}{ccc}
	0 & 0 & \delta_{i}^j \\
	0 & \omega_{ab}  & 0 \\
	-\delta_{j}^i & 0 & 0
	\end{array} \right) \ \text{up to infinitesimals $\calO(\boldsymbol p)$}
	\end{equation}
	This shows that $C$ is maximally non-integrable around the zero section of $T_\ell {}^\ast\calF $.
	Moreover, it immediately follows from~\eqref{eq:cont_thick} that the zero section of $T_\ell {}^\ast\calF $ is a coisotropic embedding (transversal to fibers of $\tau$).
	This concludes the proof.
\end{proof}

The contact manifold $(U,C|_U)$ is called a \emph{contact thickening} of $(S,C_S)$.
Now, let $NS$ be the normal bundle of $S$ in $U$.
Clearly $NS = T_\ell {}^\ast \calF$, hence $N_\ell S = T^\ast \mathcal F$.
According to the proof of Corollary~\ref{cor:L_infty_precont} the choice of the complementary distribution $G$ determines an $L_\infty[1]$-algebra structure on $\Gamma (\wedge^\bullet N_\ell S \otimes \ell)[1]  = \Gamma (\wedge^\bullet T^\ast \calF \otimes \ell)[1]$.
Moreover, such $L_\infty[1]$-algebra structure is actually independent of the choice of $G$ up to $L_\infty$-isomorphisms.
Sections of $\wedge^\bullet T^\ast \calF \otimes \ell$ are $\ell$-valued leaf-wise differential forms on $S$ and we also denote them by $\Omega^\bullet (\calF, \ell)$ (see below).

\section{The transversal geometry of the characteristic foliation}
\label{sec:transversal_geometry}

Similarly as in the symplectic case (cf.~\cite[Section 9.3]{oh2005deformations}), and in the lcs case (cf.~\cite[Section 9.1]{le2012deformations}), the multi-brackets in the $L_\infty[1]$-algebra of a pre-contact manifold can be expressed in terms of the ``geometry transversal to the characteristic foliation''.
To write down this expression, the relevant transversal geometry needs to be described.
Let $(S,C_S)$ be a pre-contact manifold, with characteristic foliation $\calF $.
Denote by $N\calF :=TS/T\calF $ the normal bundle to $\calF $, and by $N^\ast\calF =(N\calF )^\ast = T^0 \calF \subset T^\ast S$ the conormal bundle to $\calF $.

Recall that $T \calF $ is a Lie algebroid.
The standard Lie algebroid differential in $\Omega^\bullet(\calF ):=\Gamma(\wedge^\bullet T^\ast\calF )$ will be denoted by $d_{\calF }$ and called the \emph{leaf-wise de Rham differential}.
There is a flat $T \calF $-connection $\nabla$ in $N^\ast \calF $ well-defined by
\[
\nabla_X \eta :=\mathcal L_X \eta, \quad X \in \Gamma (T \calF ), \quad \eta \in \Gamma (N^\ast \calF ).
\]

\begin{remark}
	Connection $\nabla$ is ``dual to the Bott connection'' in $N \calF $.
\end{remark}

As usual, $\nabla$ determines a differential in $\Omega^\bullet(\calF ,N^\ast\calF ):=\Gamma(\wedge^\bullet T^\ast\calF \otimes N^\ast \calF )$ denoted again by $d_{\calF }$.
There exists also a flat $T \calF $-connection in $\ell$, denoted again by $\nabla$, and defined by
\begin{equation*}
\nabla_X \theta(Y) := \theta([X,Y]),\quad X\in\Gamma(T\calF ),\quad Y\in\mathfrak X(S).
\end{equation*}
The corresponding differential in $\Omega^\bullet(\calF ,\ell):=\Gamma(\wedge^\bullet T^\ast\calF \otimes \ell)$ will be also denoted by $d_\calF $.
Now, let $J^1_\perp \ell$ be the vector subbundle of $J^1\ell$ given by the kernel of the vector bundle epimorphism
\[
\varphi_\nabla :J^1\ell\longrightarrow T^\ast\calF \otimes \ell, \quad j^1_x\lambda\longmapsto(d_\calF \lambda)_x.
\]
Sections of $J^1_\perp \ell$ will be interpreted as sections of $J^1\ell$ ``transversal to $\calF $''.
Note also that the Spencer sequence $0 \rightarrow T^\ast S \otimes \ell\rightarrow J^1\ell \rightarrow \ell \rightarrow 0$ restricts to a ``transversal Spencer sequence'' $0 \rightarrow N^\ast \calF \otimes \ell\rightarrow J_\perp ^1\ell \rightarrow \ell\rightarrow 0$ which is uniquely determined by the condition to fit in the following exact commutative diagram of vector bundle morphisms
\begin{equation*}
\begin{tikzcd}
	&0\arrow[d]&0\arrow[d]&0\arrow[d]&\\
	0\arrow[r]&N^\ast\calF \otimes \ell \arrow[r]\arrow[d]&J^1_\perp \ell\arrow[r]\arrow[d]&\ell\arrow[r]\arrow[d, dash, shift left=0.2ex]\arrow[d, dash, shift right=0.2ex]&0\\
	0\arrow[r]&T^\ast S\otimes \ell\arrow[r]\arrow[d]&J^1 \ell\arrow[r]\arrow[d, "\varphi_\nabla"]&\ell\arrow[d]\arrow[r]&0\\
	0\arrow[r]&T^\ast\calF \otimes \ell\arrow[d]\arrow[r, dash, shift left=0.2ex]\arrow[r, dash, shift right=0.2ex]&T^\ast\calF \otimes \ell\arrow[r]\arrow[d]&0&&\\
	&0&0&&&
\end{tikzcd}
\end{equation*}
The above condition is well-posed because of standard diagram chasing arguments from homological algebra as in the proof of the Snake lemma.

In what follows embeddings $\gamma: T^\ast S\otimes \ell \injects J^1\ell$ and $N^\ast\calF \otimes \ell \injects J^1_\perp \ell$ will be understood, and we will identify $df\otimes\lambda$ with $j^1(f \lambda) - f j^1\lambda$, for any $f\in C^\infty (S)$, and $\lambda\in\Gamma(\ell)$.
Recall that an arbitrary $\alpha\in\Gamma(J^1\ell)$ can be uniquely decomposed as $\alpha=j^1\lambda+\eta$, with $\lambda\in\Gamma(\ell)$, and $\eta\in\Gamma(T^\ast S\otimes \ell)$.
Then, by definition, for $p\in S$, $\alpha_p$ is in $J^1_\perp \ell$ iff $\varphi_\nabla(\eta_p)=-(d_{\calF }\lambda)_p$.
Finally, there is a flat $T\calF $-connection in $J^1_\perp \ell$, also denoted by $\nabla$, well-defined as follows.
For $X\in\Gamma(T\calF )$ and $\alpha = j^1 \lambda + \eta \in \Gamma (J^1_\bot \ell)$, with $\lambda\in\Gamma(\ell)$, $\eta \in\Omega^1(S,\ell)$ such that $\varphi_\nabla (\eta)=-d_{\calF }\lambda$, put
\begin{equation}\label{eq:nabla_J1}
\nabla_X(j^1\lambda+\eta)=j^1(\nabla_X\lambda)+\calL_{\nabla_X} \eta,
\end{equation}
where $\calL_{\nabla_X}$ is the Lie derivative of $\ell$-valued forms on $S$ along derivation $\nabla_X \in \Der\, \ell$.
Accordingly, there is a differential in $\Omega^\bullet(\calF ,J^1_\perp \ell) :=\Gamma(\wedge^\bullet T^\ast\calF \otimes J^1_\perp \ell)$ which we also denote by $d_{\calF }$.

Now, note that the curvature form of $(S,C_S)$, $\omega_S:\wedge^2 C_S\rightarrow \ell$, descends to a(n $\ell$-valued) symplectic form $\omega_\perp:\wedge^2(C_S/T\calF )\rightarrow \ell$.
In particular, it determines a vector bundle isomorphism $\omega_\perp^\flat:C_S/T\calF \rightarrow(C_S/T\calF )^\ast\otimes \ell$ (cf.~Section~\ref{sec:cois_cont}).

\begin{remark}
	Let $p\in S$, $X\in\mathfrak X(S)$, and $\lambda=\theta(X)$.
	Recall that $\phi_X\in\Gamma(C_S^\ast\otimes \ell)$ is defined by $\phi_X(Y)=\theta([X,Y])$, for all $Y\in\Gamma(C_S)$ (cf.~Section~\ref{sec:cois_cont}).
	\begin{itemize}
		\item Then we have that $j_p^1\lambda\in J^1_\perp \ell$ if and only if $(\phi_X)_p\in (C_S/T\calF )^\ast\otimes \ell$.
		\item Furthermore it is easy to check, for instance using local coordinates, that when $j^1_p\lambda=0$ then the following holds:
		\begin{enumerate}
			\item $X_p\in (C_S)_p$, and
			\item $\omega_S(X_p,Y_p)=\theta([X,Y]_p)$, for all $Y\in\Gamma (C_S)$.
		\end{enumerate}
	\end{itemize}
	Therefore, if $j^1_p\lambda=0$, then $X_p\Mod T_p\calF =(\omega_\perp^\flat)^{-1}(\phi_X)_p$, and the following definition is well-posed.
\end{remark}

\begin{definition}
	Define $\sigma \widehat \Lambda{}_\perp^\sharp :J^1_\perp \ell\to N\calF $ to be the vector bundle morphism uniquely determined by:
	\begin{equation}\label{eq:sharp_bot}
	\sigma \widehat \Lambda{}_\bot^\sharp (j_p^1\lambda):=X_p{}\Mod T_p\calF -(\omega_\perp^\flat)^{-1}(\phi_X)_p,
	\end{equation}
	where $p\in M$, $\lambda\in\Gamma(\ell)$, and $X\in\mathfrak X(S)$, such that $j^1_p\lambda \in J^1_\perp L$, and $\lambda=\theta(X)$.
\end{definition}

\begin{proposition}
	There exists a vector bundle morphism $\widehat \Lambda{}_\bot:\wedge^2 J^1_\perp \ell\to \ell$ uniquely determined by putting
	\begin{equation}\label{eq:Lambda_bot}
	\widehat \Lambda{}_\bot(j^1_p \lambda,j^1_p \lambda ')=\theta([Y,Y^\prime]_p),
	\end{equation}
	where $p\in M$, $\lambda, \lambda '$ are $\nabla$-constant local sections of $\ell$ and $Y,Y^\prime\in\mathfrak X(S)$ are such that $\sigma \widehat \Lambda{}_\bot^\sharp (j^1 \lambda)=Y \Mod \Gamma (T\calF )$ and $\sigma \widehat \Lambda{}_\bot^\sharp (j^1 \lambda^\prime)=Y^\prime \Mod \Gamma (T\calF ) $.
\end{proposition}
\begin{proof}
	First of all notice that every point in $J^1_\bot \ell$ is the first jet of a $\nabla$-constant local section of $\ell$.
	Hence Definition~\eqref{eq:Lambda_bot} makes sense.
	Moreover, the rhs only depends on $\lambda, \lambda^\prime$.
	Indeed, first of all, $\theta (Y) = \lambda$, and $\theta (Y ') = \lambda '$.
	Moreover, if $Y \in \Gamma (T \calF)$, then, $0 = \nabla_Y \lambda ' = \theta ([Y,Y '])$.
	Finally, one can check, e.g.~using local coordinates, that the rhs of~\eqref{eq:Lambda_bot} does actually depend on the first jets at $p$ of $\lambda, \lambda '$.
	This shows that $\widehat{\Lambda}_\bot$ is well-defined.
\end{proof}

Vector bundle morphism $\widehat \Lambda{}_\bot :\wedge^2 J^1_\perp \ell\to \ell$ will be interpreted as the \emph{transversal version} of the bi-linear form $\widehat \Lambda{}_J$ associated with a Jacobi bi-differential operator $J$.

\section{An explicit formula for the multi-brackets: the construction}
\label{sec:multi-brackets}

Retaining notations from previous section, choose a distribution $G$ on $S$ which is complementary to $T\calF $, i.e.~$TS=G\oplus T\calF $.
There is a dual splitting $T^\ast S\simeq T^\ast\calF \oplus N^\ast\calF $ and there are identifications
$
N\calF \simeq G$, $T^\ast\calF \simeq G^0
$.
Furthermore the induced splitting of
\begin{equation*}
0\rightarrow N^\ast\calF \otimes \ell \rightarrow T^\ast S\otimes \ell\rightarrow T^\ast\calF \otimes \ell \rightarrow 0
\end{equation*}
lifts to a splitting of
\begin{equation*}
0\rightarrow J^1_\perp \ell \rightarrow J^1 \ell\rightarrow T^\ast\calF \otimes \ell \rightarrow 0.
\end{equation*}
Hence $J^1\ell\simeq J^1_\perp \ell\oplus (T^\ast\calF \otimes \ell)$.
Let $F\in\Gamma( \wedge^2 G^\ast \otimes TS/G )$ be the curvature form of $G$.
The curvature $F$ will be also understood as an element $F\in\Gamma(\wedge^2 N^\ast\calF \otimes T\calF )\subset\Gamma(\wedge^2(J^1_\perp \ell\otimes \ell^\ast)\otimes T\calF )$, where we used embedding $N^\ast\calF \otimes \ell \injects J^1_\perp \ell$.

Let $d_G:C^\infty(S)\rightarrow\Gamma(N^\ast\calF )$ be the composition of the de Rham differential $d:C^\infty(S)\rightarrow\Omega^1(S)$ followed by the projection $\Omega^1(S)\rightarrow\Gamma(N^\ast\calF )$ determined by decomposition $T^\ast S=T^\ast\calF \oplus N^\ast\calF $.
Then $d_G$ is a $\Gamma (N^\ast \calF )$-valued derivation of $C^\infty (S)$ and will be interpreted as ``transversal de Rham differential''.

\begin{proposition}\label{prop:eps}
	There exists a unique degree zero, graded $\mathbb R$-linear map $\varepsilon:\Omega(\calF )\rightarrow\Omega(\calF ,N^\ast\calF )$ such that
	\begin{enumerate}
		\item $\varepsilon|_{C^\infty(S)}=d_G$,
		\item $[\varepsilon,d_{\calF }]=0$, and
		\item the following identity holds
		\begin{equation*}
		\varepsilon(\tau \wedge \tau')=\tau \wedge \varepsilon(\tau') + (-)^{|\tau ||\tau' |} \tau' \wedge \varepsilon(\tau),
		\end{equation*}
		for all homogeneous $\tau, \tau' \in \Omega(\calF )$.
	\end{enumerate}
\end{proposition}

In order to prove Proposition~\ref{prop:eps}, the following Lemma will be useful:
\begin{lemma}\label{lem:eps}
	Let $f $ be a leaf-wise constant local function on $S$, i.e.~$d_\calF f = 0$.
	Then $d_\calF d_G f = 0$ as well.
\end{lemma}
\begin{proof}
	Let $f$ be as in the statement.
	First of all, note that $d f$ takes values in $N^\ast \calF$, so that $d_G f = df$.
	Now recall that $d_\calF d_G f = 0$ iff
	$0 = \langle d_\calF d_G f, X \rangle = \nabla_X d_G f = \calL_X d_G f $
	for all $X \in \Gamma (T \calF)$, where $\nabla$ is the canonical $T \calF$-connection in $N^\ast \calF$.
	But
	$
	\calL_X d_G f = \calL_X df = d (X f) = 0
	$.
	This completes the proof.
\end{proof}

\begin{proof}[Proof of Proposition {\ref{prop:eps}}]
	The graded algebra $\Omega (\calF)$ is generated in degree $0$ and $1$.
	In order to define $\eps$, we first define it on the degree one piece $\Omega^1 (\calF)$ of $\Omega (\calF)$.
	Thus, note that $\Omega^1 (\calF)$ is generated, as a $C^\infty (S)$-module, by leaf-wise de Rham differentials $d_\calF f \in \Omega^1 (\calF)$ of functions $f \in C^\infty (S)$.
	The only relations among these generators are the following
	\begin{equation} \label{eq:rel}
	\begin{aligned}
	d_\calF (f+g) & = d_\calF f + d_\calF g, \\
	d_\calF (fg) &= f d_\calF g + g d_\calF f , \\
	d_\calF f & = 0 \text{ on every open domain where $f$ is leaf-wise constant},
	\end{aligned}
	\end{equation}
	where $f,g \in C^\infty (S)$.
	Now define $\eps : \Omega^1 (\calF) \to \Omega^1 (\calF, N^\ast \calF)$ on generators by putting
	\[
	\eps f := d_G f \quad \text{and} \quad \eps d_\calF f := d_\calF d_G f,
	\]
	and extend it to the whole $\Omega^1 (\calF)$ by prescribing $\bbR$-linearity and the following Leibniz rule:
	\begin{equation}\label{eq:leib_dG}
	\eps (f \sigma) = f \eps (\sigma) + \sigma \otimes d_G f ,
	\end{equation}
	for all $f \in C^\infty (S)$, and $\sigma \in \Omega^1 (\calF)$.
	In order to see that $\eps$ is well defined it suffices to check that it preserves relations~\eqref{eq:rel}.
	Compatibility with the first two relations can be checked by a straightforward computation that we omit.
	Compatibility with the third relation immediately follows from Lemma~\ref{lem:eps}.
	Finally, in view of Leibniz rule~\eqref{eq:leib_dG}, $d_G$ and $\eps$ combine and extend to a well-defined derivation $\Omega (\calF) \to \Omega (\calF, N^\ast \calF)$.
	By construction, the extension satisfies all required properties.
	Uniqueness is obvious.
\end{proof}

The graded differential operator $\varepsilon$ will be also denoted by $d_G$.

Similarly, there is a ``transversal version of the first jet prolongation $j^1$''.
Namely, let $j^1_G:\Gamma(\ell)\rightarrow\Gamma(J^1_\perp \ell)$ be the composition of the first jet prolongation $j^1:\Gamma(\ell)\rightarrow\Gamma(J^1\ell)$ followed by the projection $\Gamma(J^1\ell)\rightarrow\Gamma(J^1_\perp \ell)$ determined by decomposition $J^1\ell=J^1_\perp \ell \oplus (N^\ast\calF \otimes \ell )$.
Then $j^1_G$ is a first order differential operator from $\Gamma(\ell)$ to $\Gamma(J^1_\perp \ell)$ such that
\begin{equation}\label{eq:leib_j1G}
j^1_G(f\lambda)=fj^1_G\lambda+(d_Gf)\otimes\lambda,
\end{equation}
$\lambda\in\Gamma(\ell)$ and $f\in C^\infty(S)$, where, similarly as above, we understood the embedding $N^\ast\calF \otimes \ell \injects J^1_\perp \ell$.
As announced, the operator $j^1_G$ will be interpreted as ``transversal first jet prolongation''.

\begin{proposition}\label{prop:delta}
	There exists a unique degree zero graded $\mathbb R$-linear map $\delta:\Omega(\calF ,\ell)\rightarrow\Omega(\calF ,J^1_\perp \ell)$ such that
	\begin{enumerate}
		\item $\delta|_{\Gamma(\ell)}=j^1_G$,
		\item $[\delta,d_{\calF }]=0$, and
		\item the following identity holds
		\begin{equation*}
		\delta(\tau \wedge \Omega)=\tau \wedge \delta(\omega)+d_G \tau \otimes\omega,
		\end{equation*}
		for all $\tau\in\Omega(\calF )$, and $\omega\in\Omega(\calF ,\ell)$, where the tensor product is over $\Omega(\calF )$, and we understood both the isomorphism
		\begin{equation}\label{eq:isoOmega}
		\Omega(\calF ,N^\ast\calF )\underset{\Omega(\calF )}{\otimes}\Omega(\calF ,\ell)\simeq\Omega(\calF ,N^\ast\calF \otimes \ell)
		\end{equation}
		and embedding $N^\ast\calF \otimes \ell \injects J_\perp^1 \ell$.
	\end{enumerate}
\end{proposition}

In order to prove Proposition~\ref{prop:delta}, the following Lemma will be useful:
\begin{lemma}\label{lem:delta}
	Let $\mu$ be a leaf-wise constant local section of $\ell$, i.e.~$d_\calF \mu = 0$, then $d_\calF j^1_G \mu = 0$ as well.
\end{lemma}
\begin{proof}
	Let $\mu$ be as in the statement.
	First of all note that, by the very definition of $J^1_\bot \ell$, $j^1 \mu$ takes values in $J^1_\bot \ell$ so that $j^1_G \mu = j^1 \mu$.
	Now recall that $d_\calF j^1_G \mu = 0$ iff
	$0 = \langle d_\calF j^1_G \mu, X \rangle = \nabla_X j^1_G \mu $
	for all $X \in \Gamma (T \calF)$, where $\nabla$ is the canonical $T \calF$-connection in $J^1_\bot \ell$.
	But
	$
	\nabla_X j^1_G \mu = \nabla_X j^1 \mu = j^1 \nabla_X \mu = 0
	$,
	where we used~\eqref{eq:nabla_J1}.
	This completes the proof.
\end{proof}

\begin{proof}[Proof of Proposition {\ref{prop:delta}}]
	In this proof a tensor product $\otimes$ will be over $C^\infty (S)$ unless otherwise stated.
	We can regard $\Omega (\calF , \ell) = \Omega (\calF ) \otimes \Gamma (\ell)$ as a quotient of $\Omega (\calF ) \otimes_\bbR \Gamma (\ell)$ in the obvious way.
	Our strategy is defining an operator $\delta ': \Omega (\calF ) \otimes_\bbR \Gamma (\ell) \to \Omega (\calF, J^1_\bot \ell)$ and prove that it descends to an operator $\delta : \Omega(\calF ,\ell)\rightarrow\Omega(\calF ,J^1_\perp \ell)$ with the required properties.
	Thus, for $\sigma \in \Omega (\calF)$ and $\lambda \in \Gamma (\ell)$ put
	\begin{equation}\label{eq:delta'}
	\delta ' (\sigma \otimes_\bbR \lambda):= \sigma \otimes j^1_G \lambda+ d_G \sigma\otimes_{\Omega (\calF)} \lambda \in \Omega (\calF, J^1_\bot \ell),
	\end{equation}
	where, in the second summand, we understood both isomorphism~\eqref{eq:isoOmega} and embedding $N^\ast \calF \otimes \ell \injects J^1_\bot \ell$ (just as in the statement of the proposition).
	In order to prove that $\delta '$ descends to an operator $\delta$ on $\Omega (\calF, \ell)$ it suffices to check that $\delta ' (f \sigma \otimes_\bbR \lambda) = \delta ' (\sigma \otimes_\bbR f \lambda)$ for all $\sigma, \lambda$ as above, and all $f \in C^\infty (S)$.
	This can be easily obtained using the derivation property of $d_G$ and~\eqref{eq:leib_j1G}.
	Now, Properties (1) and (3) immediately follows from~\eqref{eq:delta'}.
	In order to prove Property (2), it suffices to check that $\delta d_\calF \lambda = d_\calF j^1_G \lambda$ for all $\lambda \in \Gamma (\ell)$ (and then use Property (3)).
	It is enough to work locally.
	Thus, let $\mu$ be a local generator of $\Gamma (\ell)$ with the further property that $d_\calF \mu = 0$.
	Moreover, let $f \in C^\infty (S)$, and compute
	\begin{align*}
	\delta d_\calF ( f \mu ) & = \delta (d_\calF f \otimes \mu) = d_\calF f \otimes j^1_G \mu + d_G d_\calF f \otimes \mu = d_\calF f \otimes j^1_G \mu + d_\calF d_G f \otimes \mu \\
	&= d_\calF ( f j^1_G \mu + d_G f \otimes \mu) = d_\calF (j^1_G f \mu),
	\end{align*}
	where we used $d_\calF \mu = 0$, Proposition~\ref{prop:eps}, Lemma~\ref{lem:delta}, and~\eqref{eq:leib_j1G}.
	Uniqueness of $\delta$ is obvious.
	
\end{proof}

The graded differential operator $\delta$ will be also denoted by $j^1_G$.

Now, interpret $\widehat \Lambda{}_\bot \in\Gamma(\wedge^2(J^1_\perp \ell)^\ast\otimes \ell)$ as a section $\sharp  \in\Gamma((J^1_\perp \ell\otimes \ell^\ast)^\ast\otimes (J^1_\perp \ell)^\ast)$.
The interior product of $\sharp  $ and $F\in\Gamma(\wedge^2(J^1_\perp \ell\otimes \ell^\ast)\otimes T\calF )$ is a section $F^\sharp \in\Gamma(\operatorname{End}(J^1_\perp \ell)\otimes T\calF \otimes \ell^\ast)$.
For any $\mu \in\Omega^{k+1}(\calF ,\ell)$, with $k\geq 0$, the interior product of $F^\sharp $ and $\mu$ is a section $i_{F^\sharp }\mu\in\Omega^{k}(\calF ,\operatorname{End}J^1_\perp \ell)$.
Now, we extend
\begin{enumerate}
	\item the bi-linear map $\widehat \Lambda_{\bot} : \wedge^2 J^1_\bot \ell \to \ell $ to a degree $+1$, $\Omega (\calF)$-bilinear, symmetric form
	\[
	\langle-,-\rangle_{C}:\Omega(\calF , J^1_\perp \ell)[1]\times\Omega(\calF , J^1_\perp \ell)[1]\longrightarrow \Omega(\calF , \ell)[1]
	\]
	\item the natural bilinear map $\circ : \operatorname{End} J^1_\bot \ell \otimes  \operatorname{End} J^1_\bot \ell \to\operatorname{End} J^1_\bot \ell $ to a degree $+1$, $\Omega (\calF)$-bilinear map
	\[
	\Omega (\calF, \operatorname{End} J^1_\bot \ell )[1] \times \Omega (\calF, \operatorname{End} J^1_\bot \ell )[1] \longrightarrow \Omega (\calF, \operatorname{End} J^1_\bot \ell )[1],
	\] also denoted by $\circ$, and
	\item the tautological action $\operatorname{End} J^1_\bot \ell \otimes   J^1_\bot \ell \to J^1_\bot \ell $ to a degree $+1$, $\Omega (\calF)$-linear action
	\[
	\Omega (\calF, \operatorname{End} J^1_\bot \ell )[1] \times \Omega (\calF, J^1_\bot \ell )[1] \longrightarrow \Omega (\calF, J^1_\bot \ell )[1].
	\]
\end{enumerate}

\begin{theorem}
	\label{theor:multi}
	The first (unary) bracket in the $L_\infty[1]$-algebra structure on $\Omega(\calF ,\ell)[1] $ is $d_{\calF }$.
	Moreover, for $k > 1$, the $k$-th multi-bracket is given by
	\begin{equation}
	\label{eq:multi}
	\frakm_k (\nu_1,\ldots,\nu_k) = \frac{1}{2}\sum_{\sigma\in S_k}\epsilon(\sigma, \boldsymbol{\nu})\left\langle j^1_G\nu_{\sigma(1)},(i_{F^\sharp }\nu_{\sigma(2)}\circ\cdots\circ i_{F^\sharp }\nu_{\sigma(k-1)})j^1_G\nu_{\sigma(k)}\right\rangle_{C},
	\end{equation}
	for all homogeneous $\nu_1\,\ldots,\nu_k\in\Omega(\calF ,\ell)[1] $, where $\epsilon (\sigma, \boldsymbol{\mu})$ is the graded symmetric Koszul sign prescribed by the permutations of the $\mu$'s.
\end{theorem}

A coordinate proof of Theorem~\ref{theor:multi} is provided in the next section~\ref{subsec:proof:theor:multi}.

\begin{remark}
	The explicit form of the contact thickening (see Section~\ref{sec:contact_thickening}) shows that the Jacobi bracket is actually fiber-wise entire.
	In particular Corollaries~\ref{cor:conver} and~\ref{cor:hequi} always apply to the contact case.
\end{remark}

As seen above, on a pre-contact manifold $(S,C)$, the choice of a distribution $G$ complementary to the characteristic distribution $\ker\omega_S^\flat$ allows to construct an explicit representative of the $L_\infty[1]$-algebra of $(S,C_S)$ (which is only known up to $L_\infty$-isomorphisms).
The expressions for the multi-brackets of the $L_\infty[1]$-algebra corresponding to a choice of such a $G$ are pretty explicit, and handy from the computational point of view.
Moreover, 
 further interesting information can be inferred from them in the transversally integrable case.

\begin{definition}
	\label{def:transversally_integrable}
	A pre-contact manifold $(S,C)$ is called \emph{transversally integrable} if there exists a distribution $G$ on $S$ which is integrable and complementary to the characteristic distribution $\ker\omega_S^\flat$, i.e.~such that
	\begin{equation*}
	TS=G\oplus \ker\omega_S^\flat,\quad\text{and}\quad[\Gamma(G),\Gamma(G)]\subset\Gamma(G).
	\end{equation*}
\end{definition}

If $(S,C)$ is transversally integrable and $G$ is as in Definition~\ref{def:transversally_integrable}, then the associated $L_\infty[1]$-algebra is actually a differential graded Lie algebra (see next Proposition~\ref{prop:transversally_integrable}).
The analogous results in the symplectic and the lcs settings were proven in~\cite[Section 9.3]{oh2005deformations} and~\cite[Section 9.1]{le2012deformations} respectively (see also~\cite[Proposition 5.3]{SZ2014} for an alternative proof in the symplectic setting).

\begin{proposition}
	\label{prop:transversally_integrable}
	Let $(S,C)$ be a transversally integrable pre-contact manifold.
	Then its $L_\infty[1]$-algebra is $L_\infty$-isomorphic to a differential graded Lie algebra, up to décalage.
\end{proposition}

\begin{proof}
	Let $G$ be a distribution on $S$ complementary to $\ker\omega_S^\flat$, and let $\{\frakm_k\}$ be the $L_\infty[1]$-algebra structure on $\Omega(\calF;\ell)[1]$ associated with $S$ through $G$ as in Theorem~\ref{theor:multi}.
	If $G$ is integrable, then $F^\sharp $ vanishes and it follows immediately from~\eqref{eq:multi} that $\frakm_k=0$, for all $k>2$.
\end{proof}

\section{An explicit formula for the multi-brackets: the proof}
\label{subsec:proof:theor:multi}

In this section we provide a proof of Theorem~\ref{theor:multi}.


Let $(S, C_S)$ be a pre-contact manifold, with normal line bundle $\ell := TS / C_S$, and characteristic foliation $\calF$, and let $G$ be a complementary distribution to $T  \calF$, i.e.~$TS =  G \oplus T \calF$.
As shown in Section~\ref{sec:contact_thickening}, the bundle $T_\ell^\ast \calF := T^\ast \calF \otimes \ell$ is equipped with an hyperplane distribution $C$ (whose quotient line bundle is $L:=T_\ell{}^\ast\calF\times_S\ell\to T_\ell{}^\ast\calF$) which is contact in a neighborhood of the image $\im\mathbf{0}$ of the zero section $\mathbf 0$: the contact thickening of $(S, C_S)$.
Moreover $\mathbf 0$ is a coisotropic embedding.
In particular, there is an $L_\infty[1]$-algebra $(\Gamma (\wedge^\bullet N_\ell S \otimes \ell)[1], \{ {\mathfrak m}_k \})$ attached to $(S, C_S)$.
In this case, $NS = T_\ell^\ast \calF$, so that $\Gamma (\wedge^\bullet N_\ell S \otimes \ell) \simeq \Omega (\calF, \ell)$.
In the following we will understand this isomorphism.
We will show below that the multi-brackets ${\mathfrak m}_k$ are given by formula \eqref{eq:multi} which is the contact analogue of Oh-Park formula (see~\cite[Formula (9.17)]{oh2005deformations}).
We will do this in local coordinates.
From now on, we freely use notations and conventions from Sections~\ref{sec:contact_thickening}, \ref{sec:transversal_geometry} and~\ref{sec:multi-brackets}.

Let $(x^i,u^a,z,p_i)$ be local coordinates on $T^\ast_\ell \calF$ chosen as in the proof of Proposition~\ref{prop:contact_thickening}.
Hence, in particular,
\begin{equation*}
\Gamma(T \calF)=\left\langle \partial/\partial x^i \right\rangle, \quad \Gamma(C_S)=\left\langle \partial/\partial x^i, \mathbb C_a\right\rangle, \quad \mathbb C_a = \frac{\partial}{\partial u^a} + C_a \frac{\partial}{\partial z},
\end{equation*}
where the $C_a$'s are \emph{linear functions} of the $u^b$'s only.
The structure and curvature forms of $C_S$ are locally
\[
\theta = (dz -C_a du^a) \otimes \mu, \quad  \omega=\frac{1}{2}\omega_{ab}du^a\wedge du^b\otimes\mu,
\]
where $\omega_{ab}:=\bbC_a(C_b)-\bbC_b(C_a)$.
The skew-symmetric matrix $(\omega_{ab})$ is non-degenerate, and $(\omega^{ab})$ denotes its inverse.

The de Rham differential $d_{\calF}:\Omega(\calF;\ell)\to\Omega(\calF;\ell)$ of the Jacobi algebroid $(T\calF,\ell)$ is locally given by
\begin{equation}
\label{eq:d_F}
d_\calF\left(f_{i_1\dots i_k}d_\calF x^{i_1}\wwedge d_\calF x^{i_k}\otimes\mu\right)=\frac{\partial f_{i_1\dots i_k}}{\partial x^i}d_\calF x^i\wedge d_\calF^{i_1}\wwedge d_\calF x^{i_k}\otimes\mu.
\end{equation}
Hence, if $j,j^a,j^i,j^\circ$ denotes the local frame of $J^1\ell\to S$ defined by
\begin{equation*}
j:=j^1\mu,\quad j^a:=j^1(y^a\mu)-y^aj^1\mu,\quad j^i:=j^1(q^i\mu)-q^ij^1\mu,\quad j^\circ:=j^1(z\mu)-zj^1\mu,
\end{equation*}
then a local frame of the vector subbundle $J^1_\bot\ell\to S$ is provided by $j,j^a,j^\circ$.
In the following we will denote by $j^\alpha$ the components of the tuple $(j,j^a,j^\circ)$.

Distribution $G$ on $S$ admits a local frame consisting of vector fields of the form
\begin{equation*}
\bbG_a := \frac{\partial}{ \partial u^a} + G_a^i \frac{\partial}{\partial x^i},\quad \bbG : = \frac{\partial}{\partial z} + G^i \frac{\partial}{\partial x^i}.
\end{equation*}
Consequently $F \in \Gamma (\wedge^2 N^\ast \calF \otimes T\calF)$, the curvature form of $G$, is locally given by
\begin{equation*}
F = \left( \frac{1}{2} F_{ab}^i du^a \wedge du^b + F^i_a du^a \wedge dz \right)  \otimes \frac{\partial}{\partial x^i},
\end{equation*}
where 
\[
F^i_{ab}  := \bbG_a (G^i_b) - \bbG_b (G^i_a) \quad \text{and} \quad F^i_a = \bbG_a (G^i) - \bbG (G^i_a).
\]
We also need the degree $0$ graded first order differential operator $j^1_G:\Omega(\calF;\ell)\to\Omega(\calF;J^1_\bot\ell)$.
Locally it is completely determined by
\begin{equation}
\label{eq:j^1_G(2)}
\begin{aligned}
j^1_G(f\mu)&:=j^1_G(f)_\alpha j^\alpha=fj+(\bbG_a f)j^a+(\bbG f)j^\circ,\\
j^1_G(d_\calF x^i{\otimes}\mu)&:=j^1_G(d_\calF x^i{\otimes}\mu)_{h\alpha}d_\calF x^h{\otimes} j^\alpha\\
&\phantom{:}=d_\calF x^i{\otimes} j+\frac{\partial G^i_a}{\partial x^h}d_\calF x^h{\otimes} j^a+\frac{\partial G^i}{\partial x^h}d_\calF x^h{\otimes} j^\circ.
\end{aligned}
\end{equation}
It is easy to see that the structure form $\Theta$ of the contact distribution on the contact thickening is locally given by
\[
\Theta = \left((1 - p_i G^i)dz - (C_a + p_i G^i_a)du^a + p_i dx^i \right) \otimes \mu,
\]
Then, a careful implementation of the definitions, through a long, but straightforward computation, shows that:
\begin{itemize}
	\item the transverse Jacobi structure $\hat{\Lambda}_\bot\in\Gamma(\wedge^2(J^1_\bot\ell)^\ast\otimes\ell)$ is determined by
	\begin{equation}
	\label{eq:proof:multi:transversal_Jacobi}
	\hat{\Lambda}_\bot(j^\alpha,j^\beta)=(\bbW^{-1})^{\alpha\beta}\mu,
	\end{equation}
	\item on the contact thickening, the Jacobi structure 
	$J \in \Gamma (\wedge^2 (J^1L)^\ast\otimes L)$ is locally
	\begin{equation}
	\label{eq:proof:multi:Jacobi_contact_thickening}
	J = \left(\frac{1}{2} (\bbW_{\boldsymbol p}^{-1})^{\alpha \beta} \square_\alpha \wedge \square_\beta + \nabla^i \wedge \nabla_i \right) \otimes \mu,
	\end{equation}
\end{itemize}
where $\bbW_{\boldsymbol p} := \bbW + p_i \bbF^i$, and
\[
\bbW := \left( 
\begin{array}{ccc}
0 & C_b & -1 \\
-C_a & \omega_{ab} & 0 \\
1 & 0 & 0
\end{array}
\right) \quad \text{and} \quad \bbF^i :=  \left( 
\begin{array}{ccc}
0 & 0 & 0 \\
0 & F^i_{ab} & F_a^i \\
0 & -F^i_b & 0
\end{array}
\right).
\]
Moreover we will denote by $\square_\alpha$ the components of the tuple $(\square, \square_a, \square_\circ)$ in $\Gamma ((J^1 L)^\ast)$ given by
\[
\begin{aligned}
\square & := \mu^\ast - p_i \nabla^i, \\
\square_a & := \nabla_a - p_j \frac{\partial G^{j}_a}{\partial x^i} \nabla^i + G^i_a \nabla_i, \\
\square_\circ & := \nabla - p_j \frac{\partial G^{j}}{\partial x^i} \nabla^i + G^i \nabla_i,
\end{aligned}
\]
where $\mu^\ast$, $\nabla_i$, $\nabla_a$, $\nabla$, $\nabla_i$ is the local frame of $(J^1 L)^\ast\to T_\ell{}^\ast\calF$ defined by
\begin{equation*}
\mu^\ast (f \mu) = f,\quad \nabla_i (f \mu) = \frac{\partial f}{\partial x^i}, \quad \nabla_a (f \mu) = \frac{\partial f}{\partial u^a},\quad \nabla (f \mu) = \frac{\partial f}{\partial z} \quad \text{and} \quad \nabla^i (f \mu) := \frac{\partial f}{\partial p_i}.
\end{equation*}
Accordingly the action of vector bundle morphism $F^\sharp\in\Gamma(\End(J^1_\bot L)\otimes\ell^\ast\otimes T\calF)$ is locally given by
\begin{equation}
\label{eq:F^sharp(2)}
F^\sharp(j^\alpha)=(\bbF^i\bbW^{-1})_\beta{}^\alpha j^\beta\otimes\mu^\ast\otimes\frac{\partial}{\partial x^i},
\end{equation}
where $(\bbF^i\bbW^{-1})_\beta\phantom{}^\alpha$ is the entry of the matrix $\bbF^i\bbW^{-1}$ in the $\beta$-th row and $\alpha$-th column.
\begin{proof}[Proof of Theorem~\ref{theor:multi}]
	Denote by $\lambda_k$ the $\bbR$-linear map $\lambda_k:\Omega(\calF;\ell)[1]^{\times k}\to \Omega(\calF;\ell)[1]$ which is defined, for $k=1$, by $\lambda_1=d_\calF$, and, for $k>1$, by the right hand side of Equation~\eqref{eq:multi}.
	By their very definition, both $\lambda_k$ and $\frakm_k$ (cf.~Remark~\ref{rem:multi-brackets}) satisfy the following properties, for all $k>0$:
	\begin{enumerate}[label=(\alph*)]
		\item\label{enum:proof:multi:a} they are degree $1$ graded symmetric $\bbR$-multilinear maps,
		\item\label{enum:proof:multi:b} they are first order differential operators with scalar-type symbol in each entry.
	\end{enumerate}
	From~\ref{enum:proof:multi:b}, both the $\lambda_k$'s and the ${\mathfrak m}_k$'s are completely determined by their actions on elements in $\Omega (\calF, \ell)$ of the form $d_\calF x^i \otimes \mu$, and $f \mu$, with $f \in C^\infty (S)$.
	Furthermore, because of~\ref{enum:proof:multi:a}, for every $k$-tuple $(\eta_1,\ldots,\eta_k)$ of homogeneous elements of $\Omega(\calF,\ell)[1]$ whose degree is non-positive, both $\lambda_k(\eta_1,\ldots,\eta_k)$ and $\frakm_k(\eta_1,\ldots,\eta_k)$ vanish whenever more than two of their arguments have negative degree.
	This starting remark will simplify the proof by reducing the number of cases to be considered.
	
	Let us start by checking that ${\frakm}_1$ coincides with $\lambda_1 = d_{\calF}$.
	From Proposition~\ref{prop:CF}, and Equations~\eqref{eq:j^1_G(2)} and~\eqref{eq:proof:multi:Jacobi_contact_thickening}, we immediately get that
	\begin{align*}
	\mathfrak m_1(f\mu)\!=\!\left.\left((\bbW_{\boldsymbol p}^{-1})^{\alpha\beta}p_l(j^1_G(f\mu))_\alpha (j^1_G(d_\calF x^l{\otimes}\mu))_{k\beta}+\frac{\partial f}{\partial x^k}\right)\right|_{\bf 0}d_\calF x^k{\otimes}\mu,
	\end{align*}
	for all $f\in C^\infty(S)$, and moreover
	\begin{multline*}
	\mathfrak m_1(d_\calF x^i{\otimes}\mu)\\
	=\!\tfrac{1}{2}\tfrac{\partial}{\partial p_i}\left.\left((\bbW_{\boldsymbol p}^{-1})^{\alpha\beta}p_{l_1}p_{l_2}(j^1_G(d_\calF x^{l_1}{\otimes}\mu))_{k_1\alpha}(j^1_G(d_\calF x^{l_2}{\otimes}\mu))_{k_2\beta}\right)\right|_{\bf 0}d_\calF x^{k_1}\wedge d_\calF x^{k_2}{\otimes}\mu.
	\end{multline*}
	Comparing with~\eqref{eq:d_F} shows that indeed $\mathfrak m_1=d_{\calF}$.
	
	Let us go further by checking that ${\frakm}_k = \lambda_k$, for all $k>1$.
	Now, from Corollary~\ref{prop:multi-brackets_coordinates}, we see that ${\frakm}_{k}$ depends on derivatives of $\bbW_{\boldsymbol p}^{-1}$ wrt the $p_i$'s at $\boldsymbol p := (p_i) = 0$ up to order $k$.
	By induction on $k$ we get
	\begin{equation}
	\label{eq:induction}
	\left. \frac{\partial \bbW^{-1}_{\boldsymbol p}}{\partial p_{i_1} \cdots \partial p_{i_k}} \right|_{\boldsymbol p = 0} = (-)^k \sum_{\sigma \in S_k} \bbW^{-1} \bbF^{i_{\sigma(1)}} \bbW^{-1} \cdots \bbF^{i_{\sigma(k)}} \bbW^{-1}.
	\end{equation}
	In this section, we will denote $\bbW^{-1} \bbF^{i_1} \bbW^{-1} \cdots \bbF^{i_k} \bbW^{-1}$ simply by $\bbY^{i_1\cdots i_k}$.
	From Proposition~\ref{prop:CF}, and Equations~\eqref{eq:j^1_G(2)}, \eqref{eq:proof:multi:Jacobi_contact_thickening}, and~\eqref{eq:induction}, through a straightforward computation, we get
%
	\begin{multline*}
	\mathfrak m_{k}(d_\calF x^{i_1}\otimes\mu,\ldots,d_\calF x^{i_k}\otimes\mu)\\
	=\!\frac{1}{2}\!\sum_{\sigma\in S_k}\!(\bbY^{i_{\sigma(1)}\cdots i_{\sigma(k-2)}})^{\alpha\beta}j^1_G(d_\calF x^{i_{\sigma(k-1)}}\otimes\mu)_{s\alpha}j^1_G(d_\calF x^{i_{\sigma(k)}}\otimes\mu)_{t\beta}d_\calF x^s\wedge d_\calF x^t\otimes\mu,
	\end{multline*}
	\begin{multline*}
	\mathfrak m_k(d_\calF x^{i_1}\otimes\mu,\ldots,d_\calF x^{i_{k-1}}\otimes\mu,f\mu)\\
	=-\!\sum_{\sigma\in S_{k-1}}\!\!(\bbY^{i_{\sigma(1)}\cdots i_{\sigma(k-2)}})^{\alpha\beta}j^1_G(f\mu)_{\alpha}j^1_G(d_\calF x^{i_{\sigma(k-1)}}\otimes\mu)_{s\beta}d_\calF x^s\otimes\mu,
	\end{multline*}
	\begin{multline*}
	\frakm_k(d_\calF x^{i_1}\!\otimes\mu,\ldots,d_\calF x^{i_{k-2}}\!\otimes\mu,f\mu,g\mu)\!
	=-\!\sum_{\sigma\in S_{k-2}}\!\!(\bbY^{i_{\sigma(1)}\cdots i_{\sigma(k-2)}})^{\alpha\beta}j^1_G(f\mu)_{\alpha}j^1_G(g\mu)_{\beta}\mu.
	\end{multline*}
	With the previous notations, from~\eqref{eq:F^sharp(2)} we get the following intermediate results
	\begin{gather}
	\label{eq:intermediate_results}
	i_{F^\sharp}(f\mu)=0,\qquad
	i_{F^\sharp}(d_\calF x^i\otimes\mu)=(\bbF^i\bbW^{-1})_\beta\phantom{}^\alpha\nabla_\alpha\otimes j^\beta.
	\end{gather}
	\newline\noindent
	From Equations~\eqref{eq:j^1_G(2)}, \eqref{eq:proof:multi:Jacobi_contact_thickening} and~\eqref{eq:intermediate_results}, implementing $\lambda_k$'s definition we get
	\begin{multline*}
	\lambda_{k}(d_\calF x^{i_1}\otimes\mu,\ldots,d_\calF x^{i_k}\otimes\mu)\\
	=\frac{1}{2}\sum_{\sigma\in S_k}(\bbY^{i_{\sigma(2)}\cdots i_{\sigma(k-1)}})^{\alpha\beta}j^1_G(d_\calF x^{i_{\sigma(1)}}\otimes\mu)_{h_1\alpha}j^1_G(d_\calF x^{i_{\sigma(k)}}\otimes\mu)_{h_2\beta}d_\calF x^{h_1}\wedge d_\calF x^{h_2}\otimes\mu
	\end{multline*}
	\begin{multline*}
	\lambda_k(d_\calF x^{i_1}\otimes\mu,\ldots,d_\calF x^{i_{k-1}}\otimes\mu,f\mu)\\
	=\frac{1}{2}\sum_{\sigma\in S_{k-1}}\left\langle j^1_G(f\mu),(i_{F^\sharp}d_\calF x^{i_{\sigma(1)}}\otimes\mu\circ\dots\circ i_{F^\sharp}d_\calF x^{i_{\sigma(k-2)}}\otimes\mu)j^1_G(d_\calF x^{i_{\sigma(k-1)}}\otimes\mu)\right\rangle_C\\
	+\frac{1}{2}\sum_{\sigma\in S_{k-1}}\left\langle j^1_G(d_\calF x^{i_{\sigma(1)}}\otimes\mu),(i_{F^\sharp}d_\calF x^{i_{\sigma(2)}}\otimes\mu\circ\dots\circ i_{F^\sharp}d_\calF x^{i_{\sigma(k-1)}}\otimes\mu)j^1_G(f\mu)\right\rangle_C\\
	=-\sum_{\sigma\in S_{k-1}}(\bbY^{i_{\sigma(1)}\cdots i_{\sigma(k-2)}})^{\alpha\beta}j^1_G(f\mu)_{\alpha}j^1_G(d_\calF x^{i_{\sigma(k-1)}}\otimes\mu)_{h\beta}d_\calF x^h\otimes\mu,
	\end{multline*}
	\begin{multline*}
	\lambda_k(d_\calF x^{i_1}\!\otimes\mu,\ldots,d_\calF x^{i_{k-2}}\!\otimes\mu,f\mu,g\mu)\\
	=\frac{1}{2}\sum_{\sigma\in S_{k-2}}\left\langle j^1_G(f\mu),(i_{F^\sharp}d_\calF x^{i_{\sigma(1)}}\otimes\mu\circ\dots\circ i_{F^\sharp}d_\calF x^{i_{\sigma(k-2)}}\otimes\mu)j^1_G(g\mu)\right\rangle_C\\
	-\frac{1}{2}\sum_{\sigma\in S_{k-2}}\left\langle j^1_G(g\mu),(i_{F^\sharp}d_\calF x^{i_{\sigma(1)}}\otimes\mu\circ\dots\circ i_{F^\sharp}d_\calF x^{i_{\sigma(k-2)}}\otimes\mu)j^1_G(f\mu)\right\rangle_C\\
	=-\sum_{\sigma\in S_{k-2}}(\bbY^{i_{\sigma(1)}\cdots i_{\sigma(k-2)}})^{\alpha\beta}j^1_G(f\mu)_{\alpha}j^1_G(g\mu)_{\beta}\mu.
	\end{multline*}
	This concludes the proof.
\end{proof}

\section{Toy examples}
\label{subsec:toy_examples}

In this short section we briefly discuss the formal deformation problem for the ``simplest possible'' coisotropic submanifolds, namely \emph{Legendrian submanifolds} in a contact manifold, and their \emph{flowout along a Jacobi vector field} (or, which is the same in this case, a contact vector field).
Recall that the flowout along a Jacobi vector field of a coisotropic submanifold is again coisotropic (Example~\ref{ex:simul}.(2)).

Now, let $(M,C)$ be a contact manifold, and let $(L, J= \{-,-\})$ be the associated Jacobi structure.
In particular, $\dim M = 2n +1$ for some $n > 0$.
Recall that a \emph{Legendrian submanifold} of $(M,C)$ is a locally maximal, hence $n$-dimensional, integral submanifold of the contact distribution.
Equivalently, a Legendrian submanifold is an isotropic submanifold, which is additionally coisotropic wrt~the Jacobi structure $(L,\{-,-\})$.
Let $S \subset M$ be a Legendrian submanifold, $\ell = L|_S$, and let $\mu \in \Gamma (L)$ be such that $\mu_x \neq 0$, hence $(X_\mu)_x \notin T_x S$, for all $x \in S$.
In what follows, we denote by $\mathcal T$ the flowout of $S$ along the Hamiltonian vector field $X_\mu$, and assume that it is closed.

\begin{remark}
	There exists a canonical vector bundle isomorphism $J^1\ell\rightarrow NS$ (over the identity) given by
	$j^1\lambda|_S \mapsto X_\lambda|_S\Mod TS$, for $\lambda\in\Gamma(L)$.
	Accordingly, there are canonical vector bundle isomorphisms
	$
	N^\ast S\simeq J_1 \ell$ and $N_\ell {}^\ast S \simeq \der\, \ell
	$.
\end{remark}

Recall that $J^1 \ell$ is equipped with a canonical contact structure (see Example~\ref{ex:1jet}).
The Legendrian tubular neighborhood theorem~\cite{loose1998neighborhoods} asserts that there is a tubular neighborhood $NS \injects M $ of $S$ in $M$ such that composition $J^1 \ell \to NS \to M$ is a contactomorphism onto its image.
Since we are interested in $C^1$-small coisotropic deformations of $S$, we can assume that $M = J^1 \ell$ and identify $S$ with the image of the zero section of the natural projection $J^1 \ell \to S$.

\begin{proposition}\label{prop:leg}
	Let $\{\frakm_k \}$ be the $L_\infty[1]$-algebra structure on
	$
	\Gamma (\wedge^\bullet N_\ell S \otimes \ell)[1] = (\Der^\bullet \ell)[1]
	$
	associated with the coisotropic submanifold $S$ in the contact
	manifold $J^1 \ell$.
	Then $\frakm_k = 0$ for $k > 1$, and $\frakm_1 = -d_{\der\, \ell , \ell}$, the opposite of the de Rham differential of the Atiyah algebroid $\der\, \ell$ with values in its tautological representation on $\ell$.
\end{proposition}
\begin{proof}
	Recall that the Jacobi structure on $J^1 \ell$ is fiber-wise linear (Example~\ref{ex:J1L}).
	Accordingly, the Jacobi bracket between
	\begin{itemize}
		\item fiber-wise constant sections is trivial,
		\item a fiber-wise constant and a fiber-wise linear section is fiber-wise constant,
		\item fiber-wise linear sections is fiber-wise linear.
	\end{itemize}
	Now, the assertion immediately follows from Equations~\eqref{eq:CF1}, \eqref{eq:CF2}, \eqref{eq:CF3}.
\end{proof}

\begin{remark}\label{rem:legformal}
	As a consequence of the above proposition, the formal deformation problem for Legendrian submanifolds is unobstructed.
	Even more, one can exhibit a canonical contracting homotopy for the complex $(\Der^\bullet \ell , d_{\der\, \ell, \ell})$ (see, for instance~\cite{rubtsov1980cohomology}).
	Hence, $\mathfrak{m}_1$ is acyclic and, as known to experts, all coisotropic, hence Legendrian, sections of $J^1 \ell \to S$ are actually trivial, i.e.~they are Hamiltonian equivalent to $S$.
	In other words the moduli space of coisotropic deformations of a Legendrian submanifold is zero dimensional.
\end{remark}

\begin{proposition}\label{prop:multi-brackets_coordinates_hamiltonian_flowout}
	Let $\{\frakm_k \}$ be the $L_\infty[1]$-algebra structure on
	$
	\Gamma (\wedge^\bullet (N\mathcal T \otimes L|_{\mathcal T}^\ast) \otimes L|_{\mathcal T})
	$
	associated with the coisotropic submanifold $\calT$ in the contact
	manifold $J^1 \ell$.
	Then $\frakm_k = 0$ for $k > 2$.
\end{proposition}
\begin{proof}
	
	The characteristic foliation $\calF$ of $\mathcal T$ is one-co-dimensional.
	Accordingly, any distribution $G$ complementary to $T \calF$ has rank $1$ and, therefore, it is involutive.
	In particular, its curvature $F$ vanishes.
	Now the assertion immediately follows from Theorem~\ref{theor:multi}.
\end{proof}


\begin{remark}
	Let $\mu \in \Gamma (L)$ be as above.
	Since $\mu_x \neq 0$ for all $x \in S$, local contactomorphism $J^1 \ell \to M$ can be chosen in such a way that $\mu $ identifies with a no-where zero, fiber-wise constant section of the Jacobi bundle on $J^1 \ell$.
	In particular, $J^1 \ell \simeq J^1 (M) := J^1 \bbR_M$ and $X_\mu$ identifies with the Reeb vector field on $J^1 (M)$.
	It follows that Propositions~\ref{prop:leg} and~\ref{prop:multi-brackets_coordinates_hamiltonian_flowout} can be also proven from Proposition~\ref{prop:multi-brackets_coordinates} and the explicit form of the Jacobi structure on $J^1 (M)$ in jet coordinates (see, for instance, \cite[Exercise 2.7]{bocharov1999firstorder}).
\end{remark}

\section{A first obstructed example in the contact setting}
\label{sec:obstructed_example_contact_L-infinity}

In this section we give a conceptual interpretation based on the associated $L_\infty[1]$-algebra of a first example of co\-iso\-tro\-pic submanifold in a contact manifold whose co\-iso\-tro\-pic deformation problem is formally obstructed (see~\cite[Examples 3.5 and 3.8]{tortorella2016rigidity}).
This example was originally derived by analytical methods, and employed to illustrate that a certain subclass of coisotropic submanifolds (the so called ``integral'' ones) is not stable under small coisotropic deformations (cf.~\cite[Proposition 4.5]{tortorella2016rigidity}).
Notice that a conceptual approach to this example can be alternatively grounded on the BFV-complex (see Section~\ref{sec:obstructed_example_contact_BFV}).

Let us consider vector bundle $E:=\bbT^5\times\bbR^2\overset{\tau}{\longrightarrow} S:=\bbT^5$, $(\phi_i,y_a)\longmapsto(\phi_i)$, where $\phi_1,\ldots,\phi_5$ are the standard angular coordinates on the $5$-dimensional torus $\bbT^5$, and $y_1,y_2$ are the standard Euclidean coordinates on $\bbR^2$.
Now $E$ comes equipped with a coorientable contact structure (cf.~Remark~\ref{rem:coorientable}) by means of the contact distribution $C_E:=\ker\theta_E$, where the global contact form $\theta_E\in\Omega^1(E)$ is given by
\begin{equation*}
\theta_E:=y_1d\phi_1+y_2d\phi_2+\sin\phi_3d\phi_4+\cos\phi_3d\phi_5.
\end{equation*}
As we already know, this contact structure determines a Jacobi structure $J=\{-,-\}$ on $\bbR_E\to E$, the trivial line bundle over $E$.
It is straightforward to check that
\begin{equation*}
J=\frac{\partial}{\partial\phi_3}\wedge X+Y\wedge\Delta_E-\frac{\partial}{\partial\phi_1}\wedge\frac{\partial}{\partial y_1}-\frac{\partial}{\partial\phi_2}\wedge\frac{\partial}{\partial y_2}-Y\wedge\id,
\end{equation*}
where $\Delta_E$ denotes the Euler vector field on $E$, i.e.~$\Delta_E:=y_1\frac{\partial}{\partial y_1}+y_2\frac{\partial}{\partial y_2}$, and $X,Y\in\frakX(S)$ are defined by
\begin{equation*}
X:=\cos\phi_3\frac{\partial}{\partial \phi_4}-\sin\phi_3\frac{\partial}{\partial \phi_5},\qquad Y:=\sin\phi_3\frac{\partial}{\partial \phi_4}+\cos\phi_3\frac{\partial}{\partial \phi_5}.
\end{equation*}
Notice that $Y$ is exactly the Reeb vector field, i.e.~$\theta_E(Y)=1$ and $\theta_E([Y,\Gamma(C_E)])=0$.
Understanding the algebra embedding ${\tau^\ast}:C^\infty(S){\hookrightarrow} C^\infty(E)$, the Jacobi bracket $\{-,-\}$ on $C^\infty(E)$ is also completely determined by
\begin{equation}
\label{eq:obstructed_example_contact1}
\left\{y_a,y_b\right\}=0,\quad\left\{y_a,f\right\}=\frac{\partial f}{\partial\phi_a},\quad
\{f,g\}=\frac{\partial f}{\partial\phi_3}Xg-\frac{\partial g}{\partial\phi_3}Xf+fYg-gYf,
\end{equation}
for arbitrary  $a,b\in\{1,2\}$, and $f,g\in C^\infty(S)$.

From~\eqref{eq:obstructed_example_contact1} we get that $S$ is a regular coisotropic submanifold of $(E,C_E)$.
Hence $S$ inherits a structure of pre-contact manifold, with pre-contact distribution
\begin{equation*}
C_S:=C_E\cap TS=\left\langle\frac{\partial}{\partial\phi_1},\frac{\partial}{\partial\phi_2},\frac{\partial}{\partial \phi_3},X\right\rangle.
\end{equation*} 
The latter coincides with the kernel of the global pre-contact form $\theta_S\in\Omega^1(S)$ given by
\begin{equation*}
\theta_S:=\theta_E|_{TS}=\sin\phi_3d\phi_4+\cos\phi_3d\phi_5.
\end{equation*}
Indeed $S$ is not only regular coisotropic, but it is actually reducible, with contact reduction performed through the projection $\bbT^5=\bbT^2\times\bbT^3\to\bbT^3$, reduced contact form $\theta=\sin\phi_3d\phi_4+\cos\phi_3d\phi_5$ and reduced contact distribution $C=\left\langle\frac{\partial}{\partial \phi_3},X\right\rangle$.
For more details we refer the reader to~\cite[Section~4]{tortorella2016rigidity}.

Notice that the pre-contact manifold $(S,C_S)$ is transversally integrable (cf.~Definition~\ref{def:transversally_integrable}).
Indeed its characteristic distribution is
\begin{equation*}
T\calF=\left\langle\frac{\partial}{\partial\phi_1},\frac{\partial}{\partial\phi_2}\right\rangle;
\end{equation*}
so it admits the following involutive complementary distribution
\begin{equation*}
G:=\left\langle\frac{\partial}{\partial\phi_3},\frac{\partial}{\partial\phi_4},\frac{\partial}{\partial\phi_5}\right\rangle.
\end{equation*}

The global frame $d\phi_1|_{T\calF},d\phi_2|_{T\calF}$ of $T^\ast\calF\to S$ will be used to identify $T^\ast\calF\to S$ with the trivial rank $2$ vector bundle $E\to S$ so that $\tfrac{\partial}{\partial\phi_1}$ and $\tfrac{\partial}{\partial\phi_2}$, seen as fiberwise linear functions on $T^\ast\calF$, agrees with $y_1$ and $y_2$ respectively.
Under this identification, we immediately get that:
\begin{itemize}
	\item $(E,C_E)$ identifies with $(T^\ast\calF,\ker(\theta_G+\tau^\ast\theta_S))$, i.e.~the contact thickening of $(S,C_S)$ corresponding to the splitting $TS=T\calF\oplus G$ (cf.~Section~\ref{sec:contact_thickening}),
	\item the $L_\infty[1]$-algebra of $S$ is represented by $(\Omega^\bullet(\calF),\{\frakm_k\})$ with the multibrackets determined by $G$ according to Theorem~\ref{theor:multi}.
\end{itemize}
Moreover, since $G$ is involutive, Proposition~\ref{prop:transversally_integrable} implies that $\frakm_k=0$, for all $k>2$.
In the following we will describe more explicitly the action of $\frakm_1$ and $\frakm_2$ in order to obtain information about the coisotropic deformation problem of $S$.

Let us start finding the explicit expression of $\frakm_1$, and extracting from it information about 1) the infinitesimal coisotropic deformations of $S$ and 2) their moduli space under infinitesimal Hamiltonian equivalence.
From coorientability, $\frakm_1:\Omega^\bullet(\calF)\to\Omega^\bullet(\calF)$ reduces to the de Rham differential $d_\calF$ of the Lie algebroid $T\calF\to S$.
Hence, for all $f,g\in C^\infty(S)$, we have:
\begin{equation}
\label{eq:obstructed_example_contact2}
\begin{gathered}
\frakm_1(f)=\frac{\partial f}{\partial\phi_1}d\phi_1+\frac{\partial f}{\partial\phi_2}d\phi_2,\\
\frakm_1(fd\phi_1+gd\phi_2)=\left(\frac{\partial g}{\partial\phi_1}-\frac{\partial f}{\partial\phi_2}\right)d\phi_1\wedge d\phi_2
\end{gathered}
\end{equation} 
Now, from~\eqref{eq:obstructed_example_contact2}, we get that:
\begin{enumerate}
	\item in view of Corollary~\ref{cor:inf1}, an arbitrary section $s=fd\phi_1+gd\phi_2\in\Gamma(T^\ast\calF)$, with $f,g\in C^\infty(\bbT^5)$, is an infinitesimal coisotropic deformation of $S$ if and only if
	\begin{equation}
	\label{eq:obstructed_example_contact3}
	\frac{\partial g}{\partial\phi_1}-\frac{\partial f}{\partial\phi_2}=0,
	\end{equation}
	\item in view of Corollary~\ref{cor:infequi}, two infinitesimal coisotropic deformations $s_i=f_id\phi_1+g_id\phi_2$, with $i=0,1$, are infinitesimal Hamiltonian equivalent if and only if there is $h\in C^\infty(\bbT^5)$ such that
	\begin{equation*}
	f_1=f_0+\frac{\partial h}{\partial \phi_1},\quad g_1=g_0+\frac{\partial h}{\partial \phi_2}.
	\end{equation*}
\end{enumerate}

Going further, we will provide an explicit expression of $\frakm_2$, which will tell us something about 1) the first obstruction to the prolongability, under Hamiltonian equivalence, of infinitesimal coisotropic deformations and 2) the space of coisotropic deformations of $S$.
In view of~\eqref{eq:obstructed_example_contact1}, from Proposition~\ref{prop:CF} it follows that $\frakm_2:\Omega^\bullet(\calF)[1]\times\Omega^\bullet(\calF)[1]\to\Omega^\bullet(\calF)[1]$ is completely determined by
\begin{equation}
\label{eq:obstructed_example_contact4}
\begin{gathered}
\frakm_2(f,g)=-\{f,g\},\\
\frakm_2(f,g_1d\phi_1+g_2d\phi_2)=-\{f,g_1\}d\phi_1-\{f,g_2\}d\phi_2,\\
\frakm_2(f_1d\phi_1+f_2d\phi_2,g_1d\phi_1+g_2d\phi_2)=\left(\{f_1,g_2\}-\{f_2,g_1\}\right)d\phi_1\wedge d\phi_2,
\end{gathered}
\end{equation}
Now, from~\eqref{eq:obstructed_example_contact4}, we get that:
\begin{enumerate}
	\item in view of Proposition~\ref{prop:Kuranishi}, if $s=fd\phi_1+gd\phi_2$ is an infinitesimal coisotropic deformation of $S$ which can be prolonged to a formal coisotropic deformation of $S$ up to Hamiltonian equivalence, then there exist $h,k\in C^\infty(\bbT^5)$ such that
	\begin{equation*}
	\frac{\partial f}{\partial\phi_3}Xg-\frac{\partial g}{\partial\phi_3}Xf+fYg-gYf=\frac{\partial k}{\partial\phi_1}-\frac{\partial h}{\partial\phi_2}.
	\end{equation*}
	Hence, integrating over $\phi_1$ and $\phi_2$, we obtain the following weaker necessary condition for the prolongability of $s$
	\begin{equation}
	\label{eq:obstructed_example_contact5}
	\iint\limits_{\bbT^2}\left(\frac{\partial f}{\partial\phi_3}Xg-\frac{\partial g}{\partial\phi_3}Xf+fYg-gYf\right)d\phi_1 d\phi_2=0;
	\end{equation}
	\item in view of Corollary~\ref{cor:conver}, an arbitrary section $s=fd\phi_1+gd\phi_2\in\Gamma(T^\ast\calF)$, with $f,g\in C^\infty(\bbT^5)$, is a coisotropic deformation of $S$ if and only if $-s$ satisfies the MC equation $\frakm_1(-s)+\tfrac{1}{2}\frakm_2(-s,-s)=0$, i.e.~the following non-linear first order pde
	\begin{equation}
	\label{eq:example_cosiotropic_deformation_space_Linfty}
	\frac{\partial f}{\partial\phi_2}-\frac{\partial g}{\partial\phi_1}+\frac{\partial f}{\partial\phi_3}Xg-\frac{\partial g}{\partial\phi_3}Xf+fYg-gYf=0,
	\end{equation}
\end{enumerate}

Finally we point out that $s=\cos\phi_4d\phi_1+
\sin\phi_4d\phi_2$ is an infinitesimal coisotropic deformation of $S$ which can not be prolonged to a formal coisotropic deformation of $S$ (and, a fortiori, neither to a smooth one).
Indeed~\eqref{eq:obstructed_example_contact3} is fulfilled, but the constraint~\eqref{eq:obstructed_example_contact5} fails to be satisfied:
\begin{equation*}
\iint\limits_{\bbT^2}\left(\frac{\partial f}{\partial\phi_3}Xg-\frac{\partial g}{\partial\phi_3}Xf+fYg-gYf\right)d\phi_1 d\phi_2=(2\pi)^2\sin\phi_3\neq 0.
\end{equation*}
Hence the coisotropic deformation problem of $S$ is formally (and a fortiori smoothly) obstructed, even up to Hamiltonian equivalence.
As a consequence the moduli space of its coisotropic deformations under Hamiltonian equivalence is not smooth at the equivalence class of $S$.

\section{A second obstructed example in the contact setting}
\label{sec:second_obstructed_example_contact_L-infinity}

We describe a second example of regular co\-iso\-tro\-pic submanifold whose co\-iso\-tro\-pic deformation problem is formally obstructed (cf.~also the revised version of~\cite{LOTV}).
Differently from Section~\ref{sec:obstructed_example_contact_L-infinity} this obstructed coisotropic submanifold has a non-simple characteristic foliation.
From this perspective it is closely inspired by the analogous example in the symplectic setting which was constructed by Oh--Park~\cite{zambon2008example} and further investigated by Kieserman~\cite{kieserman2010liouville}.
In this section we provide a conceptual interpretation of this obstructed example  in terms of the associated $L_\infty[1]$-algebra.

Let us consider the $7$-dimensional coorientable contact manifold $(M,C)$, with $M:=\bbR^6\times\bbS^1$ and $C:=\ker\theta$, where global contact $1$-form $\theta\in\Omega^1(M)$ is given by
\begin{equation*}
\theta:=d\phi-\sum_{i=1}^3p_idq^i.
\end{equation*}
Above $(q^i,p_i)$ are the cartesian coordinates on $\bbR^6\simeq T^\ast\bbR^3$ and $\phi$ is the angle coordinate for $\bbS^1$.
In the following we also use polar coordinates $(r_i,\phi_i)$ on each plane $\bbR^2=\{(q_i,p^i)\}$, for $i=1,2,3$.

The contact distribution $C$ is a trivial vector bundle over $M$, with a global frame provided by
\begin{equation*}
\frac{\partial}{\partial p_i},\qquad D_i:=\frac{\partial}{\partial q^i}+p_i\frac{\partial}{\partial\phi}.
\end{equation*}
For every $f\in\Gamma(\bbR_M)= C^\infty(M)$, the corresponding contact vector field $X_f$, uniquely determined by $\theta(X_f)=f$ and $\theta([X_f,\Gamma(C)])=0$, is given by
\begin{equation*}
X_f=(D_if)\frac{\partial}{\partial p_i}-\sum_{i=1}^3\frac{\partial f}{\partial p_i}D_i+f\frac{\partial}{\partial\phi}.
\end{equation*}
In particular $\partial/\partial\phi$ is exactly the Reeb vector field $X_1$.
As we already know, the coorientable contact structure on $M$ determines a Jacobi structure $J=\{-,-\}$ on the trivial line bundle $\bbR_M\to M$.
It is straightforward to check that
\begin{equation*}
J=D_i\wedge\frac{\partial}{\partial p_i}+\id\wedge\frac{\partial}{\partial\phi}.
\end{equation*}

Consider the functions $H_i\in C^\infty(M)$ given by $H_i:=\frac{1}{2}r_i^2$, for $i=1,2,3$.
For every $\alpha>0$, set $H_\alpha:=H_1+\alpha H_2$, and define the $5$-dimensional submanifold $S_\alpha\subset M$ by setting
\begin{equation*}
S_\alpha:=H_\alpha^{-1}\left(1/4\right)\cap H_3^{-1}\left(1/2\right).
\end{equation*}
Since $\{H_\alpha,H_3\}=0$, and $\theta$, $dH_\alpha$, $dH_3$, are linearly independent on a neighborhood of $S_\alpha$, from Proposition~\ref{prop:coisotropics_contact_setting} we get that $S_\alpha$ is a regular coisotropic submanifold of $(M,C)$.
Hence $S_\alpha$ inherits a structure of pre-contact manifold, with pre-contact distribution $C_\alpha:=C\cap TS_\alpha$, i.e.~the kernel of global pre-contact form $\theta_\alpha:=\theta|_{TS_\alpha}\in\Omega^1(S_\alpha)$.
Moreover its characteristic distribution $T\calF$ is a trivial vector bundle over $S_\alpha$, with a global frame provided by
\begin{equation*}
\left.X_{(H_\alpha-\tfrac{1}{4})}\right|_{S_\alpha}\!=\!\left.\left(\frac{\partial}{\partial\phi_1}+\alpha\frac{\partial}{\partial\phi_2}-(p_1^2+\alpha p_2^2)\frac{\partial}{\partial\phi}\right)\!\right|_{S_\alpha}\!,\quad \left.X_{(H_3-\tfrac{1}{2})}\right|_{S_\alpha}\!=\!\left.\left(\frac{\partial}{\partial\phi_3}-p_3^2\frac{\partial}{\partial\phi}\right)\!\right|_{S_\alpha}\!.
\end{equation*}
So in particular all its characteristic leaves are orientable.

\begin{remark}
	Clearly, for $\alpha=1$, the characteristic foliation is simple with leaf space diffeomorphic to $\mathbb{CP}^1\times\bbS^1$.
	For all $\alpha\neq 1$ instead the characteristic foliation is not simple.
	On the one hand, for $\alpha\notin\bbQ$, each one of the characteristic leaves contained in $S_\alpha\cap H_1^{-1}(]0,1/4[)$ is dense in $S_\alpha$.
	On the other hand, for $\alpha=m/n$, with $m$ and $n$ coprime integers, there are characteristic leaves with non-trivial holonomy: precisely, characteristic leaves contained in $S_\alpha\cap H_1^{-1}(0)$ (resp.~$S_\alpha\cap H_1^{-1}(1/4)$) have cyclic holonomy group of order $m$ (resp.~$n$).
\end{remark}

Set $U_\alpha:=S_\alpha\cap H_1^{-1}(]0,1/4[)$.
Then $U_\alpha$ is an open dense subset of $S_\alpha$, and it is covered by charts of local coordinates $(u_1,u_2,x,y,z)$ defined as follows
\begin{align*}
u_1=\phi_3,\quad 
u_2=\phi_1+\alpha\phi_2,\quad 
x=H_2,\quad
y=\phi_2-\alpha\phi_1,\quad 
z=\phi+\sum_{i=1}^3 H_i(\phi_i-\frac{1}{2}\sin(2\phi_i)).
\end{align*}
The latter are actually local Darboux coordinates on $S_\alpha$, i.e.~locally $\theta_\alpha=dz-ydx$.
So locally we also have
\begin{equation}
\label{eq:second_obstructed_example_Darboux}
C_\alpha=\left\langle\frac{\partial}{\partial u_1},\frac{\partial}{\partial u_2},\frac{\partial}{\partial y},D:=\frac{\partial}{\partial x}+y\frac{\partial}{\partial z}\right\rangle,\qquad T\calF=\left\langle\frac{\partial}{\partial u_1},\frac{\partial}{\partial u_2}\right\rangle.
\end{equation}
Correspondingly the local coordinate vector fields have the following local expressions
\begin{align*}
\frac{\partial}{\partial u_1}&=X_{(H_3-1/2)},\\
\frac{\partial}{\partial u_2}&=\frac{1}{1+\alpha^2}X_{(H_\alpha-1/4)},\\
\frac{\partial}{\partial y}&=\frac{1}{1+\alpha^2}\left(\frac{\partial}{\partial\phi_2}-\alpha\frac{\partial}{\partial\phi_1}+\left(\alpha p_1^2-p_2^2\right)\frac{\partial}{\partial\phi}\right),\\
\frac{\partial}{\partial x}&=-\frac{\alpha}{r_1}\frac{\partial}{\partial r_1}+\frac{1}{r_2}\frac{\partial}{\partial r_2}+\left(\alpha\phi_1-\phi_2-\frac{\alpha}{2}\sin(2\phi_1)+\frac{1}{2}\sin(2\phi_2)\right)\frac{\partial}{\partial\phi},\\
\frac{\partial}{\partial z}&=\frac{\partial}{\partial\phi}.
\end{align*}
Notice that vector fields $\frac{\partial}{\partial u_1}$, $\frac{\partial}{\partial u_2}$, $\frac{\partial}{\partial y}$, $D$, $\frac{\partial}{\partial z}$ do not depend on the local chart of Darboux coordinates, and so they are globally defined on $U_\alpha$.
Moreover vector fields $\frac{\partial}{\partial u_1}$ and $\frac{\partial}{\partial u_2}$ (resp.~differential forms $d_\calF u_1\equiv(du_1)|_{T\calF}$ and $d_\calF u_2\equiv(du_2)|_{T\calF}$) uniquely prolong to a global frame of $T\calF$ (resp.~$T^\ast\calF$). 
As a consequence, for any $0<\eps<1/8$, we can pick a distribution $G$ on $S_\alpha$ complementary to $T\calF$ and satisfying the following additional property
\begin{equation}
\label{eq:second_obstructed_example_complementary}
\left.G\right|_{U_{\alpha,\eps}}=\left.\left\langle\frac{\partial}{\partial y},\ D,\ \frac{\partial}{\partial z}\right\rangle\right|_{U_{\alpha,\eps}},
\end{equation}
where the open subset $U_{\alpha,\eps}\subset U_\alpha$ is defined by $U_{\alpha,\eps}:=S_\alpha\cap H_1^{-1}(]\eps,1/4-\eps[)$.
From now on we assume to have fixed such a distribution $G$.
As we already know, after such choice we get that:
\begin{itemize}
	\item around $S_\alpha$ the contact manifold $(M,C)$ identifies with 
	the contact thickening of pre-contact manifold $(S_\alpha,C_\alpha)$ corresponding to the splitting $TS_\alpha=T\calF\oplus G$ (cf.~Section~\ref{sec:contact_thickening}),
	\item the $L_\infty[1]$-algebra of $S_\alpha$ is represented by $(\Omega^\bullet(\calF),\{\frakm_k\})$ with the multibrackets determined by $G$ according to Theorem~\ref{theor:multi}.
\end{itemize}

We focus on the explicit action of $\frakm_1$ and $\frakm_2$.
From coorientability, $\frakm_1:\Omega^\bullet(\calF)\to\Omega^\bullet(\calF)$ reduces to the de Rham differential $d_\calF$ of the Lie algebroid $T\calF\to S_\alpha$.
Hence, for all $f,g\in C^\infty(S_\alpha)$, the following identities hold on $S_\alpha$:
\begin{equation}
\label{eq:second_obstructed_example_contact2}
\begin{gathered}
\frakm_1(f)=\frac{\partial f}{\partial u_1}d_\calF u_1+\frac{\partial f}{\partial u_2}d_\calF u_2,\\
\frakm_1(fd_\calF u_1+gd_\calF u_2)=\left(\frac{\partial g}{\partial u_1}-\frac{\partial f}{\partial u_2}\right)d_\calF u_1\wedge d_\calF u_2.
\end{gathered}
\end{equation}
Let $J_\alpha=\{-,-\}_\alpha:C^\infty(U_\alpha)\times C^\infty(U_\alpha)\to C^\infty(U_\alpha)$ be the first-order skew-symmetric bi-differential operator defined by
\begin{equation*}
J_\alpha:=\id\wedge\frac{\partial}{\partial z}+D\wedge\frac{\partial}{\partial y}.
\end{equation*}
In view of~\eqref{eq:second_obstructed_example_Darboux} and~\eqref{eq:second_obstructed_example_complementary}, Theorem~\ref{theor:multi} implies that $\frakm_2:\Omega^\bullet(\calF)[1]\times\Omega^\bullet(\calF)[1]\to\Omega^\bullet(\calF)[1]$ acts so that the following identities hold on $U_{\alpha,\eps}$
\begin{equation}
\label{eq:second_obstructed_example_contact4}
\begin{gathered}
\frakm_2(f,g)=-\{f,g\}_\alpha,\\
\frakm_2(f,g_1d_\calF u_1+g_2d_\calF u_2)=-\{f,g_1\}_\alpha d_\calF u_1-\{f,g_2\}_\alpha d_\calF u_2,\\
\frakm_2(f_1d_\calF u_1+f_2d_\calF u_2,g_1d_\calF u_1+g_2d_\calF u_2)=\left(\{f_1,g_2\}_\alpha-\{f_2,g_1\}_\alpha\right)d_\calF u_1\wedge d_\calF u_2.
\end{gathered}
\end{equation}

We extract from~\eqref{eq:second_obstructed_example_contact2} and~\eqref{eq:second_obstructed_example_contact4} information about the infinitesimal coisotropic deformations of $S_\alpha$, their moduli space under infinitesimal Hamiltonian equivalence, and the first obstructions to their formal prolongability.
According to Corollary~\ref{cor:inf1}, an arbitrary $s=fd_\calF u_1+gd_\calF u_2\in\Omega^1(\calF)$ is an infinitesimal coisotropic deformation of $S_\alpha$ iff the following identity holds on $S_\alpha$
\begin{equation}
\label{eq:second_obstructed_example_contact_deformations}
\frac{\partial g}{\partial u_1}-\frac{\partial f}{\partial u_2}=0.
\end{equation}
In view of Corollary~\ref{cor:infequi}, two infinitesimal coisotropic deformations $s_i=f_id_\calF u_1+g_id_\calF u_2$, with $i=0,1$, are infinitesimally Hamiltonian equivalent iff there exists $h\in C^\infty(S_\alpha)$ such that the following identity holds on $S_\alpha$
\begin{equation*}
f_1=f_0+\frac{\partial h}{\partial u_1},\quad g_1=g_0+\frac{\partial h}{\partial u_2}.
\end{equation*}
Let $s=fd_\calF u_1+gd_\calF u_2$ be an infinitesimal coisotropic deformation of $S_\alpha$, such that $\operatorname{supp}(s)\subset U_\alpha$.
Since $\eps$ can be chosen arbitrarily small, from Proposition~\ref{prop:Kuranishi} it follows that, if $s$ can be prolonged to a formal coisotropic deformation, then there exist $h,k\in C^\infty(S_\alpha)$ such that
\begin{equation}
\label{eq:second_obstructed_example_contact5}
f\frac{\partial g}{\partial z}-g\frac{\partial f}{\partial z}+(Df)\frac{\partial g}{\partial y}-(Dg)\frac{\partial f}{\partial y}
=\frac{\partial k}{\partial u_1}-\frac{\partial h}{\partial u_2}.
\end{equation}
Integrating~\eqref{eq:second_obstructed_example_contact5} over a compact characteristic leaf $L$, we obtain also a weaker necessary condition for the formal prolongability of $s$:
\begin{equation}
\label{eq:second_obstructed_example_contact6}
\iint\limits_{L}\left(f\frac{\partial g}{\partial z}-g\frac{\partial f}{\partial z}+(Df)\frac{\partial g}{\partial y}-(Dg)\frac{\partial f}{\partial y}\right)d_\calF u_1 d_\calF u_2=0.
\end{equation}

\begin{proposition}
	\label{prop:second_obstructed_example_contact}
	If $\alpha\in\bbQ$, then the coisotropic submanifold $S_\alpha$ of $(M,C)$ is formally obstructed.
\end{proposition}

\begin{proof}
	Let $\alpha=\frac{m}{n}$, with coprime integers $m$ and $n$.
	In this case the characteristic foliation $\calF_\alpha$ has orientable compact leaves.
	Fix two non-constant functions $\chi\in C^\infty(\bbS^1)$ and $\rho\in C^\infty(\bbR)$ such that additionally the support of $\rho$ is contained in $]0,1/4\alpha[$.
	Then there exist two functions $f,g\in C^\infty(S_\alpha)$ uniquely determined by
	\begin{equation*}
	\label{eq:prop:second_obstructed_example_contact}
	f(u_1,u_2,x,y,z)=\rho(x),\qquad g(u_1,u_2,x,y,z)=\rho(x)\chi(ny).
	\end{equation*}
	Set $s:=fd_\calF u_1+gd_\calF u_2\in\Omega^1(\calF)$.
	Actually $s$ is an infinitesimal coisotropic deformation of $S_\alpha$ which is formally obstructed.
	Indeed $s$ satisfies~\eqref{eq:second_obstructed_example_contact_deformations}, but it fails to fulfill the constraint~\eqref{eq:second_obstructed_example_contact6}:
	\begin{equation*}
	\iint\limits_{L(\bar x,\bar y,\bar z)}\!\!\!\left(f\frac{\partial g}{\partial z}-g\frac{\partial f}{\partial z}+(Df)\frac{\partial g}{\partial y}-(Dg)\frac{\partial f}{\partial y}\right)d_\calF u_1 d_\calF u_2=\tfrac{m^2+n^2}{n}(2\pi)^2\rho(\bar x)\rho'(\bar x)\chi'(n\bar y)\neq 0,
	\end{equation*}
	where, for any $(\bar x,\bar y,\bar z)$, we have denoted by $L(\bar x,\bar y,\bar z)$ the characteristic leaf given by the level set $x=\bar x, y=\bar y, z=\bar z$.
\end{proof}
\chapter{Graded Jacobi manifolds}
\label{chap:graded_Jacobi_manifolds}


This chapter aims at setting the framework for constructing of the BFV-complex in the next chapter.
It can be read as a direct continuation of Chapters~\ref{chap:preliminaries} and~\ref{sec:abstract_jac_mfd}.

In the first section we develop Jacobi geometry in the setting of graded differential geometry.
Doing this we not only introduce Jacobi structures on graded line bundles but we even meet derivations of graded vector bundles and the Gerstenhaber--Jacobi algebra of multi-derivations of a graded line bundle.

In the second section we propose a notion of lifting of a Jacobi structure from a line bundle $L\to M$ to a certain graded line bundle $\hat L\to\hat M$.
As a first preliminary step towards the construction of the BFV-complex, we prove existence and uniqueness (up to line bundle automorphisms) of these liftings.
The latter extends to the Jacobi setting analogous results by Rothstein~\cite{Rothstein1991} (symplectic case), Herbig~\cite{herbig2007} and Sch\"atz~\cite{schatz2009bfv} (Poisson case).
In particular we describe a constructive lifting procedure based on the ``step-by-step obstruction'' technique from Homological Perturbation Theory.

\section{Multi-derivations of graded line bundles}
\label{app:graded_multi-do}

In this section we collect basic facts, including conventions and notations, concerning graded symmetric multi-derivations on graded line bundles, and, in particular, graded Jacobi structures and Jacobi bi-derivations.
Doing this we will use the language of $\bbZ$-graded differential geometry (see, e.g.,~\cite{Mehta2006}).

\subsection{Graded symmetric multi-derivations}

Let $\scrM$ be a $\bbZ$-graded manifold and let $\scrP,\scrQ$ be graded vector bundles over $\scrM$.
The notion of (linear) first order differential operator (as recalled in Section~\ref{sec:app_0}) extends immediately to the graded setting as follows.
A \emph{degree $k$ graded first order differential operator} from $\scrP$ to $\scrQ$ is a degree $k$ graded $\bbR$-linear map $\square:\Gamma(\scrP)\longrightarrow\Gamma(\scrQ)$, such that
\begin{equation*}
[[\square,a_1],a_2]=0,
\end{equation*}
for all $a_1,a_2\in C^\infty(\scrM)$, where we interpret the scalars $a_i$ as operators (multiplication by $a_i$), and denote by $[-,-]$ the graded commutator of operators.

A \emph{degree $k$ graded derivation} of $\scrP$ is a degree $k$ graded $\bbR$-linear map $\square:\Gamma(\scrP)\longrightarrow\Gamma(\scrP)$, such that there is a (necessarily unique) vector field $\sigma(\square)\in\frakX(\scrM)$, also denoted by $\sigma_\square$ and called the \emph{symbol} of $\square$, which satisfies the following graded Leibniz rule
\begin{equation*}
\square(a p)=\sigma_\square(a) p+(-)^{|a|k}a\square(p),
\end{equation*}
for all homogeneous $a\in C^\infty(\scrM)$, and $p\in\Gamma(\scrP)$.
Here, as throughout this thesis, we have denoted by $|v|$ the degree of an homogeneous element $v$ in a graded vector space.

\begin{remark}
	Every derivation of a graded vector bundle $\scrP$ is, in particular, a first order differential operator.
	In general, the converse is not true, unless $\scrP\to\scrM$ is a (graded) line bundle.
	In this case derivations of $\scrP$ are the same as first order differential operators from $\scrP$ to itself.
\end{remark}

\begin{remark}
	\label{rem:LRalgebra}
	Denote by $\Diff(\scrP)^k$ the space of degree $k$ graded derivations of $\scrP$.
	Then the space of graded derivations of $\scrP$
	\begin{equation*}
	\Diff(\scrP)^\bullet:=\bigoplus_{k\in\bbZ}\Diff(\scrP)^k
	\end{equation*}
	is a graded $C^\infty(\scrM)$-module, and also a graded Lie algebra, with the Lie bracket given by the usual graded commutator $[-,-]$.
	Additionally, the \emph{symbol map} $\Diff(\scrP)\to\frakX(\scrM),\ \square\mapsto X_\square$, is $C^\infty(\scrM)$-linear, and also a Lie algebra morphism which satisfies the following compatibility condition   
	\begin{equation*}
	[\square_1,a\square_2]=X_{\square_1}(a)\square_2+(-)^{|a||\square_1|}a[\square_1,\square_2],
	\end{equation*}
	for all homogeneous $a\in C^\infty(\scrM)$, and $\square_1,\square_2\in\Diff(\scrP)$, i.e.~the pair $(C^\infty(\scrM),\Diff(\scrP))$ is a graded Lie--Rinehart algebra (see, e.g.,~\cite{huebschmann2004lie} and~\cite{vitagliano2015homotopy}, and also cf.~Example~\ref{ex:Gerstenhaber--Jacobi}~\ref{enumitem:ex:Gerstenhaber--Jacobi_4}).
	Hence $\Diff(\scrP)$ is the module of sections of a graded Lie algebroid over $\scrM$, called the \emph{gauge algebroid} (or the \emph{Atiyah algebroid}) of the graded vector bundle $\scrP$.
	Abusing the notation, we will sometimes denote the gauge algebroid by the same symbol $\Diff(\scrP)$ as for its sections.
	For instance, we will speak about, e.g., $\Diff(\scrP)$-connections, representations of $\Diff(\scrP)$, etc., without further comments.
	Notice that, exactly like in the non-graded case, $\scrP$ carries a canonical representation of the gauge algebroid $\Diff(\scrP)$, the \emph{tautological representation}, given by the action of derivations on sections.
\end{remark}

A \emph{degree $k$ graded symmetric first order $n$-ary differential operator} from $\scrP$ to $\scrQ$ is a degree $k$ graded symmetric $\bbR$-multilinear map
\begin{equation}
\square:\underbrace{\Gamma(\scrP)\times\cdots\times\Gamma(\scrP)}_{n-\textnormal{times}}\longrightarrow\Gamma(\scrQ)
\end{equation}
which is a graded first order differential operator from $\scrP$ to $\scrQ$ in each entry.

Let $\scrL\to\scrM$ be a graded line bundle.
From now on, the trivial line bundle over the graded manifold $\scrM$ will be denoted by $\bbR_{\scrM} := \scrM \times \bbR \to\scrM$.

\begin{remark}
	We denote by $\Diff^n(\scrL,\bbR_{\scrM})^\bullet:=\bigoplus_{k\in\bbZ}\Diff^n(\scrL,\bbR_{\scrM})^k$ the space of graded symmetric first order $n$-ary differential operators from $\scrL$ to $\bbR_{\scrM}$.
	The space of graded symmetric multi-differential operators
	\begin{equation*}
	\Diff^\star(\scrL,\bbR_{\scrM})^\bullet:=\bigoplus_{n\in\bbZ}\Diff^n(\scrL,\bbR_{\scrM})^\bullet,
	\end{equation*}
	is a graded left $C^\infty(\scrM)$-module in the obvious way (we set $\Diff^0(\scrL,\bbR_{\scrM}):=C^\infty(\scrM)$, and $\Diff^n(\scrL,\bbR_{\scrM}):=0$, for all $n<0$).
	Here, and in the following, the superscripts ${}^\star$ and ${}^\bullet$ refer to, respectively, the arity and the total degree.
	The $C^\infty(\scrM)$-module $\Diff^\star(\scrL,\bbR_{\scrM})$ is, additionally, a unital associative graded commutative $\bbR$-algebra whose product is given by
	\begin{equation}
	\label{eq:product_of_multi-do}
	(\Delta\cdot\Delta')(\lambda_1,\ldots,\lambda_{k+k'})=\sum_{\tau\in S_{k,k'}}(-)^{\chi}\epsilon(\tau,\boldsymbol\lambda)\Delta(\lambda_{\tau(1)},\ldots,\lambda_{\tau(k)})\cdot\Delta'(\lambda_{\tau(k+1)},\ldots,\lambda_{\tau(k+k')}),
	\end{equation}
	for all homogeneous $\Delta\in\Diff^k(\scrL,\bbR_{\scrM}),\Delta'\in\Diff^{k'}(\scrL,\bbR_{\scrM})$, and $\lambda_1,\ldots,\lambda_{k+k'}\in\Gamma(\scrL)$.
	In~\eqref{eq:product_of_multi-do}, $\chi:=\sum_{i=1}^{k}|\Delta'||\lambda_{\tau(i)}|$, and $\epsilon(\tau,\boldsymbol{\lambda})$ is the graded symmetric Koszul sign prescribed by the $(k,k')$-shuffle of the $\lambda$'s.
	Accordingly, we have a canonical (degree $0$) graded $C^\infty(\scrM)$-algebra isomorphism
	\begin{equation*}
	\Diff^\star(\scrL,\bbR_{\scrM})\simeq S_{{\scriptscriptstyle C^\infty(\scrM)}}\Diff^1(\scrL,\bbR_{\scrM}).
	\end{equation*}
	Here, as in the remaining part of this thesis, $S_{R}M$ denotes the graded symmetric algebra of a module $M$ over a algebra $R$.
\end{remark}

\begin{remark}
	We denote by $\Diff^n(\scrL)^\bullet:=\bigoplus_{k\in\bbZ}\Diff^n(\scrL)^k$ the space of graded symmetric first order $n$-ary differential operators from $\scrL$ to $\scrL$.
	Since $\scrL$ is a line bundle, the operators in $\Diff^n(\scrL)^\bullet$ are, indeed, derivations in each argument.
	Accordingly, we will also call them (graded symmetric) \emph{multi-derivations}.
	In particular, $\Diff^1(\scrL)^\bullet=\Diff(\scrL)^\bullet$.
	The space of graded symmetric multi-derivations
	\begin{equation*}
	\Diff^\star(\scrL)^\bullet:=\bigoplus_{n\in\bbZ}\Diff^n(\scrL)^\bullet,
	\end{equation*}
	is a graded left $C^\infty(\scrM)$-module in the obvious way (we set $\Diff^0(\scrL):=\Gamma(\scrL)$, and $\Diff^n(\scrL):=0$, for all $n<0$).
	The $C^\infty(\scrM)$-module $\Diff^\star(\scrL)$ is, additionally, a graded $\Diff^\ast(\scrL,\bbR_{\scrM})$-module whose product is given by a formula similar to~\eqref{eq:product_of_multi-do}.
	Accordingly, we have a canonical (degree $0$) isomorphism of graded $\Diff^\star(\scrL,\bbR_{\scrM})$-modules
	\begin{equation*}
	\Diff^\star(\scrL)\simeq\Diff^\star(\scrL,\bbR_{\scrM})\otimes_{{\scriptscriptstyle C^\infty(\scrM)}}\Gamma(\scrL).
	\end{equation*}
\end{remark}

The notion of (graded) Gerstenhaber--Jacobi algebra has already been introduced (see Definition~\ref{definition:Gerstenhaber--Jacobi}), however we recall it here for the reader's convenience.
A \emph{Gerstenhaber-Jacobi algebra} consists of a unital associative graded commutative algebra $\calA$ and a graded $\calA$-module $\calL$, equipped with a graded Lie bracket $[-,-]$ on $\calL$ and an action by derivations of $\calL$ on $\calA$, i.e.~a degree $0$ graded Lie algebra morphism $X_{(-)}:\calL\to\operatorname{Der}(\calA)$, such that
	\begin{equation*}
	[\lambda,a\mu]=X_\lambda(a)\mu+(-)^{|a||\lambda|}a[\lambda,\mu],
	\end{equation*}
for all homogeneous $\lambda,\mu\in\calL$ and $a\in\calA$.
For our aims, as seen in the preceding chapters, the most relevant example of such structures is the Gerstenhaber--Jacobi algebra of multi-derivations of a line bundle (cf.~Section~\ref{sec:GJ_algb_multi-derivations}, and in particular Proposition~\ref{prop:Jacobi_Gerstenhaber_multi-differential}).
The next proposition extends this natural construction to the case of a graded line bundle.

\begin{proposition}
	\label{prop:GJalgebra}
	For every graded line bundle $\scrL\to\scrM$, there is a natural Gerstenhaber-Jacobi algebra structure $(\ldsb-,-\rdsb,X_{(-)})$ on $(\Diff^\star(\scrL,\bbR_{\scrM}),\Diff^\star(\scrL))$, uniquely determined by
	\begin{equation}\label{eq:SJbrackets}
	\ldsb \square,\square'\rdsb=[\square,\square'],\qquad
	\ldsb \square,\lambda\rdsb=\square(\lambda),\qquad
	\ldsb \lambda,\mu\rdsb=0,
	\end{equation}
	for all $\square,\square'\in\Diff(\scrL)$, and $\lambda,\mu\in\Gamma(\scrL)$.
	The Lie bracket $\ldsb-,-\rdsb$ is called \emph{the Schouten--Jacobi bracket}.
\end{proposition}

\begin{proof}
	Since $\scrL$ carries the tautological representation of the gauge algebroid $\Diff(\scrL)$ (cf.~Remark~\ref{rem:LRalgebra}), it is a straightforward consequence of Proposition~\ref{prop:Jacobi_algbds_GJ-algbs}.
\end{proof}

\begin{remark}
	\label{rem:Gerstenhaber_product_and_SJ_bracket}
	When computing with multi-derivations it is very helpful to have an explicit expression for the Schouten--Jacobi bracket.
	It is easy to see that
	\begin{equation}
	\label{eq:Gerstenhaber_product_and_SJ_bracket}
	\ldsb\square,\square'\rdsb=\square\bullet \square'-(-)^{|\square||\square'|}\square'\bullet\square,
	\end{equation}
	for all homogeneous $\square,\square'\in\Diff^\star(\scrL)$.
	In~\eqref{eq:Gerstenhaber_product_and_SJ_bracket}, we have denoted by $\bullet$ the \emph{Gerstenhaber product}
	(of multi-derivations).
	The latter is defined by 
	\begin{equation*}
	\square\bullet\square'(\lambda_1,\ldots,\lambda_{k+k'+1})=\!\!\!\!\sum_{\tau\in S_{k'+1,k}}\!\!\!\!\epsilon(\tau,\boldsymbol\lambda)\square(\square'(\lambda_{\tau(1)},\ldots,\lambda_{\tau(k'+1)}),\lambda_{\tau(k'+2)},\ldots,\lambda_{\tau(k+k'+1)}),
	\end{equation*}
	for all homogeneous $\square\in\Diff^{k+1}(\scrL)$, $\square'\in\Diff^{k'+1}(\scrL)$, and $\lambda_1,\ldots,\lambda_{k+k'+1}\in\Gamma(\scrL)$.
	It is also helpful to point out that, as a consequence of~\eqref{eq:Gerstenhaber_product_and_SJ_bracket}, we have
	\begin{equation}
	\label{eq:Gerstenhaber_product_and_SJ_bracket_bis}
	\square(\lambda_1,\ldots,\lambda_n)=\ldsb\ldsb\ldots\ldsb\square,\lambda_1\rdsb,\ldots\rdsb,\lambda_n\rdsb,
	\end{equation}
	for all $\square\in\Diff^n(\scrL)$, and $\lambda_1,\ldots,\lambda_n\in\Gamma(\scrL)$.
\end{remark}

\subsection{Graded Jacobi bundles}

In this section we introduce Jacobi structures on graded line bundles, and their equivalent description in terms of Jacobi bi-derivations.

\begin{definition}
	\label{def:graded_Jacobi_structure}
	A \emph{graded Jacobi structure} on a graded line bundle $\scrL\to\scrM$ is given by a \emph{graded Jacobi bracket} $\{-,-\} :\Gamma(\scrL)\times\Gamma (\scrL)\rightarrow\Gamma(\scrL)$, i.e.~a (degree $0$) graded Lie bracket which is a first order differential operator, hence a derivation, in each entry.
	A \emph{graded Jacobi bundle} (over $\scrM$) is a graded line bundle (over $\scrM$) equipped with a graded Jacobi structure.
	A \emph{graded Jacobi manifold} is a graded manifold equipped with a graded Jacobi bundle over it.
\end{definition}
A Jacobi structure on a graded line bundle $\scrL\to\scrM$ is the same as a Gerstenhaber-Jacobi algebra structure on $(C^\infty(\scrM),\Gamma(\scrL))$.

\begin{remark}
	\label{rem:shift}
	We want to emphasize that, via \emph{d\'ecalage} isomorphism, \emph{skew-symmetric} multi-deriva\-tions of a graded line bundle $\scrL$ are in one-to-one correspondence with \emph{graded symmetric} multi-deriva\-tions of the shifted line bundle $\scrL[1]$.
	According to Definition~\ref{def:graded_Jacobi_structure}, a graded Jacobi bracket is a \emph{graded skew-symmetric} bi-derivation.
	However, it turns out that formulas get much simplier if we understand graded Jacobi structures as \emph{graded symmetric} bi-derivations on a shifted line bundle, via d\'ecalage.
	This is made precise below.
\end{remark}

\begin{definition}
	\label{def:Jacobi_bi-do}
	A \emph{graded Jacobi bi-derivation} on a graded line bundle $\scrL\to\scrM$ is a degree $1$ graded symmetric bi-derivation $J\in\Diff^2(\scrL)^1$ such that $\ldsb J,J\rdsb=0$.
\end{definition}

Let $\scrL\to\scrM$ be a graded line bundle.
The next proposition establishes a one-to-one correspondence identifying, in a canonical way, Jacobi bi-derivations $J$ on $\scrL[1]$ with Jacobi structures $\{-,-\}$ on $\scrL$, thus clarifying the content of Remark~\ref{rem:shift}.

\begin{proposition}
	\label{prop:Jacobi_bi-do}
	There exists a canonical one-to-one correspondence between
	graded Jacobi brackets $\{-,-\}$ on $\scrL$ and graded Jacobi bi-derivations $J$ on $\scrL[1]$ given by the following relation
	\begin{equation*}
	\{s\lambda_1,s\lambda_2\}=(-)^{|\lambda_1|}s(J(\lambda_1,\lambda_2)),
	\end{equation*}
	for all homogeneous $\lambda_1,\lambda_2\in\Gamma(\scrL[1])$,
	where $s:\Gamma(\scrL[1])\to\Gamma(\scrL)$ denotes the suspension map.
\end{proposition}

\begin{proof}
	It follows from the same argument used in the non-graded case, see Proposition~\ref{prop:J_as_MC_element}.
\end{proof}

\begin{remark}
	\label{rem:d_J}
	From now on, we will often denote by $\{-,-\}_J$ the (graded) Jacobi bracket corresponding to a (graded) Jacobi bi-derivation $J$.
	Sometimes we will simply identify $J$ and $\{-,-\}_J$ and write $J \equiv \{-,-\}_J$ (or $J \equiv \{-,-\}$).
	
	Let $(\scrM,\scrL,\J)$ be a graded Jacobi manifold.
	There is a differential graded Lie algebra attached to $\scrM$, namely $(\Diff^\star(\scrL[1]),d_\J,\ldsb-,-\rdsb)$, with $d_\J:=\ldsb\J,-\rdsb$.
	The cohomology of $(\Diff^\star(\scrL[1]),d_\J)$ is called the Chevalley--Eilenberg cohomology of $(\scrM,\scrL,\J)$, and it is denoted by $H_{CE}(\scrM,\scrL,\J)$.
\end{remark}

We now discuss automorphisms of a graded Jacobi manifold $(\scrM,\scrL,\J\equiv \{-,-\})$.
\begin{definition}
	\label{def:Jacobi_automorphism}
	A \emph{Jacobi automorphism} of $(\scrM,\scrL,\J)$ is a degree $0$ graded automorphism $\Phi$ of the graded line bundle $\scrL\to\scrM$ such that $\Phi^\ast\J=\J$, i.e.
	\begin{equation*}
	\Phi^\ast\{\lambda_1,\lambda_2\}_{\J}\equiv\{\Phi^\ast\lambda_1,\Phi^\ast\lambda_2\}_{\J},\ \textnormal{for all}\ \lambda_1,\lambda_2\in\Gamma(\scrL).
	\end{equation*}
\end{definition}
The group of Jacobi automorphisms of $(\scrM,\scrL,\J)$ will be denoted by $\Aut(\scrM,\scrL,\J)$.
The Lie subalgebra $\mathfrak{aut}(\scrM,\scrL,\J)\subset\Diff(\scrL)^0$ of infinitesimal Jacobi automorphisms, or \emph{Jacobi derivations}, of $(\scrM,\scrL,\J)$ consists of those degree $0$ graded derivations $\square\in\Diff(\scrL)^0$ such that $\ldsb\J,\square\rdsb=0$, i.e.
\begin{equation*}
\square\{\lambda_1,\lambda_2\}_{\J}\equiv\{\square\lambda_1,\lambda_2\}_{\J}+\{\lambda_1,\square\lambda_2\}_{\J},\ \textnormal{for all}\ \lambda_1,\lambda_2\in\Gamma(\scrL).
\end{equation*}

\begin{definition}
	\label{def:Hamiltonian_family}
	For a smooth path $\{\lambda_t\}_{t\in I}\subset\Gamma(\scrL)^0$ and a smooth path $\{\Phi_t\}_{t\in I}$ of degree 0 automorphisms of $\scrL \to \scrM$, we say that $\{\lambda_t,-\}_{\J}$ \emph{integrates} to $\{\Phi_t\}_{t\in I}$ if
	\begin{equation*}
	\Phi_0=\id_{\scrL},\quad \frac{d}{dt}\Phi_t^\ast=\{\lambda_t,-\}_{\J}\circ\Phi_t^\ast.
	\end{equation*}
	In this case $\{\Phi_t\}_{t\in I}$ consists of Jacobi automorphisms and it is called the \emph{smooth path of Hamiltonian automorphisms} associated with the smooth path of \emph{Hamiltonian sections} $\{\lambda_t\}_{t\in I}$.
\end{definition}

\begin{definition}
	\label{def:Hamiltonian_single}
	A \emph{Hamiltonian automorphism} of $(\scrM,\scrL,\J)$ is a degree $0$ graded automorphism $\Phi\in\Aut(\scrM,\scrL,\J)$ such that $\Phi=\Phi_1$, for some smooth path of Hamiltonian automorphisms $\{\Phi_t\}_{t\in I}$.
\end{definition}
In the following the group of Hamiltonian automorphisms of $(\scrM,\scrL,\J)$ wil be denoted by $\Ham(\scrM,\scrL,\J)$.
The Lie subalgebra $\mathfrak{ham}(\scrM,\scrL,\J)\subset\mathfrak{aut}(\scrM,\scrL,\J)$ of infinitesimal Hamiltonian automorphisms, or \emph{Hamiltonian derivations}, of $(\scrM,\scrL,\J)$ consists of those degree $0$ graded derivations $\square\in\Diff(\scrL)^0$ of the form $\ldsb\J,\lambda\rdsb=\{\lambda,-\}_{\J}$, for some $\lambda\in\Gamma(\scrL)^0$.

\subsection{A splitting result for \texorpdfstring{$\Diff^\star(\scrL)$}{D*(scrL)}}
\label{subsec:psi_nabla}

Let $\scrM$ be a graded manifold with support $M$, and let $\scrL\to\scrM$ be a graded line bundle.
We assume the existence of a (non-necessarily canonical) graded module isomorphism 
\begin{equation}
\label{eq:Batchelor}
\Gamma(\scrL)\simeq S_{{\scriptscriptstyle C^\infty(M)}}\Gamma(F^\ast)\otimes_{{\scriptscriptstyle C^\infty(M)}}\Gamma(P), 
\end{equation}
covering a graded algebra isomorphism $C^\infty(\scrM)\simeq S_{{\scriptscriptstyle C^\infty(M)}}\Gamma(F^\ast)$, where $F\to M$ is a graded vector bundle (without component of degree $0$), and $P\to M$ is a graded line bundle.
Additionally, a $\Diff(P)$-connection in $F\to M$ determines a graded module isomorphism
\begin{equation}
	\label{eq:iso_Batch}
	\Diff^\star(\scrL)\simeq S_{{\scriptscriptstyle C^\infty(M)}}\Gamma(F^\ast\oplus F_P\oplus(J^1P)^\ast )\otimes_{{\scriptscriptstyle C^\infty(M)}}\Gamma(P),
\end{equation}
covering a graded algebra isomorphism $\Diff^\star(\scrL,\bbR_{\scrM})\simeq S_{{\scriptscriptstyle C^\infty(M)}}\Gamma(F^\ast\oplus F_P\oplus(J^1P)^\ast)$, where $F_P:=F\otimes P^\ast$.
The goal of this section is to describe isomorphism~\eqref{eq:iso_Batch} explicitly.

\begin{definition}
	\label{def:vertical_graded_derivation}
	A vector field $X\in\frakX(\scrM)$ is said to be \emph{vertical} if it is in the kernel of the canonical fibration $\scrM\to M$, i.e.~$X(f)=0$, for all $f\in C^\infty(M)\subset C^\infty(\scrM)$.
\end{definition}

Denote by $\vder(\scrM)$ the set of vertical vector fields on $\scrM$.
It is both a graded $C^\infty(\scrM)$-submodule and a graded Lie subalgebra of $\frakX(\scrM)$.

\begin{remark}
	\label{rem:A-derivations}	\

	1) By restricting vertical vector fields to $\Gamma(F^\ast)\subset C^\infty(\scrM)$, we get a degree $0$ graded $C^\infty(M)$-module isomorphism
	\begin{equation}
	\label{eq:A-derivations1}
	\vder(\scrM)\overset{\simeq}{\longrightarrow} C^\infty(\scrM)\underset{{\scriptscriptstyle C^\infty(M)}}{\otimes}\Gamma(F).
	\end{equation}
	
	2) There exists a short exact sequence of degree $0$ graded $C^\infty(\scrM)$-module morphisms
	\begin{equation}
	\label{eq:seq1}
	0 \longrightarrow\VDer(\scrM) \longrightarrow\Diff(\scrL)\longrightarrow C^\infty(\scrM)\underset{{\scriptscriptstyle C^\infty(M)}}{\otimes}\Diff(P) \longrightarrow 0.
	\end{equation}
	The arrow $\VDer(\scrM)\longrightarrow\Diff(\scrL)$, written $Z\longmapsto\bbD_Z$, is defined by setting $\bbD_Z(a\otimes\lambda)=Z(a)\otimes\lambda$, for all $Z\in\VDer(\scrM)$, $a\in C^\infty(\scrM)$, and $\lambda\in\Gamma(P)$.
	In its turn, the arrow $\Diff(\scrL)\longrightarrow C^\infty(\scrM)\otimes_{{\scriptscriptstyle C^\infty(M)}}\Diff(P)$ is obtained by restricting derivations to $\Gamma(P)\subset\Gamma(\scrL)$.
	
	3) A $\Diff(P)$-connection $\nabla$ in $F\to M$ determines a $C^\infty(M)$-linear map $\bar{\nabla}:\Diff(P)\to\Diff(\scrL)$, $\square\mapsto\bar{\nabla}_{\square}$ via
	\begin{equation*}
	\bar{\nabla}_\square(\alpha\otimes p)=(\nabla^\ast_\square\alpha)\otimes p+\alpha\otimes(\square p),
	\end{equation*}
	for all $p\in\Gamma(P)$, and $\alpha\in\Gamma(F^\ast)$, where $\nabla^\ast$ is the $D(P)$-connection in $F^\ast$ dual to $\nabla$.
\end{remark}

The above remark leads immediately to the next proposition whose proof is straightforward.

\begin{proposition}
	\label{prop:psi_nabla}
	The $C^\infty(M)$-linear map $\bar{\nabla}:\Diff(P)\to\Diff(\scrL)$ extends, by $C^\infty(\scrM)$-linearity, to a splitting of the short exact sequence~\eqref{eq:seq1}, so it determines a degree 0 graded $C^\infty(\scrM)$-module isomorphism
	\begin{equation}
	\label{eq:prop:psi_nabla}
	\phi_\nabla : \Diff(\scrL)\overset{\simeq}{\longrightarrow}\left[C^\infty(\scrM)\underset{{\scriptscriptstyle C^\infty(M)}}{\otimes}\Gamma\left(F_P\oplus (J^1P)^\ast\right)\right]\underset{{\scriptscriptstyle C^\infty(\scrM)}}{\otimes}\Gamma(\scrL).
	\end{equation}
	Moreover, there is a unique degree $0$ graded $C^\infty(\scrM)$-module isomorphism 
	\begin{equation*}
	\psi_\nabla : \Diff^\star(\scrL)\overset{\simeq}{\longrightarrow} S_{{\scriptscriptstyle C^\infty(M)}}(\Gamma(F^\ast\oplus F_P))\otimes_{{\scriptscriptstyle C^\infty(M)}}\Diff^\star(P),
	\end{equation*}
	covering a graded $C^\infty(\scrM)$-algebra isomorphism $\underline{\smash{\psi_\nabla}}:\Diff^\star(\scrL,\bbR_\scrM)\to S_{{\scriptscriptstyle C^\infty(M)}}\Gamma(F^\ast\oplus F_P\oplus (J^1P)^\ast)$,
	such that:
	\begin{enumerate}
	\item $\psi_\nabla$ agrees with the identity map on $\Gamma(\scrL)$, and
	\item $\psi_\nabla$ agrees with $\phi_\nabla$ on $\Diff(\scrL)$.
	\end{enumerate}
\end{proposition}

The preceding proposition plays a key r\^ole in lifting a Jacobi structure (first step in the construction of the BFV-complex of a coisotropic submanifold) and consequently in understanding how the moduli spaces of a coisotropic submanifold are encoded by its BFV-complex.
Namely construction of the set of contraction data~\eqref{eq:1contraction_data_1} and proof of Proposition~\ref{prop:main_result} are crucially based on Proposition~\ref{prop:psi_nabla}.

\section{Lifted graded Jacobi structures}
\label{sec:lifting_Jacobi_structures}

In this section we will describe a procedure to lift a given Jacobi structure $\J$ on an ordinary line bundle $L\to M$ to a graded Jacobi structure $\hat{\J}$ on a certain graded line bundle $\hat{L}\to\hat{M}$ (Theorems~\ref{theor:existence_BFV_brackets} and~\ref{theor:uniqueness_BFV_brackets}).
The interest in this procedure is at least two-fold.
First, it extends, exploiting simplified techniques, similar lifting procedures in~\cite{Rothstein1991} (symplectic case) and~\cite{herbig2007,schatz2009bfv} (Poisson case).
Second, as we will show in Sections~\ref{sec:BRST-charges} and~\ref{sec:BFV-complex}, the lifting of Jacobi structures represents the first step towards the construction of the BFV-complex of a coisotropic submanifold of a Jacobi manifold (see also~\cite{schatz2009bfv}).

A lifting procedure is based on a set of contraction data that Proposition~\ref{prop:1contraction_data_1} associates with every $\Diff(L)$-connection.
Indeed the proof of Theorem~\ref{theor:existence_BFV_brackets} will outline how to use these contraction data to inductively construct $\hat{\J}$ by implementing the ``step-by-step obstruction'' method of homological perturbation theory.
The same method can be used to construct a BRST charge (cf.~Theorems~\ref{theor:existence_BRST-charge} and~\ref{theor:uniqueness_BRST-charge}), and goes back to Stasheff~\cite{Stasheff1997}.
See Appendix~\ref{app:SBSO} for a version of this technique well-suited for our aims.

Notice that there is an alternative approach to the lifting.
Namely, the scheme adopted by Sch\"atz~\cite{schatz2009bfv} extends from the Poisson to the Jacobi setting.
Sch\"atz scheme involves two main ingredients:
\begin{itemize}
	\item homotopy transfer along the contraction data~\eqref{eq:1contraction_data_1} to lift the quasi-isomorphism of co-chain complexes $i_\nabla$ to a $L_\infty$-quasi-isomorphism of $L_\infty[1]$-algebras $\hat{i_\nabla}$ (see Section~\ref{subsec:first_relevant_contraction_data} for more details, including a definition, about $i_\nabla$), 
	\item transfer of formal Maurer--Cartan (MC) elements via $L_\infty$-quasi-isomorphisms to transform a Jacobi bi-derivation $\J$ on $L[1]$ into a new Jacobi bi-derivation $\hat{\J}$ on $\hat{L}[1]$.
\end{itemize}
However, we have preferred the first approach to the second one, because, in our opinion, it is simpler and does not involve unnecessarily sophisticated tools.

\subsection{The initial setting}
\label{subsec:initial_setting}

Let $E\to M$ be a vector bundle, and let $L\to M$ be a line bundle.
Define the vector bundle $E_L\to M$ by setting $E_L:=E\otimes L^\ast$.

Denote by $\hat{M}$ the graded manifold, with support $M$, represented by the graded vector bundle $\pi:E_L{}^\ast[1]\oplus E[-1]\longrightarrow M$, and by $\hat{L}$ the graded line bundle over $\hat{M}$ given by $\hat{L}:=\pi^\ast L\longrightarrow\hat{M}$.
This means that $C^\infty(\hat{M})$ and $\Gamma(\hat{L})$ are given by
\begin{gather*}
C^\infty(\hat{M})=S_{{\scriptscriptstyle C^\infty(M)}}\Gamma(E_L[-1]\oplus E^\ast[1]),\qquad
\Gamma(\hat{L})=C^\infty(\hat{M})\underset{{\scriptscriptstyle C^\infty(M)}}{\otimes}\Gamma(L).
\end{gather*}

\begin{remark}
	The algebra $C^\infty(\hat{M})=\bigoplus_{n\in\bbZ}C^\infty(\hat{M})^n$ is graded commutative wrt the $\bbZ$-grading  provided by the \emph{total ghost number} $n$, where
	\begin{equation*}
	C^\infty(\hat{M})^n:=\bigoplus_{\genfrac{}{}{0pt}{}{(h,k)\in\bbN_0 ^2}{h-k=n}}C^\infty(\hat{M})^{(h,k)},\qquad\textnormal{with}\quad C^\infty(\hat{M})^{(h,k)}:=\Gamma((\wedge^h E_L)\otimes(\wedge^k E^\ast)).
	\end{equation*}
	Hence its multiplication is also compatible with the finer $\bbN_0 ^2$-grading provided by the \emph{ghost/anti-ghost bi-degree} $(h,k)$.
	Similarly, the $C^\infty(\hat{M})$-module structure on $\Gamma(\hat{L})=\bigoplus_{n\in\bbZ}\Gamma(\hat{L})^n$ is graded wrt the $\bbZ$-grading provided by the \emph{total ghost number} $n$:
	\begin{equation*}
	\Gamma(\hat{L})^n:=\bigoplus_{\genfrac{}{}{0pt}{}{(h,k)\in\bbN_0 ^2}{h-k=n}}\Gamma(\hat{L})^{(h,k)},\qquad\textnormal{with}\quad \Gamma(\hat{L})^{(h,k)}:=\Gamma((\wedge^h E_L)\otimes(\wedge^k E^\ast)\otimes L).
	\end{equation*}
	Hence its $C^\infty(\hat{M})$-module structure is also compatible with the finer $\bbN_0 ^2$-grading provided by the \emph{ghost/anti-ghost bi-degree} $(h,k)$.
\end{remark}

\begin{remark}
	\label{rem:local_frames1}
	Let $x^i$ be an arbitrary local coordinate system on $M$.
	Fix a local frame $\xi^A$ on $E\to M$, and denote by $\xi^\ast_A$ the dual local frame on $E^\ast\to M$.
	Fix also a local frame $\mu$ on $L\to M$, and denote by $\mu^\ast$ the dual local frame on $L^\ast\to M$.
	Then the $C^\infty(M)$-algebra $C^\infty(\hat{M})$ is locally generated by
	\begin{equation*}
	\underbrace{\xi^A\otimes\mu^\ast}_{(1,0)}\qquad \underbrace{\xi^\ast_B}_{(0,1)},
	\end{equation*}
	where the subscripts denote the bi-degrees.
	Moreover, for all $(h,k)\in\bbN_0 ^2$, a local set of generators of the $C^\infty(M)$-module $C^\infty(\hat{M})^{(h,k)}$ is given by
	\begin{equation*}
	(\xi^{A_1}\otimes\mu^\ast)\cdots(\xi^{A_h}\otimes\mu^\ast)\xi^\ast_{B_1}\cdots\xi^\ast_{B_k},
	\end{equation*}
	and a local set of generators of the $C^\infty(M)$-module $\Gamma(\hat{L})^{(h,k)}$ is given by
	\begin{equation*}
	(\xi^{A_1}\otimes\mu^\ast)\cdots(\xi^{A_h}\otimes\mu^\ast)\xi^\ast_{B_1}\cdots\xi^\ast_{B_k}\otimes\mu.
	\end{equation*}
	Here, and in what follows, the symmetric product is denoted by juxtaposition.
	Hence an arbitrary $f\in C^\infty(\hat{M})$ and an arbitrary $\lambda\in\Gamma(\hat{L})$ admit the following local expression
	\begin{equation*}
	f=f^{\bfC}_{\bfB}(\xi^{\bfB}\otimes\mu^\ast)\xi_{\bfC}^\ast,\qquad\lambda=f^{\bfC}_{\bfB}(\xi^{\bfB}\otimes\mu^\ast)\xi_{\bfC}^\ast\otimes\mu.
	\end{equation*}
	Where, for any ordered $n$-tuple $\bfB=(B_1,\ldots,B_n)$, we understand the following abbreviations: $\xi^{\bfB}\otimes\mu^\ast$ for $(\xi^{B_1}\otimes\mu^\ast)\cdots(\xi^{B_n}\otimes\mu^\ast)$, and $\xi_{\bfB}^\ast$ for $\xi_{B_1}^\ast\cdots\xi_{B_n}^\ast$.
	The same notation will be adopted below without further comments.
\end{remark}

\begin{remark}
	\label{rem:compatible_gradings}
	A graded symmetric $n$-ary derivation $\Delta\in\Diff^n(\hat{L}[1],\bbR_{\hat{M}})$ is said to have (ghost/anti-ghost) bi-degree $(h,k)\in\bbZ^2$ if
	\begin{equation*}
	\Delta(\Gamma(\hat{L})^{(p_1,q_1)}\times\cdots\times\Gamma(\hat{L})^{(p_n,q_n)})\subseteq C^\infty(\hat{M})^{(h+\sum_ip_i,k+\sum_iq_i)}.
	\end{equation*}
	Denote by $\Diff^n(\hat{L}[1],\bbR_{\hat{M}})^{(h,k)}\subseteq\Diff^n(\hat{L}[1],\bbR_{\hat{M}})$ the $C^\infty(M)$-submodule of graded symmetric $n$-ary derivations of bi-degree $(h,k)$.
	The associative algebra $\Diff^n(\hat{L}[1],\bbR_{\hat{M}})$ is graded commutative wrt the $\bbZ$-grading provided by the \emph{total ghost number} $K$ provided by
	\begin{equation*}
	\Diff^\star(\hat{L}[1],\bbR_{\hat{M}}):=\bigoplus_{K\in\bbZ}\Diff^\star(\hat{L}[1],\bbR_{\hat{M}})^K,\quad\textnormal{with}\ \Diff^\star(\hat{L}[1],\bbR_{\hat{M}})^K=\!\!\!\bigoplus_{K=n+h-k}\!\!\!\Diff^n(\hat{L}[1],\bbR_{\hat{M}})^{(h,k)}.
	\end{equation*}
	Moreover, the product is compatible with both the arity, and the ghost/anti-ghost bi-degree $(h,k)$, i.e.
	\begin{equation*}
		\Diff^{n_1}(\hat{L}[1],\bbR_{\hat{M}})^{(h_1,k_1)}\cdot\Diff^{n_2}(\hat{L}[1],\bbR_{\hat{M}})^{(h_2,k_2)}\subseteq\Diff^{n_1+n_2}(\hat{L}[1],\bbR_{\hat{M}})^{(h_1+h_2,k_1+k_2)}.
	\end{equation*}
\end{remark}

\begin{remark}
	\label{rem:local_generating_system_for_do_from_calL_to_calA}
	A local set of generators of the $C^\infty(\hat{M})$-algebra $\Diff^\star(\hat{L}[1],\bbR_{\hat{M}})$ is provided by
	\begin{equation}
	\label{eq:gen_2}
	\underbrace{\mu^\ast}_{(0,0)},\quad\underbrace{\slashed\Delta_i}_{(0,0)},\quad\underbrace{\slashed\Delta_A}_{(-1,0)},\quad\underbrace{\slashed\Delta^A}_{(0,-1)},
	\end{equation}
	with $\mu^\ast$, $\slashed\Delta_i$, $\slashed\Delta_A$, and $\slashed\Delta^A$ defined by:
	\begin{align*}
	\mu^\ast(f\mu)&=f,&\mu^\ast(\xi^B)={}&\xi^B\otimes\mu^\ast,&\mu^\ast(\xi_B^\ast\otimes\mu)={}&\xi_B^\ast,\\
	\slashed\Delta_i(f\mu)&=\frac{\partial f}{\partial x^i},&\slashed\Delta_i(\xi^B)={}&0,&\slashed\Delta_i(\xi_B^\ast\otimes\mu)={}&0,\\
	\slashed\Delta_A(f\mu)&=0,&\slashed\Delta_A(\xi^B)={}&\delta_A^B,&\slashed\Delta_A(\xi_B^\ast\otimes\mu)={}&0,\\
	\slashed\Delta^A(f\mu)&=0,&\slashed\Delta^A(\xi^B)={}&0,&\slashed\Delta^A(\xi_B^\ast\otimes\mu)={}&\delta^A_B.
	\end{align*}
	Hence an arbitrary $\slashed\Delta\in\Diff^1(\hat{L}[1],\bbR_{\hat{M}})$ will be locally given by:
	\begin{equation*}
	\slashed\Delta=f^{\bfC}_{\bfB}(\xi^{\bfB}\!\otimes\!\mu^\ast)\xi_{\bfC}^\ast\mu^\ast+f^{i\bfC}_{\phantom{i}\bfB}(\xi^{\bfB}\!\otimes\!\mu^\ast)\xi_{\bfC}^\ast\slashed\Delta_i+f^{A\bfC}_{\phantom{A}\bfB}(\xi^{\bfB}\!\otimes\!\mu^\ast)\xi_{\bfC}^\ast\slashed\Delta_A+f^{\phantom{A}\bfC}_{A\bfB}(\xi^{\bfB}\!\otimes\!\mu^\ast)\xi_{\bfC}^\ast\slashed\Delta^A.
	\end{equation*}
\end{remark}

\begin{remark}
	A graded symmetric $n$-ary derivations $\square\in\Diff^n(\hat{L}[1])$ is said to have (ghost/anti-ghost) bi-degree $(h,k)$ if
	\begin{equation*}
	\square(\Gamma(\hat{L})^{(p_1,q_1)}\times\cdots\times\Gamma(\hat{L})^{(p_n,q_n)})\subseteq\Gamma(\hat{L})^{(h+\sum_ip_i,k+\sum_iq_i)}.
	\end{equation*}
	Denote by $\Diff^n(\hat{L}[1])^{(h,k)}\subseteq\Diff^n(\hat{L}[1])$ the $C^\infty(M)$-submodule of graded symmetric $n$-ary derivations of bi-degree $(h,k)$.
	Similarly as in Remark~\ref{rem:compatible_gradings}, all the algebraic structures on $\Diff^\star (\hat{L}[1])$ are compatible with both the arity, the ghost/anti-ghost bi-degree $(h,k)$, and the total ghost number $K$
	\begin{equation*}
	\Diff^\star(\hat{L}[1])=\bigoplus_{K\in\bbZ}\Diff^\star(\hat{L}[1])^K,\quad\textnormal{with}\ \Diff^\star(\hat{L}[1])^K=\!\!\!\bigoplus_{K=n-1+h-k}\!\!\!\Diff^n(\hat{L}[1])^{(h,k)}.
	\end{equation*}
	In the following, for every $(h,k)\in\bbZ^2$, we denote by $\operatorname{pr}^{(h,k)}:\Diff^\star(\hat{L}[1]) \to\Diff^\star(\hat{L}[1])^{(h,k)}$ the projection onto the homogeneous component of bi-degree $(h,k)$.
\end{remark}

\begin{remark}
	\label{rem:local_generating_system_for_do_from_calL_to_calL}
	A local set of generators for the $C^\infty(\hat{M})$-module $\Diff(\hat{L}[1])$ is provided by
	\begin{equation*}
	\underbrace{\id}_{(0,0)},\quad\underbrace{\Delta_i}_{(0,0)},\quad\underbrace{\Delta_A}_{(-1,0)},\quad\underbrace{\Delta^A}_{(0,-1)},
	\end{equation*}
	with $\id$, $\Delta_i$, $\Delta_A$, and $\Delta^A$ defined by:
	\begin{align*}
	\id(f\mu)&=f\mu,&\id(\xi^B)={}&\xi^B,&\id(\xi_B^\ast\otimes\mu)={}&\xi_B^\ast\otimes\mu,\\
	\Delta_i(f\mu)&=\frac{\partial f}{\partial x^i}\mu,&\Delta_i(\xi^B)={}&0,&\Delta_i(\xi_B^\ast\otimes\mu)={}&0,\\
	\Delta_A(f\mu)&=0,&\Delta_A(\xi^B)={}&\delta_A^B\mu,&\Delta_A(\xi_B^\ast\otimes\mu)={}&0,\\
	\Delta^A(f\mu)&=0,&\Delta^A(\xi^B)={}&0,&\Delta^A(\xi_B^\ast\otimes\mu)={}&\delta^A_B\mu.
	\end{align*}
	Hence an arbitrary $\square\in\Diff(\hat{L}[1])$ will be locally given by:
	\begin{equation*}
	\square=f^{\bfC\phantom{i}}_{\phantom{\bfA}\bfB}(\xi^{\bfB}\!\otimes\!\mu^\ast)\xi_{\bfC}^\ast\id+f^{i\bfC}_{\phantom{A}\bfB}(\xi^{\bfB}\!\otimes\!\mu^\ast)\xi_{\bfC}^\ast\Delta_i+f^{A\bfC}_{\phantom{A}\bfB}(\xi^{\bfB}\!\otimes\!\mu^\ast)\xi_{\bfC}^\ast\Delta_A+f^{\phantom{i}\bfC}_{A\bfB}(\xi^{\bfB}\!\otimes\!\mu^\ast)\xi_{\bfC}^\ast\Delta^A.
	\end{equation*}
\end{remark}

In the following, \emph{with no risk of confusion}, $\mu^\ast,\slashed\Delta_i$  will also denote the local set of generators of the $C^\infty(M)$-algebra $\Diff^\star(L[1],\bbR_M)$ defined by
\begin{equation*}
\mu^\ast(f\mu)=f,\quad\slashed\Delta_i(f\mu)=\frac{\partial f}{\partial x^i},
\end{equation*}
similarly $\id,\Delta_i$ will also denote the local set of generators of the $C^\infty(M)$-module $\Diff(L[1])$ defined by
\begin{equation*}
\id(f\mu)=f\mu,\quad\Delta_i(f\mu)=\frac{\partial f}{\partial x^i}\mu.
\end{equation*}

\subsection{Existence and uniqueness of the lifting: the statements}
\label{subsec:lifted_graded_Jacobi}

The graded line bundle $\hat{L}\to\hat{M}$ admits a tautological graded Jacobi structure, given by the canonical bi-degree $(-1,-1)$ Jacobi bi-derivation $G$, with corresponding Jacobi brackets $\{-,-\}_G$.
The latter can be seen as the natural extension, from the Poisson setting to the Jacobi one, of the so-called big bracket (cf.~\cite{kosmann1996poisson}).

\begin{definition}
	\label{def:G}
	The \emph{tautological Jacobi bi-derivation} on $\hat{L}[1]$ is the unique $G\in\Diff^2(\hat{L}[1])^{(-1,-1)}$ such that
	\begin{equation}
	\label{eq:def:G}
	G(u,\alpha)=G(\alpha,u)=\alpha(u),
	\end{equation}
	for all $u\in\Gamma(\hat{L}[1])^{(1,0)}=\Gamma(E)$, and $\alpha\in\Gamma(\hat{L}[1])^{(0,1)}=\Gamma(E^\ast\otimes L)$.
\end{definition}

\begin{remark}
	\label{rem:d_G}
	Set $d_G:=\ldsb G,-\rdsb$.
	Then $d_G$ is a cohomological derivation of $\Diff^\star(\hat{L}[1])$, with symbol $X_G$, and it is a derivation wrt $\ldsb-,-\rdsb$ as well.
	Moreover its bi-degree is $(-1,-1)$, i.e.
	\begin{equation*}
	d_G(\Diff^n(\hat{L}[1])^{(h,k)})\subseteq\Diff^{n+1}(\hat{L}[1])^{(h-1,k-1)}.
	\end{equation*}
	In particular, each section $\lambda\in\Gamma(\hat{L}[1])^{(0,0)}=\Gamma(L)[1]$ is a co-cycle wrt $d_G$, i.e.~$d_G\lambda=0$.
\end{remark}

\begin{remark}
	\label{rem:local_expression_G}
	Clearly, $G$ is locally given by $G=\slashed\Delta_A\slashed\Delta^A\otimes\mu$.
	Moreover, according to Remark~\ref{rem:Gerstenhaber_product_and_SJ_bracket}, $d_G$ is completely determined by:
	\begin{gather*}
	d_G(f\mu)=0,\quad d_G(\xi^A)=\Delta^A,\quad d_G(\xi_A^\ast\otimes\mu)=\Delta_A,\\
	d_G(\id)=G,\quad d_G(\Delta_i)=d_G(\Delta_A)=d_G(\Delta^A)=0.
	\end{gather*}
\end{remark}

\begin{definition}
	\label{def:BFV_brackets}
	A Jacobi bi-derivation $\hat{\J}$ on $\hat{L}[1]$ is said to be a \emph{lifting} of a Jacobi bi-derivation $J$ on $L[1]$ if
	\begin{enumerate}
		\item ${\operatorname{pr}^{(0,0)}}\circ\{-,-\}_{\hat{\J}}$ agrees with $\{-,-\}_G$ on $\Gamma(\hat{L})^{(1,0)}\oplus\Gamma(\hat{L})^{(0,1)}$,
		\item ${\operatorname{pr}^{(0,0)}}\circ\{-,-\}_{\hat{\J}}$ agrees with $\{-,-\}_{\J}$ on $\Gamma(\hat{L})^{(0,0)}$.
	\end{enumerate}
\end{definition}


\begin{theorem}[Existence]
	\label{theor:existence_BFV_brackets}
	Every Jacobi structure $\J$ on the line bundle $L\to M$ admits a lifting to a graded Jacobi structure $\hat{\J}$ on the graded line bundle $\hat{L}\to\hat{M}$.
	Specifically, for every fixed Jacobi structure $\J$ on $L\to M$, there exists a canonical map
	\begin{equation*}
	\{\textnormal{$\Diff(L)$-connections in $E\to M$}\}\longrightarrow\{\textnormal{liftings of $\J$ to $\hat{L}\to\hat{M}$}\},\quad \nabla\longmapsto\hat{\J}^\nabla.
	\end{equation*}
\end{theorem}

\begin{proposition}
	\label{prop:existence_BFV_brackets}
	For every Jacobi structure $\J$ on the line bundle $L\to M$, if the $DL$-connection $\nabla$ in $E\to M$ is flat, then $\hat{\J}^\nabla=G+i_\nabla\J$.
\end{proposition}

\begin{theorem}[Uniqueness]
	\label{theor:uniqueness_BFV_brackets}
	Fix an arbitrary Jacobi structure $\J$ on $L\to M$.
	Let $\hat{\J}$ and $\hat{\J}'$ be Jacobi structures on $\hat{L}\to\hat{M}$.
	If $\hat{\J}$ and $\hat{\J}'$ are both liftings of $\J$, then there exists a degree $0$ graded automorphism $\phi$ of the graded line bundle $\hat{L}\to\hat{M}$ such that
	\begin{equation*}
	\hat{\J}=\phi^\ast\hat{\J}'.
	\end{equation*}
	Moreover such $\phi$ can be chosen so to have the additional property that
	\begin{equation}
	\label{eq:theor:uniqueness_BFV_brackets}
	\phi(\Omega)-\Omega\in\bigoplus_{k\geq 1}\Gamma(\hat{L})^{(p+k,q+k)},\qquad\Omega\in\Gamma(\hat{L})^{(p,q)},\ (p,q)\in\bbN_0 ^2.
	\end{equation}
\end{theorem}

The following proposition, describing an interesting property of liftings of Jacobi structures, generalizes a similar result obtained in the Poisson setting by Herbig (cf.~\cite[Theorem 3.4.5]{herbig2007}).
\begin{proposition}
	\label{prop:lifting_CE_cohom}
	If a Jacobi structure $\hat{\J}$ on the graded line bundle $\hat{L}\to\hat{M}$ is a lifting of a Jacobi structure $\J$ on $L\to M$, then the Chevalley--Eilenberg cohomologies of $\J$ and $\hat{\J}$ are isomorphic:
	\begin{equation*}
	H_{CE}^\bullet(M,L,\J)\simeq H_{CE}^\bullet(\hat{M},\hat{L},\hat{\J}).
	\end{equation*}
	More precisely, every $\Diff(L)$-connection in $E\to M$ determines a quasi-isomorphism from $(\Diff^\star(L[1]),d_{\J})$ to $(\Diff^\star(\hat{L}[1]),d_{\hat{\J}})$.
\end{proposition}

In order to develop the necessary technical tools first, we postpone the proofs of Theorems~\ref{theor:existence_BFV_brackets} and~\ref{theor:uniqueness_BFV_brackets}, Proposition~\ref{prop:existence_BFV_brackets}, and Proposition~\ref{prop:lifting_CE_cohom} to the end of this section.

\subsection{A first relevant set of contraction data}
\label{subsec:first_relevant_contraction_data}

In this subsection we show that every (non-necessarily flat) $DL$-connection $\nabla$ in $E$ determines a set of contraction data from $\Diff^\star(\hat{L}[1])$ to $\Diff^\star(L[1])$ (Proposition~\ref{prop:1contraction_data_1}).
For the reader's convenience, we recall here that a set of \emph{contraction data} (between co-chain complexes) from $(\calK, \partial)$ to $(\underline{\smash{\calK}}, \underline{\smash{\partial}})$ consists of:
\begin{itemize}
	\item a surjective co-chain map $q:(\calK,\partial)\longrightarrow(\underline{\smash{\calK}},\underline{\smash{\partial}})$ that we simply call the \emph{projection},
	\item an injective co-chain map $j:(\underline{\smash{\calK}},\underline{\smash{\partial}})\longrightarrow(\calK,\partial)$, that we call the \emph{immersion}, such that $q\circ j=\id_{\underline{\smash{\calK}}}$,
	\item a \emph{homotopy} $h : (\calK, \partial) \to (\calK, \partial)$ between $j\circ q$ and $\id_{\calK}$.
\end{itemize}
Additionally, $q,j,h$ satisfy the following \emph{side conditions}:
\begin{equation*}
h^2=0,\qquad h\circ j=0,\qquad q\circ h=0.
\end{equation*}

In particular the set of contraction data from $\Diff^\star(\hat{L}[1])$ to $\Diff^\star(L[1])$ we are going to built represents the technical tool on which the proof of existence and uniqueness of the lifting is based (cf.~the next Subsection~\ref{subsec:existence_uniqueness_lifted_graded_Jacobi}).

\subsubsection*{The projection}
\label{subsubsec:p}
There is a degree $0$ graded module morphism $p:\Diff^\star(\hat{L}[1])\to\Diff^\star(L[1])$, covering a degree $0$ graded $\bbR$-algebra morphism $\underline{\smash{p}}:\Diff^\star(\hat{L}[1],\bbR_{\hat{M}})\to\Diff^\star(L[1],\bbR_M)$, given by
\begin{equation}
\label{eq:p}
(p\square)(\lambda_1,\ldots,\lambda_k)=\operatorname{pr}^{(0,0)}(\square(\lambda_1,\ldots,\lambda_k)),
\end{equation}
for all $\square\in\Diff^k(\hat{L}[1])$, and $\lambda_1,\ldots,\lambda_k\in\Gamma(L[1])$.
The following proposition lists some properties of $p$.

\begin{proposition}
	\label{prop:p} \ 
	
	(1) $p$ preserves the arity, i.e.~$p(\Diff^k(\hat{L}[1]))\subseteq\Diff^k(L[1])$.
	
	(2) $p$ annihilates $\Diff^\star(\hat{L}[1])^{(h,k)}$, for all $(h,k)\in\bbZ^2\setminus\{(0,0)\}$, and induces a degree $0$ graded Lie algebra morphism from $\Diff^\star(\hat{L}[1])^{(0,0)}$ onto $\Diff^\star(L[1])$, i.e.
	\begin{equation*}
	p(\ldsb\square_1,\square_2\rdsb)=\ldsb p\square_1,p\square_2\rdsb,\qquad\square_1,\square_2\in\Diff^\star(\hat{L}[1])^{(0,0)}.
	\end{equation*} 
	
	(3) $p$ is a co-chain map from $(\Diff^\star(\hat{L}[1]),d_G)$ to $(\Diff^\star(L[1]),0)$, i.e.~$p\circ d_G=0$.
\end{proposition}

\begin{proof}
	Properties (1) and (2) are immediate consequences of~\eqref{eq:p} and Remark~\ref{rem:Gerstenhaber_product_and_SJ_bracket}.
	Moreover, from Remarks~\ref{rem:d_G} and~\ref{rem:Gerstenhaber_product_and_SJ_bracket}, it follows that
	\begin{align*}
	(d_G\square)(\lambda_1,\ldots,\lambda_{k+1})&=\ldsb \ldsb \ldots\ldsb d_G\square,\lambda_1\rdsb ,\ldots\rdsb ,\lambda_{k+1}\rdsb \\
	&=d_G\underbrace{\ldsb \ldsb \ldots\ldsb \square,\lambda_1\rdsb ,\ldots\rdsb ,\lambda_{k+1}\rdsb }_{=0}\\
	&\phantom{=}+\sum_{i=1}^{k+1}(-)^{|\square|-i}\ldsb \ldsb \ldots\ldsb \ldsb \ldots\ldsb \square,\lambda_1\rdsb ,\ldots\rdsb ,\underbrace{d_G\lambda_i}_{=0}\rdsb ,\ldots\rdsb ,\lambda_{k+1}\rdsb =0,
	\end{align*}
	for all homogeneous $\square\in\Diff^k(\hat{L}[1])$, and $\lambda_1,\ldots,\lambda_{k+1}\in\Gamma(L)[1]$.
	This concludes the proof.
\end{proof}

\begin{remark}
	\label{rem:local_expression_p}
	The module morphism $p:\Diff^\star(\hat{L}[1])\to\Diff^\star(L[1])$ is locally given by
	\begin{gather*}
	p(f\mu)=f\mu,\quad p(\xi^A)=0,\quad p(\xi_A\otimes\mu)=0,\\
	p(\id)=\id,\quad p(\Delta_i)=\Delta_i,\quad p(\Delta_A)=0,\quad p(\Delta^A)=0.
	\end{gather*}
\end{remark}

\begin{remark}
	\label{rem:BFV_brackets}
	Let $\J$ be a Jacobi structure on $L\to M$.
	An arbitrary $\hat{\J}\in\Diff^2(\hat{L}[1])^1$ decomposes as follows
	\begin{equation*}
	\hat{\J}=\sum_{k=0}^\infty\hat{\J}_k,\qquad\textnormal{with}\quad\underbrace{\hat{\J}_k}_{(k-1,k-1)}.
	\end{equation*}
	It follows that the liftings of $\J$ are given by the degree $1$ graded symmetric bi-derivations $\hat{\J}$ such that
	\begin{gather}
	\label{eq:rem:BFV_brackets1}
	\hat{\J}_0=G\\
	\label{eq:rem:BFV_brackets2}
	p(\hat{\J}_1)=\J\\
	\label{eq:rem:BFV_brackets3}
	2d_G\hat{\J}_k+\sum_{i=1}^{k-1}\ldsb \hat{\J}_i,\hat{\J}_{k-i}\rdsb = 0 ,\quad\textnormal{for all }k>0.
	\end{gather}
\end{remark}

As already remarked the line bundle $L\to M$ comes equipped with a canonical flat $\Diff(L)$-connection, i.e.~the tautological representation given by the identity map $\id:\Diff(L)\to\Diff(L)$.
Accordingly every choice of a $\Diff(L)$-connection $\nabla$ in $E\to M$ determines a $\Diff(L)$-connection in $E^\ast\otimes L\to M$, denoted by $\nabla$ again.
This shows that every connection $\nabla:\Diff(L)\to \Diff(E)$ admits a canonical extension $\nabla:\Diff(L)\to\Diff(\hat{L})$.
Note here that $\Diff(L)=\Diff(L[1])$ and $\Diff(\hat{L})=\Diff(\hat{L}[1])$.

\begin{lemma}
	\label{lem:metric_connection}
	Let $\nabla$ be a $\Diff(L)$-connection in $E\to M$.
	Its canonical extension $\nabla:\Diff(L)\to\Diff(\hat{L})$ takes values in $\mathfrak{aut}(\hat{M},\hat{L},G)$, i.e., for all $\square\in\Diff(L)$, the equivalent conditions hold:
	\begin{itemize}
		\item\label{en:lem:metric_connection_1}
		$\nabla_\square$ is a Jacobi derivation, or, which is the same, an infinitesimal Jacobi automorphism of $(\hat{M},\hat{L},G)$, i.e.
		\begin{equation*}
		\nabla_\square\{\lambda_1,\lambda_2\}_G=\{\nabla_\square\lambda_1,\lambda_2\}_G+\{\lambda_1,\nabla_\square\lambda_2\}_G,\qquad\lambda_1,\lambda_2\in\Gamma(\hat{L}),
		\end{equation*}
		\item\label{en:lem:metric_connection_2}
		$\nabla_\square$ is a co-cycle of the co-chain complex $(\Diff^\star(\hat{L}[1]),d_G)$, i.e.~$d_G\nabla_\square=0$.
	\end{itemize}
\end{lemma}

\begin{proof}
	Fix an arbitrary $\square\in\Diff(L)$.
	It follows from Remark~\ref{rem:Gerstenhaber_product_and_SJ_bracket} after a straightforward computation that, for all $u\in\Gamma(E)$, $\alpha\in\Gamma(E^\ast)$, and $\lambda\in\Gamma(L)[1]$,
	\begin{equation*}
	\ldsb G,\nabla_\square\rdsb (\alpha\otimes\lambda,u)=0.
	\end{equation*}
	Since $\ldsb G,\nabla_\square\rdsb \in\Diff^2(\hat{L}[1])^{(-1,-1)}$, we get that $d_G\nabla_\square=0$.
\end{proof}

\subsubsection*{The immersion}
There is a degree $0$ graded module morphism $i_\nabla:\Diff^\star(L[1])\to\Diff^\star(\hat{L}[1])$, covering a degree $0$ graded algebra morphism $\underline{\smash{i_\nabla}}:\Diff^\star(L[1],\bbR_M)\to\Diff^\star(\hat{L}[1],\bbR_{\hat{M}})$, completely determined by
\begin{equation}
\label{eq:i_nabla}
i_\nabla\lambda=\lambda,\qquad i_\nabla\square=\nabla_\square,
\end{equation}
for all $\lambda\in\Gamma(L)$, and $\square\in\Diff(L[1])$.
In the following proposition we list some properties of $i_\nabla$.

\begin{proposition}
	\label{prop:i_nabla} \ 
	
	(1) $i_\nabla$ pr\underline{}eserves the arity, and takes values of bi-degree $(0,0)$, i.e.
	\begin{equation*}
	i_\nabla(\Diff^k(L[1]))\subseteq\Diff^k(\hat{L}[1])^{(0,0)}.
	\end{equation*}
	
	(2) $i_\nabla$ is a section of $p$, i.e.~$p\circ i_\nabla=\id$ on $\Diff^\star(L[1])$.
	
	(3) $i_\nabla$ is a morphism of co-chain map $(\Diff^\star(L[1]),0)$ to $(\Diff^\star(\hat{L}[1]),d_G)$, i.e.~$d_G\circ i_\nabla=0$.
\end{proposition}

\begin{proof}\ 
	
	(1) It is an immediate consequence of~\eqref{eq:i_nabla}.
	
	(2) Since $p\circ i_\nabla:\Diff^\star(L[1])\to\Diff^\star(L[1])$ is a module morphism, covering the algebra morphism $\underline{\smash{p}}\circ\underline{\smash{i_\nabla}}:\Diff^\star(L[1],\bbR_M)\to\Diff^\star(L[1],\bbR_M)$, it is enough to check that $p\circ i_\nabla$ agrees with the identity on $\Gamma(L)[1]$ and $\Diff(L[1])$.
	This is exactly the case because of~\eqref{eq:p} and~\eqref{eq:i_nabla}.
	
	(3) Since $d_G\circ i_\nabla:\Diff^\star(L[1]))\to\Diff^\star(\hat{L}[1])$ is a derivation along the module morphism $i_\nabla:\Diff^\star(L[1]))\to\Diff^\star(\hat{L}[1])$ (cf.~Remark~\ref{rem:derivation_along_morphism}), it is enough to check that $d_G\circ i_\nabla$ vanishes on $\Gamma(L)[1]$ and $\Diff(L[1])$, and indeed this is the case because of~\eqref{eq:i_nabla}, Remark~\ref{rem:d_G} and Lemma~\ref{lem:metric_connection}.
\end{proof}

\begin{proposition}
	\label{prop:flatness}
	The following conditions are equivalent:
	\begin{itemize}
		\item the $DL$-connection $\nabla$ in $E\to M$ is flat, i.e.
		\begin{equation}
		\label{eq:flatness_connection}
		\ldsb \nabla_{\square_1},\nabla_{\square_2}\rdsb=\nabla_{\ldsb\square_1,\square_2\rdsb},\quad\text{for all}\ \square_1,\square_2\in DL.
		\end{equation}
		\item the immersion $i_\nabla:D^\star(L[1])\to D^\star(\hat{L}[1])$ is a Lie algebra morphism, i.e.
		\begin{equation}
		\label{eq:flatness_immersion}
		\ldsb i_\nabla\square_1,i_\nabla\square_2\rdsb=i_\nabla\ldsb\square_1,\square_2\rdsb,\quad\text{for all}\ \square_1,\square_2\in D^\star(L[1]).
		\end{equation}
	\end{itemize}
\end{proposition}

\begin{proof}
	Formula~\eqref{eq:flatness_immersion} relates two objects:
	\begin{itemize}
	\item[-] the Schouten--Jacobi bracket $\ldsb-,-\rdsb$, which makes $D^\star(L[1])$ and $D^\star(\hat{L}[1])$ Gerstenhaber--Jacobi algebras over $D^\star(L[1],\bbR_M)$ and $D^\star(\hat{L}[1],\bbR_{\hat{M}})$ respectively,
	\item[-] the immersion $i_\nabla:D^\star(L[1])\to D^\star(\hat{L}[1])$, which is a module morphism covering an algebra morphism $\underline{\smash{i_\nabla}}:D^\star(L[1],\bbR_{M})\to D^\star(\hat{L}[1],\bbR_{\hat{M}})$.
	\end{itemize}
	Moreover module $D^\star(L[1])$ is generated by $D^0(L[1])=\Gamma(L)$ over $D^\star(L[1],\bbR_{M})$, and algebra $D^\star(L[1],\bbR_{M})$ is generated by $D^1(L[1],\bbR_{M})$ over $D^0(L[1],\bbR_{M})=C^\infty(M)$.
	Hence~\eqref{eq:flatness_immersion} holds iff it is satisfied when $\square_1$ and $\square_2$ have degree $0$ or $1$, i.e.
	\begin{align}
	\label{eq:flatness_immersion_a}\tag{a}
	i_\nabla\ldsb\lambda_1,\lambda_2\rdsb&=\ldsb i_\nabla\lambda_1,i_\nabla\lambda_2\rdsb,&&\text{for all}\ \lambda_1,\lambda_2\in\Gamma(L),\\
	\label{eq:flatness_immersion_b}\tag{b}
	i_\nabla\ldsb\square,\lambda\rdsb&=\ldsb i_\nabla\square,i_\nabla\lambda\rdsb,&&\text{for all}\ \square\in DL,\lambda\in\Gamma(L),\\
	\label{eq:flatness_immersion_c}\tag{c}
	i_\nabla\ldsb\square_1,\square_2\rdsb&=\ldsb i_\nabla\square_1,i_\nabla\square_2\rdsb,&&\text{for all}\ \square_1,\square_2\in DL.
	\end{align}
	On the one hand, both~\eqref{eq:flatness_immersion_a} and~\eqref{eq:flatness_immersion_b} are trivially satisfied because of the very definition of $i_\nabla$ and $\ldsb-,-\rdsb$.
	On the other hand, the identity~\eqref{eq:flatness_immersion_c} is satisfied iff its two sides, seen as derivations of $\hat{L}$, coincide when evaluated on $\Gamma(\hat{L})^{(1,0)}=\Gamma(E)$, and so it exactly amounts to the flatness of the $DL$-connection $\nabla$ in $E\to M$.
	This shows the equivalence of~\eqref{eq:flatness_connection} and~\eqref{eq:flatness_immersion}, and completes the proof.
\end{proof}

\begin{remark}
	\label{rem:local_expression_i_nabla}
	Let us keep the same notations of Remark~\ref{rem:local_frames1}.
	The tautological $\Diff(L)$-connection in $L$ is locally given by:
	\begin{align*}
	\nabla_{\id}\mu&=\mu,&\nabla_{\Delta_i}\mu&=0,
	\intertext{Clearly, if the $\Diff(L)$-connection in $E$ is locally given by:}
	\nabla_{\id}\xi^A&=\Gamma^A_B\xi^B,&\nabla_{\Delta_i}\xi^A&=\Gamma^{\phantom{i}A}_{iB}\xi^B,
	\end{align*}
	then the $\Diff(L)$-connection in $E^\ast\otimes L$, obtained by tensor product, is locally given by:
	\begin{align*}
	\nabla_{\id}(\xi^\ast_A\otimes\mu)&=(\delta_A^B-\Gamma_A^B)\xi^\ast_B\otimes\mu,&\nabla_{\Delta_i}(\xi^\ast_A\otimes\mu)&=-\Gamma^{\phantom{i}B}_{iA}\xi^\ast_B\otimes\mu.
	\end{align*}
	Accordingly its extension $\nabla:\Diff(L)\to\Diff(\hat{L})$ is locally given by:
	\begin{align*}
	\nabla_{\id}&=\id-(\delta_B^A-\Gamma_B^A)(\xi^B\otimes\mu^\ast)\Delta_A-\Gamma_A^B\xi_B^\ast\Delta^A,\\
	\nabla_{\Delta_i}&=\Delta_i+\Gamma_{iB}^{\phantom{i}A}(\xi^B\otimes\mu^\ast)\Delta_A-\Gamma_{iA}^{\phantom{i}B}\xi_B^\ast\Delta^A.
	\end{align*}  	
	Hence, the action of the module morphism $i_\nabla:\Diff^\star(L[1])\to\Diff^\star(\hat{L}[1])$ is locally given by
	\begin{align*}
	i_\nabla(f\mu)&=f\mu,\\
	i_\nabla(\id)&=\id+(\Gamma_B^A-\delta_B^A)(\xi^B\otimes\mu^\ast)\Delta_A-\Gamma_A^B\xi_B^\ast\Delta^A,\\
	i_\nabla(\Delta_i)&=\Delta_i+\Gamma_{iB}^{\phantom{i}A}(\xi^B\otimes\mu^\ast)\Delta_A-\Gamma_{iA}^{\phantom{i}B}\xi_B^\ast\Delta^A.
	\end{align*}
\end{remark}

\subsubsection*{The homotopy}
We construct the homotopy in several steps.
First of all, each $\Diff L$-connection $\nabla$ in $E$ determines a degree $0$ graded module isomorphism
\begin{equation*}
\psi_\nabla:
\Diff^\star(\hat{L}[1])\overset{\simeq}{\longrightarrow} S_{{\scriptscriptstyle C^\infty(M)}}(\Gamma(E_L[-1]\oplus E^\ast[1]\oplus E_L[-2]\oplus E^\ast[0]))\underset{{\scriptscriptstyle C^\infty(M)}}{\otimes}\Diff^\star(L[1]),
\end{equation*} 
covering a degree $0$, $C^\infty(\hat{M})$-linear, graded algebra isomorphism
\begin{equation*}
\underline{\smash{\psi_\nabla}}:
\Diff^\star(\hat{L}[1],\bbR_{\hat{M}})\overset{\simeq}{\longrightarrow} S_{{\scriptscriptstyle C^\infty(M)}}(\Gamma(E_L[-1]\oplus E^\ast[1]\oplus E_L[-2]\oplus E^\ast[0]))\underset{{\scriptscriptstyle C^\infty(M)}}{\otimes}\Diff^\star(L[1],\bbR_M),
\end{equation*}
cf.~Section~\ref{subsec:psi_nabla}.
Using the same notations of Remark~\ref{rem:local_frames1}, $\psi_\nabla$ is locally given by
\begin{gather*}
\psi_\nabla(f\mu)=f\mu,\quad\psi_\nabla(\xi^A)=\xi^A,\qquad\psi_\nabla(\xi_A^\ast\otimes\mu)=\xi_A^\ast\otimes\mu,\\
\psi_{\nabla}(\nabla_{\id})=\id,\quad \psi_\nabla(\nabla_{\Delta_i})=\Delta_i,\\ \psi_{\nabla}(\Delta_A)=\xi_A^\ast\otimes\mu,\quad \psi_{\nabla}(\Delta^A)=\xi^A.
\end{gather*}
Now we can use $\psi_\nabla$ to define a degree $0$ graded derivation $\weight:\Diff^\star(\hat{L}[1])\to\Diff^\star(\hat{L}[1])$ as follows: $\weight$ is completely determined by the following conditions:	\begin{itemize}
	\item $\psi_\nabla\circ\weight\circ\psi_\nabla^{-1}$ vanishes on $\Diff^\star(L[1])$, and
	\item $\underline{\smash{\psi_\nabla}}\circ\underline{\smash{\weight}}\circ\underline{\smash{\psi_\nabla}}^{-1}$ agrees with the identity map on $\Gamma(E_L[-1]\oplus E^\ast[1]\oplus E_L[-2]\oplus E^\ast[0])$,
\end{itemize}
where $\underline{\smash{\weight}}$ is the symbol of $\weight$.
In other words, $\weight$ counts the tensorial degree of multi-derivations w.r.t.~$E_L[-1]$, $E^\ast[1]$, $E_L[-2]$ and $E^\ast[0]$.
\begin{remark}
	Because of its very definition, $\weight$ is $C^\infty(M)$-linear, and preserves both the arity and the bi-degree, i.e.
	\begin{equation*}
	\weight(\Diff^n(\hat{L}[1])^{(h,k)})\subseteq\Diff^n(\hat{L}[1])^{(h,k)}.
	\end{equation*}
\end{remark}

\begin{remark}
	Derivation $\weight$ is locally given by:
	\begin{gather*}
	\weight(f\mu)=\weight(i_\nabla(\id))=\weight(i_\nabla(\Delta_i))=0,\\
	\weight(\xi^A)=\xi^A,\quad\weight(\xi_A^\ast\otimes\mu)=\xi_A^\ast\otimes\mu,\\\quad\weight(\Delta_A)=\Delta_A,\quad\weight(\Delta^A)=\Delta^A.
	\end{gather*}
\end{remark}
We also define a degree $(-1)$ graded derivation $\tilde H_\nabla$ of $\Diff^\star(\hat{L}[1])$ as follows: $\tilde H_\nabla$ is completely determined by the following conditions:
\begin{itemize}
	\item $\psi_\nabla\circ{\tilde H}_\nabla\circ\psi_\nabla^{-1}$ vanishes on $\Diff^\star(L[1])$,
	\item $\underline{\smash{\psi_\nabla}}\circ\underline{\smash{{\tilde H}_\nabla}}\circ\underline{\smash{\psi_\nabla}}^{-1}$ vanishes on $\Gamma(E_L[-1]\oplus E^\ast[1])$, it is $C^\infty(\hat{M})$-linear, and maps $\Gamma(E_L[-2]\oplus E^\ast[0])$ to $\Gamma(E_L[-1]\oplus E^\ast[1])$ acting as the desuspension map,
\end{itemize}
where $\underline{\smash{{\tilde H}_\nabla}}$ is the symbol of $\tilde H_\nabla$.

\begin{remark}
	\label{rem:H_nabla}
	From its very definition, $\tilde H_\nabla$ is also $C^\infty(\hat{M})$-linear, and it has bi-degree $(1,1)$, so that
	\begin{equation*}
	\tilde H_\nabla(\Diff^n(\hat{L}[1])^{(h,k)})\subseteq\Diff^{n-1}(\hat{L}[1])^{(h+1,k+1)}.
	\end{equation*}
\end{remark}

\begin{remark}
	\label{def:H^tilde}
	Derivation $\tilde H_\nabla$ is locally given by:
	\begin{gather*}
	\tilde H_\nabla(f\mu)= \tilde H_\nabla(i_\nabla(\id))= \tilde H_\nabla(i_\nabla(\Delta_i))=0,\\
	\tilde H_\nabla(\xi^A)=\tilde H_\nabla(\xi_A^\ast\otimes\mu)=0,\quad \tilde H_\nabla(\Delta_A)=\xi_A^\ast\otimes\mu,\quad \tilde H_\nabla(\Delta^A)=\xi^A.
	\end{gather*}
\end{remark}

\begin{lemma}
	\label{lem:1contraction_data_1a}
	The following identities hold:
	\begin{equation}
	\label{eq:1contraction_data_1a}
	p\circ\tilde H_\nabla=0,\qquad \tilde H_\nabla\circ i_\nabla=0,\qquad \tilde H_\nabla^2=0,\qquad [\tilde{H}_\nabla,d_G]=\weight.
	\end{equation}
\end{lemma}

\begin{proof}
	The first three identities follow immediately from the local coordinate expressions for $p$, $i_\nabla$, $d_G$, and $\tilde{H}_\nabla$.
	Moreover, from $d_G\circ i_\nabla=H_\nabla\circ i_\nabla=0$, a straighforward computation in local coordinates shows that the graded derivations $[\widetilde H_\nabla, d_G]$ and $\weight$ agree on generators.
	Hence they coincide.
\end{proof}

\begin{remark}
	\label{rem:eigenspaces_deg}
	Since both $d_G$ and $\tilde{H}_{\nabla}$ commute with $\weight$, the eigenspaces of $\weight$ are invariant under $d_G$ and ${\tilde H}_\nabla$, i.e., for every $k\in\bbN_0 $,
	\begin{align*}
	d_G(\ker(\weight-k\id))&\subseteq\ker(\weight-k\id),\\
	{\tilde H}_\nabla(\ker(\weight-k\id))&\subseteq\ker(\weight-k\id).
	\end{align*}
	Moreover,from the spectral decomposition of $\weight$ it follows that
	\begin{equation*}
	\Diff^\star(\hat{L}[1])=\bigoplus_{k\geq 0}\ker(\weight-k\id)=\ker p\oplus\im i_\nabla,
	\end{equation*}
	where $\im i_\nabla=\ker\weight$, and $\ker p=\bigoplus_{k>0}\ker(\weight-k\id)$.
\end{remark}

\begin{lemma}
	\label{lem:1contraction_data_1b}
	Let $H_\nabla:\Diff^\star(\hat{L}[1])\to\Diff^\star(\hat{L}[1])$ be the degree $(-1)$ graded $C^\infty(M)$-linear map, of bi-degree $(1,1)$, defined by setting:
	\begin{equation*}
	H_\nabla=
	\begin{cases}
	0,&\textnormal{ on }\ker\weight,\\
	-k^{-1}{\tilde H}_\nabla,&\textnormal{ on }\ker(\weight-k\id),\textnormal{ for all }k>0.
	\end{cases}
	\end{equation*}
	Then
	\begin{equation}
	\label{eq:1contraction_data_1b}
	i_\nabla\circ p-\id=[d_G,H_\nabla].
	\end{equation}
	Additionally, $p,i_\nabla$ and $H_\nabla$ satisfy the side conditions $H_\nabla^2=0$, $H_\nabla\circ i_\nabla=0$, $p\circ H_\nabla=0$.
\end{lemma}

\begin{proof}
	It is an immediate consequence of Lemma~\ref{lem:1contraction_data_1a} and Remark~\ref{rem:eigenspaces_deg}.
\end{proof}

The following proposition summarizes the above discussion.

\begin{proposition}
	\label{prop:1contraction_data_1}
	Every  $\Diff L$-connection $\nabla$ in $E$ determines the following set of contraction data:
	\begin{equation}
	\label{eq:1contraction_data_1}
	\begin{tikzpicture}[>= stealth,baseline=(current bounding box.center)]
	\node (u) at (0,0) {$(\Diff^\star(\hat{L}[1]),d_G)$};
	\node (d) at (5,0) {$(\Diff^\star(L[1]),0)$};
	\draw [transform canvas={yshift=-0.5ex},-cm to] (d) to node [below] {\footnotesize $i_\nabla$} (u);
	\draw [transform canvas={yshift=+0.5ex},cm to-] (d) to node [above] {\footnotesize $p$} (u);
	\draw [-cm to] (u.south west) .. controls +(210:1) and +(150:1) .. node[left=2pt] {\footnotesize $H_\nabla$} (u.north west);
	\end{tikzpicture}
	\end{equation}
	In particular, $p$ is a quasi-isomorphism, so that $H_{CE}(\hat{M},\hat{L},\hat{\J})\simeq\Diff^\ast(L[1])$, in a canonical way.
\end{proposition}

\subsection{Existence and uniqueness of liftings: the proofs}
\label{subsec:existence_uniqueness_lifted_graded_Jacobi}

\begin{proof}[Proof of Theorem~\ref{theor:existence_BFV_brackets}]
	Fix a $\Diff(L)$-connection $\nabla$ in $E$.
	The corresponding $\hat{J}^\nabla$ is constructed by applying Proposition~\ref{prop:SBSO_existence}.
	It is enough to use contraction data~\eqref{eq:1contraction_data_1} for contraction data~\eqref{eq:homotopy_equivalence}, and set $\calF_n:=\bigoplus_{j\geq n}\Diff^\star(\hat{L}[1])^{(i,j)}$, $N=0$, and $\bar{Q}:=G+i_\nabla(J)$.
	In this special case, the necessary and sufficient condition~\eqref{eq:prop:SBSO_existence} is trivially satisfied.
\end{proof}

\begin{proof}[Proof of Proposition~\ref{prop:existence_BFV_brackets}]
	For an arbitrary Jacobi structure $\J$ on $L\to M$, because of Proposition~\ref{prop:i_nabla}(3), and $\ldsb G,G\rdsb=0$, we just have
	\begin{equation*}
	\ldsb G+i_\nabla\J,G+i_\nabla\J\rdsb=\ldsb i_\nabla\J,i_\nabla\J\rdsb.
	\end{equation*}
	Hence, when $\nabla$ is flat, Proposition~\ref{prop:flatness} guarantees that $\ldsb G+i_\nabla\J,G+i_\nabla\J\rdsb=0$.
	In such case the step-by-step obstruction method does not add any perturbative correction to $G+i_\nabla\J$, and so the output is $\hat{\J}^\nabla=G+i_\nabla\J$.
\end{proof}

\begin{proof}[Proof of Theorem~\ref{theor:uniqueness_BFV_brackets}]
	It follows from Proposition~\ref{prop:SBSO_uniqueness} and Corollary~\ref{cor:SBSO_uniqueness}.
\end{proof}

\begin{proof}[Proof of Proposition~\ref{prop:lifting_CE_cohom}]
	\label{proof:prop:lifting_CE_cohom}
	Let $\J$ and $\hat{\J}$ be Jacobi structures on $L\to M$ and $\hat{L}\to\hat{M}$ respectively.
	Assume that $\hat{\J}$ is a lifting of $\J$, and fix a $\Diff(L)$-connection $\nabla$ in $E\to M$.
	
	Using the same terminology as in~\cite{Crainic_perturbation}, $\delta:=d_{\hat{\J}}-d_G$ provides a small perturbation of the contraction data~\eqref{eq:1contraction_data_1} determined by $\nabla$.
	Actually, from Remarks~\ref{rem:BFV_brackets} and~\ref{rem:H_nabla}, it follows that
	\begin{align*}
	\delta(\Diff^n(\hat{L}[1])^{(p,q)})&\subseteq\bigoplus_{k\geq 0}\Diff^{n+1}(\hat{L}[1])^{(p+k,q+k)},\\
	(\delta H_\nabla)(\Diff^n(\hat{L}[1])&\subseteq\bigoplus_{k\geq 1}\Diff^n(\hat{L}[1])^{(p+k,q+k)},
	\end{align*}
	so that $\delta H_\nabla$ is nilpotent and $(\id-\delta H_\nabla)$ is invertible with $(\id-\delta H_\nabla)^{-1}=\sum_{k=0}^\infty(\delta H_\nabla)^k$.
	Hence the Homological Perturbation Lemma (see, e.g., \cite{brown1967twisted,Crainic_perturbation}) applies with the contraction data~\eqref{eq:1contraction_data_1} and their small perturbation $\delta$ as input.
	The output is a new (deformed) set of contraction data
	\begin{equation*}
	\begin{tikzpicture}[>= stealth,baseline=(current bounding box.center)]
	\node (d) at (0,0) {$(\Diff^\star(L[1]),d')$};
	\node (u) at (-5,0) {$(\Diff^\star(\hat{L}[1]),d_{\hat{\J}})$};
	\draw [transform canvas={yshift=-0.5ex},-cm to] (d) to node [below] {\footnotesize $i_\nabla'$} (u);
	\draw [transform canvas={yshift=0.5ex},cm to-] (d) to node [above] {\footnotesize $p'$} (u);
	\draw [-cm to] (u.south west) .. controls +(210:1) and +(150:1) .. node[left=2pt] {\footnotesize $H_\nabla'$} (u.north west);
	\end{tikzpicture}
	\end{equation*}
	given by
	\begin{align*}
	i_\nabla'&=\sum_{k=0}^\infty(H_\nabla\delta)^ki_\nabla,&
	p'&=\sum_{k=0}^\infty p(\delta H_\nabla)^k,\\
	H_\nabla'&=\sum_{k=0}^\infty H_\nabla(\delta H_\nabla)^k,&
	d'&=\sum_{k=0}^\infty p\delta(H_\nabla\delta)^k i_\nabla.
	\end{align*}
	It follows from Propositions~\ref{prop:p} and~\ref{prop:i_nabla} that $p\delta(H_\nabla\delta)^ki_\nabla=0$, for all $k\geq 1$, and moreover
	\begin{equation*}
	d'\square=p\ldsb \hat{\J}-G,i_\nabla\square\rdsb=p\ldsb\hat{\J}_1,i_\nabla\square\rdsb=\ldsb\J,\square\rdsb=d_{\J}\square,
	\end{equation*}
	for all $\square\in\Diff^\star(L[1])$.
	Hence $d'=d_{\J}$, and $i_\nabla'$ is the desired quasi-isomorphism.
\end{proof}

\chapter{The BFV-complex of a coisotropic submanifold}
\label{chap:BFV_complex}

In this chapter, aiming at studying the coisotropic deformation problem at the non-formal level, we associate a further invariant to every coisotropic submanifold $S$ of a Jacobi manifold $(M,L,J=\{-,-\})$, namely its BFV-complex.

For any choice of a fat tubular neighborhood $(\tau,\underline{\smash{\tau}})$ of the restricted line bundle $\ell\to S$ in $L\to M$, we introduce the notion of BFV-complex attached to $S$ via $(\tau,\underline{\smash{\tau}})$ (Definition~\ref{def:BFV_complex}) extending to the Jacobi setting the similar notion presented by Sch\"atz~\cite{schatz2009bfv} in the Poisson setting.
The BFV-complex can be seen as a certain graded Jacobi manifold $(\hat{M},\hat{L},\hat{J}=\{-,-\}_{BFV})$ additionally equipped with a cohomological Hamiltonian derivation $d_{BFV}=\{\Omega_{\BRST},-\}_{BFV}$, where the graded Jacobi structure $\hat{J}$ is a lifting of $J$ (cf.~Section~\ref{sec:lifting_Jacobi_structures}) and the potential $\Omega_{\BRST}$ is a particular instance of BRST charge (cf.~Section~\ref{sec:BRST-charges}).
Combining the existence and uniqueness theorems for the BRST charges (Theorems~\ref{theor:existence_BRST-charge} and~\ref{theor:uniqueness_BRST-charge}) with the analogous results for the lifted Jacobi structures, the BFV-complex of $S$ is proven to exist and to be independent, up to isomorphisms, from the chosen fat tubular neighborhood (Theorem~\ref{theor:gauge_invariance_BFV_complex}).

As already the $L_\infty$-algebra $\frakg^\bullet(S)$ of $S$, the BFV-complex provides a cohomological resolution of the reduced Gerstenhaber--Jacobi algebra of $S$ (Corollary~\ref{cor:BFV_homological_resolution}).
But the relation between the $L_\infty$-algebra and the BFV-complex of $S$ is deeper than this.
Actually they are $L_\infty$-quasi-isomorphic (Theorem~\ref{theor:L_infty_qi}) extending to the Jacobi setting the result obtained by Sch\"atz~\cite{schatz2009bfv} in the Poisson setting.
As a consequence, the BFV-complex fully encodes the local moduli spaces, under Hamiltonian equivalence, of both the infinitesimal and the formal coisotropic deformations of $S$ (Propositions~\ref{prop:BFVinfinitesimal_moduli_space} and~\ref{prop:BFVformal_moduli_space}).
Further it also provides criteria for the unobstructedness (Proposition~\ref{prop:BFV_2nd_cohomology}) and the obstructedness (Proposition~\ref{prop:BFV_kuranishi}) of the formal coisotropic deformation problem.

Unlike the $L_\infty$-algebra, the BFV-complex controls the (non-formal) coisotropic deformation problem of $S$, even under Hamiltonian and Jacobi equivalence, with no need of any restrictive hypothesis.
In order to see this, following Sch\"atz~\cite{schatz2011moduli}, we single out the special class of so-called ``geometric'' MC elements of the BFV complex.
Then a one-to-one correspondence is canonically established between the space of coisotropic sections and the orbit space of geometric MC elements under the action of a certain group of Hamiltonian automorphisms (Theorem~\ref{theor:coisotropic_def_space}).
Additionally such one-to-one correspondence intertwines the Hamiltonian/Jacobi equivalence of coisotropic sections with the Hamiltonian/Jacobi equivalence of geometric MC elements (Theorems~\ref{theor:Mod_Ham} and~\ref{theor:Mod_Jac}).

Finally, in Section~\ref{sec:obstructed_example_contact_BFV}, the framework of the BFV-complex is employed to get a conceptual interpretation of the obstructed coisotropic deformation problem~\cite[Examples~3.5 and~3.8]{tortorella2016rigidity} which integrates its description in terms of the $L_\infty$-algebra provided in Section~\ref{sec:obstructed_example_contact_L-infinity}.

\section{BRST charges}

\label{sec:BRST-charges}

Let $S$ be a coisotropic submanifold of a Jacobi manifold $(M,L,J)$.
In analogy with the Poisson case~\cite{herbig2007,schatz2009bfv}, we will attach to $S$ an algebraic invariant, the BFV-complex.
The BFV-complex provides a cohomological resolution of the reduced Gerstenhaber-Jacobi algebra of $S$, and encodes its coisotropic deformations and their local moduli spaces.
Since we are only interested in small deformations of $S$, we can restrict to work within a fat tubular neighborhood of restricted line bundle $\ell:=L|_S\to S$ in $L\to M$.
For the reader's convenience, we recall here that, according to Definition~\ref{def:fat_tubular_neighborhod}, a fat tubular neighborhood $(\tau,\underline{\smash{\tau}})$ of $\ell\to S$ in $L\to M$ consists of two layers:
\begin{itemize}
	\item a tubular neighborhood $\underline{\smash{\tau}}:NS\to M$ of $S$ in $M$,
	\item an embedding $\tau:L_{NS}\to L$ of line bundles, over $\underline{\smash{\tau}}:NS\to M$, such that $\tau=\id$ on $L_{NS}|_S\simeq\ell$,
\end{itemize}
where $\pi:NS\to S$ is the normal bundle to $S$ in $M$, and $L_{NS}:=\pi^\ast\ell\to NS$.
By transferring Jacobi structures along a fat tubular neighborhood, we end up with the following local model for a Jacobi manifold $(M,L,J)$ around an arbitrary submanifold $S\subset M$.
\begin{itemize}
	\item The manifold $M$ is modelled on the total space $\calE$ of a vector bundle $\pi:\calE\to S$, and $S$ is identified with the image of the zero section of $\pi$.
	\item The line bundle $L\to M$ is modelled on $\pi^\ast\ell\to\calE$, for some line bundle $\ell\to S$.
\end{itemize}
In this section, working within such local model of $(M,L,J)$ around $S$, we will apply the lifting procedure described in Section~\ref{sec:lifting_Jacobi_structures} to the case when $E\to M$ is $V\calE \simeq \pi^\ast\calE\to\calE$, the vertical bundle of $\calE\to S$.
In particular, $E\to M = \calE$ admits a tautological section, that we denote by $\Omega_E$, mapping $x \in \calE$ to $(x,x) \in \pi^\ast \calE = \calE \times_S \calE$.

\subsection{Existence and uniqueness of the BRST charges: the statements}
\label{subsec:BRST-charges}

Let $u^i$ be a system of local coordinates on $S$, $\eta^A$ a local frame of $\pi:\calE\to S$, and $\mu$ a local frame of $\ell\to S$.
Denote by $\eta_A^\ast$ the local frame on $\calE^\ast\to S$ dual to $\eta^A$, and by $y_A$ the corresponding fiber-wise linear functions on $\calE$.
Then $\xi^A:=\pi^\ast\eta^A$ is a local frame of $E\to\calE$, and $\xi_A^\ast=\pi^\ast\eta_A^\ast$ is the dual frame of $E^\ast\to\calE$.
Furthermore a local frame of $E_L\to\calE$ is given by $\xi^A\otimes \pi^\ast\mu^\ast=\pi^\ast(\eta^A\otimes\mu^\ast)$, with $\xi_A^\ast\otimes \pi^\ast\mu=\pi^\ast(\eta_A^\ast\otimes\mu)$ the dual local frame of $(E_L)^\ast\to\calE$.

Recall that as first step of the lifting procedure (cf.~Section~\ref{sec:lifting_Jacobi_structures}) a graded line bundle $\hat{L}\to\hat{M}$ is constructed out of our initial data: the line bundle $L\to M$ and the vector bundle $E\to M$.
Moreover the graded line bundle $\hat{L}\to\hat{M}$ is canonically equipped with the tautological Jacobi structure $\{-,-\}_G$ (see Definition~\ref{def:G}).

\begin{proposition}
	\label{prop:delta_s}
	Let $s$ be an arbitrary section of $\pi:\calE\to S$.
	Section $\Omega_E[s]:=\Omega_E-\pi^\ast s\in\Gamma(E)$ is a MC element of $(\Gamma(\hat{L}),\{-,-\}_G)$.
	In particular $d[s]:=\{\Omega_E[s],-\}_G$ is a bi-degree $(0,-1)$ cohomological Hamiltonian derivation of the graded Jacobi manifold $(\hat{M},\hat{L},\{-,-\}_G)$.
\end{proposition}

\begin{proof}
	It is straightforward for ghost/anti-ghost bi-degree reasons.
\end{proof}

\begin{remark}
	\label{rem:local_expression_delta_s}
	The tautological section $\Omega_E$, and an arbitrary $s\in\Gamma(\pi)$  are locally given by $\Omega_E=y_A\xi^A$, and $s=g_A(u^i)\eta^A$.
	Hence $\Omega_E[s]$, and the associated cohomological derivation $d[s]$ are locally given by
	\begin{equation*}
	\Omega_E[s]=(y_A-g_A(u^i))\xi^A,\quad\textnormal{and}\quad d[s]=(y_A-g_A(u^i))\Delta^A.
	\end{equation*}
\end{remark}

Now, let $\J$ be a Jacobi structure on $L\to M$, and let $\hat{\J}$ be a lifting of $\J$ to $\hat{L}\to\hat{M}$.
Fix an arbitrary $ s\in\Gamma(\pi)$.
In general $\Omega_E[s]$ fails to be a MC element of $(\Gamma(\hat{L}),\{-,-\}_{\hat{\J}})$.
The aim of this section is to find conditions on $s$ so that $\Omega_E [s]$ can be deformed into a suitable MC element of $(\Gamma(\hat{L}),\{-,-\}_{\hat{\J}})$.
The latter will be called an \emph{$s$-BRST charge}.
It turns out that an $s$-BRST charge exists precisely when the image of $s$ is a coisotropic submanifold.
Now, suppose $S$ is coisotropic itself.
There are two reasons why $s$-BRST charges are interesting.
First of all, as it will be shown in Section~\ref{sec:BFV-complex}, the choice of a $0$-BRST charge represents the second and last step in the construction of the BFV-complex of $S$.
Moreover, as it will be shown in Section~\ref{sec:coisotropic_deformation_problem}, small coisotropic deformations of $S$ are encoded by the BFV-complex through BRST charges.
\begin{definition}
	\label{def:BRST-charge}
	An \emph{$s$-BRST charge} wrt $\hat{\J}$ is a MC element $\Omega$ of $(\Gamma(\hat{L}),\{-,-\}_{\hat{\J}})$ having $\Omega_E[s]$ as its bi-degree $(1,0)$ component.
	Explicitly, $\Omega\in\Gamma(\hat{L})^1$, $\{\Omega,\Omega\}_{\hat{\J}}=0$, and $\operatorname{pr}^{(1,0)}\Omega=\Omega_E[s]$.
\end{definition}

\begin{remark}
	Our $s$-BRST charges are analogous to what Sch\"atz calls \emph{normalized MC elements}.
	In particular, $0$-BRST charges are analogous to Sch\"atz's \emph{BFV-charges}~\cite{schatz2009bfv}.
	We adopted the terminology ``BRST charge'' because it seems to be more standard in the Physics literature on the subject.
\end{remark}

\begin{remark}
	\label{rem:BRST_charge}
	Assume that $\hat{\J}=\sum_{k=0}^\infty\hat{\J}_k$, a lifting of $\J$ to $\hat{L}\to\hat{M}$, has been decomposed as in Remark~\ref{rem:BFV_brackets}.
	Every $\Omega\in\Gamma(\hat{L})^1$ decomposes as follows
	\begin{equation*}
	\Omega=\sum_{i=0}^\infty\Omega_i,
	\end{equation*}
	with $\Omega_i\in\Gamma(\hat{L})^{(i+1,i)}$.
	Accordingly, an $s$-BRST charge wrt $\hat{\J}$ is a degree $1$ section $\Omega$ of $\hat{L}\to\hat{M}$ such that
	\begin{gather}
	\label{eq:rem:BRST_charge1}
	\Omega_0=\Omega_E[s],\\
	\label{eq:rem:BRST_charge2}
	2d[s]\Omega_h+\sum_{\genfrac{}{}{0pt}{}{i,j< h,k\geq 0}{i+j+k=h}}\{\Omega_i,\Omega_j\}_{\hat\J_k}=0,\quad\textnormal{for all } h\geq 1.
	\end{gather}
\end{remark}

Given $s\in\Gamma(\pi)$, next theorem shows that, as already announced, an $s$-BRST charge wrt $\hat{\J}$ exists precisely when the image of $s$ is coisotropic.
In this case, the $s$-BRST charge is also unique up to isomorphisms (Theorem~\ref{theor:uniqueness_BRST-charge}).

\begin{theorem}[Existence]
	\label{theor:existence_BRST-charge}
	Let $\J$ be a Jacobi structure on $L\to M$, let $\hat{\J}$ be a lifting of $\J$ to $\hat{L}\to\hat{M}$, and let $s\in\Gamma(\pi)$.
	Then there exists an $s$-BRST charge wrt $\hat{\J}$ iff  the image of $s$ is coisotropic in $(M,L,\J)$.
\end{theorem}

\begin{remark}
	\label{rem:filtration_Ham}
	Let $\{\calL^\bullet_{\geq n}\}_{n\geq 0}$ be the \emph{finite} decreasing filtration of $\Gamma(\hat{L})^\bullet$, by the graded $C^\infty(\hat{M})$-submodules $\calL^\bullet_{\geq n}$ defined as the sum of those $\Gamma(\hat{L})^{(h,k)}$, with $k\geq n$.
	Then there is a \emph{finite} decreasing filtration $\{\Ham_{\geq n}(\hat{M},\hat{L},\hat{\J})\}_{n\geq 0}$ of the group $\Ham(\hat{M},\hat{L},\hat{\J})$ of Hamiltonian automorphisms of $(\hat{M},\hat{L},\hat{\J})$.
	Namely, for $n\geq 0$, $\Ham_{\geq n}(\hat{M},\hat{L},\hat{\J})$ consists of those $\Phi\in\Ham(\hat{M},\hat{L},\hat{\J})$ such that $\Phi=\Phi_1$ for a smooth path of Hamiltonian automorphisms $\{\Phi_t\}_{t\in I}$ integrating $\{\lambda_t,-\}_{\hat{\J}}$, with $\{\lambda_t\}_{t\in I}\subset\calL^0_{\geq n}$ (cf.~Definitions~\ref{def:Hamiltonian_family} and~\ref{def:Hamiltonian_single}).
\end{remark}

\begin{theorem}[Uniqueness]
	\label{theor:uniqueness_BRST-charge}
	Let $\J$ be a Jacobi structure on $L\to M$, and let $\hat{\J}$ be a lifting of $\J$ to $\hat{L}\to\hat{M}$.
	Moreover, let $s\in\Gamma(\pi)$, and let $\Omega$ and $\Omega'$ be $s$-BRST charges wrt $\hat{\J}$.
	Then there exists $\phi\in\Ham_{\geq 2}(\hat{M},\hat{L},\hat{\J})$ such that $\phi^\ast\Omega'=\Omega$.
\end{theorem}

In order to develop the necessary technical tools first, we postpone the proofs of Theorems~\ref{theor:existence_BRST-charge} and~\ref{theor:uniqueness_BRST-charge} to the end of this section.

\subsection{A second relevant set of contraction data}
\label{subsec:second_relevant_contraction_data}

In this subsection we show that a set of contraction data from $\Gamma(\hat{L})$ to $\Gamma((\wedge^\bullet\calE_\ell)\otimes\ell)$ is associated with any section $s\in\Gamma(\pi)$ (Proposition~\ref{prop:2contraction_data}).
The latter represents the technical tool on which the proof of existence and uniqueness of $s$-BRST charges is based (cf.~the next Subsection~\ref{subsec:existence_uniqueness_BRST-charges}).

Fix a section $s\in\Gamma(\pi)$ locally given by $s=g_A(u^i)\eta^A$.
In the following, for any vector bundle $F\to S$, we will understand the canonical isomorphism $F\overset{\simeq}{\longrightarrow}(\pi^\ast F)|_{\im s}$, given by $F_x \ni v\mapsto(s(x),v)\in(\pi^\ast F)_{s(x)}$, for all $x\in S$.

\subsubsection*{The projection}

There is a degree $0$ graded module epimorphism $\wp[s]:\Gamma(\hat{L})\to\Gamma((\wedge^\bullet\calE_\ell)\otimes\ell)$, covering a degree $0$ graded algebra morphism $\underline{\smash{\wp[s]}}:C^\infty(\hat{M})\to\Gamma(\wedge^\bullet\calE_\ell)$, completely determined by:
\begin{equation*}
\wp[s]\lambda=\lambda|_{\im s},\qquad \wp[s](u)=u|_{\im s},\qquad \wp[s](\alpha)=0,
\end{equation*}
for all $\lambda\in\Gamma(L)$, $u\in\Gamma(E)$, and $\alpha\in\Gamma(E^\ast\otimes L)$.
This means that $\wp[s]$ is obtained by restricting to $\im s$ and killing the components with non-zero anti-ghost degree.
Locally
\begin{equation*}
\wp[s]\left(f_{\bfA}^{\bfB}(u^i,y_C)\xi^\ast_{\bfB}\cdot(\xi^{\bfA}\otimes \pi^\ast\mu^\ast)\cdot\pi^\ast\mu\right)=f_{\bfA}(u^i,g_A(u^i))(\eta^{\bfA}\otimes\mu^\ast)\otimes\mu.
\end{equation*}
\subsubsection*{The immersion}
There is a degree $0$ graded module monomorphism $\iota:\Gamma( (\wedge^\bullet\calE_\ell)\otimes\ell)\to\Gamma(\hat{L})$, covering a degree $0$ graded algebra morphism $\underline{\smash{\iota}}:\Gamma(\wedge^\bullet\calE_\ell)\to C^\infty(\hat{M})$, completely determined by:
\begin{equation*}
\iota\lambda'=\pi^\ast\lambda',\qquad \iota\eta=\pi^\ast\eta,
\end{equation*}
for all $\lambda'\in\Gamma(\ell)$, and $\eta\in\Gamma(\calE)$.
Locally
\begin{equation*}
\iota(f_{\bfA}(u^i)(\eta^{\bfA}\otimes\mu^\ast)\otimes\mu)=f_{\bfA}(u^i)(\xi^{\bfA}\otimes \pi^\ast\mu^\ast)\cdot\pi^\ast\mu.
\end{equation*}
It immediately follows from the definitions of $d[s]$, $\wp[s]$ and $\iota$, that
\begin{itemize}
	\item $\wp[s]$ and $\iota$ are differential graded module morphisms between $(\Gamma(\hat{L}),d[s])$ and $(\Gamma((\wedge^\bullet\calE_\ell)\otimes\ell),0)$, i.e.
	\begin{equation*}
	\wp[s]\circ d[s]=0,\qquad d[s]\circ \iota=0,
	\end{equation*}
	\item $\iota$ is a section of $\wp[s]$, i.e.~$\wp[s]\circ \iota=\id$.
\end{itemize}
Conversely $\iota \circ \wp[s]=\id$ holds only up to a homotopy of differential graded modules that we construct now.
\subsubsection*{The homotopy}
Let $\{\psi^\ast_t[s]\}_{t\in I}$ be the smooth path of bi-degree $(0,0)$ graded module endomorphisms of $\Gamma(\hat{L})$, covering a smooth path $\{\underline{\smash{\psi}}^\ast_t[s]\}_{t\in I}$ of bi-degree $(0,0)$ graded algebra endomorphisms of $C^\infty(\hat{M})$, locally given by
\begin{multline*}
\label{eq:local_psi_t}
\psi_t^\ast[s]\left(f_{\bfA}^{\bfB}(u^i,y_C)\xi^\ast_{\bfB}\cdot(\xi^{\bfA}\otimes \pi^\ast\mu^\ast)\cdot\pi^\ast\mu\right)=\\=f_{\bfA}^{\bfB}(u_i,y_C-t(y_C-g_C(u^i)))(1-t)^{|\bfB|}\xi^\ast_{\bfB}\cdot(\xi^{\bfA}\otimes\pi^\ast\mu^\ast)\cdot\pi^\ast\mu.
\end{multline*}
We remark that $\{\psi^\ast_t[s]\}_{t\in I}$ connects $\id$ to $\iota\circ \wp[s]$.
Define a smooth path $\{\square_t[s]\}_{t\in I}$ of bi-degree $(0,0)$ graded derivations of $\hat{L}\to\hat{M}$ along $\{\psi_t^\ast[s]\}_{t\in I}$ by setting
\begin{equation*}
\square_t[s]:=\frac{d}{dt}\psi_t^\ast[s],
\end{equation*}
(see~Remark~\ref{rem:derivation_along_morphism} for the meaning of derivation along a module morphism).
There is a smooth path $\{j_t[s]\}_{t\in I}$ of bi-degree $(0,1)$ graded derivations of $\hat{L}\to\hat{M}$, along $\{\psi_t^\ast[s]\}_{t\in\bbR}$, completely determined by
\begin{gather*}
j_t[s](f(u^i,y_A)\pi^\ast\mu)=-(\partial_{y_A}f)(u_i,y_C-t(y_C-g_C(u^i)))\xi_A^\ast\otimes\mu,\\
j_t[s](\xi^A)=0,\qquad j_t[s](\xi_A^\ast\otimes\pi^\ast\mu)=0.
\end{gather*}
Since $d[s]$ and $\psi_t^\ast[s]$ commute, $\{[d[s],j_t[s]]\}_{t\in I}$ is a smooth path of bi-degree $(0,0)$ graded derivations, along $\{\psi_t^\ast[s]\}_{t\in I}$, as well.
Actually, a straightforward computation in local coordinates shows that $[d[s], j_t[s]]$ and $\square_t$ agree on both sections of $\pi^\ast\ell\to\calE$ and sections of $\pi^\ast\calE\to\calE$.
Hence they coincide.

Finally, define a bi-degree $(0,1)$ graded $C^\infty(M)$-linear map $h[s]:\Gamma(\hat{L})\to\Gamma(\hat{L})$ by setting
\begin{equation*}
h[s]:=\int_0^1 j_t[s]dt.
\end{equation*}
The map $h[s]$ is a homotopy betweeen the co-chain morphisms $\id,\iota\circ \wp[s]:(\Gamma(\hat{L}),d[s])\to(\Gamma(\hat{L}),d[s])$.
Indeed
\begin{align*}
\iota\circ \wp[s]-\id=\int_0^1\frac{d}{dt}\psi_t^\ast[s]dt=\int_0^1\square_t[s]dt=\int_0^1[d[s],j_t[s]]dt=[d[s],h[s]].
\end{align*}
Additionally $\iota$, $\wp[s]$ and $h[s]$ satisfy the side conditions $h[s]^2=0$, $h[s]\circ\iota$, $\wp[s]\circ h[s]=0$.

The above discussion is summarized in the following.

\begin{proposition}
	\label{prop:2contraction_data}
	Every section $s\in\Gamma(\pi)$ determines a set of contraction data
	\begin{equation}
	\label{eq:2contraction_data_1}
	\begin{tikzpicture}[>= stealth,baseline=(current bounding box.center)]
	\node (u) at (0,0) {$(\Gamma(\hat{L}),d[s])$};
	\node (d) at (5,0) {$(\Gamma((\wedge^\bullet\calE_\ell)\otimes\ell),0)$};
	\draw [transform canvas={yshift=-0.5ex},-cm to] (d) to node [below] {\footnotesize $\iota$} (u);
	\draw [transform canvas={yshift=0.5ex},cm to-] (d) to node [above] {\footnotesize $\wp[s]$} (u);
	\draw [-cm to] (u.south west) .. controls +(210:1) and +(150:1) .. node[left=2pt] {\footnotesize $h[s]$} (u.north west);
	\end{tikzpicture}
	\end{equation}
	In particular $\wp[s]$ is a quasi-isomophism, and $H^\bullet(\Gamma(\hat{L}),d[s])\simeq\Gamma((\wedge^\bullet\calE_\ell)\otimes\ell)$ in a canonical way.
\end{proposition}

\begin{remark}
	\label{rem:2contraction__data}
	All the above constructions, and Proposition~\ref{prop:2contraction_data} still hold true after replacing the section $s\in\Gamma(\pi)$ with a smooth path $\{s_\tau\}_{\tau\in I}$ in $\Gamma(\pi)$.
	The obvious details are left to the reader.
\end{remark}

\subsection{Existence and uniqueness of the BRST charges: the proofs}
\label{subsec:existence_uniqueness_BRST-charges}

\begin{lemma}
	\label{lem:BRST_coisotropic}
	Let $\J$ be a Jacobi structure on $L\to M$, and $s\in\Gamma(\pi)$.
	For any lifting $\hat{\J}$ of $\J$ to $\hat{L}\to\hat{M}$, we have that  $\{\Omega_E[s],\Omega_E[s]\}_{\hat{\J}}\in\ker\wp[s]$ iff the image of $s$ is coisotropic in $(M,L,J)$.
\end{lemma}

\begin{proof}
	Since $\hat{\J}$ is a bi-derivation with $p(\hat{\J})=\J$, it is straightforward to check that, locally,
	\begin{equation*}
	\wp[s]\vphantom{\xi^A}\{\Omega_E[s],\Omega_E[s]\}_{\hat{\J}}=\left.(\eta^B\otimes\mu^\ast)(\eta^A\otimes\mu^\ast)\{(y_A-g_A(u^i))\mu,(y_B-g_B(u^i))\mu\}_{\J}\right|_{\im s}.
	\end{equation*}
	Hence $\wp[s]\{\Omega_E[s],\Omega_E[s]\}_{\hat{\J}}=0$ iff $\{(y_A-g_A(u^i))\mu,(y_B-g_B(u^i))\mu\}_{\J} = 0$, for every $A$ and $B$.
	The last condition means exactly that the image of $s$ is coisotropic in $(M,L,\J)$.
\end{proof}

\begin{proof}[Proof of Theorem~\ref{theor:existence_BRST-charge} (resp.~Theorem~\ref{theor:uniqueness_BRST-charge})]
	It follows immediately as a special case of Proposition~\ref{prop:SBSO_existence} (resp.~Proposition~\ref{prop:SBSO_uniqueness}).
	It will be enough to use the contraction data~\eqref{eq:2contraction_data_1} for the contraction data~\eqref{eq:homotopy_equivalence}, and set $\calF_n:=\calL_{\geq n+1}$, $N=-1$, and $\bar{Q}:=\Omega_E[s]$.
	Indeed, in this case, from Lemma~\ref{lem:BRST_coisotropic} the necessary and sufficient condition~\eqref{eq:prop:SBSO_existence} coincides with $\im(s)$ being coisotropic.
\end{proof}

\section{The BFV-complex of a coisotropic submanifold}

\label{sec:BFV-complex}

Let $(M,L,\J)$ be a Jacobi manifold, and let $S\subset M$ be a coisotropic submanifold.
Recall that $\pi:NS\to S$ denotes the normal bundle to $S$ in $M$, $\ell:=L|_S\to S$ is the restricted line bundle, and we have set $L_{NS}:=\pi^\ast\ell\to NS$.
We will use a fat tubular neighborhood $(\tau,\underline{\smash{\tau}})$ of $\ell\to S$ in $L\to M$ to identify $S$ with the image of the zero section of $\pi$, and replace the Jacobi manifold $(M,L,J)$ with its local model $(NS,L_{NS},\tau^\ast J)$ around $S$.
We will then use the lifting procedure of Section~\ref{sec:lifting_Jacobi_structures}, and the results of Section~\ref{sec:BRST-charges} with the r\^ole of $M=\calE\to S$ and $E\to M=\calE$ being now played by the normal bundle $\pi:NS\to S$ and the vertical bundle $V(NS):=\pi^\ast(NS)\to NS$ respectively.
\begin{definition}
	\label{def:BFV_complex}
	A \emph{BFV-complex} (attached to $S$ via the fat tubular neighborhood $(\tau,\underline{\smash{\tau}})$) is a differential graded Lie algebra $(\Gamma(\hat{L}),\{-,-\}_{\BFV},d_{\BFV})$ such that:
	\begin{itemize}
		\item $\{-,-\}_{\BFV}\equiv \hat{\J}$, for some lifting $\hat{\J}$ of $\J$ to a Jacobi structure on $\hat{L}\to\hat{M}$,
		\item $d_{\BFV}=\{\Omega_{\BRST},-\}_{\BFV}$, where $\Omega_{\BRST}$ is some $0$-BRST charge wrt $\hat{\J}$.
	\end{itemize}
\end{definition}
\begin{remark}
	\label{rem:BFV_complex}
	In more geometric terms, a BFV-complex can be seen, in particular, as the graded Jacobi manifold $(\hat{M},\hat{L},\hat{\J}\equiv\{-,-\}_{\BFV})$ further equipped with the cohomological Hamiltonian derivation $d_{\BFV}$.
\end{remark}
This section aims to show that the BFV-complex is actually independent, to some extent, and up to isomorphisms, of the fat tubular neighborhood, it is a cohomological resolution of the reduced Gerstenhaber-Jacobi algebra of $S$, and controls the formal coisotropic deformation problem of $S$ under both Hamiltonian and Jacobi equivalence.

\subsection{Gauge invariance of the BFV-complex}
\label{subsec:gauge_invariance_BFV}

Let $(M,L,\J)$ be a Jacobi manifold, and let $S\subset M$ be a coisotropic submanifold.
The BFV-complex of $S$ is actually independent of the choice of the (fat) tubular neighborhood, at least around $S$, as pointed out by the following

\begin{theorem}
	\label{theor:gauge_invariance_BFV_complex}
	Let $(\tau_0,\underline{\smash{\tau}}_0)$ and $(\tau_1,\underline{\smash{\tau}}_1)$ be fat tubular neighborhoods of $\ell\to S$ in $L\to M$, and set $\J^0:=\tau_0^\ast\J$ and $\J^1:=\tau_1^\ast\J$.
	Pick liftings $\hat{\J}^i$ of $\J^i$ to $\hat{L}\to\hat{M}$, and let $\Omega^i_{\BRST}$ be a $0$-BRST charge wrt $\hat\J^i$, with $i=0,1$.
	Then there exist open neighborhoods $U_0$ and $U_1$ of $S$ in $NS$, and a degree $0$ graded Jacobi bundle isomorphism $\phi:(\hat M,\hat L,\hat\J^0)|_{U_0}\longrightarrow (\hat M,\hat L,\hat\J^1)|_{U_1}$, such that $\phi^\ast(\Omega^1_{\BRST})=\Omega^0_{\BRST}$, and a fortiori $\phi^\ast d_{\BFV}^1=d_{\BFV}^0$.
\end{theorem}

\begin{proof}
	The main idea of the proof is to show the existence of open neighborhoods $U_0$ and $U_1$ of $S$ in $NS$, and a bi-degree $(0,0)$ graded line bundle isomorphism $\phi$ from $\hat L|_{U_0}\to\hat M|_{U_0}$ to $\hat L|_{U_1}\to\hat M|_{U_1}$, such that $\phi^\ast\hat{\J}^1$ is a lifting of $\J^0$ to $\hat L|_{U_0}\to\hat M_{U_0}$, and $\phi^\ast(\Omega^1_{\BRST})$ is a $0$-BRST charge wrt $\phi^\ast\hat{\J}^1$.
	After doing this, the proof will be completed using Theorems~\ref{theor:uniqueness_BFV_brackets} and~\ref{theor:uniqueness_BRST-charge}.
	
	As shown in Proposition~\ref{prop:isotopy}, the standard uniqueness of tubular neighborhoods up to isotopy (cf.~\cite[Chapter 4, Theorem 5.3]{hirsch}) can be adapted to fat tubular neighborhoods.
	Accordingly it will be enough to consider the following two special cases:
	\begin{enumerate}
		\item $\tau_1\circ F=\tau_0$, for some automorphism $F$ of the line bundle $L_{NS}\to NS$, covering an automorphism $\underline{\smash{F}}$ of the normal bundle $NS\overset{\pi}{\to} S$, such that $F=\id$ on $L_{NS}|_S\simeq\ell$,
		\item $\tau_0=\calT_0$ and $\tau_1=\calT_1$, for some smooth path $\{(\calT_t,\underline{\smash{\calT}}_t)\}_{t\in I}$ of fat tubular neighborhoods of $\ell\to S$ in $L\to M$.
	\end{enumerate}
	
	{\sc First case.}
	Let $\underline{F}^\dagger:N^\ast S\to N^\ast S$ be the inverse of the transpose of the vector bundle automorphism $\underline{F}:NS\to NS$.
	There is a bi-degree $(0,0)$ automorphism $\calF$ of the graded line bundle $\hat{L}\to\hat{M}$ uniquely determined by
	\begin{equation}
	\label{eq:proof:gauge_invariance1}
	\calF^\ast\lambda=F^\ast\lambda,\quad\calF^\ast(\pi^\ast\eta)=\pi^\ast(\underline{\smash{F}}^\ast\eta),\quad\calF^\ast((\pi^\ast\alpha)\otimes\lambda)=\pi^\ast((\underline{\smash{F^\dagger}})^\ast\alpha)\otimes F^\ast\lambda,
	\end{equation}
	for all $\lambda\in\Gamma(\hat{L})^{(0,0)}=\Gamma(L)$, $\eta\in\Gamma(NS)$, and $\alpha\in\Gamma(N^\ast S)$.
	By its very construction, $\calF$ satifies:
	\begin{equation*}
	\calF^\ast\Omega_{NS}=\Omega_{NS},\quad \calF^\ast G=G,\quad p\circ\calF^\ast=F^\ast\circ p.
	\end{equation*}
	It follows that $\calF^\ast\hat{\J}^1$ is a lifting of $\J^0$ to $\hat{L}\to\hat{M}$, and $\calF^\ast\Omega_{\BRST}^1$ is a $0$-BRST charge wrt $\calF^\ast\hat{\J}^1$.
	
	{\sc Second case.}
	We can find an open neighborhood $V$ of $S$ in $NS$, and a smooth path $\{(F_t,\underline{\smash{F}}_t)\}_{t\in I}$ of line bundle embeddings of $L_{NS}|_V\to V$ into $L_{NS}\to NS$ such that
	\begin{itemize}
		\item $\calT_0=\calT_t\circ F_t$, so that, in particular, $F_0=\id$ on $L_{NS}|_V$,
		\item $F_t$ agrees with the identity map on $L_{NS}|_S\simeq\ell$.
	\end{itemize}
	Consequently, $\J^t:=\calT_t^\ast\J$ is a Jacobi structure, and the image of the zero section of $\pi:NS\to S$ is coisotropic wrt $\J^t$.
	Additionally, $(F_t)_\ast\J^0=\J^t$ and $F_t(\im 0)=\im 0$, for all $t\in I$.
	Hence, in view of Proposition~\ref{prop:main_result}, $\{F_t\}_{t\in I}$ can be lifted to a smooth path $\{\calF_t\}_{t\in I}$ of bi-degree $(0,0)$ graded line bundle embeddings of $\hat{L}|_V\to\hat{M}|_V$ into $\hat{L}\to\hat{M}$ such that, for all $t\in I$,
	\begin{itemize}
		\item $(\calF_t)_\ast\hat{\J}^0$ is a lifting of $\J^t$ to $\hat{L}|_{F_t(V)}\to\hat{M}|_{F_t(V)}$,
		\item $(\calF_t)_\ast\Omega_{\BRST}^0$ is a $0$-BRST charge wrt $(\calF_t)_\ast\hat{\J}^0$.
	\end{itemize}
	In particular, $U_0:=V$ and $U_1:=F_1(V)$ are open neighborhoods of $S$ in $NS$, and $\phi:=\calF_1$ is a bi-degree $(0,0)$ graded line bundle isomorphism from $\hat{L}|_{U_0}\to\hat{M}|_{U_0}$ to $\hat{L}|_{U_1}\to\hat{M}|_{U_1}$ with all the desired properties.
\end{proof}

\subsection{The BFV-complex and the cohomological Jacobi reduction of a coisotropic submanifold}
\label{subsec:BFV_homological_resolution}

Let $\J$ be a Jacobi structure on the line bundle $L\to M$, and let $S\subset M$ be a coisotropic submanifold wrt $\J$.
According to Proposition~\ref{prop:conormal}, this means that $(N_\ell{}^\ast S,\ell)$ is a Jacobi subalgebroid of $(J^1L,L)$.
Set $\frakg(S):=\Gamma(\wedge^\bullet(N_\ell S)\otimes\ell)$, and denote by $d_{N_\ell{}^\ast S,\ell}:\frakg(S)\to\frakg(S)$ the de Rham differential of the Jacobi algebroid $(N_\ell{}^\ast S,\ell)$.
Differential $d_{N_\ell{}^\ast S,\ell}$ is completely determined by
\begin{equation*}
d_{N_\ell{}^\ast S,\ell}\circ P=P\circ d_{\J},
\end{equation*}
where the degree $0$ graded module epimorphism  $P:\Diff^\star(L[1])\to\frakg(S)[1]$ is the canonical projection defined by setting $P(\lambda)=\lambda|_S$, $\langle P(\square),(df\otimes\lambda)|_S\rangle=\square(f\lambda)|_S$, for all $\lambda\in\Gamma(L)$, and $f\in I_S$.
Recall that here, as throughout this thesis, $I_S\subset C^\infty(M)$ denotes the ideal of functions vanishing on $S$.
In the following we will understand the module isomorphism $\Gamma(L_{\textnormal{red}})\overset{\simeq}{\longrightarrow}H^0(\frakg(S),d_{N_\ell{}^\ast S,\ell})$, $\lambda+\Gamma_S\longmapsto[\lambda|_S]$, introduced in Section~\ref{sec:Jacobi_reduction}.

\begin{proposition}
	\label{prop:BFV_homological_resolution}
	For every BFV-complex $(\Gamma(\hat{L}),\{-,-\}_{\BFV},d_{\BFV})$ of $S$, its cohomology is canonically isomorphic to the de Rham cohomology of the Jacobi algebroid of $S$
	\begin{equation*}
	H^\bullet(\Gamma(\hat{L}),d_{\BFV})\simeq H^\bullet(\frakg(S),d_{N_\ell{}^\ast S,\ell}).
	\end{equation*}
	Specifically, from Section~\ref{subsec:second_relevant_contraction_data}, the map $\wp[0]$ is a quasi-isomorphism from $(\Gamma(\hat{L}),d_{\BFV})$ to $(\frakg(S),d_{N_\ell{}^\ast S,\ell})$.
\end{proposition}

\begin{proof}
	Let $(\Gamma(\hat{L}),\{-,-\}_{\BFV},d_{\BFV})$ be a BFV-complex of $S$.
	Assume that this latter has been constructed by choosing a fat tubular neighborhood $(\tau,\underline{\smash{\tau}})$ of $\ell\to S$ into $L\to M$, a lifting of $\J$ to a Jacobi structure $\hat{\J}=\sum_{k=0}^\infty\hat{\J}_k$ on $\hat{L}\to\hat{M}$, and a $0$-BRST charge $\Omega_{\BRST}=\sum_{k=0}^\infty\Omega_k$ wrt $\hat{\J}$.
	
	Perturbation $\delta:=d_{\BFV}-d[0]$ of the contraction data~\eqref{eq:2contraction_data_1} is small (cf.~\cite[Section 2.2]{Crainic_perturbation}).
	Indeed
	\begin{align}
	\label{eq:prop:lifting_CE_cohom1}
	\delta\in\bigoplus_{k\geq 0}\Diff(\hat{L})^{(k+1,k)},\qquad
	\delta h[0]\in\bigoplus_{k\geq 1}\End(\hat{L},\hat{L})^{(k,k)},
	\end{align}
	so that $\delta h[0]$ is nilpotent and $\id-\delta h[0]$ is invertible with $(\id-\delta h[0])^{-1}=\sum_{k=0}^\infty(\delta h[0])^k$.
	Hence the Homological Perturbation Lemma (cf., e.g.,~Appendix~\ref{sec:homological_perturbation_lemma}) can be applied taking as input the contraction data~\eqref{eq:2contraction_data_1} and the small perturbation $\delta$.
	The resulting output is given by a new deformed set of contraction data
	\begin{equation*}
	\begin{tikzpicture}[>= stealth,baseline=(current bounding box.center)]
	\node (u) at (0,0) {$(\Gamma(\hat{L}),d_{\BFV})$};
	\node (d) at (5,0) {$(\frakg(S),d')$};
	\draw [transform canvas={yshift=-0.5ex},-cm to] (d) to node [below] {\footnotesize $\iota'$} (u);
	\draw [transform canvas={yshift=0.5ex},cm to-] (d) to node [above] {\footnotesize $\wp[0]'$} (u);
	\draw [-cm to] (u.south west) .. controls +(210:1) and +(150:1) .. node[left=2pt] {\footnotesize $h[0]'$} (u.north west);
	\end{tikzpicture}.
	\end{equation*}
	From the explicit formulas for $d',\wp[0]',\iota',h[0]'$ (see Equations~\eqref{eq:rem:homological_perturbation_lemma} and~\eqref{eq:prop:homological_perturbation_lemma}), and the very definition of $\wp[0]$, it follows that
	\begin{equation*}
	\wp[0]'=\wp[0],\qquad d'=\wp[0]\delta\iota.
	\end{equation*}
	Now it remains to prove that $d'=d_{N_\ell{}^\ast S,\ell}$.
	Both $d_{N_\ell{}^\ast S,\ell}$ and $d'$ are cohomological derivations of the graded module $\frakg(S)$, hence it is enough to check that they coincide on (local) generators, i.e.~on:
	\begin{enumerate}
		\item arbitrary sections $\lambda$ of $\ell\to S$,
		\item elements of a local frame $\eta^A$ of $NS\to S$.
	\end{enumerate}
	Since $\{\Omega_E,-\}_{\hat{\J}_1}+\{\Omega_1,-\}_G$ is the $(1,0)$ bi-degree component of $\delta$, 
	it follows that
	\begin{multline*}
	d'\lambda=(\wp[0]\delta\iota)\lambda=(\wp[0]\delta)(\pi^\ast\lambda)=\left.\{\pi^\ast\lambda,\Omega_E\}_{\hat{\J}_1}\right|_S=\\=\left.(\eta^A\otimes\mu^\ast)\{\pi^\ast\lambda,y_A\mu\}_{\J}\right|_S=(P\circ d_{\J})(\pi^\ast\lambda)=d_{N_\ell{}^\ast S,\ell}\lambda.
	\end{multline*}
	for all $\lambda\in\Gamma(\ell)$.
	Moreover,
	\begin{equation}
	\label{eq:proof:BFV_homological_resolution1}
	d'\eta^A=(\wp[0]\delta\iota)\eta^A=(\wp[0]\delta)\xi^A=\left.\left(\{\xi^A,\Omega_E\}_{\hat{\J}_1}+\{\xi^A,\Omega_1\}_G\right)\right|_S.
	\end{equation}
	From~\eqref{eq:rem:BRST_charge2}, with $\Omega=\Omega_{\BRST}$ and $h=1$, it follows that, locally,
	\begin{equation}
	\label{eq:proof:BFV_homological_resolution2}
	\left.\Delta_{y_A}\{\Omega_1,\Omega_E\}_{\hat{\J}_1}\right|_S=-2\left.\{\Omega_1,\xi^A\}_G\right|_S.
	\end{equation}
	Finally, plugging~\eqref{eq:proof:BFV_homological_resolution2} into~\eqref{eq:proof:BFV_homological_resolution1}, we get
	\begin{align*}
	d'\eta^A&=\frac{1}{2}\left.\ldsb\hat{\J}_1,\Delta_{y_A}\rdsb(\Omega_E,\Omega_E)\right|_S\\&=\frac{1}{2}\eta^A\eta^B\left.\ldsb\J,\Delta_{y_A}\rdsb(y_A\mu,y_B\mu)\right|_S\\&=(P\circ d_{\J})\Delta_{y_A}\\&=d_{N_\ell{}^\ast S,\ell}\eta_A.\qedhere
	\end{align*}
\end{proof}

It is now straightforward to see that each BFV-complex of $S$ provides a cohomological resolution of the reduced Gerstenhaber-Jacobi algebra of $S$.

\begin{corollary}
	\label{cor:BFV_homological_resolution}
	The degree $0$ graded module isomorphism $\wp[0]_\ast:H^\bullet(\Gamma(\hat{L}),d_{\BFV})\to H^\bullet(\frakg(S),d_{N_\ell{}^\ast S,\ell})$ intertwines the bracket induced by $\{-,-\}_{\BFV}$ on $H^0(\Gamma(\hat{L}),d_{\BFV})$ and the reduced Jacobi bracket $\{-,-\}_{\textnormal{red}}$ on $H^0(\frakg(S),d_{N_\ell{}^\ast S,\ell})\simeq\Gamma(L_{\textnormal{red}})$.
\end{corollary}

\begin{proof}
	Pick arbitrary $d_{\BFV}$-closed degree $0$ sections $\lambda_i=\sum_{k=0}^\infty\lambda_i^k\in\Gamma(\hat{L})^0$, with $\lambda_i^k\in\Gamma(\hat{L})^{(k,k)}$, for all $k\in\bbN_0$, and $i=1,2$.
	From the construction of $\wp:=\wp[0]$ and the BFV-complex, it follows that
	\begin{align*}
	\wp_\ast\left[\{\lambda_1,\lambda_2\}_{\BFV}\right]&=\left[\wp\left(\{\lambda_1,\lambda_2\}_{\hat{\J}}\right)\right]\\&=\left[\left.\{\lambda^0_1,\lambda^0_2\}_{\J}\right|_S\right]\\&=\left\{\left[\left.\lambda^0_1\right|_S\right],\left[\left.\lambda^0_2\right|_S\right]\right\}_{\textnormal{red}}\\&=\left\{\wp_\ast[\lambda_1],\wp_\ast[\lambda_2]\right\}_{\textnormal{red}}.\qedhere
	\end{align*}
\end{proof}

Hence the reduced Gerstenhaber-Jacobi algebra of a coisotropic submanifold $S$ admits two different cohomological resolutions: one provided by the BFV-complex and another one given by the $L_\infty$-algebra.
Actually, as shown in the next section, these two resolutions are strictly related.

\subsection{The BFV-complex and the \texorpdfstring{$L_\infty$}{L-infinity}-algebra of a coisotropic submanifold}
\label{subsec:L_infty_qi}

Let $\J$ be a Jacobi structure on a line bundle $L\to M$, and let $S\subset M$ be a coisotropic submanifold of $(M,L,\J)$.
Additionally, let $(\Gamma(\hat{L}),\{-,-\}_{\BFV},d_{\BFV})$ be a BFV-complex associated with $S$ via the choice of:
\begin{itemize}
	\item a fat tubular neighborhood $(\tau,\underline{\tau})$ of $\ell\to S$ into $L\to M$,
	\item a lifting of $\J$ to a Jacobi structure $\hat{\J}$ on $\hat{L}\to\hat{M}$,
	\item a $0$-BRST charge $\Omega_{\BRST}$ wrt $\hat{\J}$.
\end{itemize}
In view of the proof of Proposition~\ref{prop:BFV_homological_resolution}, after these choices, the corresponding set of contraction data~\eqref{eq:2contraction_data_1} gets deformed into a new set of contraction data
\begin{equation}
\label{eq:deformed_second_contraction_data}
\begin{tikzpicture}[>= stealth,baseline=(current bounding box.center)]
\node (u) at (0,0) {$(\Gamma(\hat{L}),d_{\BFV})$};
\node (d) at (5,0) {$(\frakg(S),d_{N_\ell{}^\ast S,\ell})$};
\draw [transform canvas={yshift=-0.5ex},-cm to] (d) to node [below] {\footnotesize $\iota'$} (u);
\draw [transform canvas={yshift=0.5ex},cm to-] (d) to node [above] {\footnotesize $\wp[0]$} (u);
\draw [-cm to] (u.south west) .. controls +(210:1) and +(150:1) .. node[left=2pt] {\footnotesize $h[0]'$} (u.north west);
\end{tikzpicture},
\end{equation}
with $\iota'=\sum_{k=0}^\infty(h[0]\delta)^k\iota$, and $h[0]'=\sum_{k=0}^\infty h[0](\delta h[0])^k$, where $\delta$ is the ``small'' deformation of $d[0]$ given by $\delta=d_{\BFV}-d[0]$.

As described in Section~\ref{sec:linfty_algebra}, a fat tubular neighborhood $(\tau,\underline{\tau})$ allows to construct a right inverse of the canonical projection $P$ as the unique degree $0$ graded module morphism $I:\frakg(S)[1]\to\Diff^\star(L[1])$ such that $I(\lambda)=\pi^\ast\lambda$, and $I(\eta)(f\otimes\lambda)=((\pi^\ast\eta) f)\otimes\lambda$, for all $\lambda\in\Gamma(\ell)$, $\eta\in\Gamma(NS)$, and $f\in C^\infty(NS)$.
Here $\pi^\ast\eta$ is the vertical lift of $\eta$: the unique vertical vector field on $NS$ which is constant along the fibers and agree with $\eta$ along $S$.
The quadruple $(\Diff^\star(L[1]),\im I,P,\J)$ is a set of $V$-data (see Lemma~\ref{lem:linfty}), so that, according to Proposition~\ref{prop:linfty}, there is a $L_\infty$-algebra structure $\{\widetilde{\frakm}_k\}_{k\geq 1}$ on $\frakg(S)$ given by higher derived brackets 
\begin{equation*}
\widetilde{\frakm}_k(sg_1,\ldots,sg_k)=(-)^\sharp s\left(P\ldsb\ldsb\ldots\ldsb\J,Ig_1\rdsb,\ldots\rdsb,Ig_k\rdsb\right),
\end{equation*}
for all homogeneous $g_1,\ldots,g_k\in\frakg(S)$, where $s:\frakg(S)\to\frakg(S)[1]$ is the suspension map, and $(-)^\sharp$ denotes a certain sign coming from d\'ecalage.
It is easy to see that $\widetilde{\frakm}_1$ coincides with the de Rham differential of the Jacobi algebroid $(N_\ell{}^\ast S,\ell)$ associated with $S$.
Hence, from Section~\ref{sec:Jacobi_reduction}, it follows that the $L_\infty$-algebra is a cohomological resolution of the reduced Gerstenhaber-Jacobi algebra of $S$.

The following theorem constructs, by homotopy transfer, an $L_\infty$-quasi-isomorphism between the BFV-complex and the $L_\infty$-algebra.
In this way we extend, from the Poisson to the Jacobi case, a result by Sch\"atz~\cite{schatz2009bfv}.
\begin{theorem}
	\label{theor:L_infty_qi}
	For every BFV-complex $(\Gamma(\hat{L}),\{-,-\}_{\BFV},d_{\BFV})$ of $S$ there exists a quasi-isomorphism of $L_\infty$-algebras
	\begin{equation*}
	(\frakg(S),\{\widetilde{\frakm}_k\}_{k\geq 1})\longrightarrow(\Gamma(\hat{L}),\{-,-\}_{\BFV},d_{\BFV})
	\end{equation*}
	which lifts $\iota'$, i.e.~such that its first Taylor coefficient coincides with $\iota'$.
\end{theorem}

\begin{proof}
	It is an adaptation of the proof of the analogous theorem in~\cite{schatz2009bfv} and we reproduce it here for completeness and the reader's convenience.
	
	In view of Theorems~\ref{theor:uniqueness_BFV_brackets} and~\ref{theor:uniqueness_BRST-charge}, we can assume that $\hat{\J}$ and $\Omega_{\BRST}$ have been constructed as in the proofs of Theorem~\ref{theor:existence_BFV_brackets} (after the choice of a $\Diff L$-connection $\nabla$ in $E\to M$), and Theorem~\ref{theor:existence_BRST-charge}.
	The homotopy transfer along contraction data~\eqref{eq:2contraction_data_1} of the differential graded Lie algebra structure $(\{-,-\}_{\BFV},d_{\BFV})$ on $\Gamma(\hat{L})$ produces as output:
	\begin{itemize}
		\item  an $L_\infty$-algebra structure $\{\frakl_k\}_{k\geq 1}$ on $\frakg(S)$, such that $\frakl_1=d_{N_\ell{}^\ast S,\ell}$, and
		\item an $L_\infty$-quasi-isomorphism $\iota'_\infty:(\frakg(S),\{\frakl_k\}_{k\geq 1})\longrightarrow(\Gamma(\hat{L}),\{-,-\}_{\BFV},d_{\BFV})$, with $(i'_\infty)_1=\iota'$,
	\end{itemize}
	(for more details about the Homotopy Transfer Theorem see, e.g.,~\cite[Section 10.3]{loday2012algebraic},~\cite[Section~4]{fiorenza2007structures} and~\cite[Section~2.3]{schatz2009bfv}).
	Moreover the Taylor coefficients of the $L_\infty$-algebra structure and the $L_\infty$-quasi-isomorphism are explicitly given by sums over bivalent oriented rooted trees
	\begin{align}
	\frakl_1&=d_{N_\ell{}^\ast S,\ell},&\frakl_k&=\sum_{T\in|\calT[k]|}\wp[0]\circ Z_T\circ(\iota')^{\otimes k},\label{eq:L_infty_algebras_trees}\\
	(\iota'_\infty)_1&=\iota',&(\iota'_\infty)_k&=\sum_{T\in|\calT[k]|}h[0]'\circ Z_T\circ(\iota')^{\otimes k},\label{eq:L_infty_q_i_trees}
	\end{align}
	where $\left|\calT[k]\right|$ denotes the set of bivalent oriented rooted trees with $k$ leaves, and we associated a degree $(2-k)$ graded skew-symmetric $\bbR$-linear map $Z_T:\Gamma(\hat{L})^{\otimes k}\to\Gamma(\hat{L})$ with every $T\in\left|\calT[k]\right|$.
	The latter is obtained by the standard operadic rules, after decorating $T$ as follows: we attach 1) $h[0]'$ to internal edges, 2) $\{-,-\}_{BFV}$ to internal (bivalent) vertices, and 3) $\id$ to external vertices.
	
	The main idea of the proof is to show that $\frakl_k=-\widetilde{\frakm}_k$, for all $k\geq 1$.
	By construction, both $-\widetilde{\frakm}_k$ and $\frakl_k$ are degree $(2-k)$ graded skew-symmetric first order multi-differential operators from $\frakg(S)$ to $\frakg(S)$.
	Hence it is enough to check that they agree on both sections of $\ell\to S$ and sections of $NS\to S$.
	Moreover, by skew-symmetry and degree reasons,	only the following special cases have to be checked
	\begin{equation}
	\label{eq:comparison_multibrackets1}
	\begin{aligned}
	\frakl_k(s_1,\ldots,s_{k-2},\lambda,\nu)&=-\widetilde{\frakm}_k(s_1,\ldots,s_{k-2},\lambda,\nu),\\
	\frakl_k(s_1,\ldots,s_{k-1},\lambda)&=-\widetilde{\frakm}_k(s_1,\ldots,s_{k-1},\lambda),\\
	\frakl_k(s_1,\ldots,s_k)&=-\widetilde{\frakm}_k(s_1,\ldots,s_k),\\
	\end{aligned}
	\end{equation}
	for all $k\geq 1$, $\lambda,\nu\in\Gamma(\ell)$, and $s_1,\ldots,s_k\in\Gamma(NS)$.
	
	In an arbitrary tree $T\in|\calT[k]|$ there are exactly $k-1$ internal vertices, and $k-2$ internal edges.
	By construction $\wp[0]$ vanishes on $\calL_{\geq 1}$, and the image of $\iota'$ is an abelian subalgebra of $(\Gamma(\hat{L}),\{-,-\}_G)$.
	Moreover, because of their series expansions, $\{-,-\}_{\BFV}$ decreases the anti-ghost degree at most by one, whereas $h[0]'$ increases the anti-ghost degree at least by one.
	It follows that there is a unique $T_k^0\in|\calT[k]|$ giving a non-zero contribution to~\eqref{eq:L_infty_algebras_trees}: it is exactly the bivalent oriented rooted tree, with $k$ leaves, such that each one of its internal vertices is the endpoint of at least one tail.
	Moreover the contribution of $T_k^0$ to~\eqref{eq:L_infty_algebras_trees} splits into two pieces
	\begin{equation}
	\label{eq:L_infty_q_i_reduced}
	\frakl_k=\wp[0]\circ Z_{T_k^0}\circ(\iota')^{\otimes k}=\frakl_k^L+\frakl_k^R.
	\end{equation}
	where the degree $(2-k)$ graded skew-symmetric $\bbR$-linear maps $\frakl_k^L,\frakl_k^R:\frakg(S)^{\otimes k}\to\frakg(S)$ are obtained by standard operadic rules after decorating $T_k^0$ in the following two ways:
	\begin{equation*}
	\begin{tikzpicture}[baseline=(current bounding box.center)]
	\draw (0,0) -- (1,0.5);
	\draw (0,1) -- (1,0.5);
	\draw (1,0.5) -- node[very near start, above] {$h[0]$} (2,1);
	\draw (2,1) --(1,1.5);
	\draw[dotted] (2,1) -- (3,1.5);
	\draw (3,1.5) -- (2,2);
	\draw (3,1.5) -- node[very near start, above] {$h[0]$} (4,2);
	\draw (3,2.5) --(4,2);
	\draw (4,2) -- (5,2.5);
	\node[left] at (0,0) {$\wp[0]$};
	\node[below right] at (1,0.5) {$G$};
	\node[left] at (0,1) {$\iota$};
	\node[below right] at (2,1) {$G$};
	\node[left] at (1,1.5) {$\iota$};
	\node[below right] at (3,1.5) {$G$};
	\node[left] at (2,2) {$\iota$};
	\node[below right] at (4,2) {$G$};
	\node[left] at (3,2.5) {$\iota$};
	\node[right] at (5,2.5) {$h[0]d_{\BFV}^{(1,0)}\iota$};
	\draw (7,0) -- (8,0.5);
	\draw (7,1) -- (8,0.5);
	\draw (8,0.5) -- node[very near start, above] {$h[0]$} (9,1);
	\draw (9,1) --(8,1.5);
	\draw[dotted] (9,1) -- (10,1.5);
	\draw (10,1.5) -- (9,2);
	\draw (10,1.5) -- node[very near start, above] {$h[0]$} (11,2);
	\draw (10,2.5) --(11,2);
	\draw (11,2) -- (12,2.5);
	\node[left] at (7,0) {$\wp[0]$};
	\node[below right] at (8,0.5) {$G$};
	\node[left] at (7,1) {$\iota$};
	\node[below right] at (9,1) {$G$};
	\node[left] at (8,1.5) {$\iota$};
	\node[below right] at (10,1.5) {$G$};
	\node[left] at (9,2) {$\iota$};
	\node[below right] at (11,2) {$\hat{\J}_1$};
	\node[left] at (10,2.5) {$\iota$};
	\node[right] at (12,2.5) {$\iota$};
	\end{tikzpicture}
	\end{equation*}
	
	As a direct consequence of the local coordinate formulas for $h[0]$, as given in Section~\ref{subsec:second_relevant_contraction_data}, we get
	\begin{multline*}
	\wp[0]\{\pi^\ast\eta_1,h[0]\{\pi^\ast\eta_2,\ldots,h[0]\{\pi^\ast\eta_k,h[0]X\}_G\ldots\}_G\}_G\\=\frac{(-)^k}{k!}\left.\left(\bbD_{\pi^\ast\eta_1}\cdots\bbD_{\pi^\ast\eta_k}\right)(X)\right|_S\in\frakg^d(S),
	\end{multline*}  
	for all $\eta_1,\ldots,\eta_k\in\Gamma(\pi)$, and $X\in\Gamma(\hat{L})^{(d,0)}$.
	Hence, it follows from~\eqref{eq:L_infty_q_i_reduced} that
	\begin{equation}
	\label{eq:comparison_multibrackets2}
	\begin{aligned}
	\frakl_k(s_1,\ldots,s_{k-2},\lambda,\nu)&\!=\!\left.(-)^k\left(\bbD_{s_1}\cdots\bbD_{s_{k-2}}\{\lambda,\nu\}_{\J}\right)\right|_S,\\
	\frakl_k(s_1,\ldots,s_{k-1},\lambda)&=(-)^{k+1}\Bigg(\bbD_{s_1}\cdots\bbD_{s_{k-1}}\{\Omega_E,\lambda\}_{\hat{\J}_1}\\&\phantom{=}-\sum_{i=1}^k\bbD_{s_1}\cdots\hat{\bbD_{s_i}}\cdots\bbD_{s_{k-1}}\{s_i,\lambda\}_{\hat{\J}_1}\Bigg)\!\Bigg|_S,\\
	\frakl_k(s_1,\ldots,s_k)&=(-)^{k+1}\Bigg(\sum_{i=1}^k\bbD_{s_1}\cdots\hat{\bbD_{s_i}}\cdots\bbD_{s_{k-1}}\left(\{\Omega_{\BRST}^{(2,1)},s_i\}_G+\{\Omega_E,s_i\}_{\hat{\J}_1}\right)\\
	&\phantom{=}-\sum_{i<j}\bbD_{s_1}\cdots\hat{\bbD_{s_i}}\cdots\hat{\bbD_{s_j}}\cdots\bbD_{s_k}\{s_i,s_j\}_{\hat{\J}_1}\Bigg)\!\Bigg|_S,
	\end{aligned}
	\end{equation}
	for all $k\geq 1$, $\lambda,\nu\in\Gamma(\ell)$, and $s_1,\ldots,s_k\in\Gamma(NS)$.
	Finally, since $\hat{\J}_1=i_\nabla\J$ and $\Omega_{\BRST}^{(2,1)}=\frac{1}{2}h[0]\{\Omega_E,\Omega_E\}_G$, the rhs of~\eqref{eq:comparison_multibrackets2} coincides exactly with the coordinate expressions for the rhs of~\eqref{eq:comparison_multibrackets1} as  explicitly given in Section~\ref{sec:multi-brackets_coordinates}.
	This concludes the proof.
\end{proof}

%
%

\section{The BFV-complex and the coisotropic deformation problem}
\label{sec:coisotropic_deformation_problem}

Let $(M,L,\J)$ be a Jacobi manifold, and let $S\subset M$ be a coisotropic submanifold.
Throughout this section we will assume to have fixed a fat tubular neighborhood $(\tau,\underline{\smash{\tau}})$ of the restricted line bundle $\ell:=L|_S\to S$.
Accordingly, $(\tau,\underline{\smash{\tau}})$ will be used to identify $S$ with the image of the zero section of the normal bundle $\pi:NS\to S$, and to replace the Jacobi manifold $(M,L,\J)$ with its local model $(NS,L_{NS},\tau^\ast\J)$ around $S$.
Furthermore let $(\Gamma(\hat{L}),d_{\BFV},\{-,-\}_{\BFV}\equiv\hat{\J})$ be a BFV-complex attached to the coisotropic submanifold $S$ via the fat tubular neighborhood $(\tau,\underline{\smash{\tau}})$.

The aim of this section is describing the r\^ole played by the BFV-complex of $S$ in the coisotropic deformation problem.
It will be accomplished in two steps: first at the formal level, and then at the non-formal one.
In Section~\ref{subsec:BFVinfinitesimal_and_formal_coisotropic_deformations} we will show that the BFV-complex encodes the moduli spaces, under Hamiltonian equivalence, of infinitesimal and formal coisotropic deformations of $S$.
Along the way, we also state necessary and sufficient conditions for the formal coisotropic deformation problem to be unobstructed in terms of the BFV-complex.
Later, in Section~\ref{subsec:BFVcoisotropic_sections_and_moduli_spaces}, we will show that the BFV-complex conveys the whole information about coisotropic sections (i.e.~$C^1$-small coisotropic deformations of $S$ lying within the tubular neighborhood), and their moduli spaces under Hamiltonian and Jacobi equivalence.

\subsection[Infinitesimal and formal coisotropic deformations]{Infinitesimal and formal coisotropic deformations: obstructions and moduli spaces}
\label{subsec:BFVinfinitesimal_and_formal_coisotropic_deformations}

Let us start recalling that every choice of a fat tubular neighborhood $(\tau,\underline{\smash{\tau}})$ canonically determines a set of contraction data
\begin{equation*}
\begin{tikzpicture}[>= stealth,baseline=(current bounding box.center)]
\node (u) at (0,0) {$(\Gamma(\hat{L}),d[0])$};
\node (d) at (5,0) {$(\frakg(S),0)$};
\draw [transform canvas={yshift=-0.5ex},-cm to] (d) to node [below] {\footnotesize $\iota$} (u);
\draw [transform canvas={yshift=0.5ex},cm to-] (d) to node [above] {\footnotesize $\wp[0]$} (u);
\draw [-cm to] (u.south west) .. controls +(210:1) and +(150:1) .. node[left=2pt] {\footnotesize $h[0]$} (u.north west);
\end{tikzpicture}
\end{equation*}
as special case for $s=0$ of Proposition~\ref{prop:2contraction_data}.
Going further, every choice of the BFV-complex $(\Gamma(\hat{L}),d_{BFV},\{-,-\}_{BFV})$ associated with $S$ via $(\tau,\underline{\smash{\tau}})$ canonically determines  a new deformed set of contraction data by means of the Homological Perturbation Lemma, as in the proof of Proposition~\ref{prop:BFV_homological_resolution},
\begin{equation*}
\begin{tikzpicture}[>= stealth,baseline=(current bounding box.center)]
\node (u) at (0,0) {$(\Gamma(\hat{L}),d_{\BFV})$};
\node (d) at (5,0) {$(\frakg(S),d_{N_\ell{}^\ast S,\ell})$};
\draw [transform canvas={yshift=-0.5ex},-cm to] (d) to node [below] {\footnotesize $\iota'$} (u);
\draw [transform canvas={yshift=0.5ex},cm to-] (d) to node [above] {\footnotesize $\wp[0]$} (u);
\draw [-cm to] (u.south west) .. controls +(210:1) and +(150:1) .. node[left=2pt] {\footnotesize $h[0]'$} (u.north west);
\end{tikzpicture}.
\end{equation*}
In particular, there is an isomorphism $H(\Gamma(\hat{L}),d_{\BFV})\simeq H(\frakg(S),d_{N_\ell{}^\ast S,\ell})$.
Hence
\begin{itemize}
	\item the moduli space of infinitesimal coisotropic deformations of $S$, under Hamiltonian equivalence, identifies with the first cohomology group of the BFV-complex,
	\item the obstructions to the prolongability of infinitesimal coisotropic deformations to formal ones live in the second cohomology group of the BFV-complex.
\end{itemize}
We summarize the situation in Propositions~\ref{prop:BFVinfinitesimal_moduli_space} and~\ref{prop:BFV_2nd_cohomology}.

\begin{proposition}
	\label{prop:BFVinfinitesimal_moduli_space}
	The first cohomology group $H^1(\Gamma(\hat{L}),d_{BFV})$ is canonically isomorphic to the moduli space of infinitesimal coisotropic deformations, under Hamiltonian equivalence, as follows
	\begin{equation*}
	\wp[0]_\ast:H^1(\Gamma(\hat{L}),d_{BFV})\hookrightarrow \hspace{-8pt} \rightarrow H^1(\frakg(S),d_{N_\ell{}^\ast S,\ell}).
	\end{equation*}
\end{proposition}

\begin{proof}
	It follows immediately from Proposition~\ref{prop:BFV_homological_resolution} and Corollaries~\ref{cor:inf1} and~\ref{cor:infequi}.
\end{proof}

\begin{proposition}
	\label{prop:BFV_2nd_cohomology}
	If the second cohomology group $H^2(\Gamma(\hat{L}),d_{BFV})$ vanishes, then the coisotropic deformation problem of $S$ is \emph{formally unobstructed}, i.e.~every infinitesimal coisotropic deformation of $S$ can be prolonged to a formal one.
\end{proposition}

\begin{proof}
	It is a straightforward consequence of Proposition~\ref{prop:BFV_homological_resolution} and Corollary~\ref{prop:rigid}.
\end{proof}

Now, every choice of the BFV-complex $(\Gamma(\hat{L}),d_{BFV},\{-,-\}_{BFV})$ canonically determines an $L_\infty$-quasi-isomorphism (see Theorem~\ref{theor:L_infty_qi}) 
\begin{equation}
\label{eq:L_infty_qi_deformations}
\iota'_\infty:(\frakg(S),\{\frakm_k\}_{k\geq 1})\longrightarrow(\Gamma(\hat{L}),d_{\BFV},\{-,-\}_{\BFV})
\end{equation}
which lifts $\iota'$, i.e.~such that its first Taylor coefficient $(\iota'_\infty)_1$ coincides with $\iota'$.
Consequently, from the very definition of $L_\infty$-quasi-isomorphism (cf.~\cite[Definition 5.2]{lada1995strongly}), $\iota'$ induces a Lie algebra isomorphism $\iota'_\ast$ in cohomology, i.e.
\begin{equation*}
[\{\iota'\varpi_1,\iota'\varpi_2\}_{BFV}]=[\iota'\frakm_2(\varpi_1,\varpi_2)],
\end{equation*}
for all $\varpi_1,\varpi_2\in\ker\frakm_1$.
In particular, we get the following commutative diagram
\begin{equation}
\label{eq:Kuranishi_commutative_diagram}
\begin{tikzcd}
H^1(\frakg(S),\frakm_1) \arrow[r, "\operatorname{Kr}"] \arrow[d, hook, two heads, "\iota'_\ast" swap]
& H^2(\frakg(S),\frakm_1) \arrow[d, hook, two heads, "\iota'_\ast"] \\
H^1(\Gamma(\hat{L}),d_{BFV}) \arrow[r, "\operatorname{Kr}" swap]
& H^2(\Gamma(\hat{L}),d_{BFV})
\end{tikzcd}
\end{equation}
where the upper row denotes the Kuranishi map for the $L_\infty$-algebra (see~\eqref{eq:L_infty_kuranishi_map}), and the lower row is the \emph{Kuranishi map} for the BFV-complex, namely
\begin{equation}
\label{eq:BFV_kuranishi_map}
\operatorname{Kr}:H^1(\Gamma(\hat{L}),d_{BFV})\longrightarrow H^2(\Gamma(\hat{L}),d_{BFV}),\ [\nu]\longmapsto [\{\nu,\nu\}_{BFV}].
\end{equation}

\begin{proposition}
	\label{prop:BFV_kuranishi}
	Let $\alpha=[\iota's]\in H^1(\Gamma(\hat{L}),d_{BFV})$, where $s\in\Gamma(NS)$ is an infinitesimal coisotropic deformation of $S$.
	If $\operatorname{Kr}(\alpha)\neq 0$, then $s$ is \emph{formally obstructed}, i.e.~it cannot be prolonged to a formal coisotropic deformation of $S$.
	In particular, if $\operatorname{Kr}\neq 0$, then the coisotropic deformation problem of $S$ is \emph{formally unobstructed}.
\end{proposition}

\begin{proof}
	It is a straightforward consequence of~\eqref{eq:Kuranishi_commutative_diagram} and Proposition~\ref{prop:Kuranishi}.
\end{proof}

By a central result in deformation theory (cf., e.g.,~\cite[Theorem 7.8]{doubek2007deformation} and~\cite[Theorem 4.6]{kontsevich2003deformation}), the $L_\infty$-quasi-isomorphism~\eqref{eq:L_infty_qi_deformations} transforms each formal MC element $\gamma=\sum_{k=1}^\infty\gamma_k\varepsilon^k$ of $(\frakg(S),\frakm_k)$ into a formal MC element $\operatorname{MC}(\iota'_\infty)\gamma$ of $(\Gamma(\hat{L}),d_{BFV},\{-,-\}_{BFV})$ defined via  
\begin{equation*}
\operatorname{MC}(\iota'_\infty)\gamma:=\sum_{n=1}^\infty\frac{1}{n!}{(\iota'_\infty)}_n(\gamma,\ldots,\gamma).
\end{equation*}
Moreover such map descends to the moduli spaces, under gauge equivalence.
Hence it gives rise to one-to-one correspondence between:
\begin{itemize}
	\item the moduli space, under gauge equivalence, of formal MC elements of the $L_\infty$-algebra of $S$, and
	\item the moduli space, under gauge equivalence, of formal MC elements of the BFV-complex of $S$.
\end{itemize}
As a consequence Propositions~\ref{prop:mcformal} and~\ref{prop:MC_gauge} lead immediately to the following.
\begin{proposition}
	\label{prop:BFVformal_moduli_space}
	The BFV-complex of $S$ controls the formal coisotropic deformation problem of $S$ under Hamiltonian equivalence.
	Indeed there is a canonical one-to-one correspondence between
	\begin{itemize}
		\item the moduli space, under Hamiltonian equivalence, of formal coisotropic deformations $s(\eps)$ of $S$, and
		\item the moduli space, under gauge equivalence, of formal MC elements $\Omega(\eps)$ of the BFV-complex,
	\end{itemize}
	which is established by mapping the equivalence class under Hamiltonian equivalence of $s(\eps)$ to the equivalence class under gauge equivalence of $\operatorname{MC}(\iota'_\infty)(-s(\eps))$.
\end{proposition}


\subsection{Coisotropic deformations and their moduli spaces}
\label{subsec:BFVcoisotropic_sections_and_moduli_spaces}

Let us fix some notation which will be used in the following.
We will denote by $C(L,\J)$ the set of coisotropic sections, i.e.~those sections $s$ of the normal bundle $\pi:NS\to S$ whose image is coisotropic.
We will denote by $\BRST(\hat{L},\hat{\J})$ the set of those $\Omega\in\Gamma(\hat{L})$ which are $s$-BRST charges wrt $\hat{\J}$, for some arbitrary $s\in\Gamma(\pi)$.
Elements of $\BRST(\hat{L},\hat{\J})$ will be simply called \emph{BRST charges}.

\begin{proposition}
	\label{prop:coisotropic_deformation_space} \ 
	\newline\noindent
	1) The space $\BRST(\hat{L},\hat{\J})$ is invariant under the natural action of $\Ham_{\geq 2}(\hat{M},\hat{L},\hat{\J})$ on $\Gamma(\hat{L})$ (see Remark~\ref{rem:filtration_Ham} for the meaning of $\Ham_{\geq 2}(\hat{M},\hat{L},\hat{\J})$).
	\newline\noindent
	2) For any $\Omega\in\BRST(\hat{L},\hat{\J})$, there is a unique $s_\Omega\in C(L,\J)$, such that $\Omega$ is an $s_\Omega$-BRST charge.
	Section $s_\Omega$ is implicitly determined by the following relation
	\begin{equation*}
	\im(s_\Omega)=\textnormal{``zero locus of $\operatorname{pr}^{(1,0)}\Omega$''}.
	\end{equation*}
	3) There is a canonical one-to-one correspondence between $\BRST(\hat{L},\hat{\J})/\Ham_{\geq 2}(\hat{M},\hat{L},\hat{\J})$ and $C(L,\J)$ mapping $\Ham_{\geq 2}(\hat{M},\hat{L},\hat{\J}).\Omega$ to $s_\Omega$.
\end{proposition}

\begin{proof}
	It is a straightforward consequence of Definition~\ref{def:BRST-charge}, and Proposition~\ref{theor:uniqueness_BRST-charge}.
\end{proof}

Now we introduce the notion of geometric MC element of the BFV-complex by slightly adapting the analogous one given by Sch\"atz in the Poisson case (see~\cite[Section~3.4]{schatz2011moduli}).
\begin{proposition/definition}
	\label{prop:geometric_MC}
	Let $\Omega$ be a MC element of $(\Gamma(\hat{L}),\{-,-\}_{\BFV})$, and let $s\in C(L,\J)$.
	The following conditions are equivalent:
	\begin{enumerate}
		\item there is $\Phi\in\Ham_{\geq 1}(\hat{M},\hat{L},\hat{\J})$ such that $\Phi^\ast(\Omega)$ is an $s$-BRST charge wrt $\hat{\J}$,
		\item there is a section $A\in\Gamma(\operatorname{GL}_+(E))$ such that $A^\ast(\operatorname{pr}^{(1,0)}\Omega)=\Omega_E[s]$.
	\end{enumerate}
	If the equivalent conditions (1)-(2) hold, for some $s\in C(L,\J)$, then $\Omega$ is said to be a \emph{geometric MC element} of $(\Gamma(\hat{L}),\{-,-\}_{\BFV})$.
	In this case, $s$ will be also denoted by $s_\Omega$, and it is completely determined by $\Omega$ through the relation:
	\begin{equation*}
	\im s_\Omega=\textnormal{``zero locus of $\operatorname{pr}^{(1,0)}\Omega$''}.
	\end{equation*}
\end{proposition/definition}

\begin{proof}
	Let $\{\lambda_t\}_{t\in I}\subset\calL^0_{\geq 1}$ and $\{a_t\}_{t\in I}\subset\Gamma(\hat{L})^{(1,1)}=\Gamma(\End(E))$ be smooth paths such that $a_t=\operatorname{pr}^{(1,1)}\lambda_t$.
	Since the $(0,0)$ bi-degree component of $\{\lambda_t,-\}_{\hat{\J}}$ reduces to $\{a_t,-\}_G$, whose symbol is zero, Lemma~\ref{lem:integrating_graded_do} guarantees that $\{\{\lambda_t,-\}_{\hat{\J}}\}_{t\in I}$ and $\{\{a_t,-\}_G\}_{t\in I}$ integrate to smooth paths $\{\Phi_t\}_{t\in I}\subset\Ham(\hat{M},\hat{L},\hat{\J})$ and $\{\Psi_t\}_{t\in I}\subset\Ham(\hat{M},\hat{L},G)^{(0,0)}$, respectively, so that
	\begin{equation}
	\label{eq:proof:geometric_MC_1}
	\Phi_t^\ast\lambda=\Psi_t^\ast\lambda\Mod\bigoplus_{k\geq 1}\Gamma(\hat{L})^{(p+k,q+k)},\quad\textnormal{for all}\ (p,q)\in\bbN_0^2,\ \lambda\in\Gamma(\hat{L})^{(p,q)}.
	\end{equation}
	Furthermore, from the very definition of $G$, it follows that
	\begin{equation}
	\label{eq:proof:geometric_MC_2}
	\Psi_t= S_{{\scriptscriptstyle C^\infty(M)}}((A_t\otimes\id_{L^\ast})\oplus A_t^\ast)\otimes\id_L,
	\end{equation}
	where $\{A_t\}_{t\in I}\subset\Gamma(\operatorname{GL}(E))$ is the smooth path, with $A_0=\id_E$, which integrates $\{a_t\}_{t\in I}\subset\Gamma(\End(E))$, and it is explicitly given by the time-ordered exponential
	\begin{align*}
	A_t&=\calT\exp\left(\int_0^t a_sds\right)\\
	&:=\sum_{n=0}^\infty\frac{1}{n!}\int_0^t\int_0^{s_{n-1}}\dots\int_0^{s_1}\left(a_{s_{n-1}}\circ a_{s_n}\circ\ldots\circ a_{s_0}\right)ds_{n-1}\dots ds_1ds_0.
	\end{align*}
	
	Finally, \eqref{eq:proof:geometric_MC_1} and~\eqref{eq:proof:geometric_MC_2} imply that $\operatorname{pr}^{(1,0)}(\Phi_t^\ast\Omega)=A_t^\ast(\operatorname{pr}^{(1,0)}\Omega)$, for every $\Omega\in\Gamma(\hat{L})^1$, which is enough to conclude the proof.
\end{proof}

In the following, we will denote by $\MC_{\textnormal{geom}}(\hat{L},\hat{\J})$ the set of geometric MC elements of $(\Gamma(\hat{L}),\{-,-\}_{\BFV})$.
According to Proposition~\ref{prop:geometric_MC}, $\MC_{\textnormal{geom}}(\hat{L},\hat{\J})$ coincides with the orbit described by $\BRST(\hat{L},\hat{\J})$ under the natural action of $\Ham_{\geq 1}(\hat{M},\hat{L},\hat{\J})$ on $\Gamma(\hat{L})$.
The following theorem shows that, as already in the Poisson case~\cite[Theorem 2]{schatz2011moduli}, also in the Jacobi setting, the BFV-complex of a coisotropic submanifold $S$ encodes small coisotropic deformations of $S$.

\begin{theorem}
	\label{theor:coisotropic_def_space}
	A 1-1 correspondence between $\MC_{\textnormal{geom}}(\hat{L},\hat{\J})/\Ham_{\geq 1}(\hat{M},\hat{L},\hat{\J})$ and $C(L,\J)$ is canonically defined by mapping $\Ham_{\geq 1}(\hat{M},\hat{L},\hat{\J}).\Omega$ to $s_\Omega$.
\end{theorem}
\begin{proof}
	It is a straightforward consequence of Propositions~\ref{prop:coisotropic_deformation_space} and~\ref{prop:geometric_MC}.
\end{proof}

The BFV-complex does also encode the (local) moduli space of coisotropic submanifolds under Hamiltonian equivalence.

\begin{definition}
	\label{def:Ham_homotopy_geometric_MC}
	A \emph{Hamiltonian homotopy of geometric MC elements} of graded Lie algebra $(\Gamma(\hat{L}),\{-,-\}_{\BFV})$ consists of
	\begin{itemize}
		\item a smooth path $\{\Omega_t\}_{t\in I}$ of geometric MC elements of $(\Gamma(\hat{L}),\{-,-\}_{\BFV})$, and
		\item a smooth path $\{\Phi_t\}_{t\in I}$ of Jacobi automorphisms of $(\hat{M},\hat{L},\hat{\J})$, with $\Phi_0=\id_{\hat{L}}$, which integrates $\{\lambda_t,-\}_{\hat{\J}}$, for some smooth path $\{\lambda_t\}_{t\in I}\subset\Gamma(\hat{L})^0$ (cf.~Definition~\ref{def:Hamiltonian_family}),
	\end{itemize}
	such that they are related by the compatibility condition $\Phi_t^\ast\Omega_t=\Omega_0$.
	Such Hamiltonian homotopy is said to \emph{interpolate} the geometric MC elements $\Omega_0$ and $\Omega_1$.
	If geometric MC elements $\Omega_0$ and $\Omega_1$ are interpolated by an Hamiltonian homotopy, then they are called \emph{Hamiltonian equivalent}, and we write $\Omega_0\sim_{\Ham}\Omega_1$.
	Indeed $\sim_{\Ham}$ is an equivalence relation on $\MC_{\textnormal{geom}}(\hat{L},\hat{\J})$, and the equivalence class of $\Omega$ is denoted by $[\Omega]_{\Ham}$.
\end{definition}

\begin{lemma}
	\label{lem:coisotropic_moduli_space}
	Let $s$ be an arbitrary coisotropic section.
	\newline\noindent
	1) Any two $s$-BRST charges wrt $\hat{\J}$ are Hamiltonian equivalent.
	\newline\noindent
	2) If $\Omega_0,\Omega_1\in\MC_{\textnormal{geom}}(\hat{L},\hat{\J})$, with $s_{\Omega_0}=s_{\Omega_1}=s$, then $\Omega_0\sim_{\Ham}\Omega_1$.
\end{lemma}

\begin{proof}
	It is a straightforward consequence of Theorem~\ref{theor:uniqueness_BRST-charge} and Proposition~\ref{prop:geometric_MC}.
\end{proof}

The analogous notion of Hamiltonian equivalence for coisotropic sections of a Jacobi manifold, within a fat tubular neighborhood, has already been introduced in Definition~\ref{def:hequi}~\ref{enumitem:def:hequi_1}.
We rephrase it here in a more convenient way for our current purposes.

\begin{definition*}
	A \emph{Hamiltonian homotopy of coisotropic sections} consists of
	\begin{itemize}
		\item a smooth path $\{s_t\}_{t\in I}\subset C(L,\J)$ of coisotropic sections of $(L,\J)$, and
		\item a smooth path $\{(F_t,\underline{\smash{F_t}})\}_{t\in I}$ of Jacobi automorphisms of $(M,L,\J)$, with $F_0=\id_L$, which integrates $\{\lambda_t,-\}_{\J}$, for some smooth path $\{\lambda_t\}_{t\in I}\subset\Gamma(L)$,
	\end{itemize}
	such that they are related by the compatibility condition $\im s_t=\underline{\smash{F_t}}(\im s_0)$.
	Such an Hamiltonian homotopy $\{(F_t,\underline{\smash{F_t}})\}_{t\in I}$ is said to \emph{interpolate} the coisotropic sections $s_0$ and $s_1$.
	If coisotropic sections $s_0$ and $s_1$ are interpolated by an Hamiltonian homotopy, then they are called \emph{Hamiltonian equivalent}, and we write $s_0\sim_{\Ham}s_1$.
	Indeed $\sim_{\Ham}$ is an equivalence relation on $C(L,\J)$, and the equivalence class of $s$ is denoted by $[s]_{\Ham}$.
\end{definition*}

The following theorem shows that, as already in the Poisson case~\cite[Theorem 4]{schatz2011moduli}, also in the Jacobi setting, the BFV-complex of a coisotropic submanifold $S$ encodes the local moduli space of $S$ under Hamiltonian equivalence.

\begin{theorem}
	\label{theor:Mod_Ham}
	A 1-1 correspondence between $\MC_{\textnormal{geom}}(\hat{L},\hat{\J})/{\sim_{\Ham}}$ and $C(L,\J)/{\sim_{\Ham}}$ is canonically defined by mapping $[\Omega]_{\Ham}$ to $[s_\Omega]_{\Ham}$.
\end{theorem}

\begin{proof}
	We have to prove that, for all $\Omega_0,\Omega_1\in\MC_{\textnormal{geom}}(\hat{L},\hat{\J})$, the following conditions are equivalent:
	\begin{itemize}
		\item[a)] $\Omega_0$ and $\Omega_1$ are Hamiltonian equivalent,
		\item[b)] $s_0:=s_{\Omega_0}$ and $s_1:=s_{\Omega_1}$ are Hamiltonian equivalent.
	\end{itemize}
	
	a) $\Longrightarrow$ b).
	Assume there is a Hamiltonian homotopy of geometric MC elements given by $\{\Omega_t\}_{t\in I}\subset\MC_{\textnormal{geom}}(\hat{L},\hat{\J})$ and $\{\Phi_t\}_{t\in I}\subset\Aut(\hat{M},\hat{L},\hat{\J})$, with
	\begin{equation*}
	\Phi_0=\id_{\hat{L}},\quad\frac{d}{dt}\Phi_t^\ast=\{\lambda_t,-\}_{\hat{\J}}\circ\Phi_t^\ast,
	\end{equation*}
	for some $\{\lambda_t\}_{t\in I}\subset\Gamma(\hat{L})^0$.
	The latter can be canonically projected onto a Hamiltonian homotopy of coisotropic sections of $(L,\J)$, given by $\{s_t\}_{t\in I}\subset C(L,\J)$ and $\{(F_t,\underline{\smash{F_t}})\}_{t\in I}\subset\Aut(M,L,\J)$, with
	\begin{equation*}
	F_0=\id_{\hat{L}},\quad\frac{d}{dt}F_t^\ast=\{\lambda_t,-\}_{\J}\circ F_t^\ast,
	\end{equation*}
	for some $\{\lambda_t\}_{t\in I}\subset\Gamma(L)$.
	Such projection is defined by setting
	\begin{equation*}
	s_t:=s_{\Omega_t},\quad \lambda_t:=\operatorname{pr}^{(0,0)}\lambda_t,\quad F_t^\ast=\operatorname{pr}^{(0,0)}\circ\Phi_t^\ast|_{\Gamma(\hat{L})^{(0,0)}}.
	\end{equation*}
	
	b) $\Longrightarrow$ a).
	Assume there is a Hamiltonian homotopy of coisotropic sections of $(L,\J)$ given by $\{s_t\}_{t\in I}\subset C(L,\J)$ and $\{(F_t,\underline{\smash{F_t}})\}_{t\in I}\subset\Aut(M,L,\J)$.
	In view of Lemma~\ref{lem:coisotropic_moduli_space}, it is enough to show that the latter can be lifted to a Hamiltonian homotopy of geometric MC elements wrt $\hat{\J}$ intertwining an $s_0$-BRST charge and an $s_1$-BRST charge.
	
	So, let $\nabla$ be a $\Diff(L)$-connection in $L\to M$ obtained by pulling back, along $NS\overset{\pi}{\to}S$, a $\Diff(\ell)$-connection in $NS\overset{\pi}{\to}S$.
	For an arbitrary $s_0$-BRST charge $\tilde\Omega_0$ wrt $\hat{\J}^\nabla$, Proposition~\ref{prop:main_result_bis} allows us to lift $\{F_t\}_{t\in I}$ to a smooth path $\{\calF_t\}_{t\in I}\subset\Aut(\hat{M},\hat{L},\hat{\J}^\nabla)$ such that
	\begin{equation*}
	\calF_0=\id_{\hat{L}},\quad\frac{d}{dt}\calF_t^\ast=\{\lambda_t,-\}_{\hat{\J}^\nabla}\circ\calF_t^\ast,
	\end{equation*}
	for some smooth path $\{\lambda_t\}_{t\in I}\subset\Gamma(\hat{L})^0$, and additionally $\tilde\Omega_t:=(\calF_t)_\ast\tilde\Omega_0$ is an $s_t$-BRST charge wrt $\hat{\J}^\nabla$.
	In view of Theorem~\ref{theor:uniqueness_BFV_brackets}, there is an automorphism $\calG$ of the graded line bundle $\hat{L}\to\hat{M}$ such that $\calG^\ast(\hat{\J}^\nabla)=\hat{\J}$, and $\Omega_t:=\calG^\ast\tilde\Omega_t$ is an $s_t$-BRST charge wrt $\hat{\J}$.
	Hence $\{\Phi_t:=\calG^{-1}\circ\calF_t\circ\calG\}_{t\in I}$ and $\{\Omega_t\}_{t\in I}$ provide the required Hamiltonian homotopy of geometric MC elements wrt $\hat{\J}$.
\end{proof}

The BFV-complex does also encode the (local) moduli space of coisotropic submanifolds under Jacobi equivalence.

\begin{definition}
	\label{def:Jac_homotopy_geometric_MC}
	A \emph{Jacobi homotopy of geometric MC elements} of $(\Gamma(\hat{L}),\{-,-\}_{\BFV})$ consists of
	\begin{itemize}
		\item a smooth path $\{\Omega_t\}_{t\in I}$ of geometric MC elements of $(\Gamma(\hat{L}),\{-,-\}_{\BFV})$, and
		\item a smooth path $\{\Phi_t\}_{t\in I}$ of Jacobi automorphisms of $(\hat{M},\hat{L},\hat{\J})$ (cf.~Definition~\ref{def:Jacobi_automorphism}),
	\end{itemize}
	such that they are related by the compatibility condition $\Phi_t^\ast\Omega_t=\Omega_0$.
	Such a Jacobi homotopy is said to \emph{interpolate} the geometric MC elements $\Omega_0$ and $\Omega_1$.
	If geometric MC elements $\Omega_0$ and $\Omega_1$ are interpolated by a Jacobi homotopy, then they are called \emph{Jacobi equivalent}, and we write $\Omega_0\sim_{\Jac}\Omega_1$.
	Indeed $\sim_{\Jac}$ is an equivalence relation on $\MC_{\textnormal{geom}}(\hat{L},\hat{\J})$ coarser than $\sim_{\Ham}$, and the equivalence class of $\Omega$ is denoted by $[\Omega]_{\Jac}$.
\end{definition}

We present now the analogous notion of Jacobi equivalence for coisotropic sections within a fat tubular neighborhood.

\begin{definition}
	\label{def:Jac_homotopy_coisotropic_sections}
	A \emph{Jacobi homotopy of coisotropic sections} of $(L,\J)$ consists of
	\begin{itemize}
		\item a smooth path $\{s_t\}_{t\in I}\subset C(L,\J)$ of coisotropic sections of $(L,\J)$, and
		\item a smooth path $\{(F_t,\underline{\smash{F_t}})\}_{t\in I}$ of Jacobi automorphisms of $(M,L,\J)$ (cf.~Definition~\ref{def:Jacobi_mfd_morphism}),
	\end{itemize}
	such that they are related by the compatibility condition $\im s_t=\underline{\smash{F_t}}(\im s_0)$.
	Such a Jacobi homotopy is said to \emph{interpolate} the coisotropic sections $s_0$ and $s_1$.
	If coisotropic sections $s_0$ and $s_1$ are interpolated by a Jacobi homotopy, then they are called \emph{Jacobi equivalent}, and we write $s_0\sim_{\Jac}s_1$.
	Indeed $\sim_{\Jac}$ is an equivalence relation on $C(L,\J)$ coarser than $\sim_{\Ham}$, and the equivalence class of $s$ is denoted by $[s]_{\Jac}$.
\end{definition}

The following theorem shows that, in the Jacobi setting, and a fortiori in the Poisson setting, the BFV-complex of a coisotropic submanifold $S$ encodes the local moduli space of $S$ under Jacobi equivalence.
\begin{theorem}
	\label{theor:Mod_Jac}
	A 1-1 correspondence between $\MC_{\textnormal{geom}}(\hat{L},\hat{\J})/{\sim_{\Jac}}$ and $C(L,\J)/{\sim_{\Jac}}$ is canonically defined by mapping $[\Omega]_{\Jac}$ to $[s_\Omega]_{\Jac}$.
\end{theorem}

\begin{proof}
	We have to prove that, for all $\Omega_0,\Omega_1\in\MC_{\textnormal{geom}}(\hat{L},\hat{\J})$, the following conditions are equivalent:
	\begin{itemize}
		\item[a)] $\Omega_0$ and $\Omega_1$ are Jacobi equivalent,
		\item[b)] $s_0:=s_{\Omega_0}$ and $s_1:=s_{\Omega_1}$ are Jacobi equivalent.
	\end{itemize}
	a) $\Longrightarrow$ b).
	Assume there is a Jacobi homotopy of geometric MC elements given by $\{\Omega_t\}_{t\in I}\subset\MC_{\textnormal{geom}}(\hat{L},\hat{\J})$ and $\{\Phi_t\}_{t\in I}\subset\Aut(\hat{M},\hat{L},\hat{\J})$.
	The latter can be canonically projected onto a Jacobi homotopy of coisotropic sections given by $\{s_t\}_{t\in I}\subset C(L,\J)$ and $\{(F_t,\underline{\smash{F_t}})\}_{t\in I}\subset\Aut(M,L,\J)$.
	Such projection is defined by setting
	\begin{equation*}
	s_t:=s_{\Omega_t},\qquad F_t^\ast:=({\operatorname{pr}^{(0,0)}}\circ{\Phi_t^\ast})|_{\Gamma(\hat{L})^{(0,0)}}.
	\end{equation*}
	
	b) $\Longrightarrow$ a).
	Assume there is a Jacobi homotopy of coisotropic sections of $(L,\J)$ given by $\{s_t\}_{t\in I}\subset C(L,\J)$ and $\{(F_t,\underline{\smash{F_t}})\}_{t\in I}\subset\Aut(M,L,\J)$.
	In view of Lemma~\ref{lem:coisotropic_moduli_space}, it is enough to show that the latter can be lifted to a Jacobi homotopy of geometric MC elements wrt $\hat{\J}$ intertwining an $s_0$-BRST charge and an $s_1$-BRST charge.
	
	So, fix $\Omega_0$: an $s_0$-BRST charge wrt $\hat{\J}$.
	Proposition~\ref{prop:main_result} allows us to lift $\{F_t\}_{t\in I}$ to a smooth path $\{\calF_t\}_{t\in I}$ of bi-degree $(0,0)$ graded automorphisms of the graded line bundle $\hat{L}\to\hat{M}$ such that
	\begin{itemize}
		\item $\hat{\J}_t:=(\calF_t)_\ast\hat{\J}$ is a lifting of $\J$ to a Jacobi structure on $\hat{L}\to\hat{M}$,
		\item $\Omega_t:=(\calF_t)_\ast\Omega_0$ is a $s_t$-BRST charge wrt $\hat{\J}_t$.
	\end{itemize}
	Hence, in view of Theorem~\ref{theor:uniqueness_BFV_brackets}, there is also a smooth path $\{\calG_t\}_{t\in I}$ of automorphisms of the graded line bundle $\hat{L}\to\hat{M}$, with $\calG_0=\id_{\hat{L}}$, such that $(\calG_t)_\ast\hat{\J}_t=\hat{\J}$, and $(\calG_t)_\ast\Omega_t$ is an $s_t$-BRST charge wrt $\hat{\J}$.
	Finally, $\{\Phi_t:=\calG_t\circ\calF_t\}_{t\in I}$ and $\{\Omega_t:=(\Phi_t)_\ast\Omega_0\}_{t\in I}$ provide the desired Jacobi homotopy of geometric MC elements wrt $\hat{\J}$.
\end{proof}

\section{An obstructed example in the contact setting}
\label{sec:obstructed_example_contact_BFV}

It has appeared in~\cite[Examples 3.5 and 3.8]{tortorella2016rigidity} a first example, in the contact setting, of a coisotropic submanifold whose coisotropic deformation problem is obstructed at the formal level, and so a fortiori at the smooth level.
Originally this obstructed example, derived in analytical terms, was considered to illustrate that a special subclass of coisotropic submanifolds (called the ``integral'' ones in~\cite{tortorella2016rigidity}) is not stable under small coisotropic deformations.
Our goal, in the current Section, is to give an interpretation of this obstructed example within the conceptual framework provided by the associated BFV-complex (see also Section~\ref{sec:obstructed_example_contact_L-infinity} for its re-interpretation in terms of the associated $L_\infty$-algebra).

Let us consider the vector bundle $\calE:=\bbT^5\times\bbR^2\overset{\pi}{\longrightarrow} S:=\bbT^5$, $(\phi_i,y_a)\longmapsto(\phi_i)$, where $\phi_1,\ldots,\phi_5$ are the standard angular coordinates on $\bbT^5$, and $y_1,y_2$ are the standard Euclidean coordinates on $\bbR^2$.
Now $\calE$ comes equipped with a coorientable contact structure by means of the global contact form $\theta_\calE\in\Omega^1(\calE)$ given by
\begin{equation*}
\theta_\calE:=y_1d\phi_1+y_2d\phi_2+\sin\phi_3d\phi_4+\cos\phi_3d\phi_5.
\end{equation*}
Let $\{-,-\}_J$ be the transitive Jacobi structure, on the line bundle $L:=\calE\times\bbR\to \calE$, which is determined by this coorientable contact structure on $\calE$.
Then it is straightforward to compute the expression of such Jacobi bi-derivation $J\in D^2L$, i.e.
\begin{equation}
\label{eq:obstructed_example_contact_BFV_J}
J=\frac{\partial}{\partial\phi_3}\wedge X+Y\wedge\left(y_1\frac{\partial}{\partial y_1}+y_2\frac{\partial}{\partial y_2}\right)-\frac{\partial}{\partial\phi_1}\wedge\frac{\partial}{\partial y_1}-\frac{\partial}{\partial\phi_2}\wedge\frac{\partial}{\partial y_2}-Y\wedge\id,
\end{equation}
where $X,Y\in\frakX(S)$ are defined by
\begin{equation*}
X:=\cos\phi_3\frac{\partial}{\partial \phi_4}-\sin\phi_3\frac{\partial}{\partial \phi_5},\qquad Y:=\sin\phi_3\frac{\partial}{\partial \phi_4}+\cos\phi_3\frac{\partial}{\partial \phi_5}.
\end{equation*}
Since $\{y_a,y_b\}_J=0$, for all $a,b\in\{1,2\}$, we immediately get that $S$ is a regular coisotropic submanifold of $(\calE,C_\calE)$.

Denote by $\eta^1,\eta^2$ the canonical global frame of $\calE=\bbT^5\times\bbR^2\overset{\pi}{\to} S=\bbT^5$.
Then $\xi^1:=\pi^\ast\eta^1,\eta^2:=\pi^\ast\xi^2$ is the canonical global frame of $E:=\pi^\ast\calE=\bbT^5\times\bbR^2\times\bbR^2\to\calE=\bbT^5\times\bbR^2$.
In the current situation, concerning the graded manifold $\hat{M}$ and the graded line bundle $\hat{L}\to\hat{M}$ of Section~\ref{sec:lifting_Jacobi_structures}, we find that
\begin{equation*}
C^\infty(\hat{M})
=C^\infty(\calE)\otimes_{C^\infty(S)}\Gamma(S^\bullet(\calE[-1]\oplus \calE^\ast[1])=C^\infty(\calE)\otimes_{\bbR}\wedge^\bullet(\bbR^2\oplus(\bbR^2)^\ast),
\end{equation*}
and $\Gamma(\hat{L})$ coincides with $C^\infty(\hat{M})$ seen as a module over itself.

Let us start constructing the BFV bracket.
We will lift the Jacobi structure $J$ on $L$ to a graded Jacobi structure $\hat{J}^\nabla$ on $\hat{L}\to\hat{M}$ as in the proof of Theorem~\ref{theor:existence_BFV_brackets}.
In order to apply such procedure, we still have to pick up a $DL$-connection $\nabla$ in $E\to \calE$.
Since $E\to\calE$ is the trivial bundle, there is a privileged choice for $\nabla$: the trivial flat $DL$-connection in $E\to\calE$, i.e., for all $i=1,\ldots,5$, and $a,b=1,2$,
\begin{equation}
\label{eq:obstructed_example_contact_BFV_connection}
\nabla_{\tfrac{\partial}{\partial\phi_i}}\xi^b=\nabla_{\tfrac{\partial}{\partial y_a}}\xi^b=\nabla_1\xi^b=0.
\end{equation}   
Since such $\nabla$ is flat, then, according to Proposition~\ref{prop:existence_BFV_brackets}, the corresponding $\hat{J}^\nabla$ reduces to
$\hat{J}^\nabla=G+i_\nabla J$, and the BFV bracket on $\Gamma(\hat{L})$ is $\{-,-\}_{BFV}=\{-,-\}_G+\{-,-\}_{i_\nabla J}$.
From here, keeping in mind equations~\eqref{eq:obstructed_example_contact_BFV_J} and~\eqref{eq:obstructed_example_contact_BFV_connection}, and Remarks~\ref{rem:local_expression_G} and~\ref{rem:local_expression_i_nabla}, we get the explicit expression
\begin{equation}
\begin{aligned}
\hat{\J}^\nabla&=\frac{\partial}{\partial\phi_3} X+Y\left(y_1\frac{\partial}{\partial y_1}+y_2\frac{\partial}{\partial y_2}\right)-\frac{\partial}{\partial\phi_1}\frac{\partial}{\partial y_1}-\frac{\partial}{\partial\phi_2}\frac{\partial}{\partial y_2}\\
&\phantom{=}-Y\left(\id-\xi^1\frac{\partial}{\partial\xi^1}-\xi^2\frac{\partial}{\partial\xi^2}\right)+\frac{\partial}{\partial\xi^1}\frac{\partial}{\partial\xi^\ast_1}+\frac{\partial}{\partial\xi^2}\frac{\partial}{\partial\xi^\ast_2}.
\end{aligned}
\end{equation}

With the help of the BFV bracket $\{-,-\}_{BFV}$ we can give a complete description of the space of coisotropic sections of $\calE\to S$.
Along the way, as a by-product, we will also construct the BFV-differential $d_{BFV}$.
\newline\noindent
Let $s$ be an arbitrary section of $\pi:\bbT^5\times\bbR^2\to\bbT^5$.
Hence $s=f\eta^1+g\eta^2$, and $\Omega_E[s]=(y_1-f)\xi^1+(y_2-f)\xi^2$, for some arbitrary $f,g\in C^\infty(\bbT^5)$, and plugging $\Omega_E[s]$ into the BFV bracket we get
\begin{equation}
\label{eq:obstructed_example_contact_BFV_coisotropic1}
\{\Omega_E[s],\Omega_E[s]\}_{BFV}=2\left(\frac{\partial f}{\partial \phi_3}Xg-\frac{\partial g}{\partial \phi_3}Xf+\frac{\partial f}{\partial\phi_2}-\frac{\partial g}{\partial\phi_1}+y_1Yg-y_2Yf\right)\xi^1\xi^2.
\end{equation}
In the special case of $s=0$, i.e.~$f=g=0$, the latter becomes $\{\Omega_E,\Omega_E\}_{BFV}=0$.
As a consequence, the procedure for the construction of the BRST-charge, described in the proof of Theorem~\ref{theor:existence_BRST-charge}, produces as output $\Omega_{BRST}=\Omega_E$.
Hence the corresponding BFV-differential $d_{BFV}:=\{\Omega_{BRST},-\}_{BFV}$ has the following expression
\begin{equation*}
d_{BFV}=y_1\frac{\partial}{\partial\xi_1^\ast}+y_2\frac{\partial}{\partial\xi_2^\ast}-\xi^1\frac{\partial}{\partial\phi_1}-\xi^2\frac{\partial}{\partial\phi_2}+(y_1\xi^1+y_2\xi^2)Y.
\end{equation*}
Coming back the general case of an arbitrary $s\in\Gamma(\pi)$, and applying $\wp[s]$ to~\eqref{eq:obstructed_example_contact_BFV_coisotropic1}, Lemma~\ref{lem:BRST_coisotropic} provides us with a complete description of coisotropic sections:
\begin{itemize}
	\item $s=f\eta^1+f\eta^2$ is a coisotropic section iff $f,g\in C^\infty(\bbT^5)$ satisfy the following non-linear first order pde
	\begin{equation}
	\label{eq:obstructed_example_contact_BFV_nonlinear_pde}
	\frac{\partial f}{\partial \phi_3}Xg-\frac{\partial g}{\partial \phi_3}Xf+\frac{\partial f}{\partial\phi_2}-\frac{\partial g}{\partial\phi_1}+fYg-gYf=0,
	\end{equation}
	which duly agrees with Equation~\eqref{eq:example_cosiotropic_deformation_space_Linfty} found through the $L_\infty$-algebra.
\end{itemize}

Our next aim is to illustrate how the BFV-complex of $S$ allows us to:
\begin{enumerate}
	\item characterize infinitesimal coisotropic deformations of $S$, and  
	\item point out a first obstruction to their prolongability to formal ones.
\end{enumerate}
As shown in the proof of Proposition~\ref{prop:BFV_homological_resolution}, every choice of the BFV-complex of $S$ canonically determines a deformed set of contraction data
\begin{equation*}
\begin{tikzpicture}[>= stealth,baseline=(current bounding box.center)]
\node (u) at (0,0) {$(\Gamma(\hat{L}),d_{\BFV})$};
\node (d) at (5,0) {$(\frakg(S),d_{N_\ell{}^\ast S,\ell})$};
\draw [transform canvas={yshift=-0.5ex},-cm to] (d) to node [below] {\footnotesize $\iota'$} (u);
\draw [transform canvas={yshift=0.5ex},cm to-] (d) to node [above] {\footnotesize $\wp[0]$} (u);
\draw [-cm to] (u.south west) .. controls +(210:1) and +(150:1) .. node[left=2pt] {\footnotesize $h[0]'$} (u.north west);
\end{tikzpicture},
\end{equation*}
with the immersion $\iota'$ being such that $\iota'\varpi$ agrees with $\iota\varpi:=\pi^\ast\varpi$, modulo $\ker\wp[0]$, for all $\varpi\in\frakg(S)$.
Hence, Corollary~\ref{cor:inf1} tells us that the infinitesimal coisotropic deformations of $s$ are those $s=f\eta^1+g\eta^2$ such that $\pi^\ast s$ agrees with a $1$-cocycle of $(\Gamma(\hat{L}),d_{BFV})$, modulo $\ker\wp[0]$.
Since, for arbitrary $f,g\in C^\infty(\bbT^5)$ and $h^1,h^2\in C^\infty(\bbT^5\times\bbR^2)$, we have
\begin{equation*}
d_{BFV}\!\!\left(f\xi^1+g\xi^2+(h^1\xi_1^\ast+h^2\xi^\ast_2)\xi^1\xi^2\right)\!=\!\left(\!\frac{\partial f}{\partial\phi_2}-\frac{\partial g}{\partial\phi_1}+y_1(h^1-Yg)+y_2(Yf+h^2)\!\right)\!\xi^1\xi^2,
\end{equation*}
we can state that:
\begin{itemize}
	\item[(1.a)] $s=f\eta^1+g\eta^2$ is an infinitesimal coisotropic deformation of $S$ iff $f,g\in C^\infty(\bbT^5)$ satisfy the following linear first order pde
	\begin{equation}
	\label{eq:obstructed_example_contact_BFV_infinitesimal_deformation}
	\frac{\partial f}{\partial\phi_2}-\frac{\partial g}{\partial\phi_1}=0,
	\end{equation}
	which agrees with Equation~\eqref{eq:obstructed_example_contact3} found through the $L_\infty$-algebra,
	\item[(1.b)] in view of Proposition~\ref{prop:BFVinfinitesimal_moduli_space}, the moduli space, under Hamiltonian equivalence, of infinitesimal coisotropic deformations identifies with $H^1(\Gamma(\hat{L}),d_{BFV})$ as follows
	\begin{align*}
	H^1(\frakg(S),d_{N_\ell{}^\ast S,\ell})&\hookrightarrow \hspace{-8pt} \rightarrow H^1(\Gamma(\hat{L}),d_{BFV}),\\
	[f\eta^1+g\eta^2]&\mapsto[f\xi^1+g\xi^2+((Yg)\xi^\ast_1-(Yf)\xi^\ast_2)\xi^1\xi^2].
	\end{align*}
\end{itemize}
Now, for an arbitrary infinitesimal coisotropic deformation $s=f\eta^1+g\eta^2$, we can compute, according with~\eqref{eq:BFV_kuranishi_map},
\begin{equation*}
\label{eq:obstructed_example_contact_BFV_kuranishi}
\operatorname{Kr}\left[f\xi^1+g\xi^2+\left(Y(g)\xi^\ast_1-Y(f)\xi^\ast_2\right)\xi^1\xi^2\right]=\left[2\left(\tfrac{\partial f}{\partial\phi_3}Xg-\tfrac{\partial g}{\partial\phi_3}Xf+fYg-gYf\right)\xi^1\xi^2\right].
\end{equation*}
For an arbitrary $\varkappa\in\Gamma(\hat{L})^1$, i.e.~$\varkappa=F_1\xi^1+F_2\xi^2+(G^1\xi_1^\ast+G^2\xi_2^\ast)\xi^1\xi^2$, with $F_1,F_2,G^1,G^2\in C^\infty(\bbT^5\times\bbR^2)$, we have
\begin{equation*}
d_{BFV}\varkappa=\left(\frac{\partial F_1}{\partial\phi_2}-\frac{\partial F_2}{\partial\phi_1}+y_1(G^1+YF_2)+y_2(G^2-YF_1)\right)\xi^1\xi^2.
\end{equation*}
Hence, in particular, $\iint_{\bbT^2}\left(\wp[0]d_{BFV}\varkappa\right)d\phi_1d\phi_2=0$.
From what above, in view of Proposition~\ref{prop:BFV_kuranishi}, it follows that
\begin{itemize}
	\item[(2)] if an infinitesimal coisotropic deformation $s=f\eta^1+g\eta^2$ can be prolonged to a formal one, then $f,g\in C^\infty(\bbT^5)$ have to satisfy the following constraint
	\begin{equation}
	\label{eq:obstructed_example_contact_BFV_obstruction}
	\iint\limits_{\bbT^2}\left(\frac{\partial f}{\partial\phi_3}Xg-\frac{\partial g}{\partial\phi_3}Xf+fYg-gYf\right)d\phi_1 d\phi_2=0,
	\end{equation}
	which duly agrees with Equation~\eqref{eq:obstructed_example_contact5} found through the $L_\infty$-algebra.
\end{itemize}

Now we can argue as in Section~\ref{sec:obstructed_example_contact_L-infinity} (cf.~also~\cite[Example 2]{tortorella2016rigidity}) to see that $s=(\cos\phi_4)\eta^1+(\sin\phi_4)\eta^2$ is an infinitesimal coisotropic deformation of $S$ but it is formally obstructed.



\appendix
	\chapter{Auxiliary material}
\label{app:auxiliary material}

In this Appendix we present some auxiliary material which is of central importance to the development of Chapters~\ref{chap:graded_Jacobi_manifolds} and~\ref{chap:BFV_complex}.
However, if included in the main body of the thesis, it would have represented too a large deviation from the principal line of reasoning.

In Section~\ref{sec:homological_perturbation_lemma} we give a short presentation of the Homological Perturbation Lemma (cf.~\cite{brown1967twisted}).
Adapting~\cite{Crainic_perturbation} we provide a version of this classic tool from homological perturbation theory which is well-suited to the aims of this thesis.
Indeed Homological Perturbation Lemma allows to prove central properties of the BFV-complex, namely: its being a cohomological resolution of the reduced Gerstenhaber--Jacobi algebra (Proposition~\ref{prop:BFV_homological_resolution} and Corollary~\ref{cor:BFV_homological_resolution}), and (jointly with Homotopy Transfer Theorem) its being $L_\infty$-quasi-isomorphic to the $L_\infty$-algebra of a coisotropic submanifold (Theorem~\ref{theor:L_infty_qi}).

In Section~\ref{app:SBSO} we present a version of another well-known technique from homological perturbation theory: the step-by-step obstruction method (cf.~\cite{Stasheff1997}).
Our version is optimized for the main steps towards the construction of the BFV-complex, namely the existence and the uniqueness of both the lifted Jacobi structure (Theorems~\ref{theor:existence_BFV_brackets} and~\ref{theor:uniqueness_BFV_brackets}) and the BRST-charge (Theorems~\ref{theor:existence_BRST-charge} and~\ref{theor:uniqueness_BRST-charge}).

In Section~\ref{app:technical_tools}, we state and prove two propositions which are the keystones in the proofs of some central theorems about the BFV-complex, namely its gauge invariance (Theorem~\ref{theor:gauge_invariance_BFV_complex}) and its encoding the moduli spaces of coisotropic sections under Hamiltonian and Jacobi equivalence (Theorems~\ref{theor:Mod_Ham} and~\ref{theor:Mod_Jac}).

\section{Homological Perturbation Lemma}
\label{sec:homological_perturbation_lemma}

The Homological Perturbation Lemma is a classical technique from homological perturbation theory which goes back to R.~Brown~\cite{brown1967twisted}.
At least for our aims it can be essentially seen as a tool which takes as input a small perturbation of a set of contraction data and produces as output a new ``deformed'' set of contraction data.
Our short presentation mainly follows that one given in~\cite{Crainic_perturbation}.
\begin{definition}
	\label{def:contraction_data}
	\emph{Contraction data} (between co-chain complexes) from $(\calK, \partial)$ to $(\underline{\smash{\calK}}, \underline{\smash{\partial}})$
	\begin{equation}
	\label{eq:contraction_data}
	\begin{tikzpicture}[>= stealth,baseline=(current bounding box.center)]
	\node (u) at (0,0) {$(\calK,\partial)$};
	\node (d) at (5,0) {$(\underline{\smash{\calK}},\underline{\smash{\partial}})$};
	\draw [transform canvas={yshift=-0.5ex},-cm to] (d) to node [below] {\footnotesize $j$} (u);
	\draw [transform canvas={yshift=+0.5ex},cm to-] (d) to node [above] {\footnotesize $q$} (u);
	\draw [-cm to] (u.south west) .. controls +(210:1) and +(150:1) .. node[left=2pt] {\footnotesize $h$} (u.north west);
	\end{tikzpicture}
	\end{equation}
	consists of the following:
	\begin{itemize}
		\item a surjective co-chain map $q:(\calK,\partial)\longrightarrow(\underline{\smash{\calK}},\underline{\smash{\partial}})$ that we simply call the \emph{projection},
		\item an injective co-chain map $j:(\underline{\smash{\calK}},\underline{\smash{\partial}})\longrightarrow(\calK,\partial)$, that we call the \emph{immersion}, such that $q\circ j=\id_{\underline{\smash{\calK}}}$,
		\item a \emph{homotopy} $h : (\calK, \partial) \to (\calK, \partial)$ between $j\circ q$ and $\id_{\calK}$.
	\end{itemize}
	Additionally, $q,j,h$ satisfy the following \emph{side conditions}:
	\begin{equation*}
	h^2=0,\qquad h\circ j=0,\qquad q\circ h=0.
	\end{equation*}
\end{definition}
From now on in this section we will assume to have fixed the contraction data~\eqref{eq:contraction_data}.

\begin{definition}
	A \emph{deformation} of contraction data~\eqref{eq:contraction_data} is a degree $1$ map $\delta:\calK\to\calK$ such that $\partial+\delta$ is again a coboundary operator, i.e.~$(\partial+\delta)^2=0$.
	A deformation $\delta$ of~\eqref{eq:contraction_data} is said to be \emph{small} if $\id_{\calK}-\delta h$ is invertible, or equivalently if  $\id_{\calK}-h\delta$ is invertible.
\end{definition}

\begin{remark}
	\label{rem:homological_perturbation_lemma}
	Let $\delta$ be a deformation of contraction data~\eqref{eq:contraction_data}.
	Assume that $\delta h$ is nilpotent, or equivalently that $h\delta$ is nilpotent.
	In this case $\delta$ is a small perturbation of~\eqref{eq:contraction_data}, with
	\begin{equation}
	\label{eq:rem:homological_perturbation_lemma}
	(\id_{\calK}-\delta h)^{-1}=\sum_{n=0}^\infty(\delta h)^n,\qquad (\id_{\calK}-h\delta)^{-1}=\sum_{n=0}^\infty(h\delta)^n.
	\end{equation}
\end{remark}

\begin{proposition}[Homological Perturbation Lemma]
	\label{prop:homological_perturbation_lemma}
	Let $\delta$ be a small perturbation of contraction data~\eqref{eq:contraction_data}.
	Then $\delta$ determines the new perturbed contraction data
	\begin{equation*}
	\begin{tikzpicture}[>= stealth,baseline=(current bounding box.center)]
	\node (u) at (0,0) {$(\calK,\partial+\delta)$};
	\node (d) at (5,0) {$(\underline{\smash{\calK}},\underline{\smash{\partial}}')$};
	\draw [transform canvas={yshift=-0.5ex},-cm to] (d) to node [below] {\footnotesize $j'$} (u);
	\draw [transform canvas={yshift=+0.5ex},cm to-] (d) to node [above] {\footnotesize $q'$} (u);
	\draw [-cm to] (u.south west) .. controls +(210:1) and +(150:1) .. node[left=2pt] {\footnotesize $h'$} (u.north west);
	\end{tikzpicture}
	\end{equation*}
	where $q', j',h'$, and $\underline{\smash{\partial}}'$ are given by the following formulas:
	\begin{equation}
	\label{eq:prop:homological_perturbation_lemma}
	\begin{gathered}
	q':=q(\id_{\calK}-\delta h)^{-1},\quad j':=(\id_{\calK}-h\delta)^{-1}j,\\
	h':=h(\id_{\calK}-\delta h)^{-1}\equiv(\id_{\calK}-h\delta)^{-1}h,\\ 
	\underline{\smash{\partial}}':=\underline{\smash{\partial}}+q(\id_{\calK}-\delta h)^{-1}\delta j\equiv\underline{\smash{\partial}}+q\delta(\id_{\calK}-h\delta )^{-1} j.
	\end{gathered}
	\end{equation}
\end{proposition}
For the proof of the Homological Perturbation Lemma we refer the reader to~\cite{Crainic_perturbation}.

\section{A step-by-step obstruction method}
\label{app:SBSO}

Our procedure to lift a Jacobi structure to a graded one, and more generally our construction of the BFV-complex of a coisotropic submanifold are entirely based on the step-by-step obstruction method of homological perturbation theory.
Indeed the central results in the BFV-construction, namely the existence and the uniquess of the lifted Jacobi structure (see Theorems~\ref{theor:existence_BFV_brackets} and~\ref{theor:uniqueness_BFV_brackets}) and the existence and uniquess of the BRST-charges (see Theorems~\ref{theor:existence_BRST-charge} and~\ref{theor:uniqueness_BRST-charge}), have been proved through a direct application of this method.
In this section we will provide a self-contained version of the method, which is well-suited for the our objectives.
In doing this, we will adapt and integrate~\cite[Section 4.1]{Stasheff1997}.

The setting is the following:
a set of contraction data (cf.~Definition~\ref{def:contraction_data})
\begin{equation}
\label{eq:homotopy_equivalence}
\begin{tikzpicture}[>= stealth,baseline=(current bounding box.center)]
\node (u) at (0,0) {$(X,d)$};
\node (d) at (5,0) {$(Y,0)$};
\draw [transform canvas={yshift=-0.5ex},-cm to] (d) to node [below] {\footnotesize $I$} (u);
\draw [transform canvas={yshift=+0.5ex},cm to-] (d) to node [above] {\footnotesize $P$} (u);
\draw [-cm to] (u.south west) .. controls +(210:1) and +(150:1) .. node[left=2pt] {\footnotesize $H$} (u.north west);
\end{tikzpicture},
\end{equation}
and a decreasing filtration of $X$ by graded subspaces $\{\calF_n\}_{n\in\bbZ}$ such that
\begin{equation}
\label{eq:SBSO_1}
\begin{gathered}
d\calF_n\subset\calF_{n-1},\quad H\calF_n\subset\calF_{n+1},\ \textnormal{for all}\ n\in\bbZ,\\
P\calF_{N+1} = 0,\ \textnormal{for some}\ N\geq-1.
\end{gathered}
\end{equation}
Additionally, it is important to assume that the filtration $\{\calF_n\}_{n\in\bbZ}$ is finite in each homogeneous component, i.e., for every $i\in\bbZ$, we have $X^i\cap\calF_n=0$ for all $n \gg 0$.
Finally suppose there is a degree $0$ graded Lie bracket $[-,-]$ on $X$, and a degree $1$ element $\bar{Q}\in X^1\cap\calF_{-1}$, such that
\begin{equation}
\label{eq:SBSO_2}
\begin{gathered}{}
[\calF_m,\calF_n] \subset\calF_{m+n},\ \textnormal{for all}\ m,n\in\bbZ,\\
[\bar{Q},\Omega]\equiv d\Omega\Mod{\calF_n},\ \textnormal{for all}\ n\in\bbZ,\ \textnormal{and}\ \Omega\in\calF_n.
\end{gathered}
\end{equation}

The question is whether or not $\bar{Q}$ can be deformed to a Maurer--Cartan (MC) element $Q$ of $(X,[-,-])$ such that $Q\equiv\bar{Q}\Mod\calF_{N+1}$.
The step-by-step obstruction method provides an answer.
Firstly, Proposition~\ref{prop:SBSO_existence} points out a necessary and sufficient condition for the existence of such MC element, and explicitly constructs one, in the affirmative case.
Secondly, Proposition~\ref{prop:SBSO_uniqueness} establishes uniqueness, up to isomorphisms.

\begin{proposition}[Existence]
	\label{prop:SBSO_existence}
	There exists a MC element $Q$ of $(X,[-,-])$, such that $Q\equiv\bar{Q}\Mod\calF_{N+1}$, iff the following condition holds:
	\begin{equation}
	\label{eq:prop:SBSO_existence}
	[\bar{Q},\bar{Q}]\in\calF_{N}\cap\ker P.
	\end{equation}
\end{proposition}

\begin{proof} \ 
	
	$(\Longrightarrow)$
	Let $Q$ be a MC element of $(X,[-,-])$ such that $Q\equiv\bar{Q}\Mod\calF_{N+1}$.
	Then we get immediately
	\begin{equation*}
	0=[Q,Q]=[\bar{Q},\bar{Q}]+2[\bar{Q},Q-\bar{Q}]+[Q-\bar{Q},Q-\bar{Q}]\equiv[\bar{Q},\bar{Q}]+2d(Q-\bar{Q})\Mod\calF_{N+1}.
	\end{equation*}
	Hence, from~\eqref{eq:SBSO_1} and $P\circ d=0$, it follows that $P([\bar{Q},\bar{Q}])=0$, and $[\bar{Q},\bar{Q}]\in\calF_N$.
	
	$(\Longleftarrow)$
	Assume that $[\bar{Q},\bar{Q}]\in\calF_N\cap\ker P$.
	The main idea of the proof is to construct the desired MC element $Q$ through a perturbative expansion
	\begin{equation*}
	Q=\bar{Q}+\sum_{n>N} Q_n,
	\end{equation*}
	where $\{Q_n\}_{n>N}$ is a (necessarily finite) sequence in $X^1$, with $Q_n\in\calF_n$, such that, for every $k\geq N$,
	\begin{equation}
	\label{eq:perturbative_expansion}
	Q(k):=\bar{Q}+\sum_{N<h\leq k}Q_h\Longrightarrow [Q(k),Q(k)]\in\calF_k.
	\end{equation}
	Clearly if sequence $\{Q_n\}_{n>N}$ exists, then the finite sum~\eqref{eq:perturbative_expansion} actually provides a MC element $Q$ of $(X,[-,-])$ such that $Q\equiv\bar{Q}\Mod\calF_{N+1}$.
	Now, we show how to set up a recursive procedure to construct the $Q_n$.
	
	{\sc Proof of the Main Idea.}
	Assume we constructed the sequence $\{Q_n\}_{n>N}$ up to the term of filtration degree $k$, for some $k\geq N$.
	Then the next term in the sequence can be obtained by setting $Q_{k+1}:=\frac 12H[Q(k),Q(k)]$.
	Indeed, from either the hypotheses on $\bar{Q}$ (if $k=N$), or the inductive hypotheses (if $k>N$), it follows, in any case, that $[Q(k),Q(k)]\in\ker P$.
	Moreover, from Jacobi identity, we get
	\begin{equation*}
	0=[Q(k),[Q(k),Q(k)]]\equiv[\bar{Q},[Q(k),Q(k)]]\Mod\calF_k\equiv d[Q(k),Q(k)]\Mod\calF_{k}.
	\end{equation*}
	Hence $[Q(k),Q(k)]$ is annihilated by $P$, and is $d$-closed up to terms of filtration degree $k$, so that
	\begin{align*}
	[Q(k)+Q_{k+1},Q(k)+Q_{k+1}]&=[Q(k),Q(k)]+2[Q(k),Q_{k+1}]+[Q_{k+1},Q_{k+1}]\\
	&\equiv(\id+d\circ H)[Q(k),Q(k)]\Mod\calF_{k+1}\\
	&\equiv(I\circ P-H\circ d)[Q(k),Q(k)]\Mod\calF_{k+1}\\
	&\equiv 0\Mod\calF_{k+1}.\qedhere
	\end{align*}
\end{proof}

\begin{proposition}[Uniqueness]
	\label{prop:SBSO_uniqueness}
	Let $Q^0,Q^1$ be MC elements of $(X,[-,-])$, such that $Q^i\equiv\bar{Q}\Mod\calF_{N+1}$, $i=0,1$.
	Then there exists an automorphism $\phi$ of $(X,[-,-])$ such that $\phi(Q^0)=Q^1$.
	Moreover $\phi$ can be chosen so that $\phi(\Omega)\equiv\Omega\Mod\calF_{n+N+2}$, for all $n\in\bbZ$, and $\Omega\in\calF_n$.
\end{proposition}

\begin{proof}
	The main idea of the proof is to check that, for every $n\geq N+1$, if $Q^0$ and $Q^1$ coincide up to terms of filtration degree $n$, i.e.
	\begin{equation}
	\label{eq:proof:prop:SBSO_uniqueness1}
	Q^1\equiv Q^0\Mod\calF_n,
	\end{equation}
	then there is $R\in\calF_{n+1}\subset\calF_1$ such that $Q^1$ and $(\exp R)Q^0:=\sum_{k=0}^\infty\frac{1}{k!}\operatorname{ad}_{R}^k Q^0$, with $\operatorname{ad}_R:=[R,-]$, coincide up to terms of filtration degree $n+1$, i.e.
	\begin{equation}
	\label{eq:proof:prop:SBSO_uniqueness2}
	(\exp R)(Q^0)\equiv Q^1\Mod\calF_{n+1}.
	\end{equation}
	The statement of the Proposition will then follow because the decreasing filtration is finite, and $Q^0\equiv Q^1\Mod\calF_{N+1}$, by hypothesis.
	
	{\sc Proof of the Main Idea.}
	If~\eqref{eq:proof:prop:SBSO_uniqueness1} holds, then~\eqref{eq:proof:prop:SBSO_uniqueness2} is satisfied by setting $R:=H(Q^1-Q^0)$.
	Indeed, from the Maurer--Cartan equation for $Q^0$ and $Q^1$, it follows that 
	\begin{equation*}
	0=[Q^1,Q^1]=2[Q^0,Q^1-Q^0]+[Q^1-Q^0,Q^1-Q^0]\equiv 2d(Q^1-Q^0)\Mod\calF_n,	\end{equation*}
	i.e.~$Q^1-Q^0$ is $d$-closed up to terms of filtration degree $n$.
	Hence we also get
	\begin{equation*}
	Q^1-Q^0=(I\circ P-d\circ H-H\circ d)(Q^1-Q^0)\equiv[R,Q^0]\Mod\calF_{n+1},
	\end{equation*}
	and so $(\exp R)Q^0\equiv Q^1\Mod\calF_{n+1}$, as needed.
\end{proof}

By the same argument used in the proof of Proposition~\ref{prop:SBSO_uniqueness}, we get the following.

\begin{corollary}
	\label{cor:SBSO_uniqueness}	
	Let $N\geq 0$, and let $Q$ be a MC element of $(X,[-,-])$, with $Q\equiv\bar{Q}\Mod(\calF_{N}\cap\ker P)$.
	Then there is an automorphism $\phi$ of $(X,[-,-])$ such that $\phi(Q)\equiv\bar{Q}\Mod\calF_{N+1}$.
	Additionally $\phi$ can be chosen so that $\phi(\Omega)\equiv\Omega\Mod\calF_{n+N+1}$, for all $n\in\bbZ$, and $\Omega\in\calF_n$.
\end{corollary}

\begin{remark}
	Notice that the \emph{side conditions} satisfied by the contraction data, do not play any r\^ole in the proofs of Propositions~\ref{prop:SBSO_existence}, \ref{prop:SBSO_uniqueness} and Corollary~\ref{cor:SBSO_uniqueness}, and they could be actually relaxed.
	Actually, there is only one place in this work where the side conditions are truly relevant.
	Namely, we need the homotopy equivalence~\eqref{eq:2contraction_data_1} to be a set of true contraction data when proving Theorem~\ref{theor:L_infty_qi}.
	Indeed, the side conditions are necessary to implement the \emph{homotopy transfer} that generates the higher brackets in the $L_\infty$-algebra from the BFV-complex.
\end{remark}

\section{Some auxiliary technical results}
\label{app:technical_tools}

The aim of this section is to state and prove Propositions~\ref{prop:main_result} and~\ref{prop:main_result_bis}, which represent the technical tools at the basis of most of the main results in Chapter~\ref{chap:BFV_complex}, namely: the gauge independence of the BFV complex (cf.~Theorem~\ref{theor:gauge_invariance_BFV_complex}), and the fact that the BFV complex encodes the local moduli spaces of coisotropic sections under Hamiltonian and Jacobi equivalence (cf.~Theorems~\ref{theor:Mod_Ham} and~\ref{theor:Mod_Jac}).
Actually these theorems have been proven as straightforward applications of Propositions~\ref{prop:main_result} and~\ref{prop:main_result_bis}.

In this section we will work within the local model established in Section~\ref{sec:BRST-charges}.
Let us start with two preliminary lemmas.
\begin{lemma}
	\label{lem:technical_lemma}
	Let $s$ be a section of the vector bundle $\pi:NS\to S$, and let $\{e_t\}_{t\in I}$ be a smooth path of sections of the pull-back vector bundle $E\to NS$.
	Suppose that
	\begin{enumerate}
		\item\label{enum:lem:technical_lemma1}
		$e_0=\Omega_E[s]$,
		\item\label{enum:lem:technical_lemma2}
		$\textnormal{``zero locus of $e_t$''}=\im(s)$,
		\item\label{enum:lem:technical_lemma3}
		$\left.e_t\right|_{E_x}$ intersects transversally the restriction to $E_x$ of the zero section of $E\to NS$, for all $x\in S$.
	\end{enumerate}
	Then there exists a smooth path $\{A_t\}_{t\in I}\subset\Gamma(\operatorname{GL}(E))$, with $A_0=\id_E$, such that
	\begin{equation}
	\label{eq:lem:technical_lemma}
	e_t=A_te_0,
	\end{equation}
	or equivalently there is a smooth path $\{a_t\}_{t\in I}\subset\End(E)$ such that
	\begin{equation}
	\label{eq:lem:technical_lemma_bis}
	\frac{d}{dt}e_t=a_te_t.
	\end{equation}
\end{lemma}

\begin{proof}
	The proposition has a local character, and it is enough to work in a neighborhood of an arbitrary point $x\in NS$.
	We distinguish two cases: $x\notin\im(s)$ and $x\in\im(s)$.
	
	\textsc{First Case.}
	Set $N':=NS\smallsetminus\im(s)$, and equip the vector bundle $E':=E|_{N'}\to N'$ with a Riemannian metric.
	Denote by $F_t\to N'$ the rank-$1$ vector subbundle of $E'\to N'$ generated by $e_t|_{N'}$.
	A smooth path $\{a_t\}_{t\in I}\subset\End(E)$ satisfying~\eqref{eq:lem:technical_lemma_bis} can be obtained by the following composition of vector bundle morphisms: $E'\overset{P_t}{\longrightarrow}F_t\overset{I_t}{\longrightarrow}E'$, where $P_t:E'\to F_t$ is the orthogonal projection, and $I_t:F_t\to E'$ is given by $I_t(e_t|_{N'})=\left(\frac{d}{dt}e_t\right)|_{N'}$.
	
	\textsc{Second Case.}
	Since $e_t\in\Gamma(\hat{L})^{(1,0)}$, and $d[s]\in\Diff(\hat{L})^{(0,1)}$, it follows, by bi-degree reasons, that $d[s](e_t)=0$.
	Even more, $e_t$ is actually a co-boundary of $(\Gamma(\hat{L}),d[s])$.
	Indeed hypothesis~\eqref{enum:lem:technical_lemma2} guarantees that $\wp[s]e_t=0$.
	Hence, setting $A_t:=-h[s]e_t$, in view of Proposition~\ref{prop:2contraction_data}, we get
	\begin{equation*}
	e_t=d[s]A_t=\{\Omega_E[s],A_t\}_G=A_t\Omega_E[s]=A_te_0,
	\end{equation*}
	where, in the last step, we used hypothesis~\eqref{enum:lem:technical_lemma1}.
	Finally, a simple computation in local coordinates shows that hypothesis~\eqref{enum:lem:technical_lemma3} is equivalent to the fiberwise invertibility of $A_t:=-h[s]e_t$ on $\im(s)$, and so also on some open neighborhood of $\im(s)$ in $NS$.
\end{proof}

\begin{lemma}
	\label{lem:integrating_graded_do}
	Fix a smooth path $\{\square_t\}_{t\in I}\subset\Diff(\hat{L})^0$, with $I:=[0,1]$.
	Denote by $\{\tilde\square_t\}_{t\in I}\subset\Diff(\hat{L})^{(0,0)}$, $\{\slashed\square_t\}_{t\in I}\subset\Diff(L)$ and $\{X_t\}_{t\in I}\subset\frakX(M)$ the smooth paths defined by setting $\tilde\square_t:=\operatorname{pr}^{(0,0)}\square_t$, $\slashed\square_t:=\left.\square_t^0\right|_{\Gamma(L)}$, and $X_t:=\sigma_{\slashed\square_t}$ respectively.
	The following conditions are equivalent:
	\begin{enumerate}
		\item\label{enum:integrating_graded_do1}
		$\{\square_t\}_{t\in I}$ integrates to a smooth path $\{\Phi_t\}_{t\in I}$ of automorphisms of $\hat{L} \to \hat{M}$;
		\item\label{enum:integrating_graded_do2}
		$\{\tilde\square_t\}_{t\in I}$ integrates to a smooth path $\{\tilde\Phi_t\}_{t\in I}$ of bi-degree $(0,0)$ automorphisms of $\hat{L}\to\hat{M}$;  
		\item\label{enum:integrating_graded_do3}
		$\{\slashed\square_t\}_{t\in I}$ integrates to a smooth path $\{(F_t,\underline{\smash{F_t}})\}_{t\in I}$ of automorphisms of $L \to M$;
		\item\label{enum:integrating_graded_do4}
		$\{X_t\}_{t\in I}$ integrates to a smooth path $\{\underline{\smash{F_t}}\}_{t\in I}$ of diffeomorphisms of $M$.
	\end{enumerate}
	Moreover, if the equivalent conditions (1)-(4) are satisfied, then the following relations hold:
	\begin{gather*}
	\Phi_t^\ast\lambda=\tilde\Phi_t^\ast\lambda\Mod\bigoplus_{k\geq 1}\Gamma(\hat{L})^{(p+k,q+k)},\ \textnormal{for all}\ (p,q)\in\bbN_0^2,\ \lambda\in\Gamma(\hat{L})^{(p,q)},\\
	\tilde\Phi^\ast\lambda=F_t^\ast\lambda,\ \textnormal{for all}\ \lambda\in\Gamma(L).	
	\end{gather*}
\end{lemma}

\begin{proof}
	It is straightforward.
\end{proof}

Let us fix the setting for Proposition~\ref{prop:main_result}.
Assume we have the following smooth paths
\begin{itemize}
	\item $\{\J_t\}_{t\in I}\subset\Diff^2(L[1])$ such that $\J_t$ is a Jacobi structure on $L\to M$,
	\item $\{s_t\}_{t\in I}\subset\Gamma(\pi)$, such that $\im s_t$ is a coisotropic submanifold of $(M,L,\J_t)$,
	\item A smooth path $\{(F_t,\underline F_t)\}_{t\in I}$ of automorphisms of $L \to M$, with $\underline{\smash{F}}_0=\id_{NS}$ and $F_0=\id_L$, such that $\im s_t=\underline F_t(\im s_0)$ and $\J_t=F_t^\ast\J_0$.
\end{itemize}
Fix moreover the following objects:
\begin{itemize}
	\item $\hat{\J}_0$, a lifting of $\J_0$ to a Jacobi structure on $\hat{L}\to\hat{M}$,
	\item $\Omega_0$, an $s_0$-BRST charge wrt $\hat{\J}_0$.
\end{itemize}
Our aim is to find a lifting of $\{F_t\}_{t\in I}$ to a suitable smooth path $\{\calF_t\}_{t\in I}$ of bi-degree $(0,0)$ automorphisms of $\hat{L}\to\hat{M}$, with $\calF_0=\id_{\hat{L}}$.
This is accomplished by the following.

\begin{proposition}
	\label{prop:main_result}
	There exists a smooth path $\{\calF_t\}_{t\in I}$ of bi-degree $(0,0)$ automorphisms of $\hat{L}\to\hat{M}$, with $\calF_0=\id_{\hat{L}}$, such that
	\begin{enumerate}
		\item\label{enum:main_result1}
		$\{\calF_t\}_{t\in I}$ is a lifting of $\{F_t\}_{t\in I}$, i.e.~$\calF_t|_L=F_t$,
		\item\label{enum:main_result2}
		$\hat\J_t:=(\calF_t)_\ast\hat\J_0$ is a lifting of $\J_t$ to $\hat{L}\to\hat{M}$,
		\item\label{enum:main_result3}
		$\Omega_t:=(\calF_t)_\ast\Omega_0\in\MC(\Gamma(\hat{L}),\{-,-\}_{\J_t})$ is an $s_t$-BRST charge wrt $\hat{\J}_t$.
	\end{enumerate}	
\end{proposition}

\begin{proof}
	We will show explicitly how to construct a smooth path $\{\square_t\}_{t\in I}\subset\Diff(\hat{L})^{(0,0)}$ integrating to a smooth path $\{\calF_t\}_{t\in\ I}\subset$, so that
	\begin{equation*}
	\calF_0=\id_{\hat L},\quad\frac{d}{dt}\calF_t^\ast=\square_t\circ\calF_t^\ast,
	\end{equation*}
	and moreover the latter satisfies conditions~\eqref{enum:main_result1}--\eqref{enum:main_result3} in the statement.
	
	Fix an arbitrary $\Diff(\ell)$-connection $\nabla$ in $\pi:NS\to S$.
	By pull-back along $\pi$, we also get a $\Diff(L)$-connection in $E\to NS$, denoted by $\nabla$ again.
	Arguing as in Sections~\ref{subsec:first_relevant_contraction_data} and~\ref{subsec:psi_nabla}, the latter $\nabla$ determines a degree $0$ graded $C^\infty(\hat{M})$-module isomorphism, of bi-degree $(0,0)$,
	\begin{equation*}
	\Diff(\hat{L})\simeq C^\infty(\hat{M})\underset{{\scriptscriptstyle C^\infty(M)}}{\otimes}\big(\underbrace{\Gamma((E_L)^\ast)}_{(-1,0)}[1]\oplus\underbrace{\Gamma(E)}_{(0,-1)}[-1]\oplus\underbrace{\Diff(L)}_{(0,0)}\big)
	\end{equation*}
	Focussing on the ghost/anti-ghost bi-degree $(0,0)$ component, we get, in particular, an isomorphism of $C^\infty(M)$-modules
	\begin{equation*}
	\Diff(\hat{L})^{(0,0)}\simeq\End(E)\oplus\End(E^\ast)\oplus\Diff(L).
	\end{equation*}
	
	Consequently, for any path $\{\square_t\}_{t\in I}\subset\Diff(\hat{L})^{(0,0)}$ there exist $\{a_t\}_{t\in I}\subset\End(E)$, $\{b_t\}_{t\in I}\subset\End(E^\ast)$, and $\{\slashed\square_t\}_{t\in I}\subset\Diff(L)$ uniquely determined by
	\begin{equation*}
	\square_t=\bbD_{a_t}+\bbD_{b_t}+\nabla_{\slashed\square_t}.
	\end{equation*}
	Here we interpret the endomorphisms $a_t,b_t$ as vertical vector fields on $\hat{M}$ via~\eqref{eq:A-derivations1} (with $F = (E_L)^\ast [1] \oplus E[-1]$) and use the connection $\bbD:\VDer(\hat{M})\to D(\hat{L})$ in Remark~\ref{rem:A-derivations} (with $\scrM = \hat{M}$ and $\scrL = \hat{L}$).
	For every $Z\in\VDer(\hat{M})$, the symbol of $\bbD_Z$ is $Z$ and vanishes on $C^\infty(M)$.
	Hence Lemma~\ref{lem:integrating_graded_do} guarantees that the condition on $\{\square_t\}_{t\in I}\subset\Diff(\hat{L})^{(0,0)}$ to integrate to a smooth path $\{\calF_t\}_{t\in I}$ of bi-degree $(0,0)$ automorphisms of $\hat{L}\to\hat{M}$, is equivalent to both the following
	\begin{itemize}
		\item $\{\nabla_{\slashed\square_t}\}_{t\in I}\subset\Diff(\hat{L})^{(0,0)}$, and integrates to a smoooth path $\{\Phi_t\}_{t\in I}$ of bi-degree $(0,0)$ automorphisms of $\hat{L}\to\hat{M}$,
		\item $\{\slashed\square_t\}_{t\in I}\subset\Diff(L)$ integrates to a smooth path $\{\phi_t\}_{t\in I}$ of automorphisms of $L \to M$.
	\end{itemize}
	Now, it follows that $\{\Phi_t^{-1}(\bbD_{a_t}+\bbD_{b_t})\}_{t\in I}=\{\bbD_{{\tilde a}{}_t}+\bbD_{\tilde b_t}\}_{t\in I}$, for some $\{\tilde a_t\}_{t\in I}\subset\End(E)$, and $\{\tilde b_t\}_{t\in I}\subset\End(E^\ast)$.
	Therefore Lemma~\ref{lem:integrating_graded_do} again implies that $\{\Phi_t^{-1}(\bbD_{a_t}+\bbD_{b_t})\}_{t\in I}$ integrates to $\Psi_t:= S_{{\scriptscriptstyle C^\infty(M)}}((A_t\otimes\id_{L^\ast})\oplus B_t)\otimes\id_L$, for some smooth paths $\{A_t\}_{t\in I}\subset\Gamma(\operatorname{GL}(E))$ and $\{B_t\}_{t\in I}\subset\Gamma(\operatorname{GL}(E^\ast))$, with $A_0=\id_E$ and $B_0=\id_{E^\ast}$.
	So we get the following decomposition
	\begin{equation*}
	\calF_t=\Phi_t\circ\Psi_t.
	\end{equation*}
	
	The construction of $\{\Phi_t\}_{t\in I}$ and $\{\Psi_t\}_{t\in I}$ guarantees that
	\begin{equation*}
	\Phi_t|_L=\phi_t,\qquad \Phi_t G=G,\qquad
	\Psi_t|_L=\id_L,\qquad \Psi_t|_{\hat{L}^{(1,0)}}=A_t,\qquad \Psi_t|_{\hat{L}^{(0,1)}}=B_t\otimes\id_L.
	\end{equation*}
	Now condition~\eqref{enum:main_result1} can be equivalently rewritten as
	\begin{equation}
	\label{eq:proof:main_resultI}
	\phi_t=F_t,
	\end{equation}
	which completely determines $\Phi_t$, hence $\slashed\square_t$.
	Condition~\eqref{enum:main_result2} splits into two conditions: $\phi_t\J_0=\J_t\equiv F_t\J_0$, and $\Psi_tG=G$.
	Therefore it reduces to $B_t^\ast A_t=\id_E$, for all $t\in I$, and it can be equivalently rewritten as
	\begin{equation*}
	\tilde b_t^\ast+\tilde a_t=0.
	\end{equation*}
	As a consequence $\bbD_{\tilde a_t}+\bbD_{\tilde b_t}=\{\tilde a_t,-\}_G$, and $\{\Psi_t\}_{t\in I}$ is obtained by integration of
	\begin{equation}
	\label{eq:proof:main_resultII}
	\Psi_0=\id_{\hat{L}},\quad \frac{d}{dt}\Psi_t=\{\tilde a_t,-\}_G\circ\Psi_t.
	\end{equation}
	The latter does not yet tell us anything about $\tilde a_t$.
	However, condition~\eqref{enum:main_result3} becomes $A_t\Omega_E[s_0]=\Phi_t^\ast(\Omega_E[s_t])$, or equivalently
	\begin{equation}
	\label{eq:proof:main_resultIII}
	\frac{d}{dt}e_t=\tilde a_t(e_t),
	\end{equation}
	where $e_t:=\Phi_t^\ast(\Omega_E[s_t])$, for all $t\in I$.
	From the choice of $\nabla$ and the construction of $\{\Phi_t\}_{t\in I}$, the smooth path $\{e_t\}_{t\in I}\subset\Gamma(\hat{L})^{(1,0)}=\Gamma(E)$ satisfies the hypotheses of Lemma~\ref{lem:technical_lemma}, so equation~\eqref{eq:proof:main_resultIII} admits a (non-unique) solution $\{\tilde a_t\}_{t\in I}\subset\Gamma(\hat{L})^{(1,0)}$.
	This concludes the proof.
\end{proof}

Now, we fix the setting for Proposition~\ref{prop:main_result_bis}.
Fix a Jacobi structure $\J$ on $L\to M$.
Suppose we have the following:
\begin{itemize}
	\item a smooth path $\{s_t\}_{t\in I}\subset C(L,\J)$ of coisotropic sections of $(L,\J)$, and
	\item a smooth path $\{(F_t,\underline F_t)\}_{t\in I}$ of automorphisms of $(M,L,\J)$, with $F_0=\id_L$, integrating $\{\lambda_t,-\}_{\J}$, for some smooth path $\{\lambda_t\}_{t\in I}\subset\Gamma(L)$, and $\im s_t=\underline F_t(\im s_0)$.
\end{itemize}
Choose a $\Diff(\ell)$-connection $\nabla$ in $\pi:NS\to S$, and pull it back along $\pi$ to get a $\Diff(L)$-connection in $L\to M$ which will be again denoted by $\nabla$.
Fix moreover $\Omega_0$, an $s_0$-BRST charge wrt $\hat{\J}^\nabla$.

Our aim is to find a lifting of $\{F_t\}_{t\in I}\subset\Aut(M,L,\J)$ to a suitable smooth path $\{\calF_t\}_{t\in I}\subset\Aut(\hat{M},\hat{L},\hat{\J}^\nabla)$, with $\calF_0=\id_{\hat{L}}$.
This aim is accomplished by the following

\begin{proposition}
	\label{prop:main_result_bis}
	There exist a smooth path $\{\hat{\lambda}_t\}_{t\in I}\subset\Gamma(\hat{L})^0$, and a  smooth path $\{\calF_t\}_{t\in I}\subset\Aut(\hat{M},\hat{L},\hat{\J}^\nabla)$, which integrates $\{\hat{\lambda}_t,-\}_{\hat{\J}^\nabla}$, such that
	\begin{enumerate}
		\item\label{enum:main_result_bis1}
		$\{\calF_t\}_{t\in I}\subset\Aut(\hat{M},\hat{L},\hat{\J}^\nabla)$ is a lifting of $\{F_t\}_{t\in I}\subset\Aut(M,L,\J)$, i.e.~$\operatorname{pr}^{(0,0)}\circ\calF_t|_L=F_t$,
		\item\label{enum:main_result_bis2}
		$\Omega_t:=(\calF_t)_\ast\Omega_0\in\MC(\Gamma(\hat{L}),\{-,-\}_{\hat{\J}^\nabla})$ is an $s_t$-BRST charge wrt $\hat{\J}^\nabla$.
	\end{enumerate}	
\end{proposition}

\begin{proof}
	Setting $\slashed\square_t=\{\lambda_t,-\}_{\J}$, from the local coordinate expression of $i_\nabla$ in Remark~\ref{rem:local_expression_i_nabla}, it is straightforward to get
	\begin{equation}
	\label{eq:main_result_bis1}
	\nabla_{\slashed\square_t}=\{\lambda_t,-\}_{i_\nabla\J}.
	\end{equation}
	The scheme used in the proof of Proposition~\ref{prop:main_result} applies as well in the current special situation, where $\hat{\J}^0=\hat{\J}^\nabla$, and $\J_t=\J$, and guarantees the existence of a smooth path $\{\tilde a_t\}_{t\in I}\subset\Gamma(\hat{L})^{(1,1)}=\Gamma(\End(E))$ and a smooth path $\{\tilde\calF_t\}_{t\in I}$ of bi-degree $(0,0)$ automorphisms of $\hat{L} \to \hat{M}$, with
	\begin{equation}
	\label{eq:main_result_bis2}
	\tilde\calF_0=\id_{\hat{L}},\quad\frac{d}{dt}\tilde\calF_t^\ast=(\{\tilde a_t,-\}_G+\nabla_{\slashed\square_t})\circ\tilde\calF_t^\ast,
	\end{equation}
	such that, in particular, $\tilde\calF_t$ is a lifting of $F_t$, and $pr^{(1,0)}(\tilde\calF_t)_\ast\Omega_0=\Omega_E[s_t]$.
	Define a smooth path $\{\hat{\lambda}_t\}_{t\in I}\subset\Gamma(\hat{L})^0$ by setting $\hat{\lambda}_t:=\lambda_t+\tilde a_t$.
	Because of Lemma~\ref{lem:integrating_graded_do}, from~\eqref{eq:main_result_bis1} and~\eqref{eq:main_result_bis2} it follows that $\{\hat{\lambda}_t,-\}_{\hat{\J}^\nabla}$ integrates to a smooth path $\{\calF_t\}_{t\in I}$, with $\calF_0=\id_{\hat{L}}$, satifying conditions~\eqref{enum:main_result_bis1} and ~\eqref{enum:main_result_bis2}.
\end{proof}


\backmatter



\end{document}